\documentclass{amsart}
\usepackage{hyperref}
\usepackage{mathabx}
\usepackage{amsmath, amsthm, amssymb}
\usepackage{mdwlist} 
\usepackage{graphicx}
\usepackage{ytableau}
\usepackage[OT2,T1]{fontenc}
\usepackage[utf8]{inputenc}
\usepackage[russian,english]{babel}

\usepackage[author-year]{amsrefs}
\usepackage{enumerate}
\usepackage[nomain,symbols]{glossaries}
\usepackage[linesnumbered,boxruled]{algorithm2e}

\usepackage[mathscr]{euscript}
\usepackage{mathrsfs}

\usepackage{pictexwd,dcpic}

\newtheorem{theorem}{Theorem}[subsection]
\newtheorem{conditionnumbertheorem}[theorem]{Condition Number Theorem}
\newtheorem{corollary}[theorem]{Corollary}
\newtheorem{proposition}[theorem]{Proposition}
\newtheorem{lemma}[theorem]{Lemma}

\theoremstyle{definition}
\newtheorem{definition}[theorem]{Definition}
\newtheorem*{convention}{Important remark on notations}
\newtheorem{problem}{Problem}
\newtheorem{newproblem}{Problem}

\theoremstyle{remark}
\newtheorem{remark}[theorem]{Remark}
\newtheorem{example}[theorem]{Example}

\definecolor{absolutezero}{rgb}{0.0,0.28,0.73}
\definecolor{darkgreen}{rgb}{0.0,0.38,0.0}

\newcommand{\changed}[1]{#1}
\newcommand{\secrev}[1]{#1}
\newcommand{\thirdrev}[1]{#1}
\newcommand{\fourthrev}[1]{#1}
\newcommand{\aboutsampling}[2]{#1} 

\renewcommand{\Re}{\mathrm{Re}}
\renewcommand{\Im}{\mathrm{Im}}

\newcommand{\codim}{\mathrm{codim}}

\DeclareMathOperator*{\bigE}{\mathbb E}
\DeclareMathOperator*{\bigProb}{\mathrm Prob}

\newcommand{\ghostrule}[1]{\rule{0cm}{#1}}
\newcommand{\expected}[2]{
\bigE_{\substack{#1}}\left({#2}\right)}
\newcommand{\probability}[2]{
\bigProb_{\substack{#1}}\left[{#2}\right]}
\newcommand{\conditional}{ \ \resizebox{!}{3ex}{$|$}\ }

\newcommand{\condexpected}[4]{
	\bigE_{\substack{#1}}\left( \ghostrule{5ex}{#2} \conditional {#3} \right.
	\\
	#4 \ghostrule{5ex} \right)}

\newcommand{\conv}[1]{\mathrm{Conv}(#1)}
\newcommand{\diam}[1]{\mathrm{diam}(#1)}

\newcommand{\Li}{\mathrm{Li}_2}
\newcommand{\vol}{\mathrm{Vol}}
\newcommand{\dd}{\,\mathrm{d}}

\newcommand{\binomial}[2]{\ensuremath{\left( \begin{matrix}#1 \\ #2 \end{matrix} \right)}}

\newcommand{\partialat}[2]{\ensuremath{
	\frac{\partial}{\partial #1}{\raisebox{-0.5ex}{$|_{\raisebox{-0.5ex}{$\scriptstyle #1 = #2$}}$}}}}
\newcommand{\partialhigh}[3]{\ensuremath{
	\frac{\partial^{#1}}{\partial #2^{#1}}{\raisebox{-0.5ex}{$|_{\raisebox{-0.5ex}{$\scriptstyle #2 = #3$}}$}}}}

\newcommand{\defun}[5]{\ensuremath{\begin{array}{lrcl}
#1:&#2 & \longrightarrow & #3\\&#4 & \longmapsto & #5\end{array}}}

\newcommand{\defeq}{\stackrel{\mathrm{def}}{=}}
\newcommand{\diag}[1]{\mathrm{diag}\left(#1\right)}

\makenoidxglossaries
\newglossaryentry{PA}{name={\ensuremath{\mathscr P_A}},
                      description={Space of Laurent polynomials with support $A \subseteq \mathbb Z^n$}}
\newglossaryentry{Ais}{name=\ensuremath{A_i},
                       description={Support for the $i$-th equation}}
\newglossaryentry{Si}{name={\ensuremath{S_i}}, description={Input size: $2 \le S_i=\#A_i < \infty$ and $S=\sum_{i=1}^n S_i \ge 2n \ge 4$}}
\newglossaryentry{VA}{name={\ensuremath{V_A}},
                      description={Veronese embedding for support $A$}}
\newglossaryentry{Via}{name={\ensuremath{V_{i \mathbf a}}},
		      description={$\mathbf a$-th coordinate of $V_{A_i}$}}
\newglossaryentry{rhoia}{name={\ensuremath{\rho_{i \mathbf a}}},
		      description={Constant coefficient \changed{in $V_{i\mathbf a}(\mathbf z)=\rho_{i\mathbf a}e^{\mathbf a \mathbf z}$}}}
\newglossaryentry{FA}{name={\ensuremath{\mathscr F_A}},
                      description={Space of exponential sums with support $A$}}
\newglossaryentry{MM}{name={\ensuremath{\secrev{\mathscr M}}},
                      description={Domain of the main chart for the toric variety}}
\newglossaryentry{Lambda}{name={\ensuremath{\Lambda}},
		      description={Integer lattice spanned by the union of the sets $A_i - A_i$}}
\newglossaryentry{quotient}{name={\ensuremath{[\cdot]}},
                      description={Natural quotient or projection, e.g. in multi-projective space}}
\newglossaryentry{VV}{name={\ensuremath{\mathcal V}},
		      description={Toric variety associated to $(A_1, \dots, A_n)$}}
\newglossaryentry{AA}{name={\ensuremath{\mathcal A_i}},
		      description={Convex hull of $A_i$}}
\newglossaryentry{mi}{name={\ensuremath{\mathbf m_i}},
                      description={The momentum map for the $i$-th support}}
\newglossaryentry{deltaix}{name={\ensuremath{\delta_i(\mathbf x)}},
		      description={\secrev{Radius of $i$-th support around $m_i(\mathbf x)$}
		      }}
\newglossaryentry{deltai}{name={\ensuremath{\delta_i}},
		      description={\secrev{Radius of $i$-th support around $m_i(\mathbf 0)$}
		      }}
\newglossaryentry{normi}{name={\ensuremath{\|\cdot\|_{i,\mathbf x}}},
                      description={$i$-th toric metric at the point $\mathbf x \in \secrev{\mathscr M}$}}
\newglossaryentry{norm}{name={\ensuremath{\| \cdot\|_{\mathbf x}}},
                      description={Toric norm,$\| \cdot\|_{\mathbf x}^2=\sum_i \| \cdot\|_{i,\mathbf x}^2$ }}
\newglossaryentry{mu}{name={\ensuremath{\mu(\mathbf f, \mathbf x)}},
                      description={Toric condition number of $\mathbf f$ at $\mathbf x \in \secrev{\mathscr M}$}}
\newglossaryentry{M}{name={\ensuremath{M(\mathbf f, \mathbf x)}},
                      description={Unscaled condition matrix for $\mathbf f$ at $\mathbf x$}}
\newglossaryentry{nui}{name={\ensuremath{\nu_i}},
                      description={Distortion invariant for the $i$-th support}}
\newglossaryentry{nu}{name={\ensuremath{\nu}},
                      description={Distortion invariant}}
\newglossaryentry{S}{name={\ensuremath{\secrev{\mathscr S}}},
                      description={Solution variety}}
\newglossaryentry{Lold}{name={\ensuremath{\mathcal L(\mathbf f_t, \mathbf z_t;\, a,b)}}, description={Condition length of the path $(\mathbf f_t, \mathbf z_t)_{t \in [a,b]}$}}
\newglossaryentry{R}{name={\ensuremath{R_i, R}}, description={Renormalization operator}} 
\newglossaryentry{elli}{name={\ensuremath{\ell_i}},
		      description={$\mathbf u \mapsto \max_{\mathbf a \in A_i} \mathbf a \Re(\mathbf u)$}}
\newglossaryentry{alpha}{name={\ensuremath{\alpha(\mathbf f, \mathbf z)}},
                      description={Smale's alpha invariant for $\mathbf f$ at $\mathbf z$}}
\newglossaryentry{gamma}{name={\ensuremath{\gamma(\mathbf f, \mathbf z)}},
                      description={Smale's gamma invariant for $\mathbf f$ at $\mathbf z$}}
\newglossaryentry{beta}{name={\ensuremath{\beta(\mathbf f, \mathbf z)}},
                      description={Smale's beta invariant for $\mathbf f$ at $\mathbf z$}}
\newglossaryentry{N}{name={\ensuremath{N}},
                      description={Newton iteration}}
\newglossaryentry{L}{name={\ensuremath{\mathscr L(\mathbf f_t, \mathbf z_t;\, a,b)}},
		      description={Renormalized condition length for $(\mathbf f_t \cdot R(\mathbf z_t))_{t \in [a,b]}$}}
\newglossaryentry{MixedVolume}{name={\ensuremath{V=V(\mathcal A_1, \dots, \mathcal A_n)}},
		      description={Mixed volume of the tuple $(\mathcal A_1, \dots, \mathcal A_n)$}}
\newglossaryentry{MixedSurface}{name={\ensuremath{V'=V'(\mathcal A_1, \dots, \mathcal A_n)}},
		      description={Mixed \changed{area} of the tuple $(\mathcal A_1, \dots, \mathcal A_n)$}}
\newglossaryentry{Normal}{name={\ensuremath{\secrev{\mathcal N}(\mathbf f,\Sigma^2; \mathscr F)}},
		      description={Normal distribution with mean $\mathbf f$ and covariance $\Sigma^2$ in real or complex space $\mathscr F$. }}
\newglossaryentry{zeroset}{name={\ensuremath{Z(\mathbf q)}},
                      description={Zero set of $\mathbf q$}}
\newglossaryentry{boundedzeroset}{name={\ensuremath{Z_H(\mathbf q)}},
		      description={Set of zeros of $\mathbf q$ with $\max(|z_i|)\le H$}}
\newglossaryentry{sigma}{name={\ensuremath{\Sigma_i^2 = \diag{\sigma_{i\mathbf a}^2}}},
		      description={Covariance matrix of $\mathbf g_i \in \mathscr F_{A_i}$}}
\newglossaryentry{kappa}{name={\ensuremath{\kappa_{\rho_i}}},
		      description={Imbalance invariant for the coefficients $\rho_{i\mathbf a}$}}
\newglossaryentry{lambdai}{name={\ensuremath{\lambda_i}},
                      description={Minkowski support function of $A_i$}}
\newglossaryentry{Aixi}{name={\ensuremath{A_i^{\boldsymbol \xi}}},
		      description={Extremal points of $A_i$ in the direction $\mathbf \xi$}}
\newglossaryentry{cone}{name={\ensuremath{C(B_1, \dots, B_n)}},
		      description={Open cone above $B_1 \subseteq A_1$, \dots, $B_n \subseteq A_n$}}
\newglossaryentry{etas}{name={\ensuremath{\eta_i, \eta}}, description={Face gap invariants}}
\newglossaryentry{Fanj}{name={\ensuremath{\mathfrak F_j}}, description={$j$-th stratum of the fan for $(A_1, \dots, A_n)$}}
\newglossaryentry{FanR}{name={\ensuremath{\mathfrak R}}, description={set of norm-1 representatives of $\mathfrak F_0$.}}
\newglossaryentry{SigmaINF}{name={\ensuremath{\Sigma^{\infty}}},
                      description={Variety of systems with a root at toric infinity}}
\newglossaryentry{roff}{name={\ensuremath{r=r(\mathbf f)}},
		      description={Polynomial vanishing on $\Sigma^{\infty}$}}
\newglossaryentry{dr}{name={\ensuremath{d_r}},
                      description={Degree of the polynomial $r$}}
\newglossaryentry{exclusion}{name={\ensuremath{\Lambda_{\epsilon}, \Omega_H, Y_K}},
                      description={Exclusion sets.}}
\newglossaryentry{kappaf}{name={\ensuremath{\kappa_{\mathbf f}}},
		      description={Imbalance invariant for the system $\mathbf f$}}
\newglossaryentry{muf}{name={\ensuremath{\mu_{\mathbf f}}},
		      description={Maximal renormalized condition for the system $\mathbf f$}}

\newglossaryentry{Q}{name={\ensuremath{Q}},
                      description={Geometric invariant of the tuple $(A_1, \dots, A_n)$}}
\newglossaryentry{Ihat}{name={\ensuremath{I_{{\mathbf f}, \Sigma^2}}},
                        description={  
\secrev{Expectation of $\| M(\mathbf f+\mathbf g ,\mathbf z)^{-1}\|_{\mathrm F}^2$
for $\mathbf g \sim \secrev{\mathcal N}(0,\Sigma^2)$.}}}
\newglossaryentry{MixedVolume2}{name={\ensuremath{V_W(\mathcal A_1, \dots, \mathcal A_j)}},
		      description={Mixed volume of $(\mathcal A_1, \dots, \mathcal A_j)$ as subsets of $W$.}}

\renewcommand{\glossarypreamble}{The following convention applies: vectorial quantities are typeset in boldface 
(e.g. $\mathbf f$) while scalar quantities are not (e.g. $f$). Multi-indices and vectors of coefficients of polynomials and exponential sums
are treated as covectors (row vectors), as well as the momentum map.}

\setglossarystyle{long3col}
    {\begin{longtable}{lp{\glsdescwidth}r}}%
    {\end{longtable}}%

\author{Gregorio Malajovich}
\title
{Complexity of sparse polynomial solving 2: Renormalization}
\address{Departamento de Matemática Aplicada, Instituto de Matemática, Universidade Federal do
Rio de Janeiro. Caixa Postal 68530, Rio de Janeiro, RJ 21941-909, Brasil.}
\email{gregorio@im.ufrj.br}
\date{April 16, 2022}
\thanks{This research was partially funded by the {\em Coordenação de Aperfeiçoamento de Pessoal de Nível Superior} (CAPES), grants PROEX and PRINT, and by the {\em Fundação Carlos Chagas Filho de Amparo à Pesquisa do Estado do Rio de Janeiro} (FAPERJ),
grant E-26/211.557/2021}

\begin{document}
\begin{abstract}
Renormalized homotopy continuation on toric varieties is introduced as a tool for solving sparse systems of
polynomial equations, or sparse systems of exponential sums. The cost of continuation 
	depends on a renormalized condition length, \changed{defined as} a line integral of the \changed{condition
	number along all the lifted renormalized} paths.

The theory developed in this paper leads to a continuation algorithm tracking all the solutions 
	between two \changed{generic} systems \changed{with} the same structure. The algorithm is randomized, in the sense that it
	follows a random path between the two systems. \changed{The probability of success is one.}
In order to produce an expected cost bound, several invariants
depending solely of the supports of the equations are introduced.
	For instance, the mixed \changed{area} is a quermassintegral that generalizes \changed{surface area} in the same way that
	mixed volume generalizes ordinary volume. The facet gap measures \secrev{for each 1-cone in the fan and for each support polytope, how close is 
the supporting hyperplane to the nearest vertex}. 
Once the supports are fixed, the expected cost depends on the input coefficients solely through two invariants: 
the renormalized toric condition number and the imbalance of the absolute values of the coefficients. 
	This leads to a non-uniform \secrev{polynomial} complexity bound for polynomial solving in terms of those two invariants.
\end{abstract}
\maketitle
\newpage
\setcounter{tocdepth}{2}
\tableofcontents
\renewcommand{\glsglossarymark}[1]{}
\printnoidxglossary[type={symbols}, sort={def}, title={List of Notations}]
\section{Introduction}

Classical foundational results on \secrev{solving polynomial systems} refer to the {\em possibility}
of solving \secrev{those systems} by an algorithm such as elimination or homotopy.
A theory capable to explain and predict
the {\em computational cost} of solving polynomial systems over $\mathbb C$ 
using homotopy algorithms was developed over the last thirty years
\cites{Smale-algorithms,Kostlan,Shub-projective,
Bezout1,Bezout2,Bezout3,Bezout4,Bezout5,
DedieuShub,
BePa05e,BePa09,Beltran-Pardo,
Bezout6,Bezout7,BeltranShub-topology,BeltranShub2009,
BDMS1,BDMS2,
Beltran-homotopia,
BC-annals,
Dedieu-Malajovich-Shub,
ABBCS,
Lairez, Lairez-rigid}.
As explained
in the books by  \ocite{BCSS} and \ocite{BC}, most results in this theory were obtained
through the use of unitary symmetry. \changed{The reach of this theory was 
limited}  
to the realm of dense \changed{homogeneous or multi-homogeneous} polynomial systems.

This paper extends the theory of homotopy algorithms to more general sparse systems.
A common misconception is to consider sparse systems as a particular case of dense systems, 
with some vanishing coefficients. 
This is not true from the {\em algorithmic} viewpoint. The vanishing coefficients introduce 
exponentially many artifact solutions. To see that, compare
the classical Bézout bound to the mixed volume bound in 
Theorems~\ref{BKK} and ~\ref{BKK2} below.

A theory of homotopy algorithms featuring {\em toric varieties} as a replacement for 
the classical {\em projective space}
was proposed by \ocite{toric1} in a previous attempt. Unfortunately, no clear complexity bound could be
obtained independently of integrals along the homotopy path.
Much stronger results are derived here through the introduction of
another symmetry group, that I call renormalization. Essentially, renormalization lifts the 
algorithm domain from the toric variety to its tangent space. Before going further, it is
necessary to explain the basic idea of renormalization and how it replaces  
unitary invariance.  

\subsection{Symmetry and renormalization}

Solutions for systems of $n$ homogeneous polynomial equations
in $n+1$ variables are complex rays through the origin,
so the natural solution locus is projective space $\mathbb P^n$. 
The unitary group $U(n+1)$ acts transitively and isometrically 
on projective space, and this
induces an action on the space $\mathcal H_d$ of degree $d$ homogeneous 
polynomials. 

A rotation $Q \in U(n+1)$ acts on a polynomial $f$ by composition $f \circ
Q^*$, so that every pair $(f, [\mathbf X]) \in \mathcal H_d \times \mathbb P^n$ 
with $f(\mathbf X)=\mathbf 0$ is mapped to the pair $(f \circ Q^*, [Q \mathbf X])$, and 
$(f \circ Q^*)(Q \mathbf X)=f(\mathbf X)=\mathbf 0$.
For the correct choice of a Hermitian inner product in $\mathcal H_d$, the 
group $U(n+1)$ acts by isometries. As a consequence, all of the invariants 
used in the theory are $U(n+1)$-invariants.

The canonical argument \secrev{in the theory} of dense polynomial solving goes as follows: suppose that one wants
to prove a \secrev{lemma} for some system $\mathbf f = (\mathbf f_1, \dots, \mathbf f_n)$ of polynomials
$\mathbf f_i \in \mathcal H_{d_i}$, $d_i \in \mathbb N$, at some point $[\mathbf X] \in 
\mathbb P^n$. \secrev{The hypotheses and conclusions are phrased in terms of $U(n+1)$ invariants.} Then one assumes without loss of generality that 
$\mathbf X = \begin{pmatrix}1 & 0 & \dots & 0\end{pmatrix}^T$. Intricate lemmas
become simple calculations.

Early tentatives to develop a complexity theory for
solving sparse polynomial systems were hindered by the lack of a 
similar action \cites{MRMomentum, MRHigh}. For instance, the complexity 
bounds obtained by
\ocite{toric1} depend on a condition length, which is the line integral
along a path of solutions $\mathbf (\changed{\mathbf f}_t, \mathbf X_t)$ of the condition number, times
a geometric distortion invariant $\nu(\mathbf X_t)$. No bound on the expectation of
this integral is known.

It is customary in the sparse case to look at roots $\mathbf X \in \mathbb C^n$
with $X_i \ne \mathbf 0$, that is on the multiplicative group $\mathbb C^n_\times$. 
The {\em toric variety} from equation~\eqref{toric-variety} below is a convenient
closure of $\mathbb C^n_\times$.
The contributions in this paper stem from the \changed{transitive} action of the multiplicative
group $\mathbb C_{\times}^n$ onto itself, and onto spaces of sparse
polynomials. Each element $\mathbf U \in \mathbb C_{\times}^n$ acts on $\mathbf X \in
\mathbb C_{\times}^n$ by componentwise multiplication. Let $A \subseteq
\mathbb Z^n$ be finite, and let \gls{PA}
be the set of Laurent polynomials
of the form
\[
	F(\mathbf Z)=\sum_{\mathbf a \in A} f_{\mathbf a} 
	Z_1^{a_1}
	Z_2^{a_2}
	\dots 
	Z_n^{a_n}
.
\]
The element $\mathbf U \in \mathbb C_{\times}^n$ acts 
\changed{on $\mathscr P_A$} by sending \changed{$\mathbf Z \mapsto F(\mathbf Z)$} into
\changed{$\mathbf Z \mapsto F(\secrev{\mathbf U} \mathbf Z)$}. 
\changed{For short, we will use notations $F(\cdot)$ and $F(\secrev{\mathbf U}\cdot)$ respectively. This action will be used}
to send a pair $(F(\cdot), \mathbf X)$ into the pair
$(F(\secrev{\mathbf X} \, \cdot \,), \mathbf 1)$
where $\mathbf 1 = \begin{pmatrix} 1 & 1 & \dots & 1 \end{pmatrix}^{T}$ 
is the unit of $\mathbb C_{\times}^n$. 
\secrev{This is the renormalization used \changed{here}.}
One can also replace the unit of 
$\mathbb C_{\times}^n$ by an arbitrary point. 

The main results in this paper can now be \changed{informally stated}. 
They will be
formalized later, using logarithmic coordinates that make polynomials
into exponential sums. While this last formulation is sharper and more
elegant, we start with the primary results.

\subsection{Sketch of the main results}
\subsubsection{Renormalization}
\changed{The renormalized Newton iteration applied to a pair $(\mathbf F(\cdot),$ $\mathbf X)$ is essentially Newton iteration applied to $(\mathbf F(\secrev{\mathbf X} \, \cdot),\mathbf 1)$. The result is then subject to the group action for $\mathbf X$. The precise construction of renormalized Newton iteration appears 
in Section~\ref{sec:renorm} and uses logarithmic coordinates.

The renormalization operator takes the homotopy path $(\mathbf F_t(\cdot), \mathbf X_t)_{t \in [t_0,T]}$ with $\mathbf F_t(\mathbf X_t) \equiv \mathbf 0$,
into the homotopy path $(\mathbf F_t(\secrev{\mathbf X_t}\ \cdot),\mathbf 1)_{t \in [t_0,T]}$. 
Given $\mathbf F_t(\cdot)$ and an approximate solution of $\mathbf X_{t_0}$
for $\mathbf F_{t_0}$, the
Renormalized Homotopy algorithm in Definition~\ref{def-recurrence} produces
an approximate  
lifting $\mathbf X_t$ for all $\mathbf F_t$.
Theorem~\ref{th-A} bounds the computational cost of this homotopy algorithm
linearly on \secrev{an} invariant,
the {\em renormalized condition length} of the homotopy path.
This invariant
is the sum of line integrals of the {\em toric condition number} \cite{toric1}
along
the renormalized path $(\mathbf F_t(\secrev{\mathbf X_t}\, \cdot),\mathbf 1)_{t \in [t_0,T]}$ and also along $(\mathbf X_t)_{t \in [t_0,T]}$. Departing from the previous
approach, the last integral is taken for $\mathbf X_t$ 
\secrev{not in the toric variety itself, but rather in a particular coordinate chart, namely the tangent plane to the toric variety at the fixed point $\mathbf 1$.}
{\em Caveat:} all the results in this paper will be formally stated in logarithmic coordinates, so what is stated as $\mathbf 1$ in this introduction will later become $\mathbf 0=[0, \dots, 0]^T$.}

\subsubsection{Expected condition}
Let $\mathbf q$ be a Gaussian, \changed{random} sparse polynomial system and let $Z(\mathbf q)$ denote the
set of roots of $\mathbf q$ in the toric variety. 
In order to investigate the renormalized condition length, one would like to bound the average of the sum
over $Z(\mathbf q)$ of the \changed{squared
condition number after renormalization. The bounds obtained
in this paper are more technical: Theorems~\ref{E-M2} and~\ref{old-mainB} provide a conditional bound, only the roots
away from `toric infinity' are counted.} The most troubling issue is that the upper bound
does not depend solely on the {\em mixed volume}, but also on the {\em mixed \changed{area}}. This is another 
quermassintegral generalizing the \changed{surface area} of a convex body.
\changed{The ratio of the mixed area to the mixed volume can be bounded below 
by the isometric inequality. No general upper bound can possibly exist.
Please see example~\ref{isoperimetric} for a particular upper bound. 
The mixed area quermassintegral is closely related to
the mixed area measure introduced by 
\ocite{Aleksandrov}.
The precise connection is explained in Remark~\ref{rem-Aleksandrov}
below.}

\subsubsection{Toric infinity}
Will the roots close to `toric infinity' make the bound from Theorems~\ref{E-M2} and \ref{old-mainB} worthless? 
\changed{The Condition Number Theorem}~\ref{cond-num-infty} establishes a perturbation bound in terms of the distance to the locus
of sparse systems with solution at `toric infinity'. The degree of this locus is bounded
in Theorem~\ref{degenerate-locus}.
In the particular case where the supports are general enough (strongly mixed supports),
this degree is no larger than the number of \secrev{{\em rays} (1-cones)} in the {\em fan} of the tuple of supports. 
Those two results can be used in Theorem~\ref{vol-omega} 
to bound the probability that a linear homotopy path \changed{fails} 
the condition in Theorem~\ref{old-mainB}, that is the probability that it crosses the set of
systems with at least one root close to `toric infinity'.

\subsubsection{Expected condition length, conditional}

Since this is an exploratory paper, we choose for simplicity a homotopy path
of the form $\mathbf g + t \mathbf f$, where 
\secrev{$\mathbf f$ and $\mathbf g$ have the same support,} $\mathbf g$ has iid Gaussian coefficients and $\mathbf f$ 
is fixed and outside a certain variety. If the supports are strongly mixed, the only requirement is that
the coefficients of $\mathbf f$ are non-zero. 

With probability one, this homotopy path lifts to \changed{$n! V/\det(\Lambda)$} solution paths, where $V$ is Minkowski's mixed volume \changed{and $\Lambda$ is the lattice determinant, to be defined later}. 
The global cost of Renormalized Homotopy along this homotopy path, $0 \le t \le \infty$, is given by
the sum over all the solution paths of the condition length from Theorem~\ref{th-A}.
Theorem~\ref{mainD} implies that given a fixed support $(A_1, \dots, A_n)$ 
and for a generic system $\mathbf f$, 
with probability at least $3/4$, the sum of condition lengths is no more than
\begin{equation} \label{main-bound-0}
\secrev{	O\left(
	\mu_{\mathbf f}^2 \kappa_{\mathbf f}^{\frac{3}{2}}
		\text{\ and logarithmic terms}\right).
}
\end{equation}
\changed{This} bound depends \changed{on the coefficients of the target system 
$\mathbf f$ solely through the 
	{\em renormalized condition number} $\mu_{\mathbf f} =
\max_{X \in Z(\mathbf f)} \mu(\mathbf f \secrev{(\mathbf X \cdot)}, \mathbf 1)$ 
and the imbalance invariant
$\kappa_{\mathbf f} = \max(\|\mathbf f_i\|/|f_{i\mathbf a}|)$.}
The bound stated above is non-uniform. \changed{The dependency on the
support polytopes is hidden inside of the big O notation. 
The statement of Theorem~\ref{mainD} is
uniform in terms of the input size, 
the mixed volume, the mixed area, the facet gap, and of course $\mu_{\mathbf f}$
and $\kappa_{\mathbf f}$. Once again, since the results in this paper are stated in logarithmic coordinates, the point $\mathbf 1$ in the definition of
$\mu_{\mathbf f}$ must be replaced by $\mathbf 0$ and the multiplicative action becomes additive.}

\subsubsection{The cost of homotopy}
Theorem \ref{mainD} allows to solve a random system $\mathbf g$, given the set of solutions
of a suitable system \changed{$\mathbf f$} with same support. \secrev{Reversely, it} also allows to solve a suitable arbitrary system
$\mathbf f$ of same support, given
the solutions of the random system $\mathbf g$. In order to obtain a more decisive complexity bound,
we consider the problem of finding the
set of solutions of a \changed{fixed system $\mathbf f$} in terms of the set of solutions of \changed{another fixed system $\mathbf h$}. The procedure goes through
a random system $\mathbf g$, in a manner akin to the {\em Cheater's Homotopy} suggested by
\ocite{Li-Sauer-Yorke}.
\changed{For a fixed support $A_1, \dots, A_n$ and for generic systems
$\mathbf f$ and $\mathbf h$ \changed{properly scaled},
the randomized algorithm in Theorem~\ref{mainE} will perform this task with probability one and expected cost
\begin{equation}\label{main-bound-1}
O
	\left( \left(
	\secrev{\mu_{\mathbf f}^2 
	\kappa_{\mathbf f}^{\frac{3}{2}}
+
	\mu_{\mathbf h}^2
	\kappa_{\mathbf h}^{\frac{3}{2}}}
	\right) \text{\ and logarithmic terms}\right) 
.
\end{equation}
Again, this bound is non-uniform. The precise statement provides a uniform
bound, which depends on invariants of the support such as the mixed volume,
the mixed area and the facet gap.}
In particular, once one convenient start system $\mathbf h$ with small $\kappa_{\mathbf h}$ and small
condition number is known, we obtain a non-uniform complexity bound:
the cost of solving a polynomial system $\mathbf f$ with the same support as $\mathbf h$ is
is
\begin{equation}\label{main-bound-2}
	\secrev{O\left(\mu_{\mathbf f}^2 
	\kappa_{\mathbf f}^{\frac{3}{2}}
	 \text{\ and logarithmic terms}\right)
	.}
\end{equation}

\subsection{Related work}

\ocite{Bezout 6} introduced the condition length in the solution variety and related it to the number of Newton steps in a homotopy continuation method. The step selection problem 
\secrev{in terms of the condition length} 
was dealt independently by
\ocite{Beltran-homotopia} and \ocite{Dedieu-Malajovich-Shub}. 
\fourthrev{The step selection problem was also independently dealt in the
paper by \ocite{BC-annals}, in terms of another integral}. The integral bounds obtained in \fourthrev{all} those papers would
apply to any subspace of the space of dense polynomials. 
As explained before, lack of unitary invariance prevented obtaining global complexity bounds in this setting. Recently, \citeauthor{EPR}~\ycites{EPR,EPR2} introduced new techniques in the context of {\em real} polynomial solving that may overcome this difficulty.

\changed{\ocite{Verschelde-toric} and later \ocite{DTWY} suggested another 
type of homotopy in toric varieties, using Cox coordinates. This adds one
new variable for every face or for every 
direction in the 0-fan, while the approach in
this paper does not add new variables. No complexity bounds on this other
toric homotopy are known.}

\subsection{Organization of the paper}
In \changed{section \ref{sec:background}} we revisit the notations, basic definitions and
facts needed in the sequel. \changed{The renormalization operator, renormalized Newton iteration and the main related invariants are discussed in section
~\ref{sec:renorm}.
The main results are formally stated in section \ref{sec:results}.}
The proof of the main statements is postponed
to Sections~\ref{sec-homotopy} to ~\ref{sec-linear}.
The final section lists some open problems and other remaining issues.

\subsection{Acknowledgments}

Special thanks to Bernardo Freitas Paulo da Costa and to Felipe Diniz who endured an
early seminar on some of this material. 
Peter Bürgisser convinced me to rework the introduction in terms of classical
polynomial systems rather than exponential sums, which is no minor improvement.
Also,
I would like to thank Matías Bender, Paul Breiding, Alperen Ergür, Josué Tonelli Cueto 
and Nick Vannieuwenhoven for
their input and suggestions. I also thank
Mike Shub, Jean-Claude Yakoubsohn, Marianne Akian,
Stéphane Gaubert for helping to clarify some of the issues in this paper.
\changed{An anonymous referee provided \secrev{two rounds of} valuable and extensive comments,
which substantially improved the presentation of this paper.}

\section{Background and notations}
\label{sec:background}
\subsection{The toric variety}
\label{sec:toric}

This paper is built on top of the theory
of Newton iteration and homotopy on toric varieties proposed
by \ocite{toric1}. 
We review in this section the notations and results that are necessary to
formally state the main theorems of this paper. 

As in the previous work, logarithmic coordinates are used to represent 
polynomial roots, exact or approximate. 
Polynomials get replaced by exponential sums. For instance if 
\changed{
\[
	F(\mathbf X) = \sum_{\mathbf a \in A} f_{\mathbf a} X_1^{a_1} X_2^{a_2} \cdots X_n^{a_n} = \mathbf 0
\]
for a finite set $A \subseteq \mathbb Z^n$, then we write
\[
	f(\mathbf x) = \sum_{\mathbf a \in A} f_{\mathbf a} e^{\mathbf a\mathbf x} 
\hspace{1em}
\text{,}
\hspace{1em}
\changed{X_i=e^{x_i}}
\]}
so that \changed{$f=F \circ \exp$} and $f(\changed{\mathbf x}) = \mathbf 0$. 

\ocite{toric1} considered the action of the additive group 
$\mathbb R^n$ by shifting supports. This leads us to consider more
general exponential sums, where we cannot assume any more that $A
\subseteq \mathbb Z^n$. We assume instead that $A - A = \{ \mathbf a -
\mathbf a': \mathbf a, \mathbf a' \in A\}$ is a finite subset of $\mathbb Z^n$.
It obviously contains the origin.

We will actually deal with {\em systems} of equations with possibly
different supports. \secrev{Those supports are always} finite sets
$\glsdisp{Ais}{A_1, \dots, A_n} \subseteq \mathbb R^n$ such that 
$A_i-A_i \subseteq \mathbb Z^n$. \changed{We also assume through this paper
that $\gls{Si} \defeq \#A_i \ge 2$ and that $n \ge 2$.}
To each $\mathbf a \in A_i$ we associate a function
\[
	\defun{\gls{Via}}{\mathbb C^n}{\mathbb C}{\mathbf x}{V_{i \mathbf a}(\mathbf x)=\rho_{i \mathbf a} e^{\mathbf a \mathbf x}}
\]
where $\gls{rhoia} >0$ is a fixed real number. We denote by 
\glsdisp{FA}{$\mathscr F_{A_i}$}
the complex vector space of exponential sums of the form
\[
	\defun{\secrev{f_i}}{\mathbb C^n}{\mathbb C}{\mathbf x}{\sum_{\mathbf a \in A_i} 
\secrev{f_{i\mathbf a}} V_{i\mathbf a}(\mathbf x)}
\]
\secrev{with $f_{i\mathbf a} \in \mathbb C$.} Solving systems of sparse polynomial systems with support $(A_1, \dots,$ $A_n)$
is equivalent to solving systems of exponential sums in $\mathscr F = \mathscr F_{A_1}
\times \cdots \times \mathscr F_{A_n}$.
The following conventions apply: the inner product 
on $\mathscr F_{A_i}$ is the inner product that makes the basis
$(V_{i \mathbf a})_{\mathbf a \in A_i}$ orthonormal. Objects in $\mathscr F_{A_i}$ will be
represented in coordinates as `row vectors' $\mathbf f_i = (\dots f_{i \mathbf a} \dots)_{\mathbf a \in A_i}$ and objects in \secrev{the dual space} $\mathscr F_{A_i}^*$ will be
represented as column vectors. 
We denote by \glsdisp{VA}{$V_{A_i}$} the vector valued {\em \secrev{sparse}~Veronese map}
\[
	\defun{V_{A_i}}{\mathbb C^n}{\mathscr F_{A_i}^* \simeq \mathbb C^{\secrev{S_i}}}{\mathbf x}{V_{A_i}(\mathbf x)=
\begin{pmatrix} \vdots \\ V_{i\mathbf a}(\mathbf x) \\ \vdots \end{pmatrix}_{\mathbf a \in A_i}.}
\]
Then evaluation of $f_i$ at $\mathbf x$ is given by the \changed{coupling}
\changed{\begin{equation}\label{pairing1}\color{black}
	f_i(\mathbf x) = \mathbf f_i \cdot V_{A_i}(\mathbf x) =
\begin{pmatrix}
	\cdots f_{i\mathbf a} \cdots
\end{pmatrix}_{\mathbf a \in A_i}
\begin{pmatrix} \vdots \\ V_{i\mathbf a}(\mathbf x) \\ \vdots \end{pmatrix}_{\mathbf a \in A_i}
.
\end{equation}}

\begin{example}
Let $A=A_1=\{0, 1, \dots, d\} \subseteq \mathbb Z$. The roots of the polynomial
	$F(X)=\sum_{a=0}^d F_a X^a$ are of the form $X=e^{x}$, where $x$ is a solution of the exponential sum equation 
$f(x) \defeq \sum _{a=0}^d F_a e^{ax} = 0$. Notice that
if $f(x)=0$ then $f(x + 2 k \pi \sqrt{-1})=0$ for all $k \in \mathbb Z$,
so roots of $F(X)=0$ in $\mathbb C$ are in bijection with roots of $f(x)=0$
in $\mathbb C \mod 2 \pi \sqrt{-1}\ \mathbb Z$.
\end{example}

\begin{example}[Weyl metric]\label{ex:Weyl} The unitary invariant inner product introduced by \ocite{Weyl} and also known as the Bombieri inner product plays a prominent role in the theory of dense homotopy algorithms~\cite{BCSS}. 
	Let $A=\{ \mathbf a \in \mathbb N_0^n \text{ s.t. } \sum_{j=1}^n a_j \le d\}$.
	If $F, G$ are degree $d$ polynomials in $n$ variables, 
	$F(\mathbf X)=\sum_{\mathbf a \in A} F_a\mathbf X^{\mathbf a}$ and $G(\mathbf X)=\sum_{\mathbf a \in A} G_a\mathbf X^{\mathbf a}$, Weyl's inner product is by definition
\[
	\langle F, G \rangle_{\mathscr P_{d,n}} = \sum_{a_1+\dots+a_n  \le d} 
	\frac{F_{\mathbf a} \bar G_{\mathbf a}}{\binomial{d}{\mathbf a}}
\]
where the multinomial coefficient
\[
\binomial{d}{\mathbf a} = \frac{d!}{a_1!\ a_2!\ \dots a_n!\ 
	\secrev{(d-\sum_{j=1}^n a_j)} !}
\]
is the coefficient of $W_1^{a_1} W_2^{a_2} \cdots W_n^{a_n}$ in $(1+W_1+\dots+W_n)^{d}$.
We set
\[
	\rho_{\mathbf a} = \sqrt{\binomial{d}{\mathbf a}},
	\hspace{4em}
	V_{\mathbf a}(\mathbf x) = \rho_{\mathbf a} e^{ \mathbf a \mathbf x}
	,\hspace{1em} \text{ and}
	\hspace{3em}	f_{\mathbf a} = \frac{F_{\mathbf a}}{\rho_{\mathbf a}} 
.
\]
As before, $\mathbf f \cdot V_A(\mathbf x) = F(e^{\mathbf x})$. The
exponential sum
$\mathbf f$ is represented in orthonormal coordinates $f_{\mathbf a}$ 
with respect to 
	Weyl's metric.
\end{example}

Once we fixed the supports (finite sets) $A_1, \dots, A_n$ \secrev{with} $A_i - A_i \in \mathbb Z^n$, and 
picked the coefficients $\rho_{i\mathbf a}$, we would like to be able to solve
\secrev{systems of equations of the form}
\[
\mathbf f(\mathbf x) = 
\begin{pmatrix}
f_1(\mathbf x) \\
\vdots \\
f_n(\mathbf x)
\end{pmatrix}
=
0, 
\]
with $\mathbf f$ in $\mathscr F_{A_1} \times \cdots \times \mathscr F_{A_n}$. 
If $\mathbf x \in \mathbb C^n$ is a solution of $\mathbf f(\mathbf x)=\mathbf 0$,
then $\mathbf f(\mathbf x + 2 \pi \sqrt{-1} \mathbf k)=\mathbf 0$ for all $\mathbf k \in
\mathbb Z^n$. It makes sense therefore to consider solutions in $\mathbb C^n
\mod 2 \pi \sqrt{-1} \,\mathbb Z^n$ instead. 
It turns out that in many situations we can do better.  

\begin{example}[Generalized biquadratic trick]\label{ex:biquadratic} Let $A=\{0, d, 2d, \dots cd\}$
for $c,d \in \mathbb N$. The degree $cd$ polynomial
\[
	F(X) = \sum_{a \in A} F_a X^a
\]
can be solved by finding the roots of the degree $c$ polynomial equation $G(W)=\sum_{i=0}^c F_{id} W^i=0$ and then taking $d$-th roots. This is the same as solving the
exponential sum $g(w)=\sum_{i=0}^c F_{id} e^{iw}$ in $\mathbb C \mod 2 \pi \sqrt{-1}$ and dividing by $d$. Or solving
	$f(x) = \sum_{a \in A} F_{a} e^{ax}= 0$ in $\mathbb C \mod \frac{2 \pi}{d} \sqrt{-1}$.
\end{example}

There is a multi-dimensional analogous to the situation in
example~\ref{ex:biquadratic}. A lot of work can be saved by exploiting this
fact. 
After we fixed the $A_i-A_i$'s, we  
want to declare $\mathbf x$ and $\mathbf w \in \mathbb C^n \mod 2 \pi \sqrt{-1}\,\mathbb Z^n$ equivalent if for all $\mathbf f =(f_1, \dots, f_n) \in \mathscr F$,
\begin{equation}
\label{equiv-roots}
\mathbf f (\mathbf x) = \mathbf 0 
\
\Leftrightarrow
\
\mathbf f (\mathbf w) = \mathbf 0 
.
\end{equation}
To do this formally, let
\[
	\defun{[V_{A_i}]}{\mathbb C^n}{\mathbb P(\mathscr F_{A_i}^*)}{x}{[V_{A_i}(x)]}
\]
be the differentiable map induced by $V_{A_i}$.
The equivalence relation below has the properties of \eqref{equiv-roots}
\begin{equation}\label{eq-def-M}
	\mathbf x \sim \mathbf w \hspace{1em} \mathrm{iff} \hspace{1em} \forall i, \ [V_{A_i}(\mathbf x)]=[V_{A_i}(\mathbf w)].
\end{equation}
Then we quotient $\gls{MM}= \mathbb C^n/\sim$.
If the mixed volume $V(\conv{A_1}, \dots, \conv{A_n})$ is non-zero, then
$\secrev{\mathscr M}$ turns out to be $n$-dimensional \cite{toric1}*{Lemma 3.3.1 and Remark 3.3.2}. 
In general,
the natural projection $\mathbb C^n \mod 2 \pi \sqrt{-1}\, \mathbb Z^n \rightarrow \secrev{\mathscr M}$ is many-to-one, and its degree is given by the determinant
of a certain lattice.
More precisely, let $\gls{Lambda} \subseteq \mathbb Z^n$ 
be the $\mathbb Z$-module spanned by the
union of all the $A_i-A_i$. Assuming again non-zero mixed volume,
$\Lambda$ has rank $n$. This means that the linear span
of $\Lambda$ is an $n$-dimensional vector space. In example~\ref{ex:biquadratic}, we had $\Lambda=d \mathbb Z$. 
Before going further, let us recall some basic definitions about lattices. For further details,
the reader is referred to the textbook by \ocite{Lovasz}.

\begin{definition}
\begin{enumerate}[(a)]
\item A {\em full rank lattice} $\Lambda \subseteq \mathbb R^n$ is a $\mathbb Z$-module
so that there are $\mathbf u_1, \dots, \mathbf u_n \in \Lambda$ linearly independent
over $\mathbb R$, and such that 
every $\mathbf u \in \Lambda$ is an integral
linear combination of the $\mathbf u_i$. A list $(\mathbf u_1, \dots, \mathbf u_n)$
with that property is called a {\em basis} of $\Lambda$.
\item If $\Lambda \subseteq \mathbb R^n$ is a full rank lattice, then
we define its determinant as $\det \Lambda = |\det U|$ where $U$ is a matrix with 
		\changed{rows} $\mathbf u_1, \dots, \mathbf u_n$ of a basis of $U$. The determinant
does not depend on the choice of the basis.
\item
The dual of a full rank lattice $\Lambda \subseteq \mathbb R^n$ 
is the set 
\[
	\Lambda^* = \{ v \in (\mathbb R^n)^*:
	\forall \mathbf u \in \Lambda, \mathbf v(\mathbf u) \in \mathbb Z	
		\}
.
\]
\end{enumerate}
\end{definition}
It turns out that $\Lambda^*$ is also a full rank lattice. If $\Lambda$ is full rank and
a basis of $\Lambda$ is given by the \changed{rows} of a matrix $U$, then $U$ is invertible and a basis
for $\Lambda^*$ is given by the \changed{columns} of $U^{-1}$. In general, \changed{if} the \changed{rows} of $U$ are a basis for
a general lattice $\Lambda$, then the \changed{columns} of
its Moore-Penrose pseudo-inverse $U^{\dagger}$ are a basis for $\Lambda^*$.
We can now give a more precise description of $\secrev{\mathscr M}$
\secrev{as a product of $\mathbb R^n$ by a fundamental domain
of $\Lambda^*$}:

\begin{lemma}
\[
	\secrev{\mathscr M} = \mathbb C^n \mod 2 \pi \sqrt{-1}\,\Lambda^*
\]
\end{lemma}
\begin{proof}
The relation $\mathbf x \sim \mathbf w$ in equation~\eqref{eq-def-M} 
is equivalent to:
\[
	\forall i,\ \exists s_i \in \mathbb C \setminus \{\mathbf 0\} \ 
	\text{such that}\ 
\forall \mathbf a \in A_i,\ 
\mathbf a (\mathbf x - \mathbf w) \equiv s_i \mod 2 \pi \sqrt{-1}\,\mathbb Z^n.
\]
We can eliminate the $s_i$ to obtain an equivalent statement,
\[
\forall i, \ 
\forall \mathbf a, \mathbf a' \in A_i,\ 
(\mathbf a - \mathbf a')(\mathbf x -\mathbf w)
	\equiv \mathbf 0 \mod 2 \pi \sqrt{-1}\, \secrev{\mathbb Z^n}
.
\]
This is the same as 
\[
\forall \boldsymbol \lambda \in \Lambda, \ 
\boldsymbol \lambda(\mathbf x -\mathbf w)
\equiv \mathbf 0 \mod 2 \pi \sqrt{-1} \secrev{\mathbb Z^n}
.
\]	
Thus, $\mathbf x \sim \mathbf w$ is equivalent to
\[
\mathbf x \equiv \mathbf w \mod 2 \pi \sqrt{-1}\, \Lambda^*.
\]
\end{proof}

There is a natural metric structure on $\secrev{\mathscr M}$. Recall that each
$V_{A_i}$
induces a differentiable map
\[
	\defun{[V_{A_i}]}{\secrev{\mathscr M}}{\mathbb P(\mathscr F_{A_i}^*)}{x}{[V_{A_i}(x)].}
\]
Let $\omega_{A_i}$
denote the pull-back of the Fubini-Study metric in $\mathbb P(\mathscr F_{A_i}^*)$ to $\secrev{\mathscr M}$. The Hermitian inner product associated to this
Kähler form is denoted by $\langle \cdot , \cdot \rangle_{i}$.
\begin{example} If $A=\{\mathbf 0, \mathrm e_1, \dots, \mathrm e_n\}$ and
$\rho_{a}=1$, then $\mathbb P(\mathscr F_A) = \mathbb P^n$ and 
$\langle \cdot , \cdot \rangle_{i}$ is just the pull-back of the Fubini-Study
metric.
More generally, in the setting of example~\ref{ex:Weyl}
\secrev{with $A=\{ \mathbf a \in \mathbb N_0^n, 
	\sum a_i \le d\}$},  
	we notice that 
	\[\|V_A(x)\| = \|(1,X_1, \dots, X_n)\|^d.\]
As a consequence, the inner
product is $d^2$ times the Fubini-Study
metric.
\end{example}

Let $\mathscr F = \mathscr F_{A_1} \times \cdots \times \mathscr F_{A_n}$.
Let $\mathbf V=(V_{A_1}, \dots, V_{A_n})$.
The coordinatewise coupling is denoted by
\begin{equation}\label{pairing2}
	\mathbf f \cdot \mathbf V(\mathbf x) = \begin{pmatrix} f_1 \cdot V_{A_1}(\mathbf x) \\ \vdots \\ f_n \cdot V_{A_n}(\mathbf x)
\end{pmatrix}
\end{equation}
\changed{where the coupling $\mathbf f_i \cdot V_{A_i}$ was defined in \eqref{pairing1}. \secrev{It produces an element} of $\mathbb C^n$.}
The zero-set of $\mathbf f$ is
\[
Z(\mathbf f) = \{ \mathbf x \in \secrev{\mathscr M}: \mathbf f \cdot \mathbf V(\mathbf x)=\mathbf 0 \}
\]
Assuming again that $\Lambda$ has full rank, the immersion
\[
	\defun{\glsdisp{quotient}{[\mathbf V]}}{\secrev{\mathscr M}}{\mathbb P(\mathscr F_{A_1}^*) \times \cdots
\times \mathbb P(\mathscr F_{A_n}^*)}{\mathbf x}{
\begin{pmatrix} [V_1(\mathbf x)] \\ \vdots \\ [V_n(\mathbf x)] \end{pmatrix}}
\]
turns out to be an embedding \changed{\cite{toric1}*{Lemma 3.3.1}}. The $n$-dimensional toric variety 
\begin{equation}\label{toric-variety}
	\gls{VV} = \overline{\{[\mathbf V(\mathbf x)]:\mathbf x \in \secrev{\mathscr M}\}}
\end{equation}
is the natural locus for roots of sparse polynomial systems (aka
exponential sums). Points in $\mathcal V$ that are not of the form
$[\mathbf V(\mathbf x)]$ are said to be at {\em toric infinity}.
The {\em main chart} for $\mathcal V$ is the map $[\mathbf V]: \secrev{\mathscr M} 
\rightarrow \mathcal V$. Its range contains the `finite' points of $\mathcal V$,
that is the points not at toric infinity. 

\subsection{The momentum map}
The {\em momentum map} 
\[
	\defun{\gls{mi}}{\secrev{\mathscr M}}{\glsdisp{AA}{\mathcal A_i}=\conv {A_i}}{\mathbf x}
{\mathbf m_i(\mathbf x) = \sum_{\mathbf a \in A_i} \frac{|V_{i\mathbf a}(\mathbf x)|^2}{\|V_{A_i}(\mathbf x)\|^2}\mathbf a}
\]
is a surjective volume preserving map (up to a constant) from 
$(\secrev{\mathscr M}, \langle \cdot , \cdot \rangle_i)$ into the interior 
of $\mathcal A_i$.
The constant is precisely $\pi^n=n! \vol (\mathbb P^n)$,
so that a generic $\mathbf f \in \mathscr F_{A_i}^n$ has $n! \vol_n \mathcal A_i$
roots in $\secrev{\mathscr M}$ (see \ocite{toric1} and references).
\changed{The derivative of $[V_{A_i}(\mathbf x)]$ can be expressed in terms
of the momentum map.}

\changed{\begin{lemma} 
The differential $D[V_{A_i}]: T_{\mathbf x}\secrev{\mathscr M} \rightarrow T_{[V_{A_i}](\mathbf x)}$ is precisely
	\begin{equation}\label{derivative-projection}
D [V_{A_i}](\mathbf x) \mathbf u = 
P_{V_{A_i}(\mathbf x)^{\perp}} \frac{1}{\|V_{A_i}(\mathbf x)\|^2}DV_{A_i}(\mathbf x) \mathbf u
\end{equation}
	\secrev{where the projection operator is
	$P_{V_{A_i}(\mathbf x)^{\perp}} = I - 
	\frac{1}{\|V_{A_i}(\mathbf x)\|^2}
	V_{A_i}(\mathbf x)
	V_{A_i}(\mathbf x)^*$.}
Moreover,
	\begin{equation}\label{derivative-momentum}
D [V_{A_i}](\mathbf x) \mathbf u = 
\frac{1}{\|V_{A_i}(\mathbf x)\|}
		\left( DV_{A_i}(\mathbf x) - V_{A_i}(\mathbf x) \mathbf m_i(\mathbf x) \right).
\end{equation}
\end{lemma}}
\changed{
\begin{proof}
	We choose 
$\frac{ V_{A_i}(\mathbf x)}{\|V_{A_i}(\mathbf x)\|}$ as the representative
	of the projective point $[V_{A_i}](\mathbf x) \in \mathbb P(\mathscr F_{A_i})$. 
Differentiating 
$\frac{ V_{A_i}(\mathbf x)}{\|V_{A_i}(\mathbf x)\|}$ one obtains
\[
		D [V_{A_i}](\mathbf x) \mathbf u = 
		\left(I - \frac{1}{\|V_{A_i}(\mathbf x)\|^2}
V_{A_i}(\mathbf x)V_{A_i}(\mathbf x)^*	
\right)\frac{1}{\|V_{A_i}(\mathbf x)\|}DV_{A_i}(\mathbf x) \mathbf u 
.
\]
The term inside the parenthesis is the projection operator
	$P_{V_{A_i}(\mathbf x)^{\perp}}$, hence \eqref{derivative-projection}. 
\secrev{Equation} \eqref{derivative-momentum} comes
	from writing the momentum map as 
\[
	\mathbf m_i(\mathbf x)=\frac{1}{\|V_{A_i}(\mathbf x)\|^2} V_{A_i}(\mathbf x)^*
DV_{A_i}(\mathbf x).
\]
\end{proof}
}

\secrev{In order to produce a coarse, although handy bound of the 
toric norm in terms of the Hermitian norm, we first introduce the {\em radius} of the support
at $\mathbf x$, viz.
\begin{equation}\label{deltaix}
	\gls{deltaix}
	\defeq \max_{\mathbf a \in A_i} \| \mathbf a - \mathbf m_i(\mathbf x)\|
.
\end{equation}
}
\begin{lemma}\label{coarse-bound-DV} Let $\mathbf x \in \secrev{\mathscr M}$ and $\mathbf u \in T_{\mathbf x} \secrev{\mathscr M} \simeq \mathbb C^n$.
	Let $\| \cdot \|$ be the canonical Hermitian norm
\changed{and let $\glsdisp{normi}{\| \mathbf u\|_{i,\mathbf x} = 
\|D[V_{A_i}](\mathbf x) \mathbf u \|}$ be the pull-back of Fubini-Study
metric on $\mathbb P(\mathscr F_{A_i}^*)$}.
	Then \changed{$\|\mathbf u\|_{i, \mathbf x} \le \max_{\mathbf a \in A_i}
	|(\mathbf a - m_i(\mathbf x))\mathbf u|$, in particular}
	\[
		\| \mathbf u\|_{i,\mathbf x}
		\le \delta_i(\mathbf x)
		\|\mathbf u\|.
	\]
	where 
	\[
		\secrev{	\frac{1}{2} \diam {\conv{A_i}} \le
	\delta_i(\mathbf x) \le \diam {\conv{A_i}}.}
	\]
\end{lemma}
\begin{proof}[Proof of Lemma~\ref{coarse-bound-DV}]
\changed{We normalized the representative
	of $[V_{A_i}](\mathbf x)$ in such a way that the Fubini-Study norm and
	the $\mathscr F_{A_i}^*$ norm coincide.}
	The $\mathbf a$-th coordinate of \eqref{derivative-momentum} is bounded by \[
\left|
\left(D [V_{A_i}](\mathbf x) \mathbf u\right)_{\mathbf a}
\right|
	=
	\frac{|V_{\changed{i}\mathbf a}(\mathbf x)|}{\|V_{A_i}(\mathbf x)\|}
	|(\mathbf a - \mathbf m_i(\mathbf x)) \mathbf u|
\le
	\changed{
	\frac{|V_{\changed{i}\mathbf a}(\mathbf x)|}{\|V_{A_i}(\mathbf x)\|}
	\max_{\mathbf a \in A_i} |(\mathbf a - \mathbf m_i(\mathbf x)) \mathbf u|.}
	\]
	\changed{Hence,
	\[
		\| \mathbf u\|_{i, \mathbf x}
		\le
	\max_{\mathbf a \in A_i} |(\mathbf a - \mathbf m_i(\mathbf x)) \mathbf u|.\]
}
	The upper bound $\delta_i\secrev{(\mathbf x)} \le \diam {\conv{A_i}}$ is trivial. For the
lower bound, let $\mathbf a, \mathbf a' \in A_i$ maximize 
$\| \mathbf a - \mathbf a'\|=\diam {\conv{A_i}}$. 
Let $\mathbf c = \frac{1}{2}(\mathbf a + \mathbf a')$.
Assume without loss
of generality that 
$\| \mathbf m_i(\mathbf x) - \mathbf a\| \ge \| \mathbf m_i(\mathbf x) - \mathbf a'\|$. Then,
$\| \mathbf m_i(\mathbf x) - \mathbf a\| \ge \| \mathbf c - \mathbf a\| =
\frac{1}{2}\diam {\conv{A_i}}$.
	\changed{The Cauchy-Schwartz inequality for canonical inner product and
	norms yields
\[
	\|\mathbf u \|_{i,\mathbf x} 
		\le \max_{\mathbf a \in A_i} \|\mathbf a - m_{i}(\mathbf x) \|
		\|\mathbf u\| = \delta_i\secrev{(\mathbf x)} \|\mathbf u\|.
\]}
\end{proof}

We define an inner product on $\secrev{\mathscr M}$ as the pull-back of the Fubini-Study \secrev{inner product} in \changed{the toric variety} $\mathcal V$, namely
\begin{equation}\label{metricM}
	\langle \cdot , \cdot \rangle_{\mathbf x} = 
\langle \cdot , \cdot \rangle_{1,\mathbf x}  
+ \cdots
+
\langle \cdot , \cdot \rangle_{n,\mathbf x}  
.
\end{equation}
Its associated norm is denoted by $\gls{norm}$.
The previous complexity analysis by \ocite{toric1} relied on two main 
invariants.

The {\em toric condition number}
is defined for $\mathbf f \in \mathscr F$ and
$\mathbf x \in \secrev{\mathscr M}$ by
\[
	\gls{mu} = 
\left\| M(\mathbf f,\mathbf x)^{-1} \diag { \|\mathbf f_i\|}
\right\|_{\mathbf x}
\]
where 
\changed{
\[
	\gls{M} = 
\begin{pmatrix}
	\mathbf f_1 \cdot D[V_{A_1}](\mathbf x) \\
\vdots \\
	\mathbf f_n \cdot D[V_{A_n}](\mathbf x) \\
\end{pmatrix}
\]
and $\| \cdot \|_{\mathbf x}$ is the operator norm for linear maps from
$\mathbb C^n$ (canonical norm assumed) into $(\secrev{\mathscr M}, \| \cdot \|_x)$.
Conceptually, it is the norm of the derivative of the implicit function
$G: \mathbb P(\mathscr F_{A_1}) \times \cdots \times \mathbb P(\mathscr F_{A_n})\rightarrow (\mathscr M, \langle \cdot , \cdot \rangle_{\mathbf x})$ for the
equation $\mathbf f \cdot \mathbf V(\mathbf x)=\mathbf 0$. However the definition above
is more general.

We may replace formula \eqref{derivative-projection} inside the definition of
$M(\mathbf f, \mathbf x)$ to obtain the expression
\[
M(\mathbf f, \mathbf x)
=
\begin{pmatrix}
\frac{1}{\|V_{A_1}(\mathbf x)\|} \mathbf f_1 \cdot P_{V_{A_1}(\mathbf x)^{\perp}} DV_{A_1}(\mathbf x) \\
\vdots \\
\frac{1}{\|V_{A_n}(\mathbf x)\|} \mathbf f_n \cdot P_{V_{A_n}(\mathbf x)^{\perp}} DV_{A_n}(\mathbf x) \\
\end{pmatrix}
\]
Replacing by \eqref{derivative-momentum} instead, one obtains
\[
M(\mathbf f,\mathbf x) = 
\begin{pmatrix}
\frac{1}{\|V_{A_1}(\mathbf x)\|} \mathbf f_1 \cdot \left( DV_{A_1}(\mathbf x) - V_{A_1}(\mathbf x) m_1(\mathbf x)\right) \\
\vdots \\
\frac{1}{\|V_{A_n}(\mathbf x)\|} \mathbf f_n \cdot \left( DV_{A_n}(\mathbf x) - V_{A_n}(\mathbf x) m_n(\mathbf x)\right) \\
\end{pmatrix}
.
\]
}
If $\mathbf m_i(\mathbf x)=\mathbf 0$ for all $i$, or if $\mathbf x$ is a zero for $\mathbf f$, we have
just
\begin{equation}\label{M-zero}
M(\mathbf f,\mathbf x) = 
\begin{pmatrix}
\frac{1}{\|V_{A_1}(\mathbf x)\|} \cdot \mathbf f_1 DV_{A_1}(\mathbf x) \\
\vdots \\
\frac{1}{\|V_{A_n}(\mathbf x)\|} \cdot \mathbf f_n DV_{A_n}(\mathbf x) \\
\end{pmatrix}
.
\end{equation}

\secrev{The second invariant, called {\em distortion invariant},} bounds the distortion when passing from $\|\cdot\|_{i,\mathbf x}$ to $\|\cdot\|_{i, \mathbf y}$ \cite{toric1}*{Lemma 3.4.5}. 
\changed{It is defined as}
$\gls{nu}(\mathbf x) = \max_i \nu_i(\mathbf x)$ with
\secrev{
\begin{equation}\label{distortion}
	\gls{nui}(\mathbf x) \secrev{\defeq} 
	\sup_{\mathbf a \in A_i}
	\sup_{\|\mathbf u\|_{i,\mathbf x} \le 1} 
	|(\mathbf a - \mathbf m_i(\mathbf x)) \mathbf u|
	=
	\sup_{\mathbf a \in A_i}
	\|(\mathbf a - \mathbf m_i(\mathbf x)) \|_{i, \mathbf x}^*
\end{equation}}
Thus, it can also be understood as the `radius' of the support with respect to the
momentum $\mathbf m_i(\mathbf x)$, in the dual metric 
\secrev{$\| \cdot \|_{i,\mathbf x}^*$} to $\|\cdot\|_{i, \mathbf x}$. This radius can be computed effectively for a fixed $\mathbf x$.
It can also be bounded for $\mathbf x=\mathbf 0$ as in Lemma~\ref{lem-nu-kappa}
below. \secrev{For instance, if $\rho_{i\mathbf a} = 1$ for all $i, \mathbf a$, one has
$\nu_i(\mathbf 0) \le \sqrt{S_i}$.}

Expressions for both invariants can be simplified by shifting
each support $A_i$, so that $\mathbf m_i(\mathbf x) = \mathbf 0$. By shifting supports, we still obtain finite sets $A_i \subseteq \mathbb R^n$ with the property that $A_i-A_i \in \mathbb Z^n$. 

\secrev{The projective distance in $\mathscr F_{A_i}$ is defined by
$d_{\mathbb P}(\mathbf f_i, \mathbf g_i) = \inf_{\lambda \in \mathbb C} \frac{\|\mathbf f_i - \lambda \mathbf g_i\|}{\|\mathbf f_i\|}$.
It is equal to the sine of the Riemannian metric defined by Fubiny-Study in $\mathbb P(\mathscr F_{A_i})$. We will define
the {\em multiprojective distance} in $\mathscr F$
as the product metric, viz.
\begin{equation}\label{multi-projective-distance}
	d_{\mathbb P}(\mathbf f, \mathbf g) = \sqrt{\sum_i \inf_{\lambda \in \mathbb C} \frac{\|\mathbf f_i - \lambda \mathbf g_i\|^2}{\|\mathbf f_i\|^2}}
.
\end{equation}
This induces a metric in $\mathbb P(\mathscr F_{A_1}) \times \cdots \times \mathbb P(\mathscr F_{A_n})$, not to be confused with
the product Fubini-Study Riemannian distance.}

The following estimates will be needed:
\begin{proposition}\label{aggregates} 
	\label{prop-mu}
	Let $\mathbf f,\mathbf g \in \mathscr F$.
Let $\mathbf x \in \secrev{\mathscr M}$. 
	\begin{enumerate}[(a)]
		\item \label{prop-mu-a} \secrev{Unconditionally,}
\[
	1 \le \mu(\mathbf f,\mathbf x) 
.
\]
\item \label{prop-mu-b} 
			If \changed{$d_{\mathbb P}(\mathbf f, \mathbf g) \mu(\mathbf f,\mathbf x) < 1$,} then
			\secrev{we have the Lipschitz properties}
\[
	\frac{ \mu(\mathbf f,\mathbf x) }{1 + d_{\mathbb P}(\mathbf f,\mathbf g) \mu(\mathbf f,\mathbf x)}
		\le \mu(\mathbf g,\mathbf x) \le 
			\frac{ \mu(\mathbf f,\mathbf x) }{1 - d_{\mathbb P}(\mathbf f,\mathbf g) \mu(\mathbf f,\mathbf x)}
			\]
	\end{enumerate}
\end{proposition}
\secrev{Proposition \ref{aggregates}(\ref{prop-mu-a}) is 
Equation (5) in \ocite{toric1}. Proposition \ref{aggregates}(\ref{prop-mu-b}) 
is a particular case of Theorem 4.3.1 ibidem, with $s=0$.
}

\section{Renormalization}
\label{sec:renorm}
\subsection{Action}

\begin{figure}
\centerline{
	\resizebox{\textwidth}{!}{ 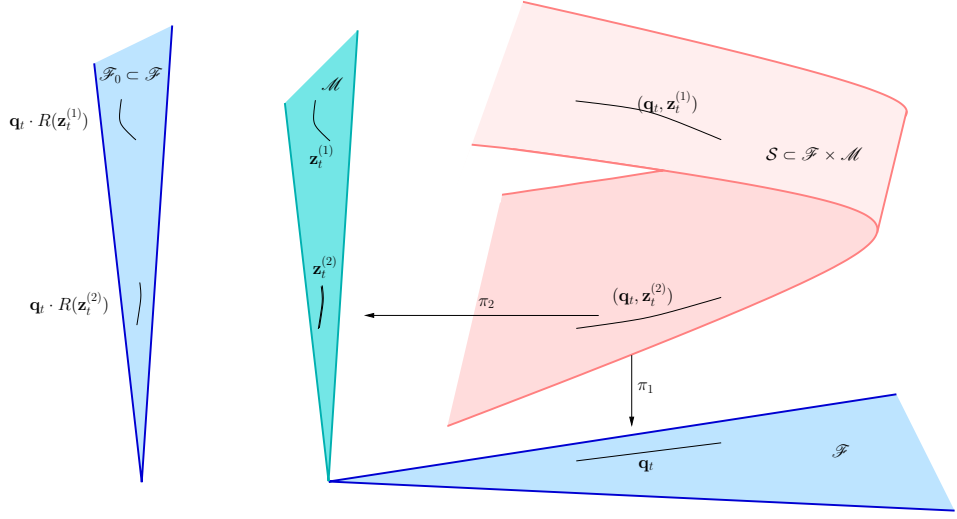} }
	\caption{The solution variety. Each path in the space $\mathscr F$ of equations lifts
	onto possibly several paths. Each of those corresponds to a different 
	renormalized path in $\mathscr F_0 = \{ \mathbf f \in \mathscr F: \mathbf f \cdot V(\mathbf 0) = \mathbf 0 \}$.
\label{solution-variety}}
\end{figure}

The main result in my previous paper \cite{toric1} was a step count for path-following
in terms of \changed{the} condition length.  
Let $\mathscr F = 
\mathscr F_{A_1} \times \cdots \times \mathscr F_{A_n}$ and let
\[\gls{S} = \{ (\mathbf f,\mathbf z) \in \mathscr F
\times \secrev{\mathscr M}: \mathbf f \cdot \mathbf V(\mathbf z)=\mathbf 0\}
\] 
be the {\em solution variety} (Figure~\ref{solution-variety}).
The condition length of the path $(\mathbf f_t, \mathbf z_t)_{t \in [a,b]} \in \secrev{\mathscr S}$ 
is \secrev{defined as}
\begin{equation}\label{lold}
	\gls{Lold} \defeq
\int_a^b \sqrt{\| \dot {\mathbf f}_t \|_{\mathbf f_t}^2 + \|\dot {\mathbf z}_t\|_{\mathbf z_t}^2} 
\ \mu(\mathbf f_t, \mathbf z_t) \nu(\mathbf z_t) \ \dd t,
\end{equation}
where $\|\dot {\mathbf f}\|_{\mathbf f}$ is the \changed{Fubini-Study} norm in
\secrev{the multi-projectivization of $\mathscr F$,}
namely
$\|\dot {\mathbf f}\|_{\mathbf f}^2=
\sum \|P_{\mathbf f_i^{\perp}} \dot {\mathbf f}_i\|^2/\|\mathbf f_i\|^2$\secrev{, and the distortion invariant
$\nu=\max \nu_i$ was defined in \eqref{distortion}.} It is assumed that the
path is smooth enough for the integral to exist.
At this time, no average estimate
of $\mathcal L$ is known in the sparse setting. 
The main obstructions to obtaining
such bound seem to be the invariant $\nu(\mathbf z)$ and the dependence of the
condition number on the toric norm at the point $\mathbf z$. When $\mathbf z$ approaches
`toric infinity', that is $\mathbf m_i(\mathbf z)$ approaches $\partial \mathrm{Conv}(A_i)$
for some $i$, the invariant $\nu_i(\mathbf z)$ can become arbitrarily large. This occurs in particular during polyhedral or `cheater's' homotopy, 
where the starting system has roots at `toric infinity'.

The renormalization approach in this paper is intended to overcome
those difficulties. 
First, we fix once and for all a privileged point
in \secrev{$\secrev{\mathscr M}$}. In this paper, we take this point to be the origin $\mathbf 0$.
By shifting the supports $A_i$, we can further assume that $\mathbf m_i(\mathbf 0)=\mathbf 0$.

\begin{definition} The {\em renormalization operator} $\glsdisp{R}{R}=R(\mathbf u)$\secrev{, $u \in \secrev{\mathscr M}$,}
	is \secrev{defined as} the operator 
\[
\defun{R(\mathbf u)}{\mathscr F}{\mathscr F}{\mathbf f}{\mathbf f \cdot R(\mathbf u)
	\secrev{\defeq}
\begin{pmatrix}
f_1 \cdot R_1(\mathbf u)\\
\vdots \\
f_n \cdot R_n(\mathbf u)\\
\end{pmatrix}
}
\]
with
\[
\defun{R_i(\mathbf u)}{\mathscr F_{A_i}}{\mathscr F_{A_i}}
{\mathbf f_i=[ \dots, f_{i\mathbf a}, \dots ]} 
	{\secrev{
\mathbf f_i \cdot R_i(\mathbf u) 
=
[ \dots, f_{i\mathbf a} e^{\mathbf a\mathbf u 
}, \dots ]
.
	}} 
\]
\end{definition}
\changed{Since $R(\mathbf u)$ is an operator and not an element of $\mathscr F^*$, no confusion can arise with the coupling of $\mathscr F$ and $\mathscr F^*$
as in equations ~\eqref{pairing1} and~\eqref{pairing2}. 
\secrev{The notation above} is
consistent with the coupling, in the sense that
$\mathbf f \cdot R(\mathbf u) \cdot V(\mathbf x) \secrev{\defeq} 
(\mathbf f \cdot R(\mathbf u)) \cdot V(\mathbf x)$ is uniquely defined.
In particular, $\mathbf f_i \cdot V_{A_i}(\mathbf x) = 
\mathbf f_i \cdot R(\mathbf x) \cdot V_{A_i}(\mathbf 0)$
and also, 
\[
\mathbf f \cdot \mathbf V(\mathbf x) = 
\mathbf f \cdot R(\mathbf x) \cdot \mathbf V(\mathbf 0).
\]} 
%

\secrev{
	In order to state Theorem~\ref{th-renormalization} below, we define
\begin{equation}\label{def-elli}
	\defun{\gls{elli}}{\mathbb C^n}{\mathbb R}{\mathbf z}{\ell_i(\mathbf z)
	\defeq \max_{\mathbf a \in A_i} \mathbf a \Re(\mathbf z) .}
\end{equation}
}

\begin{remark}
	$\ell_i(\mathbf u) = \lambda_i(\Re(\mathbf u))$, where
	$\lambda_i$ is the Minkowski support function of $A_i$,
	to be defined in Equation \eqref{def-lambdai}.
\end{remark}

\begin{theorem}\label{th-renormalization}Under the notations above:
	\begin{enumerate}[(a)]
	\item
		The renormalization operator induces an action also denoted $R(\mathbf u)$ of the additive group $\mathbb C^n$ into the 
			solution variety $\secrev{\mathscr S \defeq
			\{ (\mathbf f,\mathbf z) \in \mathscr F \times \secrev{\mathscr M}: \mathbf f \cdot \mathbf V(\mathbf z)= \mathbf 0\}}$, 
			namely $(\mathbf f, \mathbf z) \mapsto \left(\mathbf f \cdot R(\mathbf u), (\mathbf z - \mathbf u)\right)$.
\item
If $\mathbf u$ is pure imaginary, then for all $\mathbf a \in A_i$,
$|e^{\mathbf a \mathbf u}|=1$. Hence,
the map $R(\mathbf u): \secrev{\mathscr S} \rightarrow \secrev{\mathscr S}$ is an isometry, as well as the coordinate maps $\mathbf f \mapsto \mathbf f \cdot R(\mathbf u)$ and $\mathbf z \mapsto \mathbf z - \mathbf u$.
\item In general, \secrev{
\[
e^{-\ell_i(-\mathbf u)} \| \mathbf f_i\| \le
\| \mathbf f_i \cdot R_i(\mathbf u) \| \le 
e^{\ell_i(\mathbf u)} \| \mathbf f_i\|
\]
and 
\[
			e^{-\max_i(\ell_i(-\mathbf u))} \| \mathbf f\| \le
 \| \mathbf f \cdot R(\mathbf u) \| \le 
			e^{\max_i(\ell_i(\mathbf u))} \| \mathbf f\| 
.\]
}
	\end{enumerate}
\end{theorem}

The action of $\mathbb R^n$ by imaginary renormalization $R(\mathbf u \sqrt{-1})$
is also known as {\em toric action}. The word toric comes from the fact that this action is actually an action of $\mathbb R^n \mod 2 \pi \,\mathbb Z^n \simeq (S^1)^n$.
A more elaborate version of the real action was used by \ocite{Verschelde-toric},
with additional variables. \changed{See also \ocite{DTWY}.}

\secrev{
\begin{remark}
	The subgroup $2 \pi \sqrt{-1} \Lambda^*$ fixes
	$\mathscr S$, hence the statement of
	Theorem~\ref{th-renormalization} holds verbatim for the action
	of the quotient subgroup $\mathbb C^n \mod 2 \pi \sqrt{-1} \Lambda^*$.
\end{remark}
}

\secrev{
	\begin{proof}[Proof of Theorem~\ref{th-renormalization}]
		For item (a), assume that $\mathbf f \cdot \mathbf V(\mathbf z) = \mathbf 0$,
		then for any $\mathbf u \in \mathbb Z^n$ we have
		trivially $(\mathbf f \cdot R(\mathbf u)) \cdot \mathbf V(\mathbf z - \mathbf u) \mathbf = 0$. Therefore, the solution variety
$\secrev{\mathscr S}$ is preserved.
		In order to establish that the mapping is a group action,
		just notice that
\[
	\mathbf f_i \cdot R_i(\mathbf u + \mathbf v) 
=
		[\dots, f_{1\mathbf a} e^{\mathbf a (\mathbf u+\mathbf v)}, \dots ]
=
		(\mathbf f_i \cdot R_i(\mathbf u) ) \cdot R_i(\mathbf v) 
\]
and $\mathbf z - (\mathbf u + \mathbf v) = (\mathbf z - \mathbf u) - \mathbf v$.

		In order to prove item (b), assume that $\mathbf u$ is pure
		imaginary. 
		For all $\mathbf a \in A_i$, $|e^{\mathbf a \mathbf u}|=1$
		and $R_i(\mathbf u)$ is an isometry of the linear space
		$\mathscr F_{A_i}$. Moreover,
		for any $\dot {\mathbf z} \in T_{\mathbf z}(\secrev{\mathscr M})$,
\[
	\| \dot {\mathbf z}\|_{i,\mathbf z - \mathbf u}
	=
		\|DV_{A_i}(\mathbf z - \mathbf u) \dot {\mathbf z} \|
	=
		\|\diag{e^{\mathbf a \mathbf u}} DV_{A_i}(\mathbf z - \mathbf u) \dot {\mathbf z} \|
=
		\| DV_{A_i}(\mathbf z - \mathbf u) \dot {\mathbf z} \|
=
		\| \dot {\mathbf z}\|_{i,\mathbf z}
.
\]
		Thus, $R(\mathbf u)$ is an isometry of $\secrev{\mathscr S}$.

		Unconditionally,
\[
	\| \mathbf f_i R_i(\mathbf u) \| =
	\| (\dots, f_{i\mathbf a} e^{\mathbf a \mathbf u},
		\dots) \|
		\le \|\mathbf f_i\| \max |e^{\mathbf a \mathbf u}|
		\le \|\mathbf f_i\| e^{\ell_i(\mathbf u)},
\]
	and the lower bounds are similar. This establishes item (c).
	\end{proof}
}
\subsection{Newton renormalized}
The toric Newton operator was defined in \cite{toric1} by
\[
	\defun{\gls{N}}{\mathscr F \times \secrev{\mathscr M}}{\secrev{\mathscr M}}
{(\mathbf f,\mathbf z)}{\mathbf z - (\mathbf f \cdot P_{\mathbf V(\mathbf z)^{\perp}} D\mathbf V(\mathbf z))^{-1} (\mathbf f \cdot \mathbf V(z))}
.
\]
Notice that the Newton operator $N$ is invariant by independent scaling of
each $f_i$, so it also defines an operator in 
\secrev{$\mathbb P(\mathscr F_{A_1}) \times \dots \times
\mathbb P(\mathscr F_{A_n}) \times \secrev{\mathscr M}$}.
Its analysis \secrev{required a} toric version of
Smale's invariants $\gls{alpha}=\beta(\mathbf f, \mathbf z) \gamma(\mathbf f, \mathbf z)$.

\secrev{Before introducing those invariants, one needs to produce a `local
function' $S_{\mathbf f, \mathbf z}$ around a point
$\mathbf z$. This local function was suggested by the analysis of the
projective Newton operator and is fully described in \cite{toric1}*{Sec. 3.3}. 
In this paper we will only need a simplified expression, obtained under the
hypothesis that
$m_i(\mathbf z)=\mathbf 0$ for all $i$. In this case,  
the {\em local function} at $\mathbf z$ for the toric Newton operator is
\[
	S_{f,\mathbf z} (\dot{\mathbf z}) = 
	\begin{pmatrix}
		\mathbf f_1 \cdot \left( \frac{1}{\|V_{A_1}(\mathbf z)\|} V_{A_1}(\mathbf z + \dot {\mathbf z}) \right)
\\
\vdots
\\
		\mathbf f_n \cdot \left( \frac{1}{\|V_{A_n}(\mathbf z)\|} V_{A_n}(\mathbf z + \dot {\mathbf z}) \right)
	\end{pmatrix}
	.
\]
The toric versions of $\beta$ and $\gamma$ are the Smale invariants associated to this local function are precisely
\[
	\gls{beta} = \left\| DS_{\mathbf f,\mathbf z}(\mathbf 0)^{-1} S_{\mathbf f,\mathbf z}(\mathbf 0) \right\|_{\mathbf z} 
\]
and 
\[
	\gls{gamma} = 
\sup_{k \ge 2}
\left(\frac{1}{k!}
\left\|
		DS_{\mathbf f,\mathbf z}(\mathbf 0)^{-1}
		D^kS_{\mathbf f,\mathbf z}(\mathbf 0)
\right\|_{\mathbf z}\right)^{1/(k-1)}
.
\]
The expressions above are also invariant by independent scalings
of each $\mathbf f_i$. The expressions above require $m_i(\mathbf z)=\mathbf 0$, and a more
general definition for the toric invariants can be found in \cite{toric1}*{Sec.3.3 and 3.5}.  

In this paper, we consider the {\em renormalized} Newton operator instead.
We replace the pair $(\mathbf f, \mathbf z)$ by the pair $(\mathbf g, \mathbf 0)$ where $\mathbf g = \mathbf f \cdot R(\mathbf z)$. This way, we only need
to shift all the supports $A_i$ once, so that all centers of gravity are
at the origin. Therefore, $m_i(\mathbf 0)=\mathbf 0$ for all $i$. By doing so the points
of $A_i$ may cease to be integral, but the differences $A_i - A_i
=\{\mathbf a - \mathbf a': \mathbf a, \mathbf a' \in A_i\}$ are integral
points.

Let $\mathbf x_0$ be a point of $\secrev{\mathscr M}$.
The renormalized iterates of $\mathbf x_0$ are defined inductively 
by: 
\begin{equation}\label{eq-ren-Newton}
	\mathbf x_{i+1}=N(\mathbf f \cdot R(\mathbf x_i), \mathbf 0) +\mathbf x_i
.
\end{equation}
This can be described as a procedure: apply the group action 
to send the pair
$(\mathbf f, \mathbf x_i)$ to $(\mathbf f \cdot R(\mathbf x_i), \mathbf 0)$,
then apply one Newton iteration, then undo the group action to
move back to $(\mathbf f, \mathbf x_{i+1})$. 
Renormalized Newton iterates} can be compared to the actual Newton iterates of $\changed{\mathbf x_0}$ in 
$T_{\mathbf 0} \secrev{\mathscr M} = \mathbb C^n$ for a suitable function,
namely:
\begin{lemma}\label{toric-to-classical}
Assume $\mathbf m_i(\mathbf 0)=\mathbf 0$ for all $i$.
Let
\[
	\defun{\mathbf F}{T_{\mathbf 0}\secrev{\mathscr M}}{\mathbb C^n}{\changed{\mathbf x}}{\mathbf F(\changed{\mathbf x})= \mathbf f \cdot \mathbf V(\mathbf{\mathbf x})}.
\]
Then,
	\begin{enumerate}[(a)] 
		\item $N(\mathbf f \cdot \mathbf R(\mathbf x), \mathbf 0) = N(\mathbf F, \changed{\mathbf x})-\changed{\mathbf x}$
		\item $\beta (\mathbf f \cdot \mathbf R(\mathbf x), \mathbf 0) = \beta(\mathbf F,\changed{\mathbf x})$ 
		\item $\gamma (\mathbf f \cdot \mathbf R(\mathbf x), \mathbf 0) = \gamma(\mathbf F,\changed{\mathbf x})$ 
\end{enumerate}
where the left-hand-sides use the notations in \cite{toric1}
and the right-hand-sides are the classical Smale's invariants in $(T_{\mathbf 0}\secrev{\mathscr M},\|\cdot\|_{\mathbf 0})$.
\end{lemma}

\begin{proof}[Proof of Lemma~\ref{toric-to-classical}]
We establish first item (a):
\begin{eqnarray*}
N(\mathbf f \cdot \mathbf R(\mathbf x), \mathbf 0) &=&
- (\mathbf f \cdot \mathbf R(\mathbf x) \cdot D\mathbf V(\mathbf 0))^{-1} (\mathbf f \cdot \mathbf R(\mathbf x)\cdot \mathbf V(\mathbf 0))
\\
&=&
	\changed{- (\mathbf f \cdot D\mathbf V(\mathbf x))^{-1} (\mathbf f \cdot \mathbf V(\mathbf x))}\\
&=&
	- (D\mathbf F(\changed{\mathbf x})^{-1} \mathbf F(\changed{\mathbf x})\\
	&=& - \changed{\mathbf x} + N(\mathbf F,\changed{\mathbf x})
\end{eqnarray*}
Items (b) and (c) are similar, so we just prove item (c).
For each $k \ge 2$,
\[
\begin{split}
\frac{1}{k!}
\left\|
(\mathbf f \cdot \mathbf R(\mathbf x) \cdot D\mathbf V(\mathbf 0))^{-1}
(\mathbf f \cdot \mathbf R(\mathbf x) \cdot D^k\mathbf V(\mathbf 0))
\right\|_{\mathbf 0}
=&\\
&\hspace{-2em}\frac{1}{k!}
	\left\| D\mathbf F(\changed{\mathbf x})^{-1} D^k\mathbf F(\changed{\mathbf x})
\right\|_{\mathbf 0}
.
\end{split}
\]
Taking $k-1$-th roots and taking the sup, we obtain
\[
	\gamma(\mathbf f \cdot \mathbf R(\mathbf x) , \mathbf 0) = \gamma(\mathbf F,\changed{\mathbf x}) 
\]
as stated.
\end{proof}

Lemma~\ref{toric-to-classical} immediately implies a renormalized version of the classical Smale's theorems on quadratic convergence of Newton iteration\cites{BCSS,Malajovich-nonlinear, Malajovich-UIMP} without the necessity of dealing with different metrics at different points like \ocite{toric1}. 
See also Theorems 2.1.1, 2.1.2 and references ibidem. 
\secrev{In this context, Smale's
quadratic convergence theorems can be rephrased:}

\begin{theorem}[$\gamma$-theorem]\label{th-gamma}
Let $\boldsymbol \zeta \in \secrev{\mathscr M}$ be a non-degenerate zero of $\mathbf f
	\cdot \mathbf V(\changed{\mathbf x})=\mathbf 0$.
If $\mathbf x_0 \in \secrev{\mathscr M}$ satisfies
\[
	\|\boldsymbol \zeta - \mathbf x_0\|_{\mathbf 0} \gamma(\mathbf f \cdot R( \boldsymbol \zeta), \mathbf 0) \le \frac{3 - \sqrt{7}}{2} ,
\]
then the sequence $\mathbf x_{i+1} =  N(\mathbf f \cdot R(\mathbf x_i), \mathbf 0)
	\secrev{+\mathbf x_i}$ is well-defined
and
\[
	\| \boldsymbol \zeta - \mathbf x_i\|_{\mathbf 0} \le 2^{-2^i+1} \|\boldsymbol \zeta - \mathbf x_0\|_{\mathbf 0} . 
\]
\end{theorem}
\begin{theorem}[$\alpha$-theorem]\label{th-alpha}
Let \[
\alpha 
\le
\alpha_0 = \frac{13 - 3 \sqrt{17}}{4},
\]
\[
	r_0 = \changed{r_0(\alpha)=}
\frac{1+\alpha-\sqrt{1-6\alpha+\alpha^2}}{4\alpha} 
\text{ and }
	r_1 =\changed{r_1(\alpha)=}
\frac{1-3\alpha-\sqrt{1-6\alpha+\alpha^2}}{4\alpha} 
.
\]
If $\mathbf x_0 \in \secrev{\mathscr M}$ satisfies 
$\alpha(\mathbf f \cdot R(\mathbf x_0),\mathbf 0) 
\defeq \beta(\mathbf f \cdot R(\mathbf x_0),\mathbf 0) \gamma(\mathbf f \cdot R(\mathbf x_0),\mathbf 0)
\le \alpha$, then the sequence defined recursively by
	$\mathbf x_{i+1} =  N(\mathbf f \cdot R(\mathbf x_i),\mathbf 0)	\secrev{+\mathbf x_i}$
is well-defined and converges to a limit $\boldsymbol \zeta$ so that
${\mathbf f} \cdot V(\boldsymbol \zeta)=\mathbf 0$. Furthermore,
\begin{enumerate}[(a)]
\item
$
		\| \mathbf x_i - \boldsymbol \zeta \|_{\mathbf 0} \le 2^{-2^i+1} \| \mathbf x_1 - \mathbf x_0 \|_{\mathbf 0}
$
\item 
$
		\| \mathbf x_0 - \boldsymbol \zeta\|_{\mathbf 0} \le r_0 \beta(\mathbf f \cdot R( \mathbf x_0),\mathbf 0)
$
\item
$
\| \mathbf x_1-\boldsymbol \zeta \|_{\mathbf 0}
		\le r_1 \beta(\mathbf f \cdot R( \mathbf x_0), \mathbf 0).
$
\end{enumerate}
\end{theorem}

\secrev{
	The $\gamma$ invariant is defined in terms of the higher derivatives of $\mathbf f$. Those are hard to evaluate, so it is desirable to bound $\gamma$ in terms of the condition number $\mu$. Such a bound is known in the dense, unitary invariant setting~\cite{BCSS}*{Theorem 2 p. 267}. It was generalized to
the toric setting:
\begin{theorem}
	\cite{toric1}*{Theorem 3.6.1} 
	\label{prop-mu-gamma} Let $\mathbf f \in \mathscr F$
	and $\mathbf x \in \secrev{\mathscr M}$. Then,
			$\gamma(\mathbf f,\mathbf x) \le \frac{1}{2} \mu(\mathbf f,\mathbf x) \nu(\mathbf x)$,
	\fourthrev{where $\nu(\mathbf x) = \max \nu_i(\mathbf x)$ was defined in \eqref{distortion}.}
\end{theorem}
}
\secrev{
\begin{convention}
From now on, we will omit the argument $\mathbf 0$ whenever 
	it is obvious from context.
	For instance, $\beta(\mathbf f) = \beta(\mathbf f, \mathbf 0)$,
	$\gamma(\mathbf f) = \gamma(\mathbf f, \mathbf 0)$, $\mu(\mathbf f)=\mu(\mathbf f,\mathbf 0)$, $\nu=\nu(\mathbf 0)$.
\end{convention}
A particular case of Theorem~\ref{prop-mu-gamma} can be written as:
\begin{equation}\label{eq-mu-gamma}
	\gamma(\mathbf f) \le \frac{1}{2} \mu(\mathbf f) \nu .
\end{equation}
	}

\section{Statement of main results}
\label{sec:results}
\subsection{Homotopy}
\changed{We can now properly state the main results in this paper. We start
by the renormalized homotopy algorithm. The input is a path $(\mathbf q_{\tau})_{\tau \in [t_0,T]}$ of sparse exponential sums and an {\em approximate root} of the initial system $\mathbf q_{t_0}$ associated to an exact root $\mathbf z_{t_0}$. The output is an approximate root
of the target system $\mathbf q_T$ associated to the root $\mathbf z_T$ such that a continuous lifting $\mathbf z_t$, $\mathbf q_t (\mathbf z_t) \equiv \mathbf 0$
exists.}

Loosely speaking, approximate roots of $\mathbf f$ are points $\mathbf x_0$ 
satisfying the 
the conclusions of Theorem~\ref{th-gamma} (approximate roots of the {\em second kind})
or of Theorem~\ref{th-alpha}(a) (approximate roots of the {\em first kind}). 
\changed{For algorithmic purposes} it is useful to have an effective certification of the hypotheses
of either \changed{theorem. This} can be achieved by replacing the invariant 
$\secrev{\gamma(\mathbf f \cdot R(\mathbf x_0))=}
\gamma(\mathbf f \cdot R(\mathbf x_0),\mathbf 0)$ 
by its upper bound
$\secrev{\frac{1}{2} \mu(\mathbf f \cdot R(\mathbf x_0)) \secrev{\nu} =}
\frac{1}{2} \mu(\mathbf f \cdot R(\mathbf x_0),\mathbf 0) \secrev{\nu}(\mathbf 0)$ 
from
\secrev{\eqref{eq-mu-gamma}, where $\nu$ was defined in \eqref{distortion}}, \changed{whence:
\begin{definition} \label{def-approximate-root} A {\em certified approximate root} of $\mathbf f \in \mathscr F$ 
	is some $\secrev{\mathbf x} \in \secrev{\mathscr M}$ satisfying
\[
	\frac{1}{2} \secrev{\beta(\mathbf f \cdot R(\mathbf x)) \mu(\mathbf f \cdot R(\mathbf x)) \secrev{\nu}}
\le \alpha
.
\]
for some $0<\alpha \le \alpha_0$.
Let $(\mathbf x_r)$ be the sequence of renormalized Newton iterates
of $\mathbf x$, viz. $\mathbf x_0=\mathbf x$ and
$\mathbf x_{r+1} = N( \mathbf f \cdot R(\mathbf x_r), \mathbf 0) + \mathbf x_r$.
We say that $\boldsymbol \zeta = \lim_{r \rightarrow \infty} \mathbf x_r$
	is the \secrev{root associated to} $\mathbf x$.
\end{definition}}

\begin{definition}\label{def-recurrence}
Let $(\mathbf q_{\tau})_{\tau \in [t_0, T]}$ be a path in 
$\mathscr F = \mathscr F_{A_1} \times \cdots \times \mathscr F_{A_n}$ and
	$\mathbf x_0 \in \secrev{\mathscr M}$ be a \changed{certificate approximate
	root} for $\mathbf q_{t_0}$. 
	\changed{Let $\secrev{\nu}=\nu(\mathbf 0)$ be the value of the distortion invariant at the origin.}
The {\em renormalized} homotopy at the origin is given by the recurrence:
\begin{equation}\label{rec-zero}
\left\{
\begin{array}{rcl}
	\mathbf x_{j+1} &=& N( \mathbf q_{t_j} \changed{\cdot} \mathbf R(\mathbf x_{j}), \mathbf 0) + \mathbf x_j 
\\
t_{j+1} &=& \min\left(\rule{0pt}{2.6ex}T, \inf \left\{\rule{0pt}{2.2ex} t>t_j:
\frac{1}{2} 
	\secrev{\beta( \mathbf q_{t} \changed{\cdot} \mathbf R(\mathbf x_{j+1}))}
\right. \right.\\
&&\hspace{9em} 
	\left. \left.
	\times \ 
	\secrev{\mu( \mathbf q_{t} \changed{\cdot} \mathbf R(\mathbf x_{j+1}))}\ \secrev{\nu} 
\ge \alpha_* \rule{0pt}{2.2ex} \right\}\rule{0pt}{2.6ex} \right) .
\end{array}
\right.
\end{equation}
\end{definition}
\changed{This recurrence hopefully produces a mesh of points $(t_j, \mathbf x_j)$, $j=1, \dots, N$,
$t_N=T$,
approximating the values of a continuous solution path $(\mathbf z_{\tau})_{\tau \in [t_0,T]}$ to the 
homotopy path $(\mathbf q_{\tau})_{\tau \in [t_0,T]}$. The number $N$ of 
\secrev{{\em renormalized homotopy}} steps will be bounded in terms
of invariants associated to the {\em renormalized} homotopy path
$(\mathbf q_{\tau} \cdot \mathbf R(\mathbf z_{\tau}))_{\tau \in [t_0,T]}$ with a fixed solution
at the origin $\mathbf 0 \in \secrev{\mathscr M}$. The main invariant is the {\em renormalized condition length} defined
below,} \secrev{not to be confused with the {\em ordinary} condition length $\mathcal L$ as in~\eqref{lold}.}

\begin{definition}\label{def-rencondlength}
	Let $(\mathbf q_{\tau}, \mathbf z_{\tau})_{\tau \in [t, t']}$ be a path in the solution variety $\secrev{\mathscr S}$. 
Its {\em renormalized condition length} is defined by:
	\secrev{
\[
\glsdisp{L}{\mathscr L((\mathbf q_{\tau}, \mathbf z_{\tau});\, t,t')}
= \int_{t}^{t'} 
\left(
\left\| \frac{\partial}{\partial \tau} \mathbf p_{\tau} \right\|_{\mathbf p_{\tau} }
+\secrev{\nu} \|\dot {\mathbf z}_{\tau}\|_{\mathbf 0} \right)
\mu( \mathbf p_{\tau}) \dd \tau
\]
	where $\mathbf p_{\tau}= \mathbf q_{\tau} \cdot R(\mathbf z_{\tau})$ is
the renormalized path,
	\changed{multi-projective Fubini-Study metric is assumed for the first summand} 
\[
	\left\| \frac{\partial}{\partial \tau} \mathbf p_{\tau}\right\|_{\mathbf p_{\tau} }
=
	\sqrt{
	\sum_{i=1}^n
	\frac{\left\| P_{\mathbf p_{i\tau}^{\perp}} \dot {\mathbf p}_{i\tau} \right\|^2}
	{\left\| \mathbf p_{i\tau} \right\|^2 }}
\]
with $P_{\mathbf p_{i\tau}^{\perp}}$ the orthogonal projection onto 
	$\mathbf p_{i\tau}^{\perp} \subseteq \mathscr F_{A_i}$. The notation
$\|\cdot\|_{\mathbf 0}$ stands for the $T_{\mathbf 0}\secrev{\mathscr M}$-metric, and $\secrev{\nu} = \nu(\mathbf 0)$ is fixed}. 
\end{definition}
Definition \ref{def-rencondlength} makes sense for $-\infty \le t \le t' \le \infty$. For instance, we can consider a linear homotopy $\mathbf q_t = \mathbf f + t \mathbf g$ for $0 \le t \le \infty$. This line projects onto a finite path
in  $\mathbb P(\mathscr F_{A_1}) \times \cdots \times \mathbb P(\mathscr F_{A_n})$.
If the condition number $\secrev{\mu(\mathbf q_t \cdot R(\mathbf z_t))}$ is bounded for all $t$ and for all solutions
paths $\mathbf z_{t}$, then the renormalized condition length will be finite (see Section~\ref{sec-linear}).
Of course, it may happen to the condition length to be infinite.
The number of \secrev{renormalized} homotopy steps required by Definition~\ref{def-recurrence} can be bounded in terms of the condition length:

\begin{theorem}\label{th-A}
There are constants $0< \alpha_* \simeq 0.074\dots < \alpha_0$, $0<u_*=u_*(\alpha_*) \simeq 0.129\dots$,
$0< \delta_*=\delta_*(\alpha_*)\simeq0.085\dots$ 
with the following properties. Let $-\infty \le t_0 < T \le \infty$. For any path
	$(\mathbf q_t)_{t \in [t_0,T]}$ of class \changed{$\mathcal C^{1}$} in
$\mathscr F = \mathscr F_{A_1} \times \cdots \times \mathscr F_{A_n}$ and for any
$x_0 \in \mathbb C^n$, if
the pair $(\mathbf q_{t_0}, \mathbf x_{0})$ 
satisfies
\begin{equation}\label{thA-a0}
\frac{1}{2} 
	\secrev{
\beta( \mathbf q_{t_0} \cdot \mathbf R(\mathbf x_0))\ \mu( \mathbf q_{t_0} \cdot \mathbf R(\mathbf x_0) )\ 
	\nu} 
\le \alpha_*,
\end{equation}
	the recurrence \eqref{rec-zero} \secrev{for the sequence $(t_j, \mathbf x_j)$} is well-defined and there
	is a \changed{$\mathcal C^{1}$} path $(\mathbf q_t, \mathbf z_t) \in \secrev{\mathscr S}$ 
satisfying, for all $t_j \le t < t_{j+1}$, 
\begin{equation}\label{thA-ui}
u_j(t) \defeq \frac{1}{2} 
\| \mathbf z_t - \mathbf x_{j+1}\|_{\mathbf 0}\ 
	\secrev{
	\mu( \mathbf q_{t} \cdot \mathbf R(z_t) ) 
		\nu} 
\le u_*.
\end{equation}
Moreover, 
	$\mathscr L(\changed{(\mathbf q_{t}, \mathbf z_t);}\, t_j, t_{j+1}) \ge \delta_*$ whenever $t_{j+1}<T$. 
	In particular, whenever $\mathscr L(\changed{(\mathbf q_{t}, \mathbf z_t);}\,t_0,T)$ is finite, there is $N \in \mathbb N$ with
$t_N = T$ and 
\[
	N \le 1+ \frac{1}{\delta_*} \mathscr L(\changed{(\mathbf q_{t}, \mathbf z_t);}\,t_0,T).
\]
If we set 
\[
\mathbf x_{N+1} = N( \mathbf q_{T} R(\mathbf x_{N}), \mathbf 0) + \mathbf x_N 
\]
then
\begin{equation}\label{thA-aN}
\frac{1}{2} 
	\secrev{
	\beta( \mathbf q_{T} \cdot \mathbf R(\mathbf x_{N+1}))\ \mu( \mathbf q_{T} \cdot \mathbf R(\mathbf x_{N+1}))\ 
		\nu} 
\le \alpha_*,
\end{equation}
\end{theorem}
In conclusion, the recurrence terminates after at most
\[
	1+ \frac{1}{\delta_*} \mathscr L(\changed{(\mathbf q_t, \mathbf z_t);}\, t_0,T)
\]
iterations. With an extra iteration more, we recover a {\em \secrev{certified} approximate root} of
$\mathbf q_T$, in the sense of \changed{Definition~\ref{def-approximate-root}}.

Theorem~\ref{th-A} provides a way to produce \secrev{certified} approximate roots for $\mathbf q_T$ from the knowledge of 
\secrev{certified} approximate roots of $\mathbf q_0$. \secrev{Certified} approximate roots with different associated roots
of $\mathbf q_0$ give rise to \secrev{certified} approximate roots for $\mathbf q_T$ with different associated roots.
In case the renormalized condition length associated to all homotopy paths is finite, this
allows to `approximately solve' $\mathbf q_T$ in terms of a maximal set of `\secrev{certified} approximate solutions' 
to $\mathbf q_0$. The renormalized condition length will be infinite if for instance one of the roots
is degenerate or at `toric infinity'.
Theorem~\ref{th-A} is proved in Section~\ref{sec-homotopy} below.
\begin{remark}
The word {\em renormalization} is taken here in the sense of dynamical
systems or cellular automata. Strictly speaking, the renormalization
operator should be allowed to be time-dependent as in the renormalized
Graeffe iteration by Malajovich and Zubelli \ycites{MZ1,MZ2}. 
\end{remark}
\changed{
\begin{remark}
The solution path $(\mathbf q_t, \mathbf z_t)$ in Theorem~\ref{th-A}
has the same regularity as $\mathbf q_t$. If one assumes class
$\mathcal C^{1 + \text{Lip}}$, that is a Lipschitz first derivative of
$(\mathbf q_t, \mathbf z_t)$ with respect to $t$, then $\mathbf z_t$ also turns out
to have Lipschitz first derivative. This is the case in particular
for geodesics of the Lipschitz-Riemann structure induced by the
condition metric as explained by \ocite{BDMS2}*{section 2}.
\end{remark}}

\subsection{Expectation of the renormalized condition number}

Theorem~\ref{th-A} reduces the problem of obtaining
a global complexity estimate to \changed{the problem of
bounding the
condition length from} Definition~\ref{def-rencondlength}. We will take
a random homotopy path, with one of the endpoints fixed. Before computing
its renormalized condition length, we will need to compute the expected
squared condition number for a given time $\tau$. We will actually obtain
a {\em conditional} expectation. 
This will be enough
to produce an algorithm with a bounded absolute expectation, as it
will be explained in section~\ref{sec:unconditional}.
This bound on the conditional expectation of the squared condition number depends 
on a generalization of
Minkowski's {\em mixed volume}, that we call the {\em mixed \secrev{area}}. 
We recall the definition of mixed volume first.
\secrev{Some references on mixed volume are
\cites{Bernstein, Lutwak, Khovanskii, GKZ, Malajovich-nonlinear, Malajovich-mixed, Jensen}.}

\begin{figure}
\centerline{
	\resizebox{\textwidth}{!}{ 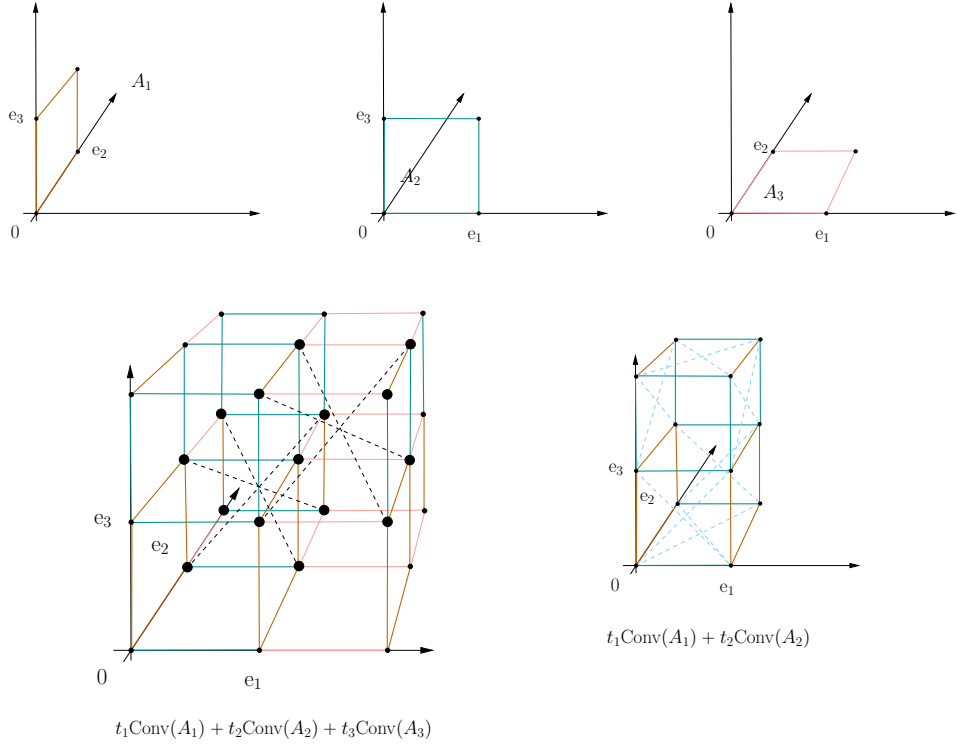} }
\caption{The mixed volume and the mixed \secrev{area}. Top, the support for the system of polynomials
	$F_1(X_1, X_2, X_3) = 1 + X_2 + X_3 + X_2 X_3$,
	$F_2(X_1, X_2, X_3) = 1 + X_1 + X_3 + X_1 X_3$,
	$F_3(X_1, X_2, X_3) = 1 + X_1 + X_2 + X_1 X_2$. Each support is represented in a different
	color. Bottom left, the Minkowski linear combination
	of the supports. Only the two cubes with edges of all colors, aka `mixed cells', have a volume term in $t_1 t_2 t_3$. The mixed volume $V= \frac{1}{3!} 2 = 1/3$, and this means that if the coefficients are replaced by generic coefficients the system has 2 roots in $\mathbb C_{\times}^3$. Bottom right, the cells with surface multiple of $t_1 t_2$ are one of the parcels in the mixed \secrev{area}. In this example there are 6 of them, therefore $V(\conv{A_1}, \conv{A_2}, B^3) = 1$ and by permuting supports, the mixed \secrev{area} $V'$ is equal to $3$. 
\label{mixed-surface}}
\end{figure}

\begin{definition}\label{def-mixed} The {\em mixed volume} of an $n$-tuple 
	of compact
convex sets
$(\mathcal A_1,$ $\dots, \mathcal A_n)$ 
in $\mathbb R^n$ is
\[
	\glsdisp{MixedVolume}{V(\mathcal A_1, \dots, \mathcal A_n)}
\defeq
\frac{1}{n!}\ 
\frac{\partial^n}
{\partial t_1 \partial t_2 \cdots \partial t_n}
\
	\vol_n(t_1 \mathcal A_1 + \cdots + t_n \mathcal A_n)
\]
where $t_1, \dots, t_n \ge 0$ and the derivative is taken at
$t_1=\dots=t_n=0$. 
\end{definition}
 The normalization factor $1/n!$ ensures that
\[
\vol_n(\mathcal A) = V(\mathcal A, \dots, \mathcal A)  
.
\]
\changed{
\begin{remark}\label{properties-mixed-volume}
The mixed volume is also known to be monotonic, symmetric, translation
invariant and linear in each $\mathcal A_i$ with
respect to Minkowski linear combinations. Those five properties also
define the mixed volume. 
\end{remark}}
\secrev{Bernstein's Theorem} 
can be stated in terms of polynomials or exponential sums, we state it
here in terms of exponential sums.
\begin{theorem}[\citeauthor{Bernstein}, \citeyear{Bernstein}]
\label{BKK}
Let $A_1, \dots, A_n$ be finite subsets of $\mathbb Z^n$, 
$\mathcal A_i = \conv{A_i}$, $i=1, \dots, n$, and let
$\mathbf f_i \in \mathscr F_{A_i}$. Then, the system
\begin{eqnarray*}
\mathbf f_1 V_{A_1}(\mathbf z) &=& 0 \\
	& \vdots & \\
\mathbf f_n V_{A_n}(\mathbf z) &=& 0 \\
\end{eqnarray*}
has at most $n! V(\mathcal A_1, \dots, \mathcal A_n)$ isolated
roots in $\mathbb C^n \mod 2 \pi \sqrt{-1}\,\mathbb Z^n$
\changed{counted according to their multiplicity}. This bound
is attained for generic $\mathbf f$.
\end{theorem}
\noindent
This statement is equivalent to the polynomial version, because
\[
	\defun{\exp}{\mathbb C^n \mod 2 \pi \sqrt{-1}\,\changed{\mathbb Z^n}}
	{\mathbb C_{\times}^n}{\mathbf z}{\exp(\mathbf z) = 
	\begin{pmatrix}e^{z_1} & \dots & e^{z_n}\end{pmatrix}}
\]
is a bijection.
We saw in section \ref{sec:toric} that the natural map $\mathbb C^n \mod 2 \pi \sqrt{-1}\, \mathbb Z^n
\rightarrow \secrev{\mathscr M}$ \secrev{maps} $(\det \Lambda)$ to 1. \changed{In the statement above, the multiplicity of all isolated roots is a multiple of $\det \Lambda$. We may restate Theorem~\ref{BKK} by counting roots in $\secrev{\mathscr M}$ instead, so we can drop the multiplicities}:

\begin{theorem}\label{BKK2}
Let $A_1, \dots, A_n$ be finite subsets of $\mathbb C^n$, 
$\mathcal A_i = \conv{A_i}$, $i=1, \dots, n$, and let
$\mathbf f_i \in \mathscr F_{A_i}$. Then, the system
\begin{eqnarray*}
\mathbf f_1 V_{A_1}(\mathbf z) &=& 0 \\
	& \vdots & \\
\mathbf f_n V_{A_n}(\mathbf z) &=& 0 \\
\end{eqnarray*}
has at most $n! V(\mathcal A_1, \dots, \mathcal A_n)/(\det \Lambda)$ isolated
roots in $\secrev{\mathscr M}$. This bound
is attained for $\mathbf f$ generic.
\end{theorem}
\begin{remark}
The bound in Theorem ~\ref{BKK2} is basis invariant in the following
sense: if one replaces the $A_i$ by $A_i M$ for an arbitrary 
matrix $M$ with integer coefficients, invertible over $\mathbb Q$,
then the number $ \frac{n! V(\mathcal A_1, \cdots, \mathcal A_n)}
{\det \Lambda}$ does not change \secrev{since both numerator and denominator
	are multiplied by $\det(M)$.}
\end{remark}

\medskip
\par
The mixed volume was originally defined by \ocite{Minkowski} in connection with
surface and curvature of convex bodies. 
Assume that $\mathcal A \subseteq \mathbb R^n$ is a compact convex body with
smooth boundary. Its {\em surface} or $n-1$ dimensional volume of its boundary
$\partial A$ is given by 
\begin{equation}\label{eq-area}
	V'= V'(\mathcal A) = \partialat{\epsilon}{0}
\vol_n(\mathcal A + \epsilon B^n) =
n V(\mathcal A, \dots, \mathcal A, B^n) 
\end{equation}
where $B^n$ stands for the unit Euclidean $n$-ball.
\begin{figure}
\centerline{
	\resizebox{0.5\textwidth}{!}{ 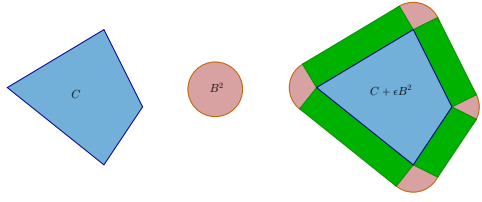} }
\caption{The Steiner polynomial of a convex body $C$ is the volume of $C + \epsilon B^n$. In two dimensions,
	$\vol_2(C+ \epsilon B^2) = \vol_2(C) + 2 \epsilon \mathrm{Length}(\partial C)+ \pi \epsilon^2$. With the
	proper normalization, the coefficients of this polynomial are known as the {\em intrinsic volumes} of a convex polytope.
\label{fig-steiner}}
\end{figure}

A generalization of Minkowski's surface {\em quermassintegral} turns out
to be an important invariant for homotopy continuation in the sparse case,
namely
\begin{equation}\label{eq-mixed-area}
	\begin{split}
		\glsdisp{MixedSurface}{V'(\mathcal A_1, \dots, \mathcal A_n)}
		&= \partialat{\epsilon}{0}
V(\mathcal A_1+\epsilon B^n, \dots, \mathcal A_n+\epsilon B^n)
\\
&= 
\sum_i
V(\mathcal A_1, \dots, \stackrel{i\text{-th}}{B^n}, \dots, \mathcal A_n).
	\end{split}
\end{equation}
This quermassintegral \changed{is} called {\em mixed \changed{area}} in analogy with the
ordinary \changed{area} of the boundary of a convex set. Definitions
\eqref{eq-area} and \eqref{eq-mixed-area} coincide when
$\mathcal A_1 = \dots = \mathcal A_n = \mathcal A$. Another manifestation
of this invariant
arises 
when one of the supports in replaced by  a unit
simplex $\Delta_{n} = \{\mathbf 0, \mathrm e_1, \dots, \mathrm e_n\}$ 
as a result of eliminating one variable,
see for instance \cite{HJS} and references.
Recall that $\Delta_n$ has circumscribed radius $\sqrt{1-1/n}$ so that
\changed{$\Delta_n \subseteq \sqrt{1-1/n}\, B^n$. By the monotonicity and 
multilinearity properties of the mixed volume,} \secrev{Equation \eqref{eq-mixed-area} implies that}
\[
\sum_i V(\mathcal A_1, \dots, \stackrel{i\text{-th}}{\conv{\Delta_n}}, \dots, \mathcal A_n)
\le
\sqrt{1-\frac{1}{n}}\, V' \le V'.
\]
\changed{\begin{remark}\label{rem-Aleksandrov}
The mixed area quermassintegral is closely related to Aleksandrov's mixed
area measure \cite{Aleksandrov}. This is a measure on the unit sphere
that depends on 
	$n-1$ convex bodies. \secrev{Aleksandrov defined a volume form $\dd F_{\mathcal A_1, \dots, \mathcal A_{n-1}}$ on the unit sphere. It has the property that for every convex body $\mathcal B$ with support function
	$\lambda(\mathcal B)$,} 
	\[
		V(\mathcal A_1, \dots, \mathcal A_{n-1}, \mathcal B)
		=
		\frac{1}{n}\int_{S^{n-1}} \lambda_{\mathcal B}(\Theta)
		\ \dd F_{\mathcal A_1, \dots, \mathcal A_{n-1}}(\Theta)
.	\]
	If we set $\mathcal B$ equal to the unit ball $B^n$, the support function is constant and equal to one. The area $V'$ in \eqref{eq-area} is just
	the integral of the area measure $\secrev{\dd F}_{\mathcal A, \dots, \mathcal A}$
	while the mixed area $V'(\mathcal A_1, \dots, \mathcal A_n)$ in \eqref{eq-mixed-area} is the average over all subsets of $n-1$ convex bodies
of the integral of the appropriate mixed area measure.
References and more recent results related to mixed areas can be found in the
survey by \ocite{Lutwak}.
\end{remark}
}

\begin{example} \label{isoperimetric}Suppose that
the convex sets have the same {\em shape}, say 
$\mathcal A_i = d_i \mathcal A$ for $d_i > 0$. Then,
\[
V=V(\mathcal A_1, \dots, \mathcal A_n) = 
	\secrev{\left( \prod_{i=1}^n d_i \right)} \vol_n{\mathcal A}
\]
	and \changed{from the monotonicity property,}
\[
\left( \min_{1\le j \le n} \prod_{i \ne j} d_i \right) V'(\mathcal A)
\le
V'(\mathcal A_1, \dots, \mathcal A_n) 
\le \left( \max_{1 \le j \le n} \prod_{i \ne j} d_i \right) V'(\mathcal A)
.
\]
In this example,
\[
\frac{V(\mathcal A_1, \dots, \mathcal A_n)}
{V(\mathcal A) \max d_j}
\le
\frac{V'(\mathcal A_1, \dots, \mathcal A_n)}{V'(\mathcal A)}
\le
\frac{V(\mathcal A_1, \dots, \mathcal A_n)}
{V(\mathcal A) \min d_j}
\]
	and the isoperimetric inequality \cite{Khovanskii} $V'(\mathcal A) \ge n V(\mathcal A)^{\frac{n-1}{n}} \vol_{n}(B^n)^{\frac{1}{n}}$
provides the bound
\[
V(\mathcal A_1, \dots, \mathcal A_n) \le 
\frac{(\max d_i) \sqrt{\pi} V(\mathcal A)^{1/n}}{n \Gamma(n/2)^{1/n}}
V'(\mathcal A_1, \dots, \mathcal A_n)
\]
where $\Gamma(n/2)^{1/n} \simeq \sqrt{\frac{n}{2e}}$.

\end{example}
\begin{example}
	In the specific case when \secrev{$\mathcal A_i = d_i \conv{\Delta_n}$},
$n! V = \prod d_i$ is the Bézout number and 
$n! V' = (n^2 + n\sqrt{n+1}) \sum_i \prod_{j \ne i} d_j$.
\end{example}

\begin{example}
Let $p,q \in \mathbb N$ be relatively prime. By Euclid's algorithm there are $r,s \in \mathbb Z$ so that
$pr+qs=1$. Let $A_1=\{ [0,0], [p,q]\}$,$A_2=\{[0,0],[-s,r] \}$ and $\mathcal A_i = \conv{A_i}$. 
Then $V(\mathcal A_1, \mathcal A_2)=\frac{1}{2}$ and $V'(\mathcal A_1, \mathcal A_2)= 
\sqrt{p^2+q^2} + \sqrt{r^2+s^2}$. We see from this example that $V'/V$ can be arbitrarily large.
\end{example}

\medskip
\par

We assume from now on that the mixed volume $V(\mathcal A_1, \dots, \mathcal A_n)$ is non-zero, and in particular $\Lambda$ has rank $n$.
Then we consider a random Gaussian complex polynomial system $\mathbf q$,
$\mathbf q_i \in \mathscr F_{A_i}$. For simplicity let
$\secrev{\mathbf{f_i}} = \mathbb E(\mathbf q_i)$ be the average, $\mathbf g_i = \mathbf q_i-\secrev{\mathbf {f_i}}$ and 
$\Sigma_i^2 = \mathbb E( \mathbf g_i^* \mathbf g_i)$ be the covariance matrix. 
\changed{In this paper, the covariance matrix is assumed to be diagonal and 
positive definite.} Recall
that $\mathbf g_i$ is a covector, so $\mathbf g_i^*$ is a vector \secrev{and} $\Sigma_i^2$ is
indeed \changed{a real diagonal} matrix. In this sense, $\mathbf q_i \sim \secrev{\mathcal N}(
\secrev{\mathbf f_i}, \Sigma_i^2;
\mathscr F_{A_i})$. For short,
\[
\mathscr F = \mathscr F_{A_1} \times \cdots \times \mathscr F_{A_n},
\]
\[
	\mathbf q \sim \gls{Normal}
\hspace{2em}
\text{and}
\hspace{2em}
\mathbf g = \mathbf q-\secrev{\mathbf f} \sim \secrev{\mathcal N}(\mathbf 0, \Sigma^2; \mathscr F)
\]
\secrev{We define the zero-set
\[
	\gls{zeroset} = \{
\mathbf z \in \secrev{\mathscr M}: \mathbf q \cdot \mathbf V(\mathbf z)=\mathbf 0 \}
\]
and the set of bounded zeros
\[
	\gls{boundedzeroset} \defeq \{ \mathbf z \in Z(\mathbf q): \| \Re(\mathbf z)\|_{\infty} \le H \}
\]
depending on a parameter $H$.}
For generic $\mathbf q$, Theorem~\ref{BKK2} implies that
\[
	\# Z_H(\mathbf q) \le 
\# Z(\mathbf q) \defeq \frac{n! V(\mathcal A_1, \cdots, \mathcal A_n)}
{\det \Lambda} .
\]

\secrev{
The following result bounds the integral of a non-scaled condition number
computed at the set of {\em bounded} zeros in terms of 
the mixed \changed{area}
\[
V'= \sum_{i=1}^n V( \stackrel{1\text{-st}}{\mathcal A_1}, \cdots, \stackrel{i\text{-th}}{B^n}, \cdots, \stackrel{n\text{-th}}{\mathcal A_n})
\]
defined in \eqref{eq-mixed-area}.
}

\begin{theorem}\label{E-M2} Let $A_1, \dots, A_n \in \mathbb R^n$
be such that the $\mathbb Z$-module $\Lambda$ generated by
$\bigcup_{i=1}^n A_i-A_i$ is a subset of $\mathbb Z^n$,
and suppose that the mixed volume $V(\conv{A_1}, \dots,$ $\conv{A_n})$
is non-zero. 
Let $\mathscr F = 
\mathscr F_{A_1} \times \cdots \times \mathscr F_{A_n}$.
Let $S_i=\# A_i = \dim_{\mathbb C}(\mathscr F_{A_i}) \ge 2$.
For each $i$, let $\glsdisp{sigma}{\Sigma_i}
\changed{=\diag{\sigma_{i\mathbf a}}_{\mathbf a \in A_i}}$, \changed{$\sigma_{i\mathbf a}>0$.}
Fix some real number $H>0$, and
denote by $Z_H(\mathbf q)$ be the set of isolated roots 
of $\mathbf q \in \mathscr F$ 
	in $\secrev{\mathscr M}$ with $\| \Re(\mathbf z)\|_{\infty} \le H$.
	Then, \secrev{for all $\mathbf f \in \mathscr F$},
	\[
		\expected{\mathbf q \sim \secrev{\mathcal N}(\secrev{\mathbf f}, \Sigma^2)} { \sum_{\mathbf z \in Z_H(\mathbf q)} \| M(\mathbf q ,\mathbf z)^{-1}\|_{\mathrm F}^2 }
\le
	\frac{2H\sqrt{n}}{\det(\Lambda)}
	\frac{1}{\min_{\mathbf a} \sigma_{k \mathbf a}^2}
	\thirdrev{n!} \, V' \secrev{.}
\]
\end{theorem}

\begin{remark}
The interest of varying the $\sigma_{i \mathbf a}$ while keeping
the $\rho_{i \mathbf a}$ fixed arises from the theoretical
analysis of non-linear homotopy paths. 
For instance, one can consider a Gaussian system
$\mathbf g$, and produce a non-linear homotopy by setting
	$\secrev{ q_{i \mathbf a}(t) \defeq } e^{-t b_{i\mathbf a}} g_{i \mathbf a}$
for random real coefficients $b_{i\mathbf a}$. 
This is equivalent to 
polyhedral homotopy
as described by  
\ocite{Verschelde-Verlinden-Cools} or 
\ocite{Huber-Sturmfels}. 
No {\em a priori} step count bound is known at this time.
This avenue of research will be pursued in a future paper.
\end{remark}

\begin{remark} Let $d \in \mathbb N$. If we replace each $A_i$ in Theorem~\ref{old-mainB} by $d A_i$, we will replace the mixed \changed{area} $V'$ by $d^{n-1} V'$. 
Moreover, $\det d \Lambda = d^n \det \Lambda$. To keep the same solution set,
we must replace $H$ by $H/d$. The right hand bound is therefore invariant. 
It is not, unfortunately, a lattice basis invariant.
\end{remark}

\changed{Theorem~\ref{E-M2} can be directly compared to Theorem 18.4 by \ocite{BC}. 
The left-hand-side of their formula is actually averaged
by the number of paths, that is by the Bézout number.} 
\secrev{Their result becomes:}
\begin{theorem}[\citeauthor{BC}, \citeyear{BC}]

	Assume that $A_1 =  \dots = A_n
	= \left\{ \mathbf a \in \mathbb N_0:
\sum a_j \le d\right \}$. Let $\rho_{i \mathbf a}^2 = \binomial{d}{\mathbf a}$. Let $\sigma > 0$.
Then,
\[
	\expected{\mathbf q \sim \secrev{\mathcal N}(\mathbf f, \sigma^2I)}{ \frac{
		\sum_{\mathbf z \in Z(\mathbf q)} \| M(\mathbf q ,z)^{-1}\|_2^2 }{\changed{d^n}}}
\le
	\frac{e(n+1)}{2 \sigma^2}
\]
\end{theorem}
while the particular case of Theorem~\ref{E-M2} would be
\[
	\expected{\mathbf q \sim \secrev{\mathcal N}(\secrev{\mathbf f}, \sigma^2I)}{\frac{\sum_{\mathbf z \in Z_H(\mathbf q)} \| M(\mathbf q ,z)^{-1}\|_{\mathrm F}^2 }{\changed{d^n}}}
\le
\secrev{O\left( \frac{H n^{\frac{\thirdrev{5}}{2}}}{d \sigma^2}\right)}
.\]
\changed{\begin{proof}
	\secrev{The surface \fourthrev{area} is
$V'=\frac{d^{n-1}}{\changed{(n-1)!}}(n+\sqrt{n})$.}
\end{proof}}
As we see, the price to pay for greater generality is of modest
$O(H\thirdrev{n^{\frac{3}{2}}}/d)$.
\smallskip

\changed{We will use Theorem~\ref{E-M2} to bound the expectation of
the renormalized condition number \secrev{in Theorem~\ref{old-mainB} below}. This more technical statement will
allow comparison to previous results in the subject.}
\changed{Recall that the condition number}
depends on the metric we choose on the $\mathscr F_{A_i}$'s. In this paper, the metric is specified by the choice of the coefficients $\rho_{i \mathbf a}$. 

\changed{We will need }
a distortion bound for these coefficients.
Let $A_i'~\secrev{\subseteq A_i}$ denote the set of vertices of $\conv{A_i}$. 
We set
\begin{equation}\label{eq-rho}
	\gls{kappa}
	\defeq \frac{
\sqrt{ \sum_{\mathbf a \in A_i} \rho_{i,\mathbf a}^2 }
}
{\min_{\mathbf a \in A_i'} \rho_{i,\mathbf a}}.
\end{equation}
\begin{remark}\label{rem-kappa-rho}
In the particular case $\rho_{i\mathbf a}=1$ for all $\mathbf a$, 
we have $\kappa_{\rho_i}=\sqrt{S_i}$. 
\end{remark}
\begin{remark}
	In the dense case of degree $d_i$
with the
coefficients of Example~\ref{ex:Weyl}, $\kappa_{\rho_i} = (n+1)^{\frac{d_i}{2}}$ so this distortion can be much larger than $\sqrt{S_i} = \binomial{d_i+n}{n}^{\frac{1}{2}}$.
\end{remark}

\secrev{We can now give a bound for the expected value of the squared,
renormalized condition number. Recall from our previous convention that
we write $\mu(\mathbf p)$ for $\mu(\mathbf p, \mathbf 0)$. According to the same convention, we will write $\gls{deltai} \defeq \delta_i(\mathbf 0)$
for the radius around $m_i(\mathbf 0)$. Let $V'$ denote the mixed area as in \eqref{eq-mixed-area}.}
\begin{theorem}\label{old-mainB} Let $A_1, \dots, A_n \in \mathbb R^n$
be such that the $\mathbb Z$-module $\Lambda$ generated by
$\bigcup_{i=1}^n A_i-A_i$ is a subset of $\mathbb Z^n$,
and suppose that the mixed volume $V(\conv{A_1}, \dots,$
	$ \conv{A_n})$
is non-zero. 
	Let $\secrev{\mathbf f \in }\mathscr F = 
\mathscr F_{A_1} \times \cdots \times \mathscr F_{A_n}$.
	Let $S_i=\# A_i = \dim_{\mathbb C}(\mathscr F_{A_i}) \ge 2$.
For each $i$, let $\Sigma_i
\changed{=\diag{\sigma_{i\mathbf a}}_{\mathbf a \in A_i}}$, \changed{$\sigma_{i\mathbf a}>0$.}
	Let $L = \max_i \frac{\| \mathbf f_i \Sigma_i^{-1} \|}{\sqrt{S_i}}$.
Then, for any fixed real number $H>0$,
\[
\begin{split}
\expected{\mathbf q \sim \secrev{\mathcal N}(\mathbf f, \Sigma^2; \mathscr F)}{
\sum_{\mathbf z \in Z_H(\mathbf q)} 
\secrev{\mu^2(\mathbf q \cdot R(z))}}
\hspace{-5em}&\hspace{5em} \le \\
	\le& \frac{2.5 e H\sqrt{n}}{\det(\Lambda)}
\left( 1 + 3L +\sqrt{\frac{\log(n)}{\min (S_i)}+2\log(3/2)}\right)^2
\\
		& \times	\frac{\max_{i}\left( S_i \kappa_{\rho_i}^2 \max_{\mathbf a \in A_i} (\sigma_{i, \mathbf a}^2)\right)}
{\min_{i,\mathbf a} (\sigma_{i, \mathbf a}^2)}
\left( \sum \delta_i^2 \right)
	\ \thirdrev{n!}\, V'
.
\end{split}
\]
\end{theorem}

\begin{remark} The choice of the coefficients $\rho_{i \mathbf a}$ allows
for more flexibility to the theory. Besides the trivial choice 
$\rho_{i \mathbf a} = 1$ and the Weyl metric $\rho_{i \mathbf a} =
\sqrt{\binomial{d}{\mathbf a}}$, another interesting possibility is
$\rho_{i \mathbf a} = |f_{i \mathbf a}|$ for a fixed
system $\mathbf f$. In this case, we recover a condition number
$\mu(\mathbf f, \mathbf z)$ (resp. a renormalized condition number
$\mu(\mathbf f \cdot \mathbf R(\mathbf z))$ with respect to
coefficientwise relative error. 
Numerical evidence supporting this choice was presented by
\ocite{MRMomentum}, in the context of the non-renormalized condition
number.
\end{remark}

\changed{If we assume
a centered,
uniform Gaussian distribution with coefficients $\rho_{i \mathbf a}=1$
instead, we obtain a
simplified statement of Theorem~\ref{old-mainB}.}
The
\secrev{expectation} grows mildly in terms of $n$ and $H$. However, it grows as
the square of the generalized degrees. 

\begin{corollary}\label{cor-mainB}
	Under the hypotheses of Theorem \ref{old-mainB} with $\rho_{i \mathbf a}=\rho_i$, \secrev{$\Sigma_i=I$ and $\| \mathbf f_i\|=\sqrt{S_i}$},
\[
\begin{split}
\expected{\mathbf g \sim \secrev{\mathcal N}(\mathbf 0, I; \mathscr F)}{
	\sum_{\mathbf z \in Z_H(\mathbf g)} \secrev{\mu^2(\mathbf g \cdot R(z))}}
\le
	\frac{2.5 e H\sqrt{n}}{\det(\Lambda)}
	\left( \thirdrev{4} +\sqrt{\frac{\log(n)}{\min (S_i)}+2\log(3/2)}\right)^2
\\
	\times \max_i (S_i^2)	
\left( \sum \delta_i^2 \right)
	\ \thirdrev{n!} V'
\end{split}
\]
\end{corollary}

\begin{example}[dense case]\label{ex-dense} 
Assume that $n, d \in \mathbb N$ and $A_i = \{ \mathbf a \in \mathbb N_0^n: \sum a_i \le d\}$
	In that case $V=\frac{d^n}{n!}$, $V'=\frac{d^{n-1}}{\changed{(n-1)!}}(n+\sqrt{n})$
	and $S_i=\binomial{n+d}{n}$.
	Choosing the center
	of gravity as the origin, $\delta_i = d \left(1-\frac{n+1}{(n+1)^2}\right) \le d$.
	\secrev{From Corollary \ref{cor-mainB}, we} obtain a bound of
\[
	O\left(H n^{\thirdrev{5}/2} d^{n+1} \binomial{d+n}{n}^2\right).
\]
\end{example}

The only known results that are vaguely similar to Theorem~\ref{old-mainB} are Theorems~1 and 5 by \ocite{MRHigh}. 
The mixed case \fourthrev{there} (Theorem 5) depends on a quantity called the {\em mixed dilation}. This is equal to $1$
in the unmixed case, but the mixed dilation is hard to bound in general.
The definition of the condition number \secrev{there} is different but coincides up to
scaling, in the unmixed case, with the definition here:
what appears as $\mu$  
is actually $\sqrt{n} \mu$ in this paper.
Volumes are also differently scaled.
There is no renormalization, so there is no need for $H$. We obtain an imperfect comparison to
Theorem 1 of \ocite{MRHigh}, after rescaling:

\begin{theorem} Let $A_1 = \dots = A_n$ and $\rho_{i \mathbf a} = 1$.
Then,
\[
	\probability{\mathbf g \sim \secrev{\mathcal N}(\mathbf 0, I; \mathscr F)}{\max_{\mathbf z \in Z(\mathbf g)} \secrev{\mu(\mathbf g, \mathbf z)} > \epsilon^{-1}}
\le
	n(n+1) (S_i-1)(S_i-2)
	\ n! V \epsilon^{4}
.
\]
\end{theorem}

Recall that $\mu(\mathbf g \secrev{,\mathbf z}) \ge 1$. Integrating with respect to $t=\epsilon^{-2}$, we recover
\begin{eqnarray*}
\expected{\mathbf g \sim \secrev{\mathcal N}(\mathbf 0, I; \mathscr F)}{
	\max_{\mathbf z \in Z(\mathbf g)} \mu^2(\mathbf g \secrev{, \mathbf z})}
	&\le&
n(n+1) (S_i-1)(S_i-2)
	\ n! V \int_1^{\infty} t^{-2} \ \dd t
\\
	&=&
n(n+1) (S_i-1)(S_i-2)
\ n! V 
\end{eqnarray*}
and none of the bounds implies the other.

\subsection{On infinity}\label{infinity}

Theorem~\ref{old-mainB} suggests that roots with large infinity
norm are a hindrance to renormalized homotopy. There are two obvious
remedies. One of them is to change coordinates `near infinity' and use
another sort of renormalization.  
The other remedy is to show that roots with a large infinity
norm have low probability. In this paper we pursue the
latter choice.

Recall that the toric variety was defined as the Zariski closure
\[
\mathcal V = 
\overline{\{[\mathbf V(\mathbf x)]:\mathbf x \in \secrev{\mathscr M}\}}
=
\overline{\{[(V_1(\mathbf x)], \dots, [V_n(\mathbf x)]):\mathbf x \in \secrev{\mathscr M}\}}
\subseteq \mathbb P( \mathscr F_{A_1}) \times \cdots \times \mathbb P( \mathscr F_{A_n}).
\]
Points of the form $[\mathbf V(\mathbf x)]$, $\mathbf x \in \secrev{\mathscr M}$ are
deemed {\em finite}, all the other points are said to be at {\em toric infinity}.
\begin{example}[Linear case]
Assume that $A_i = \Delta_n = \{ \mathbf 0, \mathrm e_1, \dots, \mathrm e_n\}$ 
for $i=1, \dots, n$. 
Then for any nonempty proper subset $B$ of $A_i$ we set
$y_j=0$ for $j \in B$, $y_j=-1$ for $j \not \in B$. For any $1 \le j \le n$,
choose $-\pi < \omega_j \le \pi$. Define
\[
\mathbf x(t) = \begin{pmatrix} 
	(y_1 - y_0)t + \omega_1 \sqrt{-1}\\
	(y_2 - y_0)t + \omega_2 \sqrt{-1}\\
\vdots\\
	(y_n - y_0)t + \omega_n \sqrt{-1}
\end{pmatrix}
\]
	so the point $[\mathbf w]=\lim_{t \rightarrow +\infty} [\mathbf V(\mathbf x(t))]$
	is a point at toric infinity.
The reader can easily check that each choice
of $B$ defines a different set of points
	at toric infinity in \changed{$\mathcal V$}, 
with $w_{\mathbf a} \ne 0$ if and only if $\mathbf a \in B$. 
\end{example}

In order to clarify what does `toric infinity' looks like in general, 
we may introduce the \changed{Minkowski support function}
\secrev{of $\conv{A_i}$,}
\begin{equation}\label{def-lambdai}
\defun{\gls{lambdai}}{\mathbb R^n}{\mathbb R}{\boldsymbol\xi}{
\lambda_i(\boldsymbol \xi) \defeq \max_{\mathbf a \in A_i} \mathbf a \boldsymbol\xi.}
\end{equation}
Always, $\mathbf a \boldsymbol\xi - \lambda_i(\boldsymbol \xi) \le 0, \ \forall \mathbf a \in A_i$. The convex closure $\conv{A_i}$ is the intersection of all half-spaces 
$\mathbf x \boldsymbol\xi - \lambda_i(\boldsymbol \xi) \le 0$
for $\boldsymbol \xi \in S^{n-1}$. Its `supporting' facet  in the direction
$\boldsymbol \xi$ is $\conv{A_i^{\boldsymbol \xi}}$ for the subset
\[
	\gls{Aixi} 
	= \{ \mathbf a \in A_i: \lambda_i(\mathbf \xi) - \mathbf a \xi = 0\}
.
\]
We also define \changed{a `second largest value' function or {\em facet gap}}
\[
	\defun{\eta_i}{\mathbb R^n}{\mathbb R}{\boldsymbol \xi}{\eta_i(\boldsymbol \xi) \defeq \min_{\mathbf a \in A_i \setminus A_i^{\boldsymbol \xi}}
\ \lambda_i(\boldsymbol \xi)-\mathbf a \boldsymbol\xi}
\]
Each function $\eta_i$ is \secrev{upper} semi-continuous, with $\inf_{\boldsymbol \xi \in S^{n-1}} \eta_i(\boldsymbol \xi)=0$.
Yet, the $\eta_i$ will provide us with an important
invariant to assess the `quality' of a tuple of supports
$(A_1, \dots, A_n)$.

\changed{For explanatory purposes, we consider first the unmixed situation
$A=A_1=\dots = A_n$, $\eta=\eta_i$. Suppose that $\conv{A_i}$ is not
contained in \secrev{a} hyperplane. Then, let $F_0$ be the set of outer \secrev{unit} normal
vectors to the $n-1$-facets of $\conv{A}$. The number 
$\eta=\min_{\boldsymbol \xi \in F_0} \eta(\boldsymbol \xi)$ is 
the smallest distance from the tangent space to
an ($n-1$)-dimensional face 
to some $\mathbf a \in A$
not in that tangent space. In particular $\eta > 0$.}
\changed{\begin{remark}\label{eta-bounds-by-Cramer}
An anonymous referee suggested a lower bound on $\eta$ based
	on the {\em unary facet complexity} of a polytope~\cite{Straszak-Vishnoi}. This is the integer
\[
	\mathrm{fc}(\mathcal A) = \max_{\boldsymbol \xi \in F_0^{\mathbb Z}} \max_i |\xi_i|
\]
where $F_0^{\mathbb Z}$ is the set of minimal integer outer normal vectors,
that is of vectors of the form $q \boldsymbol \xi \in \mathbb Z^n$ with
	$\boldsymbol \xi \in F_0$ and $q>0$ minimal. As in \cite{B-interior}*{Prop.2.7}, 
\begin{equation}\label{fclower}
	\eta \ge \frac{1}{\sqrt{n}\, \mathrm{fc}(\mathcal A)}
.
	\end{equation}
Following~\ocite{B-interior}*{Lemma 4.2 or Theorem 4.5}, 
	one can use Cramer's rule to obtain a \secrev{lower} bound on $\mathrm{fc}(\mathcal A)$. Using Hadamard's inequality
	to estimate the determinant, one obtains a lower bound
\[
	\eta \ge \frac{1}{\sqrt{n}\max_{\mathbf a \in A} (1+\|\mathbf a\|^2)^{n/2}}
.
\]
\end{remark}}
Before \changed{generalizing this invariant to a general tuple
$A_1$, \dots, $A_n$}, we 
need to introduce the associated {\em fan}, which is the
dual structure to this tuple.
\changed{More details and references can be found in \ocite{GKZ}.}
\changed{
\begin{definition} 
	\label{defan}
	A {\em fan} is a finite collection $\mathscr F$ of
convex polyhedral cones, such that
\begin{enumerate}
	\item Every face of every cone in $\mathscr F$ belongs to $\mathscr F$
	\item The intersection of two cones from $\mathscr F$ is a face
	of both cones.
\end{enumerate}
A fan is said to be complete if the union of its cones is $\mathbb R^n$.
\end{definition}
For instance, if $A \subseteq \mathbb Z^n$ is finite and $\conv{A}$
has non-empty interior, then for every subset $B \subseteq A$ one can define 
the 
supporting cone $C(B) = \{ 
\mathbf \xi \in \mathbb R^n:
B=A^{\boldsymbol \xi}\}$ \secrev{and its topological closure $\bar C(B)$ in
$\mathbb R^n$}. 
\secrev{It may happen that $C(B)=\{\mathbf 0\}$, for instance if $B=A$ and
$\conv{A}$ has non-zero volume.}
The set of the closed cones of the form $\bar C(B)$,
where $B \subseteq A$, is known as the normal fan of $\conv{A}$. 
\secrev{It is complete if and only if $\conv{A}$ has non-zero volume.}
To construct
the fan associated to a tuple of polytopes, we define:
}
\begin{definition}
For any tuple of non-empty subsets $B_1 \subseteq A_1, \dots , B_n
\subseteq A_n$, the {\em open cone above $(B_1, \dots, B_n)$} is
\secrev{defined as}
\[
	\gls{cone} \defeq \{ 
	\boldsymbol \xi \in \mathbb R^n:
	B_i = A_i^{\boldsymbol \xi} \}.
\]
	The {\em closed cone} $\bar C(B_1, \dots, B_n)$ is the topological closure of
$C(B_1, \dots, B_n)$ in $\mathbb R^n$, namely
\[
\bar C(B_1, \dots, B_n) = \{ 
	\boldsymbol \xi \in \mathbb R^n:
	B_i \supseteq A_i^{\boldsymbol \xi} \} \secrev{.}
\]
\end{definition}
For $j=0, \dots, n-1$, let $\gls{Fanj}$ be the set of non-empty $j+1$-dimensional
closed cones of the form $\bar C(B_1, \dots, B_n)$ for $\emptyset \ne B_k \subseteq A_k$. \secrev{Also, define $\mathscr F_{-1}$ as the set containing 
uniquely the zero-dimensional cone $\{ \mathbf 0 \}$.}
\begin{definition}
	\label{defan2}
The {\em fan} associated to the tuple $(A_1, \dots, A_n)$ is the tuple
	$(\mathfrak F_{n-1},$ $\dots,\mathfrak F_{\secrev{-1}})$. 
	\thirdrev{Elements
	of $\mathfrak F_0$ will be called {\em rays}, and each ray
	in $\mathfrak F_0$ can be represented as
	$\{ t \mathbf \xi: t \ge 0\}$ with $\mathbf \xi \in S^{n-1}$.
	We define therefore 
	\begin{equation}\label{FanR}
	\gls{FanR} \defeq \{ \mathbf \xi \in S^{n-1} \cap c: c \in \mathfrak F_0\}.
	\end{equation}
	}
\end{definition}

\begin{figure}
	\centerline{\hspace{-5em}
	\resizebox{\textwidth}{!}{ 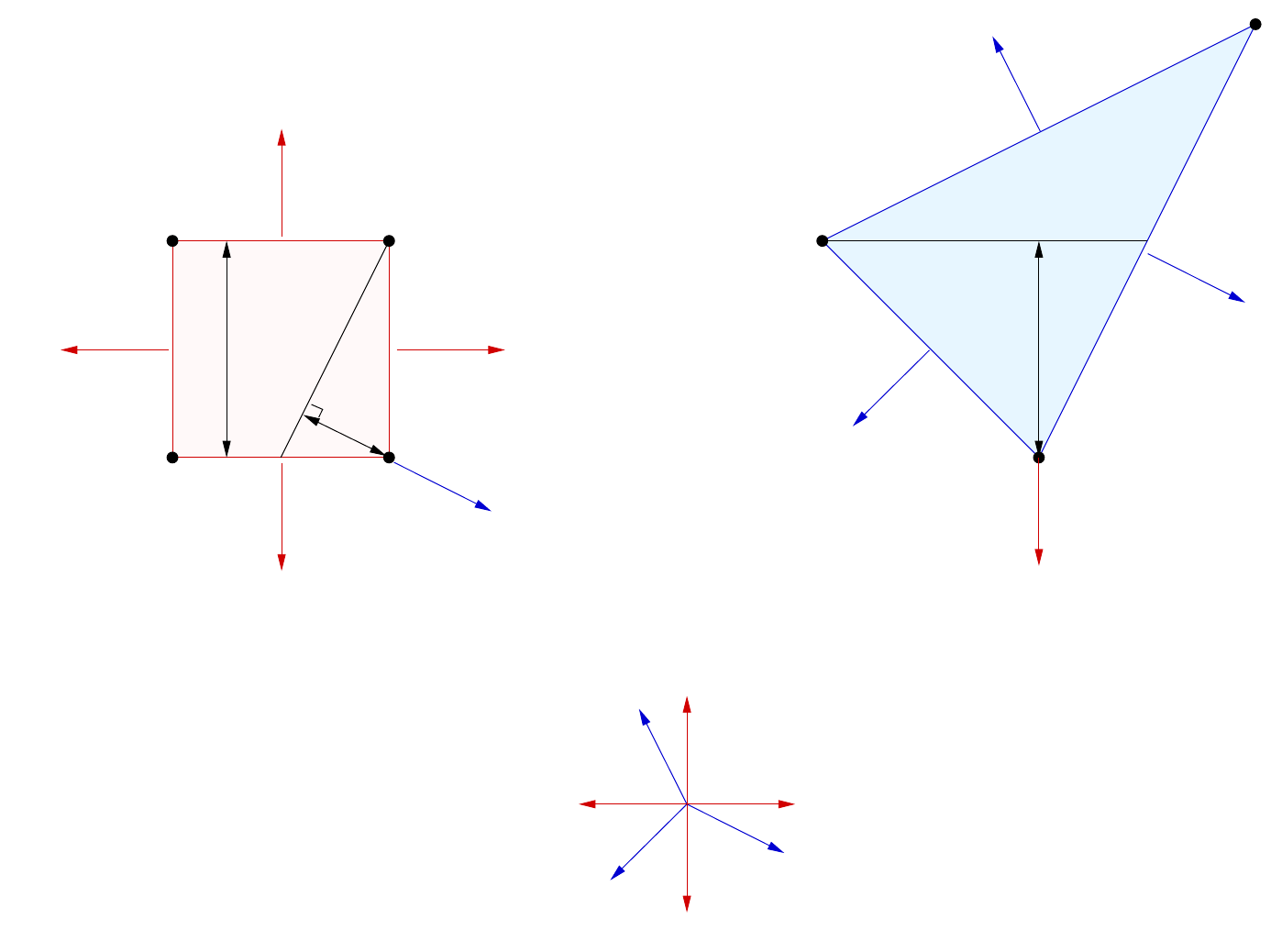} }
\vspace{1ex}
\centerline{\begin{tabular}{c|ccccccc|c}
	$\mathbf \xi \in \thirdrev{\mathfrak R}$ & $\mathrm e_1$ &
$\mathrm e_2$ &
$-\mathrm e_1$ &
$-\mathrm e_2$ &
	$\mathbf a$ & $\mathbf b$ & $\mathbf c$ & Min.value \\ \hline
	$\eta_1(\mathbf \xi)$ & 
	$2$ & $2$ & $2$ & $2$ & $2\sqrt{2}$ & 
	$\frac{2 \sqrt{5}}{5}$& 	$\frac{2 \sqrt{5}}{5}$ &$\frac{2 \sqrt{5}}{5}$\\
	$\eta_2(\mathbf \xi)$ &
	$2$ & $2$ & $2$ & $2$ & $3 \sqrt{2}$ & 
	$\frac{6 \sqrt{5}}{5}$& 	$\frac{6 \sqrt{5}}{5}$ &$2$\\
\end{tabular}}
\caption{
\label{fig-eta}
	The computation of $\eta$ for the polynomial system
	$F_1(X,Y) = 1 + X^2 + Y^2 + X^2 Y^2$, $F_2(X,Y)=X^2 + Y^2 + X^4 Y^4$.
	Top line, the supports and the normal vectors. \changed{Second line the fan 
	$\mathfrak F_1=\{C_1, \dots, C_7\}$ and $\thirdrev{\mathfrak R} = \{ \pm \mathrm e_1, \pm \mathrm e_2,
	\mathbf a, \mathbf b, \mathbf c\}$. 
	Bottom line, the value of $\eta_i$ at each element of \thirdrev{$\mathfrak R$}.
	}}
\end{figure}
\secrev{If the mixed volume of $\conv{A_1}$, \dots, $\conv{A_n}$ does not
vanish,
$\mathfrak F_0$ is the set of rays $C(B_1, \dots, B_n)$
supported by
maximal subsets $B_1 \subseteq A_1$, \dots, $B_n \subseteq A_n$. In the unmixed case $A=A_1=\dots=A_n$, those are spanned by outer unit normal vectors to 
the $n-1$-dimensional faces of $\conv{A}$.} 

The quality or {\em condition} of the tuple of supports can now be measured
in terms of the facet gaps
\begin{equation}\label{defeta}
	\glsdisp{etas}{\eta_i} = \min_{\boldsymbol \xi \in \thirdrev{\mathfrak R}} \eta_i(\boldsymbol \xi)
\hspace{1em}
\text{and}
\hspace{1em}
{\eta} = \min_i \eta_i
.
\end{equation}
(See figure~\ref{fig-eta}).
\changed{Geometrically, the facet gap $\eta$ is the minimal distance over all
\secrev{rays $\boldsymbol \xi \in \mathfrak R$ from
some $\mathbf a' \in A_i \setminus A_i^{\boldsymbol \xi}$ to the tangent hyperplane $\mathbf a + \boldsymbol \xi^{\perp}$, $\mathbf a \in A_i^{\boldsymbol \xi}$.}
\begin{remark}\label{remark-on-eta}
	\secrev{The facet gap $\eta_i$ cannot be larger than the
	diameter of $\conv{A_i}$.} 
	From Lemma~\ref{coarse-bound-DV}, we immediately recover
	the estimate
\begin{equation} \label{bound-eta}
	\eta \le \eta_i \le \diam{\conv{A_i}} \le 2\delta_i 
.
\end{equation}
\end{remark}
\changed{
	It is possible to \secrev{extend} the bounds in 
Remark~\ref{eta-bounds-by-Cramer} to the mixed setting. 
\begin{proposition} 
\[
	\eta \ge \frac{1}{\sqrt{n} \left(1+ \frac{\max(\diam{A_i})^2}{2} + \frac{n}{2} \right)^{\frac{2n-1}{2}} }
.
\]
\end{proposition}
}
\changed{
\begin{proof}
Without loss of generality, we assume that the circumscribed center of
$A_i$ is within distance $\frac{\sqrt{n}}{2}$ from the origin.
Indeed, $\eta$ is invariant by integer translations of each $A_i$.
We save for later usage the bound
\begin{equation}\label{radius-bound}
	\begin{split}
		\max_{i} \max_{\mathbf a \in A_i} 
		\secrev{\left(1 + \|\mathbf a\|^2\right)} &\le
	1 + \left(\frac{\max(\diam{A_i})}{2} + \frac{\sqrt{n}}{2}\right)^2
\\
		&\le
	1 + \frac{\max(\diam{A_i})^2}{2} + \frac{n}{2}
	\end{split}
\end{equation}

The unary complexity of the fan is
\[
	\mathrm{fc}(\mathfrak F) = \max_{\boldsymbol \xi \in \mathfrak R^{\mathbb Z}} \max_i |\xi_i|
\]
where $\mathfrak R^{\mathbb Z}$ is the set of irreducible non-zero
integral representatives of each cone in $\mathfrak F_0$,
	that is the set of vectors of the form $\mathbf 0 \ne \boldsymbol \xi \in \mathbb Z^n$ with \secrev{$\boldsymbol \xi \in c$ for some $c \in \mathfrak F_0$, and with relatively prime coefficients.}
	As before, \secrev{by \eqref{fclower}},
\[
\eta \ge \frac{1}{\sqrt{n}\, \mathrm{fc}(\mathfrak F_0)}
.
\]
	In order to estimate $\mathrm{fc}(\mathfrak F_0)$, we need to produce the Cayley matrix\secrev{.}  
For each
$\boldsymbol \xi \in \mathfrak F_0^{\mathbf Z}$,
we define the Cayley matrix $C$ for 
$A_1^{\boldsymbol \xi}, \dots, A_n^{\boldsymbol \xi}$	
as the matrix with rows $(-\mathrm e_i \ \mathbf a)$ for
	each $\mathbf a \in A_i^{\boldsymbol \xi}$, $i=1, \dots, n$. This matrix has $2n$ columns
and rank $2n-1$. The system
	\begin{eqnarray*}
C 
\begin{pmatrix} \boldsymbol \lambda \\ \boldsymbol \xi \end{pmatrix} 
&=& 0\\
\left(-\mathrm e_i, \mathbf a \right)
\begin{pmatrix} \boldsymbol \lambda \\ \boldsymbol \xi \end{pmatrix} 
&<& 0, 
		\hspace{2em} \mathbf a \in A_i \setminus A_i^{\boldsymbol \xi},\  i=1, \dots n
\end{eqnarray*}
admits a solution for $\mathbf \xi$ as above and for
$\lambda_i = \lambda_i(\mathbf \xi)$ Minkowski's support function.
(This trick is better explained in \cite{Malajovich-mixed}).
	By replacing $C$ by a full-rank, $(2n-1)\times 2n$ \secrev{submatrix} $C'$, Cramer's
rule yields an expression for a solution of the system of equations and
	inequations above: every solution $\begin{pmatrix}\boldsymbol \lambda\\ \boldsymbol \xi\end{pmatrix}$ belongs to the kernel of $C'$.
Therefore we can rescale $\boldsymbol \lambda$ and $\boldsymbol \xi$
so that $\lambda_i$ and $\xi_i$ are given up to a sign as   
$(2n-1)\times(2n-1)$ integer subdeterminants of $C'$. This expression is
not guaranteed to provide a reduced solution, but Hadamard's bound applies
anyway. Bounding the norm of each row by \eqref{radius-bound},
one obtains
\[
	\eta \ge \frac{1}{\sqrt{n} \left(1+\frac{\max \diam{A_i}}{2} + \frac{n}{2} \right)^{\frac{2n-1}{2}}}
\]
\end{proof}}

We will prove that if a system of sparse exponential sums has
a root with a large infinity norm, then it is close to the variety of
systems with a root at toric infinity.
	\secrev{As usual, assume we are given supports $A_1, \dots, A_n$ and that the
	mixed volume of $\conv{A_1}, \dots, \conv{A_n}$ does not vanish.
	Let \thirdrev{$\mathfrak R$} be the set of unit norm representatives
	of $\mathfrak F_0$,} as in Definition~\ref{defan2}. Let $\eta_i$
	be the facet gap of $A_i$ as in
	Equation \eqref{defeta}.
	}

\changed{\begin{conditionnumbertheorem}\color{black}\label{cond-num-infty}

	Let $H > 0$. Let $\mathbf q \in \mathscr F$ and suppose that there is
	$\mathbf z \in Z(\mathbf q)$ with $\|\Re(\mathbf z)\|_2 \ge H$. Then, there are
	$\boldsymbol \xi \in \thirdrev{\mathfrak R}$ and $\mathbf h \in \mathscr F$ such that
\[
	\frac{\| \mathbf h_i \|}{\| \mathbf q_i \|} \le 
\kappa_{\rho_i}
e^{-\eta_i H/n},
	\hspace{2em} i=1, \dots, n
\]
with the property that $\lim_{t \rightarrow \infty} [\mathbf V(\mathbf z + t \boldsymbol \xi)]$ is a root at infinity for $\mathbf q + \mathbf h$.
\end{conditionnumbertheorem}}
Let $\gls{SigmaINF}$ be the set
of all systems $\mathbf p \in \mathscr F$ that admit a root at toric
infinity. The condition number theorem above says that the
`reciprocal condition' of $\mathbf q$ with respect to $\Sigma^{\infty}$
is bounded above by $\kappa_{\rho_i}
e^{-\eta H/n}$.
The role of $\Sigma^{\infty}$ can be clarified by 
translating Bernstein's second theorem into the language of this paper.
\begin{theorem}[\citeauthor{Bernstein}, \citeyear{Bernstein}, Theorem B]
\label{BKK3}
Assume that the mixed volume $V(\conv{A_1}, \dots, \conv{A_n})$
does not vanish. 
	The mixed volume bound in Theorem~\ref{BKK} (resp.~\ref{BKK2}) \changed{is equal to the number of
	isolated roots counted according to their multiplicity,} if and only if 
	$\mathbf f \not \in \Sigma^{\infty}$.
\end{theorem}

\changed{\begin{remark} A system $\mathbf f \not \in \Sigma^{\infty}$ can have degenerate isolated roots, for instance the exponential sum $f(x) = (e^x -1)^2$ has support $A=\{0,1,2\}$ and coefficient covector $\mathbf f=(1,-2,1) \not \in \Sigma^{\infty}=\{ \mathbf f: f_0 f_2=0\}$.
\end{remark}}

\medskip

At this point we need to mention a particular class of tuples $(A_1, \dots,
A_n)$ that guarantee that $\Sigma^{\infty}$ is empty, 
\secrev{provided $q_{i{\mathbf a}} \ne 0$ for all $\mathbf a$.} 

\secrev{
\begin{definition}\label{trop-strongly} 
	\begin{enumerate}[(a)]
	\item
	A tuple of supports  $(A_1, \dots, A_n)$
is {\em strongly mixed} if and only if, for all $\boldsymbol \xi \ne \mathbf 0$,
there is $1 \le i \le n$ and a unique $\mathbf a \in A_i$
with the property that for all $\mathbf a' \in A_i$,
$\mathbf a' \ne \mathbf a$, 
$\mathbf a' \boldsymbol \xi < 
\mathbf a \boldsymbol \xi $.
\item
	A system of $n$-variate polynomials 
$G_1, \dots, G_n$ with support $(A_1, \dots, A_n)$, viz.
\[
	G_i(\mathbf X) = \sum_{\mathbf a \in A_i} G_{i \mathbf a} \mathbf X^{\mathbf a}
	\hspace{4em} i=1, \dots, n,
\]
	with $G_{i \mathbf a} \ne 0$ for all
	$\mathbf a \in A_i$, is said to be strongly mixed if the tuple
	$(A_1, \dots, A_n)$ is strongly mixed.
\item	The same definition holds for exponential sums.
	\end{enumerate}
\end{definition}}

\begin{remark}
Let $n \ge 2$. 
A system that is {\em strongly mixed} cannot be {\em unmixed}.
It cannot have a repeated support. It cannot have a support that is
a scaled translation of another support.
\end{remark}

\secrev{There is a natural interpretation of strongly mixed supports and polynomial
systems in the language of tropical algebraic geometry.}
Recall that the {\em tropical semi-ring} is $\mathbb R \cup \{-\infty\}$,
with sum $x \boxplus y = \max(x,y)$ and product $x \boxtimes y = x+y$. 
We will keep the \changed{notations} $x^a = 
\underbrace{x \boxtimes \cdots \boxtimes x}_{a\ \text{times}}$
\changed{and $\mathbf x^{\mathbf a}=x_1^{a_1} \boxtimes \cdots \boxtimes
x_n^{a_n}$}.
\secrev{A tropical polynomial $\boxplus_{\mathbf a \in A} f_{\mathbf a} \boxtimes \mathbf X^{\mathbf a},$ $f_{\mathbf a} \in \mathbb R$, corresponds in classical language to the maximum over a set of affine functions on variables $x_1, \dots, x_n$. The set $X$ of points
where this maximum is attained for at least two values of $\mathbf a$
is known as the
{\em tropical hypersurface} associated to the tropical polynomial.}
To a system of tropical polynomial equations, one associates the
{\em tropical prevariety}, that is the intersection of the tropical surfaces
from each equation. 
\begin{example}[Dense linear case] If $A_1=\dots=A_m=\Delta_n=\{\mathbf 0, \mathrm e_1,
\dots, \mathrm e_n\}$, then the tropical prevariety is the union of all
the cones in $\mathfrak F_{n-2}$.
\end{example}
\noindent \\
\secrev{In this language,}
\secrev{
\begin{lemma}\label{lem-strongly} 
	A tuple of supports  $(A_1, \dots, A_n)$
is {\em strongly mixed} if and only if, the tropical polynomial
system
\[
H_i(\boldsymbol \xi) = 
\bigboxplus_{\mathbf a \in A_i} \boldsymbol \xi^{\mathbf a} 
\]
	has tropical prevariety $\{\mathbf 0\}$. 
\end{lemma}
}

We expect generic systems of $m$ tropical polynomials in $n$
variables to have $n-m$ dimensional tropical prevarieties, but the system 
\changed{$\mathbf H(\boldsymbol \xi)$ in \secrev{Lemma}~\ref{lem-strongly}} above
is not generic. All the tropical monomials have identical coefficient,
so this particular tropical prevariety is a union of polyhedral cones,
plus the origin. More precisely, the tropical prevariety is the union of the origin and
the cones of a (possibly empty) subset of $\mathfrak F_{n-2} \cup \cdots \cup \mathfrak F_0$. 
\begin{remark}
The {\em tropicalization} of the polynomial $G_i(X)$ is
usually defined as
	\[
		\mathrm{trop}(G_i) = 
	\bigboxplus_{\mathbf a \in A_i} (-v(G_{i \mathbf a}))\boxtimes \boldsymbol \xi^{\mathbf a} 
\]
where $v$ is a non-trivial valuation. This is different from the
polynomials $H_i(\boldsymbol \xi)$ above.
For more details on tropical geometry, the reader is
referred to the book by \ocite{Maclagan-Sturmfels}. 
\end{remark}

\begin{lemma}\label{tropical} 
Let $A_1, \dots, A_n \subseteq \mathbb R^n$ be finite, with $A_i - A_i \subseteq \mathbb Z^n$. The following are equivalent:
\begin{enumerate}[(a)] \item \secrev{The tuple $(A_1, \dots, A_n)$}
is strongly mixed.
\item
	For each $\boldsymbol \xi \in \thirdrev{\mathfrak R}$, 
there is $i$ such that 
$\# A_i^{\boldsymbol \xi} = 1$.
\item $\Sigma^{\infty}$ is contained in an union of $\#\mathfrak F_0$ hyperplanes
of the form
\[
\mathbf q_{i{\mathbf a}} = 0.
\]
\end{enumerate}
\end{lemma}

Before stating the next result, we introduce the polynomial
\secrev{
	\begin{equation}\label{v-poly}
	v(t) \defeq 
n!\partialat{\epsilon}{0}
V( \conv{A_1}+\epsilon B^n+t \mathcal A, \dots, \conv{A_n}+\epsilon B^n
+t \mathcal A)
	\end{equation}
}
where $B^n$ denotes the unit $n$-ball, 
$\mathcal A \secrev{\defeq \conv{A_1} + \dots + \conv{A_n}}$ and $V$ is the mixed volume
\secrev{of $\conv{A_1}, \dots, \conv{A_n}$}. 
When $A_1=\dots=A_n$, \secrev{then}
$v(t)=(1+tn)^{n-1} V'$. 
\secrev{We can also expand $v(t)$ as 
\begin{equation}\label{v-k}
	v(t) = \sum_{k=0}^n \frac{v_k}{k!} t^k .
\end{equation}}
\secrev{The coefficient $v_0$ is equal to the mixed \changed{area} $V'$.} 
The coefficients $v_k$ can be seen up to scaling as mixed, non-smooth
analogues of the curvature integrals in Weyl's tube formula \cite{BC}*{Theorem 21.9}.

Now we need to consider homotopy paths, and we pick the easiest example: set
$\mathbf q(t) = \mathbf g + t \mathbf f$, where $\mathbf f$ is
fixed and $\mathbf g$ is a Gaussian random variable. But we will
need paths avoiding $\Sigma^{\infty}$. This is possible if we take $t$
real. \secrev{\fourthrev{The} problem is to produce an equidimensional real algebraic
variety containing the paths going through
$\Sigma^{\infty}$, and to bound its degree.}
\begin{theorem}\label{degenerate-locus} Let $A_1, \dots, A_n \subseteq \mathbb Z^n$ be finite, with non-zero
	mixed volume. Under the notations above, the following hold:
	\begin{enumerate}[(a)]
\item The set $\Sigma^{\infty}$ is contained in the zero set of a polynomial
	$\glsdisp{roff}{r}$ of degree 
\[
	\gls{dr}\le 
			\frac{\changed{\max_i \diam {A_i}}}{\changed{2} \det \Lambda}
			\max_{\thirdrev{0 \le k \le n}} (\changed{\thirdrev{e^k} v_k})
			\ \# \mathfrak F_0 
\]
\item If the system is strongly mixed, then $d_r \changed{\le} \# \mathfrak F_0$ and \changed{$d_r \le S=\sum_i \#A_i$, with} $\Sigma^{\infty}=Z(r)$.
\item Assume that $r(\mathbf f) \ne 0$. Then,
the set 
\[
	\Sigma^{\infty}_{\mathbf f} \secrev{\defeq} \{ \mathbf g \in \mathscr F: \exists t \in \mathbb R, \mathbf g + t \mathbf f \in \Sigma^{\infty} \}
\]
is contained, as a subset of $\mathbb R^{2S}$,
			in the zero set of a real polynomial $s=s(\mathbf g)$ of degree at most $d_r^2$.
\end{enumerate}
\end{theorem}

\changed{\begin{remark} \secrev{The set of systems with a root at toric
	infinity is not always an hypersurface or and equidimensional surface. The number $d_r$ is the degree of an hypersurface containing this locus.} In the unmixed case, $d_r$ is the sum of the degree of the $A$-resultants obtained as follows: for each $n-1$ facet \secrev{$A'$} of
	$A_1 = \dots = A_n$ \secrev{with outer normal vector $\boldsymbol \xi$, let $T$ be the linear span of the set of differences $A'-A'$ and let $\Lambda' \subset T$ be the
	lattice generated by $A' - A'$. The lattice $\Lambda'$ admits a lattice
	basis $\mathbf b_1, \dots, \mathbf b_{n-1}$. Let 
	$B$ be the $n \times (n-1)$ matrix with columns 
	$\mathbf b_1, \dots, \mathbf b_{n-1}$. As a linear map, $B$ 
	maps $\mathbb Z^{n-1}$ into
	$\Lambda$ bijectively, and induces an isomorphism of
	$\mathbb R^{n-1}$ into $T$. Let $A$ be the preimage of $A'$ by $B$.
	The set of polynomial systems with a root at $\boldsymbol \xi$-infinity
	is the zero-set 
	of the $A$-resultant for this $A$, and for the system of polynomials with coefficient $f_{i (\mathbf a B)},$ where $\mathbf a \in A$.}
\end{remark} }

When $\mathbf f$ is fixed with $r(\mathbf f) \ne 0$, the set of paths with some root of large
norm is a neighborhood of $\Sigma^{\infty}_{\mathbf f}$ in the usual topology. 
More specifically, \secrev{we introduce now a set $\Omega_{\mathbf f, T, H}$
of all $\mathbf g$ with the property that the path $\mathbf f + t \mathbf g$, $t \in [0,T]$,
will have a large zero. Namely,}
\[
	\Omega_H \secrev{\defeq} \Omega_{\mathbf f, T, H} \secrev{\defeq}
\left\{
\mathbf g \in \mathscr F: \exists t \in [0,T], 
\exists \mathbf z \in Z(\mathbf g+t \mathbf f), 
\|\Re(\mathbf z)\|_{\infty} \ge H \right \}
.
\]

At this point and in the next section, we choose always
coefficients $\rho_{i \mathbf a}=1$ and variance $\Sigma^2 = I$, in order
to keep statements short.

\begin{theorem}\label{vol-omega}
Assume that $\rho_{i,\mathbf a}=1$ for all $i, \mathbf a$.
	\secrev{Let $r$ be the polynomial of Theorem \ref{degenerate-locus}
	and let $d_r$ be its degree}. 
Suppose that $r(\mathbf f) \ne 0$. Let 
	$0 < \delta < \frac{1}{2(2d_r^2+1)S}$ and assume that
\[
H \ge \frac{n}{\eta} \log\left( 
	16 e \delta^{-1} d_r^2 S \thirdrev{\sqrt{\max_i (S_i)}}
	\left(1 + \frac{T \|\mathbf f\|}{ \sqrt[2S]{\delta} \sqrt{S}}\sqrt{e}\right) \right)
.
\]
Then,
\[
\probability{\mathbf g \sim \secrev{\mathcal N}(\mathbf 0,I; \mathscr F)}{\mathbf g \in \Omega_H}
\le
\delta
.
\]
\end{theorem}

\begin{remark} The coefficients $v_k$ can be bounded in terms of a more
	classical-looking Quermassintegral. Indeed, let
	\changed{
		\begin{eqnarray*}
	w(\tau) &\defeq& 
	n! V( \conv{A_1}+\tau B^n, \dots, \conv{A_n}+\tau B^n )
\\
			&=& \sum_{k=0}^n \frac{w_k}{k!} \tau^k .
		\end{eqnarray*}}
	The first terms are $w_0 = V$ and $w_1=V'$. \changed{Since
	$\mathcal A \subseteq \left(\sum_{i=0}^n \frac{\diam{A_i}}{2}\right) B^n$,
	we bound
\[
	v_k \le \left(\sum_{i=0}^n \frac{\diam{A_i}}{2}\right)^{k} w_{k+1}.
\]}
In particular,
	\[
	\log(d_r) \le O\left( 
	\log(\# \mathfrak F_0) + \thirdrev{n} + \changed{n \log\left(\sum_{i=0}^n \diam{A_i}\right) + \log (\max w_k)} \right) 
\]
\end{remark}

\changed{
	\begin{remark}\label{intrinsic}
The coefficients $w_k$ from the preceding remark are closely related to well-known
invariants in convex geometry, the {\em intrinsic volumes} 
$V_j(C)$ of a convex body $C$.
In the unmixed case $A_1 = A_2 = \cdots = A_n$, one has
\[
	\binomial{n}{j} w_{n-j} = \vol_{n-j}(B^{n-j}) V_j(\conv{A_1})
.
\]
while in the general case,
\[
	\binomial{n}{j} w_{n-j} \le \vol_{n-j}(B^{n-j}) V_j(\mathcal A)
.
\]
	so $\log (\max w_k) \le O(\log \max_j V_j(\mathcal A))$.
\end{remark}}

\subsection{Analysis of the homotopy algorithm}
\label{sec:unconditional}

The `renormalization' process comes at a cost. We will define below
three sets that must be excluded from the choice of \secrev{the random initial system} $\mathbf g$ for the
algorithm to behave well. Choices of $\mathbf g$ in one of those sets
may lead to a `failed' computation. \secrev{In this case we will start over with another random $\mathbf g$.} 

The first of those sets is easy to describe and easy to avoid. It corresponds
to paths $\mathbf q_t = \secrev{\mathbf g} + t \mathbf f$ with no known  
decent upper bound for the renormalized speed vector
\[
	\left\| \frac{\partial}{\partial t} ( \mathbf q_t \cdot R(\mathbf z_t) )
	\right\|_{\mathbf q_t \cdot R(\mathbf z_t)}
\]
\secrev{valid for all continuous solutions} $\mathbf z_t \in Z(\mathbf q_t)$. We will see that this set is
confined to 
a product of \secrev{thin} slices 
\secrev{around semi-axes $[0,-f_{i\mathbf a}\infty)$ 
in the complex plane }, one slice from each coordinate:
\[
	\glsdisp{exclusion}{\Lambda_{\epsilon}} \defeq \left\{ \mathbf g \in \mathscr F: 
\ \exists 1 \le i \le n, \ \exists \mathbf a \in A_i, \ 
|\arg \left(\frac{g_{i{\mathbf a}}}{f_{i\mathbf a}}\right)| \ge \pi-\epsilon
\right\}
.
\]
Notice that $\probability{\mathbf g \sim \secrev{\mathcal N}(\mathbf 0,I; \mathscr F)}{\mathbf g \in \Lambda_{\epsilon}} \le S \frac{\epsilon}{\pi}$, \secrev{where} $S=\dim_{\mathbb C}(\mathscr F)$. For simplicity we will take $\epsilon = \frac{\pi}{72 S}$ so that
\secrev{we have} $\probability{g \sim \secrev{\mathcal N}(\mathbf 0, I; \mathbf F)}{\mathbf g \in \Lambda_{\epsilon}}\le 1/72$.

The second exclusion set is more subtle. We want to 
remove the set
\[
	{\Omega_H} = \Omega_{\mathbf f, T, H} 
\]
from Theorem~\ref{vol-omega}. \changed{The precise value of $T$ will
be specified later}.
While we do not know {\em a priori} if $\mathbf g \in \Omega_H$, we can
execute the path-following algorithm and stop in case of failure,
that is in case $\|\Re(\mathbf z(t))\|_{\infty} \ge H$.
Failure is also likely if $\mathbf f$ has a root at infinity.
We pick $H$ so that
$\delta = \probability{g \sim \secrev{\mathcal N}(\mathbf 0, I; \mathbf F)}{\mathbf g \in \Omega_{H}}
\le \frac{1}{2 (2d_r^2+1) S}$ in order to apply
Theorem~\ref{vol-omega}. Since $S \ge 4$ and $d_r \ge \# \mathfrak F_0 \ge 2$,
we have always $\delta \le \frac{1}{72}$.
\changed{If furthermore $\mathbf f$ is scaled so that $\| \mathbf f\|=\sqrt{S}$,} we deduce that with probability $\ge 71/72$, $\mathbf g \not \in \Omega_H$
for
\begin{equation}\label{eq-H}
	H \secrev{\defeq} \frac{n}{\eta}
O\left( \log (d_r) + \log (S) + \log(T) \right)
\end{equation}

The third set to avoid is
\begin{equation}\label{eqYK}
	{Y_K} \secrev{\defeq} \{ \mathbf g \in \mathscr F: \exists i, \| \mathbf g_i \| \ge K \sqrt{S_i}\}
\end{equation}
Using the large deviations estimate, we will show \secrev{in Section~\ref{overview-linear}} that for
$K=1+\sqrt{\frac{\log(n)+\log(10)}{\min (S_i)}}$ we have $\probability{g \sim \secrev{\mathcal N}(\mathbf 0, I; \mathbf F)}{\mathbf g \in Y_K}\le 1/10$.

\secrev{Before stating the next theorem, we introduce a couple of notations.
For any system 
$\mathbf f \in \mathscr F_{A_1} \times \cdots \times \mathscr F_{A_n}$ 
we define its {\em imbalance invariant}
\begin{equation}\label{kappaf}
	\gls{kappaf} \defeq \max_{i, \mathbf a} \frac{\| \mathbf f_i \|}
	{|\mathbf f_{i\mathbf a}|} 
\end{equation}
and its {\em condition}
\begin{equation}\label{muf}
	\gls{muf} \defeq 
	\max_{\mathbf z \in Z(\mathbf f)} \mu(\mathbf f \cdot R(\mathbf z)) =
\end{equation}
}

\begin{theorem}\label{mainD}
There is a constant $C$ with the following property:
\begin{enumerate}[(a)]
	\item \label{mainD-supports} Assume that $A_1, \dots, A_n \subseteq \mathbb R^n$
are such that the $\mathbb Z$-module $\Lambda$ generated by
		$\bigcup_{i=1}^n A_i-A_i$ is a subset of $\mathbb Z^n$, and that the mixed volume \secrev{$V=V(\mathcal A_1, \dots,$ $\mathcal A_n)$} is non-zero,
		$\mathcal A_i = \conv{A_i}$. 
\item \label{mainD-spaces} Let $\mathscr F = 
\mathscr F_{A_1} \times \cdots \times \mathscr F_{A_n}$ where it is
assumed that for all $i, \mathbf a \in A_i$, $\rho_{i, \mathbf a}=1$. 
Denote by $S_i$ the complex dimension of $\mathscr F_{A_i}$,
		$S_i=\# A_i$, \fourthrev
	{$S=\sum S_i = \dim_{\mathbb C} (\mathscr F)$} and assume that \secrev{$n \ge 2$}
		and $S_i \ge 2$.
\item \label{mainD-fixed}
Let $\mathbf f \in \mathscr F$ with $r(\mathbf f) \ne 0$ where $r$ is the polynomial from Theorem~\ref{degenerate-locus}. Suppose also
that $\mathbf f$ is scaled in such a way that
$\|\mathbf f_i \| = \sqrt{S_i}$. 
\item \label{mainD-random}
Take $\mathbf g \sim \secrev{\mathcal N}(\mathbf 0,I; \mathscr F)$, and and consider the
random path
		$\mathbf q_t = \mathbf g + t \mathbf f \in \mathscr F$\secrev{, $0 \le t < \infty$.} To this path \secrev{in coefficient space,}
		associate the set $\mathscr Z(\mathbf q_t)$ \secrev{of continuous paths $z_t$ in solution space with } $\mathbf q_t \cdot \mathbf V(\mathbf z_t) \equiv \mathbf 0$
		\secrev{, for all $0 \le t < \infty$}. 
\end{enumerate}
	Then with probability 1, \secrev{there are $\frac{n! V}{\det(\Lambda)}$ paths $\mathbf z_t \in \mathscr Z(\mathbf q_{\tau})$ and all of them extend to continuous paths in $\secrev{\mathscr M}$ for $t \in [0, \infty]$.} With 
	probability $\ge 3/4$\secrev{, we have} \fourthrev{$g \not \in \Omega_H$ and}
\[
\sum_{\mathbf z_{\tau} \in \mathscr Z(\mathbf q_{\tau})} \mathscr L (\mathbf q_t, \mathbf z_t; 0, \infty)
	\le C 
	Q \changed{n^{\thirdrev{\frac{5}{2}}}} S\max_i (\thirdrev{S_i}) 
	K \thirdrev{\max(K,\sqrt{\kappa_{\mathbf f}})} 
	\kappa_{\mathbf f} \mu_{\mathbf f}^2 \secrev{\nu}
	\ \secrev{\mathrm{LOGS}_{\mathbf f}}
\]
	\secrev{where
\begin{equation}\label{mainD-K}
K \defeq \left(1+\sqrt{\frac{\log(n)+\log(10)}{\min (S_i)}}\right), 
\end{equation}
	\begin{equation}\label{mainD-Q}
	\gls{Q}	\defeq \eta^{-2} \left(\sum_{i=1}^n \delta_i^2\right) \frac {\max( n! V , \thirdrev{n!} V' \eta)}{\det \Lambda},
	\end{equation}
	\fourthrev{the distortion invariant $\nu$ was defined in \eqref{distortion},
	the radii $\delta_i=\delta_i(0)$ were defined in \eqref{deltaix},}
		the facet gap $\eta$ was defined in \eqref{defeta} and $V'$ is the mixed \changed{area as in Equation \eqref{eq-mixed-area}}.
	The notation $\mathrm {LOGS}_{\mathbf f}$ stands for
	\begin{equation}\label{mainD-LOGS}
		\secrev{\mathrm{LOGS}_{\mathbf f}} \defeq
		\log (d_r) +\log(S) + \log(\mu_{\mathbf f})+\log(\kappa_{\mathbf f}) 
	\end{equation}
	and $d_r$ is the degree of the polynomial $r$.} 
\end{theorem}

\begin{remark}\label{num-paths}
Remark~\ref{remark-on-eta} implies that
$\frac{1}{2} \le \frac{n}{4} \le \eta^{-2} \sum \delta_i^2$. The number of paths
satisfies with probability one:
\[
	\# \mathcal Z(\mathbf q_t) = \frac{n! V}{\secrev{\det(\Lambda)}} \le 2 Q.
\]
\end{remark}
\begin{remark}
The quantity $Q$ is invariant by
uniform integer scaling $A_i \mapsto d A_i$, $d \in \mathbb N$.
\end{remark}

\begin{remark}
	In the particular case \secrev{where} 
	$\mathbf f$ is strongly mixed, \secrev{we have}
	$d_r \le \# \mathfrak F_0$. Otherwise a bound for $d_r$ is provided by Theorem~\ref{degenerate-locus}.
\end{remark}

The proof of Theorem~\ref{mainD} is postponed to Section~\ref{sec-linear}.
\changed{In Definition~\ref{def-approximate-root} we \secrev{defined} {\em certified approximate
roots}. Next we define a global object to approximate the zero-set of a 
non-degenerate system of equations:}

\begin{definition} Assume that $\mathbf f \in \mathscr F$ is non-degenerate
with no root at infinity. 
	A {\em certified solution set} for $\mathbf f$ is a set $X \secrev{\,\subseteq \mathscr M}$ of certified
approximate roots for $\mathbf f$, with different associated roots,
	\changed{and such that}
\[
	\secrev{	\# X = \frac{n! V(\conv{A_1}, \dots, \conv{A_n})}{\det \Lambda}
	}.
\]
\end{definition}

\secrev{Theorems \ref{th-A} and \ref{mainD} combined will provide a probabilistic estimate for a randomized homotopy algorithm between two fixed $\mathbf f, \mathbf h \in \mathscr F$. We will assume that both $\mathbf f$
and $\mathbf h$ satisfy the hypothesis (\ref{mainD-fixed}) in Theorem~\ref{mainD}, viz. $r(\mathbf f) \ne 0$ and $r(\mathbf h) \ne 0$.
Given a certified solution set $X_{\mathbf h}$ for $\mathbf h$, we may 
attempt to proceed as \ocite{Li-Sauer-Yorke}:
\medskip

\begin{algorithm}[H]\label{algorithm-1}
\SetKwInput{KwInput}{Input}                
\SetKwInput{KwOutput}{Output}              
\SetKwRepeat{Repeat}{repeat}{end}
\KwInput{$\mathbf f, \mathbf h \in \mathscr F$ and a certified
	solution set $X_{\mathbf h}$ for $\mathbf h$.}
\KwOutput{A certified solution set $X_{\mathbf f}$ for $\mathbf f$.}
Pick $\mathbf g \in \secrev{\mathcal N}(\mathbf 0,I; \mathscr F)$ and 
	let $X_{\mathbf g}\fourthrev{, X_{\mathbf f}} \leftarrow \emptyset$\;
For each $\mathbf x_0 \in X_{\mathbf h}$, apply the recurrence 
\eqref{rec-zero} to the path $\mathbf q_t = \mathbf h + t \mathbf g$
	for $t \in [0,\infty]$. \fourthrev{If $\mathbf x_N$ is the last
	value obtained, store $\mathbf x_{N+1} \defeq N(\mathbf g \cdot R(\mathbf x_N))+\mathbf x_N$
	in $X_{\mathbf g}$}\;
For each $\mathbf x_0 \in X_{\mathbf g}$, apply the recurrence 
	\eqref{rec-zero} to the path 
	$\tilde{\mathbf q}_t = \mathbf g + t \mathbf f$
for $t \in [0,\infty]$.  \fourthrev{If $\mathbf x_N$ is the last
	value obtained, store $\mathbf x_{N+1} \defeq N(\mathbf f \cdot R(\mathbf x_N))+\mathbf x_N$
	in $X_{\mathbf f}$}\;

\caption{Randomized homotopy between sparse systems.}
\end{algorithm}
\medskip

\fourthrev{
	Recall from Definition~\ref{def-rencondlength} that the condition length of a path $(\mathbf q_t)$ 
is scaling invariant in each coordinate $\mathbf q_{it}$. 
The notation $(\mathbf q_t)_{t \in [0,\infty]}$ stands
for the topological closure of $([\mathbf q_{t}])_{t \in [0,\infty)}$
in the product of the $\mathbb P(\mathscr F_{A_i})$'s. The limit point
$\lim_{t \rightarrow \infty} [\mathbf q_t]$ is precisely $\mathbf g$.

If the renormalized condition length of the path
$(\mathbf q_t)_{t \in [0, \infty]}$ is finite, Theorem~\ref{th-A}
guarantees that the recurrence \eqref{rec-zero} terminates 
with
\[
	N \le 1+ \frac{1}{\delta_*} \mathscr L(\changed{(\mathbf q_{t}, \mathbf z_t);}\,0,\infty)
\]
and that each $\mathbf x_{N+1}$ produced in step 2 is a certified approximate
root for $[\mathbf q_{\infty}]=[\mathbf g]$.
If the renormalized
condition length of $(\mathbf q_t)_{t \in [0, \infty]}$ is infinite, then step 2 may not terminate.

The second path was defined by
$\tilde{\mathbf q}_t = \mathbf g + t \mathbf f$ for $t\in [0,\infty]$.
It is precisely the path $\mathbf f + \tau \mathbf g$, $\tau \in [0, \infty]$,
in the opposite sense, so the condition length is the same. 
The same notations apply, but this time, $\lim_{t \rightarrow \infty} [\tilde {\mathbf q}_t]=[\mathbf f]$.

If the condition
length is finite, then Theorem~\ref{th-A} guarantees again that 
\[
	N \le 1+ \frac{1}{\delta_*} \mathscr L(\changed{(\tilde{\mathbf q}_{t}, \mathbf z_t);}\,0,\infty)
\]
and that $\mathbf x_{N+1}$ is a certified approximate
root of $[\tilde{\mathbf q}_{\infty}] = [\mathbf f]$. But if the renormalized
condition length is infinite, step 3 may not terminate.
}

\fourthrev{The path $\mathbf q_t$ satisfies the renormalized condition length bound of Theorem~\ref{mainD} with probability 
$\ge \frac{3}{4}$. The path $\tilde{\mathbf q}_t$ satisfies the same bound with
probability $\ge \frac{3}{4}$. Hence, the probability that step 2 fails or that it takes longer than the theoretical bound from Theorem~\ref{mainD} is 
$\le 1/4$. The same is true for step 3.
Hence, with probability at least $1/2$, Algorithm~\ref{algorithm-1} 
terminates within the theoretical time bound from
the two applications of Theorem~\ref{mainD}. 
}

From Remark \ref{num-paths}, there are at most $2Q$ paths to follow for
step 2, and also for step \fourthrev{3. Each path may require up to $N+1$ renormalized Newton iterations.}
The properties of Algorithm~\ref{algorithm-1} are summarized below. 
First define
\begin{equation}\label{L}
	L_{\mathbf f} \defeq 
	 \thirdrev{\max(K,\sqrt{\kappa_{\mathbf f}})} 
 \kappa_{\mathbf f} \mu_{\mathbf f}^2 \mathrm{LOGS}_{\mathbf f}
	\hspace{1em}
	\text{and}
	\hspace{1em}
	L_{\mathbf h} \defeq \thirdrev{\max(K,\sqrt{\kappa_{\mathbf f}})} 
			\kappa_{\mathbf h} \mu_{\mathbf h}^2 \mathrm{LOGS}_{\mathbf h}
.
\end{equation}
}

\begin{proposition}\label{composed-homotopy}
	Let $\mathbf f, \mathbf h \in \mathscr F$ be given \changed{such that} $r(\mathbf f) \ne 0$,
	$r(\mathbf h) \ne 0$, together
with a certified solution set $X_{\mathbf h}$ for $\mathbf h$. 
	\changed{We assume that $\mathbf f$ and $\mathbf h$ are scaled so that
	$\|\mathbf f_i\|=\|\mathbf h_i\|=\sqrt{S_i}$.}
	Let $\mathbf g \in \secrev{\mathcal N}(\mathbf 0,I; \mathscr F)$.
	Then with probability $\ge 1/2$, \secrev{Algorithm~\ref{algorithm-1}} will produce a certified solution set 
	$X_{\mathbf f}$ for $\mathbf f$ in at most
	\secrev{
	\begin{equation}\label{N-star}
	\changed{N_*} \defeq		4Q \left(2+ \frac{C}{\delta_*} n^{\thirdrev{\frac{5}{2}}} S \max_i (\thirdrev{S_i}) 
	K \nu
			\left(  
L_{\mathbf f} 
			+  
			L_{\mathbf h}\right)\right)
	\end{equation}}
	renormalized Newton iterations, \changed{where $\delta_*$ is the constant in Theorem~\ref{th-A}}, \secrev{$\mathrm{LOGS}_{\mathbf f}$ was defined in
	\eqref{mainD-LOGS} and accordingly,}
	\[
		\secrev{\mathrm{LOGS}_{\mathbf h}= 
	\log (d_r) +\log(S) + \log(\mu_{\mathbf h}) + \log (\kappa_{\mathbf h})
	.}
		\]
\end{proposition}

This algorithm can be modified 
to `give up' for an unlucky choice of $\mathbf g$ and start again.
However, the results in this paper would be completely useless if one
really needed to know $Q$, $d_r$ and $\mu_{\mathbf f}$ in order to produce a 
probability one algorithm for homotopy. 
\changed{This is why we do not assume $N_*$ known}.
Since the probability 1/2 procedure of Proposition~\ref{composed-homotopy}
requires at most $N_*$ renormalized Newton iterations, 
we will \secrev{modify Algorithm~\ref{algorithm-1} 
so that it counts the number of renormalized Newton iterations and fails
if this number exceeds a previously stipulated bound.}
\medskip

\secrev{
\begin{algorithm}[H]\label{algorithm-2}
\SetKwInput{KwInput}{Input}                
\SetKwInput{KwOutput}{Output}              
\SetKwRepeat{Repeat}{repeat}{end}
\KwInput{$\mathbf f, \mathbf h \in \mathscr F$ and a certified
	solution set $X_{\mathbf h}$ for $\mathbf h$.}
\KwOutput{A certified solution set $X_{\mathbf f}$ for $\mathbf f$.}
Stipulate an arbitrary value $N_0 \in \mathbb N$ and set $k=0$\;
	\Repeat{}{
	Execute the Algorithm~\ref{algorithm-1} with input $\mathbf f, 
	\mathbf h, X_{\mathbf h}$
	up to $N_k$ renormalized Newton iterations, and set $X_{\mathbf f}=
	\emptyset$ in case Algorithm~\ref{algorithm-1} does not terminate naturally\;
	\If{ $X_{\mathbf f} \ne \emptyset$}{\Return $X_{\mathbf f}$}
	Set $N_{k+1} \leftarrow \sqrt{2} N_k$\;
	Increase $k$ by one \;
	}
\caption{Probabilistic homotopy algorithm}
\end{algorithm}
}
\medskip


Eventually for some value of $k$, $N_k \le N_* < N_{k+1}$ so 
\secrev{Algorthm~\ref{algorithm-2}} 
succeeds with probability one.
The expected number of renormalized Newton iterations is 
\[
	\bar N \defeq N_0 + \dots + N_k + \frac{1}{2} N_{k+1} +
	\frac{1}{4} N_{k+2} + \dots
\]
Trivially,
\[
	N_0 + \dots + N_{k} \le 
N_* \left(1 + \frac{1}{\sqrt{2}} + \left(\frac{1}{\sqrt{2}}\right)^2 + \dots\right)
=
(2+\sqrt{2}) N_*
\]
while
\[
\frac{1}{2} N_{k+1} +
	\frac{1}{4} N_{k+2} + \dots
\le
N_* \left(\frac{1}{\sqrt{2}} + \left(\frac{1}{\sqrt{2}}\right)^2 + \dots\right)
=
(1+\sqrt{2}) N_* 
\]
It follows that 
\[
	\bar N \le (3+2\sqrt{2}) N_*.
\]
\fourthrev{The reader can check that this bound is optimal for the class of algorithms where
step 7 is replaced by $N_{k+1} \leftarrow s N_k$, $s>1$.} We proved:
\begin{theorem}\label{mainE}
There is a probability 1 algorithm with input 
$n, A_1, \dots, A_n, \mathbf f, \mathbf h,$
	$X_{\mathbf h}$ and output $X_{\mathbf f}$ with the following properties. If  
$r(\mathbf f) \ne 0$, $r(\mathbf h) \ne 0$, $N_*$ is finite and $X_{\mathbf h}$
	is a \changed{certified} solution set for $\mathbf h$, then $X_{\mathbf f}$ is a \changed{certified}
	solution set for $\mathbf f$. This algorithm will perform at most
\[
	(3+2\sqrt{2}) N_*
\]
	renormalized Newton iterations on average, \fourthrev{where $N_*$ was defined in \eqref{L}--\eqref{N-star}}.
\end{theorem}

\secrev{To better understand what was achieved, we can bound $L_{\mathbf f}$
in \eqref{N-star} by
\[
	K 
 \kappa_{\mathbf f}^{\frac{3}{2}} \mu_{\mathbf f}^2 \mathrm{LOGS}_{\mathbf f}
\]
and a similar bound holds for $L_{\mathbf h}$.
The invariant
	$\kappa_{\rho_i}$ form Equation \eqref{eq-rho} 
can be used to
bound $\nu(\mathbf x)$ in general. We will deduce a bound for
$\nu=\nu(\mathbf 0)$ from \eqref{N-star} in terms of $\max S_i$:}
\begin{lemma}\label{lem-nu-kappa} Unconditionally \secrev{for all $\mathbf x$}, 
\[
1 \le \nu_i(\mathbf x).
\]
At the origin,
\[
\secrev{\nu_i}=\nu_i(\mathbf 0) \le \kappa_{\rho_i} .
\]
If one sets $\rho_{i,\mathbf a} \equiv 1$, then
\[
\secrev{\nu_i} \le \sqrt{S_i}.
\]
\end{lemma}

\secrev{The proof is postponed to the end of this section.
	From Lemma~\ref{lem-nu-kappa},
$\nu \le \max \sqrt{S_i}$. Hence, there is a bound of the form
\[
	N_* \le P Q \,
	\left(\kappa_{\mathbf f}^{\frac{3}{2}} \mu_{\mathbf f}^2 \mathrm{LOGS}_{\mathbf f} +\kappa_{\mathbf h}^{\frac{3}{2}} \mu_{\mathbf h}^2 \mathrm{LOGS}_{\mathbf h}\right) 
\]
where $P$ is a polynomial in the input size $S$. The quantity $Q$ defined in
\eqref{mainD-Q} depends only on the geometry of the tuple $(A_1, \dots, A_n)$
through the mixed volume, the mixed area, the facet gap $\eta$ and the radii $\delta_i$. This establishes the informal complexity bound \eqref{main-bound-1}. The argument for the informal complexity bound \eqref{main-bound-0} is similar.

A {\em non-uniform} algorithm is a family of algorithm, one for each possible
tuple $(A_1, \dots, A_n)$ with non-zero mixed volume. 
\fourthrev{In this context, it is enough to use a good starting system which is known to exist, but is not known explicitly}.
The non-uniform complexity bound \eqref{main-bound-2} is formalized below as a corollary of
Theorem~\ref{mainE}.

\begin{corollary} Let the tuple $(A_1, \dots, A_n)$ be fixed
and satisfy the conditions of 
Theorem~\ref{mainD}(\ref{mainD-supports}).
Then
there is a randomized algorithm that finds all the roots of $\mathbf f \in \mathscr F$,
$r(\mathbf f) \ne 0$, with probability 1 and
expected cost of 
\[
	O(\mu_{\mathbf f}^2 
	\kappa_{\mathbf f}^{\frac{3}{2}}
	(\log(\mu_{\mathbf f}) +
	\log(\kappa_{\mathbf f})) .
\]
renormalized Newton iterations.
\end{corollary}

\begin{proof}
	Let $\mathscr F = \mathscr F_{A_1} \times \cdots \times \mathscr F_{A_n}$ be as in Theorem~\ref{mainD}(\ref{mainD-spaces}). The set of all $\mathbf h \in \mathscr F$ with $r(\mathbf h)=0$ has zero measure, as well as the set of $\mathbf h$ with $\kappa_{\mathbf h} = \infty$ or $\mu_{\mathbf h}=\infty$. Therefore there exists  
$\mathbf h \in \mathscr F$ with  
	$r(\mathbf h) \ne 0$, $\mu_{\mathbf h}< \infty$ and $\kappa_{\mathbf h}< \infty$. This $\mathbf h \in \mathscr F$ is non-degenerate, so $X_{\mathbf h}=Z(\mathbf h)$ is a certified solution set. 
According to Theorem~\ref{mainE},
	Algorithm~\ref{algorithm-2}
	with input $(\mathbf h, Z(\mathbf h), \mathbf f)$ has cost
\[
	O\left( \kappa_{\mathbf f}^{\frac{3}{2}} \mu_{\mathbf f}^2 \mathrm{LOGS}_{\mathbf f} 
	\right)
\]
where the big-$O$ hides the terms $P$, $Q$, and those depending on $\mathbf h$.
The degree $d_r$ that appears inside a logarithm in $\mathrm{LOGS}_{\mathbf f}$ is now a constant. Therefore, the cost measured in renormalized Newton iterations is bounded by   
\[
	O(\mu_{\mathbf f}^2 
	\kappa_{\mathbf f}^{\frac{3}{2}}
	(\log(\mu_{\mathbf f}) +
	\log(\kappa_{\mathbf f})) .
\]
\end{proof}

\begin{proof}[Proof of Lemma~\ref{lem-nu-kappa}]
	\secrev{From Lemma~\ref{coarse-bound-DV},
	\[
		\| \mathbf u\|_{i,\mathbf x}
		\le \delta_i(\mathbf x)
		\|\mathbf u\|.
	\]
	It follows that
	$\{ \mathbf u: \delta_i(\mathbf x) \| \mathbf u\| \le 1 \} \subseteq
	\{ \mathbf u: \| \mathbf u\|_{i,\mathbf x} \le 1\}$.}
	This provides the lower bound
\[
	1 \le 
	\sup_{\delta_i \| \mathbf u\| \le 1} \sup_{\mathbf a \in A_i}
	|(\mathbf a - \mathbf m_i(\mathbf x)) \mathbf u| \le \nu_i(\mathbf x)
	\secrev{
	=
	\sup_{\mathbf a \in A_i}
	\sup_{\| \mathbf u\|_{i,\mathbf x} \le 1} 
	|(\mathbf a - \mathbf m_i(\mathbf x)) \mathbf u| \le \nu_i(\mathbf x)
	}
	.
\]
	\secrev{Let $A_i' \subseteq A_i$ be the set of vertices of
	$\conv{A_i}$.}
	At the origin,
\[
\|\mathbf u\|_{i \mathbf 0}\, \kappa_{\rho_i} =
	\frac{ \sqrt{ \sum_{\mathbf a} (\rho_{i\mathbf a} (\mathbf a-\mathbf m_i(\mathbf 0)) \mathbf u)^2}}
	{\min_{\mathbf a \in A_i'} \rho_{i\mathbf a} }
\ge
	\max_{\mathbf a \in A_i'} |(\mathbf a-\mathbf m_i(\mathbf 0))  \mathbf u|
	=
	\max_{\mathbf a \in A_i} |(\mathbf a-\mathbf m_i(\mathbf 0))  \mathbf u|
.
\]
	Therefore, $\| \mathbf u\|_{i\mathbf 0} \le 1$ implies that 
	$\max_{\mathbf a \in A_i} |(\mathbf a-\mathbf m_i(\mathbf 0))  \mathbf u| \le \kappa_{\rho_i}$.
	Therefore, $\secrev{\nu = } \nu(\mathbf 0) \le \kappa_{\rho_i}$. If furthermore
	$\rho_{i\mathbf a} \equiv 1$, then $\kappa_{\rho_i}=\sqrt{S_i}$.
\end{proof}}

\aboutsampling{
\secrev{
\subsection{Sampling versus solving}
\label{sampling}

It is worth to revisit two open problems proposed in my previous paper \cite{toric1}. The first one was a sparse version of Smale's 17$^{\mathrm{th}}$ problem:

\begin{problem}\label{probA} Can a finite zero of a random sparse polynomial system (...) be found approximately, on the average, in time polynomial in $S=\sum_i \#A_i$ with a uniform algorithm?
\end{problem}

Suppose that the answer is yes. Then we may modify Algorithm~\ref{algorithm-2}
to find one zero of an arbitrary, fixed system $\mathbf f$
with $r(\mathbf f) \ne 0$ and finite $\kappa_{\mathbf f}$ and $\mu_{\mathbf f}$.
The modified algorithm is:
\medskip

\begin{algorithm}[H]\label{algorithm-3}
\SetKwInput{KwInput}{Input}                
\SetKwInput{KwOutput}{Output}              
\SetKwRepeat{Repeat}{repeat}{end}
\KwInput{$\mathbf f, \mathbf h \in \mathscr F$ and a certified
	solution $\mathbf x_0$ for $\mathbf h$.}
\KwOutput{A certified solution for $\mathbf f$.}
Stipulate an arbitrary value $N_0$ and set $k=0$\;
\Repeat{}{
		Find a solution $\mathbf x_0$ for a random 
		$\mathbf g \in \secrev{\mathcal N}(\mathbf 0, I; \mathscr F)$\;
		Apply the recurrence 
		\eqref{rec-zero} to the path 
		$\mathbf q_t = \mathbf g + t \mathbf f$
		for $t \in [0,\infty]$ up to 
		$N \le N_k$ renormalized Newton steps,
		and in case the recurrence did not terminate set
		$N=N_k+1$\;
	\If {$N \le N_k$}{\Return 
		$N(\mathbf f \cdot R(\mathbf x_N))+\mathbf x_N$}
		Set $N_{k+1} \leftarrow \sqrt{2} N_k$\;
		Increase $k$ by one\;
	}
\caption{Homotopy for finding just one root}
\end{algorithm}	
\medskip


This algorithm will provide a certified approximate root of $\mathbf f$. 
The cost of this procedure may be divided in two parts: the cost of solving
random systems $\mathbf g$, and the cost of solving $\mathbf f$ given a
certificed approximate solution for $\mathbf g$. We can bound the second
part by {\em half} of the bound in Theorem~\ref{mainE}, viz.
\[
	\frac{3+2\sqrt{2}}{2} N_* .
\]
Recall from \eqref{N-star} that $N_*$ depends linearly on the normalized
mixed volume $n! V$, where $V \defeq V(\conv{A_1}, \dots, \conv{A_n})$, and
also on the mixed area $V'$.

We can improve this bound by restating Problem~\ref{probA} in terms of {\em sampling} instead of {\em solving}. Recall that the solution variety is the set 
$\mathscr S \defeq 
\{(\mathbf g, \boldsymbol \zeta) \in \mathscr F \times \mathscr M:
\mathbf g \cdot V(\boldsymbol \zeta) =\mathbf  0\}$. The space $\mathscr F$ is
endowed with the unit normal Gaussian probability density $\secrev{\mathcal N}(\mathbf 0,I; 
\mathscr F)$. The pull-back of this density by the canonical projection
$\pi_1:\mathscr S \rightarrow \mathscr F$ has total mass $n! V$. From now
on, we assume that the solution variety $\mathscr S$ is endowed with the
probability measure
\[
	\frac{1}{n! V} \pi_1^* \secrev{\mathcal N}(\mathbf 0,I; \mathscr F) .
\]
{\em Sampling} means finding a random pair $(\mathbf g, \boldsymbol \zeta)$
with respect to that measure. This pair may be represented by
$(\mathbf g, \mathbf x_0)$ where $\mathbf x_0$ is a certified approximate
root for $\boldsymbol \zeta$. Now we can ask:

\begin{newproblem}\label{probA2}
Is it possible to sample in $\mathscr S$
with a uniform algorithm, in expected time polynomial in
$S$ and linear in $\frac{Q \det \Lambda}{n!V} \nu (\log(\nu) + \log(d_r))$ ?
\end{newproblem}

Suppose that the answer is yes.
Let $\mathbf q_t = \mathbf g + t \mathbf f$ as in Theorem~\ref{mainD}. With probability $3/4$,
\[
\sum_{\mathbf z_{\tau} \in \mathscr Z(\mathbf q_{\tau})} \mathscr L (\mathbf q_t, \mathbf z_t; 0, \infty)
	\le C 
	Q \changed{n^{\thirdrev{\frac{5}{2}}}} S\max_i (S_i) 
	K L_{\mathbf f} 
	\kappa_{\mathbf f} \mu_{\mathbf f}^2 \secrev{\nu}
	\ \secrev{\mathrm{LOGS}_{\mathbf f}}
.
\]
Given the bound above, Markov's inequality guarantees that
at least half of the paths satisfy
\begin{equation}\label{lucky}
\mathscr L (\mathbf q_t, \mathbf z_t; 0, \infty)
	\le 2 C 
	\frac{Q \det(\Lambda)}{n!V}\changed{n^{\thirdrev{\frac{5}{2}}}} S\max_i (S_i) 
	K L_{\mathbf f}
	\kappa_{\mathbf f} \mu_{\mathbf f}^2 \secrev{\nu}
	\ \secrev{\mathrm{LOGS}_{\mathbf f}}
.
\end{equation}
The algorithm below is essentially Algorithm~\ref{algorithm-3} with line 3 replaced
by a sampling. With probability $\ge 3/8$, it will track a path satisfying
\eqref{lucky}.
\medskip

\begin{algorithm}[H]\label{algorithm-4}
\SetKwInput{KwInput}{Input}                
\SetKwInput{KwOutput}{Output}              
\SetKwRepeat{Repeat}{repeat}{end}
\KwInput{$\mathbf f, \mathbf h \in \mathscr F$ and a certified
	solution $\mathbf x_0$ for $\mathbf h$.}
\KwOutput{A certified solution for $\mathbf f$.}
Stipulate an arbitrary value $N_0$ and set $k=0$\;
\Repeat{}{
		Sample $(\mathbf g, \mathbf x_0) \in \mathscr S$\;
		Apply the recurrence 
		\eqref{rec-zero} to the path 
		$\mathbf q_t = \mathbf g + t \mathbf f$
		for $t \in [0,\infty]$ up to 
		$N \le N_k$ renormalized Newton steps,
		and in case the recurrence did not terminate set
		$N=N_k+1$\;
	\If {$N \le N_k$}{\Return 
		$N(\mathbf f \cdot R(\mathbf x_N))+\mathbf x_N$}
		Set $N_{k+1} \leftarrow \sqrt{2} N_k$\;
		Increase $k$ by one\;
	}
\caption{Homotopy with sampling for finding just one root}
\end{algorithm}	
\medskip


Arguing as before, we obtain
\begin{corollary}\label{mainF}
	Suppose that Problem~\ref{probA2} has a positive answer.
	Then there is a randomized algorithm with input 
$n, A_1, \dots, A_n, \mathbf f \in \mathscr F$
	satisfying the following properties. If the algorithm 
	terminates, it produces 
	a certified approximate root $\mathbf x$ of $\mathbf f$.
	If 
	$A_1, \dots, A_n$ satisfy the conditions of Theorem~\ref{mainD}
	(\eqref{mainD-supports}-\eqref{mainD-spaces}) and if
	$r(\mathbf f) \ne 0$, then with probability one
	the algorithm terminates within
	expected number of renormalized Newton iterations
	polynomial on the input size and linear on 
\[
\eta^{-2} \left(\sum_{i=1}^n \delta_i^2\right) 
	\max\left(1, \frac{V'}{V} \right) 
	\kappa_{\mathbf f}^{\frac{3}{2}} \mu_{\mathbf f}^2 {\mathrm LOGS}_{\mathbf f}
.
\]
\end{corollary}

Thus, a solution for Problem~\ref{probA2} implies a mostly efficient algorithm for finding
one solution of a fixed $\mathbf f \in \mathscr F$. It is still
linear in the 
isoperimetric ratio $V'/V$ and quadratic on the reciprocal facet gap $\delta^{-1}$. 
I also asked:

\begin{problem}\label{probB} Can every finite zero of a random polynomial system be
found approximately, on the average, in time polynomial in the input size $S$
with a uniform algorithm running in parallel, one parallel process for every
expected zero?
\end{problem}

The algorithm in Theorem~\ref{mainD} can clearly be parallelized. The methods in this paper do not seem to provide any insight for the parallel complexity of tracking each path. Yet, using the Markov inequality, an affirmative answer to problem~\ref{probB} would provide a parallel algorithm guaranteed to produce approximate solutions for at least {\em half} of the zeros 
within the bound in Corollary~\ref{mainF}.
However, if one drops
the parallel complexity assumption in Problem~\ref{probB}, it makes sense to ask:

\begin{newproblem}\label{probB2}
	Is it possible to find an approximate solution set
for a random polynomial 
with a uniform algorithm, in expected time polynomial in
$S$ and linear in $Q \log(d_r)$ ?
\end{newproblem}

A positive answer would imply an uniform algorithm for solving sparse
polynomials, within a time bound similar to the one in Theorem~\ref{mainD}.
}
}{}

\section{Renormalized homotopy}
\label{sec-homotopy}
\smallskip
\centerline{
\begin{tabular}{|r|r|r|}
\hline
\hline
\multicolumn{3}{|c|}{Constants in Theorem~\ref{th-A} and its proof}
\\
\hline
$\alpha_* = 0.074,609,958\cdots$
&
$\alpha_{**} = 0.096,917,682\cdots$
&
$\delta_* = 0.085,180,825\cdots$ 
\\
\hline
$u_* = 0.129,283,177\cdots$
&
$u_{**}= 0.007,556,641\cdots$
&
$u_{***}=0.059,668,617\cdots$
\\
\hline
\hline
\end{tabular}}
\medskip

The objective of this section is to prove Theorem~\ref{th-A}.
We first prove a technical result for later use.

\subsection{Technical Lemma}

\fourthrev{Let $\nu_i$ be the distortion invariant as defined in \eqref{distortion}, and $\nu=\max \nu_i$.}

\begin{lemma}\label{lem:fRdist}
Let $\mathbf f_i \in \mathscr F_{A_i}$ and $\mathbf x \in \mathbb C^n$.
Assume that  
$\mathbf m_i(\mathbf 0)=\mathbf 0$ for $i=0,\dots,n$.
Then for all $i=1, \dots, n$, 
\[
	d_{\mathbb P}(\mathbf f_i, \mathbf f_i \cdot R_i(\mathbf x)) \le \sqrt{5}\ \|\mathbf x\|_{i,\mathbf 0}\  \secrev{\nu_i}.
\]
Moreover,
\[
d_{\mathbb P}(\mathbf f, \mathbf f \cdot R(\mathbf x)) \le \sqrt{5}\ \|\mathbf x\|_{\mathbf 0}\  \secrev{\nu}.
\]

\end{lemma}
\begin{proof}
Without loss of generality, let $\|\mathbf f_i\|=1$. 
\secrev{Also, recall that $\nu = \max \nu_i$.
Since we required that $m_i(\mathbf 0)=\mathbf 0$, the expression for $\nu_i$ in
\eqref{distortion} simplifies to 
\[
\nu_i = 
\sup_{\mathbf a \in A_i}
\sup_{\|\mathbf u\|_{i,\mathbf 0} \le 1} |\mathbf a \mathbf u|
.
\]
We will first prove the lemma for $\mathbf x$ real, then
for $\mathbf x$ imaginary, and finally deduce the full statement.}
\medskip

\noindent
\secrev{{\bf Part 1:} Assume} that $\mathbf x$ is a real vector. Then,
\begin{eqnarray*}
d_{\mathbb P}(\mathbf f_i, \mathbf f_i \cdot R_i(\mathbf x)) &=& 
\inf_{t \in \mathbb C} \frac{ \| \mathbf f_i - t \mathbf f_i \cdot R_i(\mathbf x)\|}
{\|\mathbf f_i\|} \\
	&\le& \sqrt{ \sum_{\mathbf a \in A_i} |f_{i,\mathbf a}|^2 |1-e^{\mathbf a\mathbf x-\changed{\ell_i}(\mathbf x)}|^2}
\end{eqnarray*}
	where we set \secrev{$t = e^{-\ell_i(\mathbf x)}$} and $\changed{\ell_i}(\mathbf x) = \max_{\mathbf a\in A_i} \mathbf a \Re(\mathbf x)$  
	\secrev{was defined in \eqref{def-elli}.}
	For all $\mathbf a \in A_i$, $\mathbf a\mathbf x-\changed{\ell_i}(\mathbf x) \le 0$.
\secrev{The inequality $0 \le 1 - e^{-\lambda} \le \lambda$ for $\lambda \ge 0$ implies that}
\secrev{
\[
d_{\mathbb P}(\mathbf f_i, \mathbf f_i \cdot R_i(\mathbf x)) \le 
	\sqrt{ \sum_{\mathbf a \in A_i} |f_{i,\mathbf a}|^2 |\mathbf a\mathbf x-\changed{\ell_i} (\mathbf x)|^2}
.
\]
Since $\|f_i\|=1$,
\begin{equation}\label{tech1}
d_{\mathbb P}(\mathbf f_i, \mathbf f_i \cdot R_i(\mathbf x)) \le 
\max_{\mathbf a\in A_i} |\mathbf a\mathbf x-\changed{\ell_i}(\mathbf x)| 
.
\end{equation}
\secrev{
\medskip

\noindent
	{\bf Part 2:} Assume} that $\mathbf x$ is pure imaginary.
Similarly, we obtain
\begin{eqnarray*}
d_{\mathbb P}(\mathbf f_i, \mathbf f_i \cdot R(\mathbf x)) 
	&=& \inf_{t \in \mathbb C} \frac{ \| \mathbf f_i - t \mathbf f_i \cdot R_i(\mathbf x)\|}
{\|\mathbf f_i\|}\\
&\le&
\inf_{t \in \mathbb C} \| \mathbf f_i - \mathbf f_i \cdot R_i(\mathbf x)\|
	\\
&\le& \sqrt{ \sum_{\mathbf a \in A_i} |\mathbf f_{i,\mathbf a}|^2 |1-e^{\mathbf a \mathbf x}|^2}
\end{eqnarray*}
	using $\| \mathbf f_i \|=1$. For all real $\lambda$, $|1-e^{i \lambda}| \le |\lambda|$ and hence,
	\begin{equation}\label{tech2}
d_{\mathbb P}(\mathbf f_i, \mathbf f_i \cdot R(\mathbf x)) 
\le \max_{\mathbf a \in A_i} |\mathbf a \mathbf x|
\end{equation}
}
\secrev{\medskip

\noindent
{\bf Part 3:}
For a general $\mathbf x \in \mathbb C^n$, we can decompose
the operator $R_i(\mathbf x) = R_i(\Re(\mathbf x))\cdot R_i(\Im(\mathbf x))$.
The triangular inequality and Theorem 
	\ref{th-renormalization}(c) imply that
\begin{eqnarray*}
d_{\mathbb P}(\mathbf f_i, \mathbf f_i \cdot R_i(\mathbf x)) 
&\le&
d_{\mathbb P}(\mathbf f_i, \mathbf f_i \cdot R_i(\Re(\mathbf x))) + \\
	&& + d_{\mathbb P}(\mathbf f_i \cdot R_i(\Re(\mathbf x)),
\mathbf f_i \cdot R_i(\Re(\mathbf x))\cdot R_i(\Im(\mathbf x))) 
.
\end{eqnarray*}
Let 
\[
\mathbf g_i = 
\frac{\mathbf f_i \cdot R_i(\Re(\mathbf x))}
{\|\mathbf f_i \cdot R_i(\Re(\mathbf x))\|}
.
\]
Since $\|\mathbf g_i\|=1$, we have now
\[
d_{\mathbb P}(\mathbf f_i, \mathbf f_i \cdot R_i(\mathbf x)) 
\le
d_{\mathbb P}(\mathbf f_i, \mathbf f_i \cdot R_i(\Re(\mathbf x))) +
d_{\mathbb P}(\mathbf g_i, \mathbf g_i \cdot R_i(\Im(\mathbf x))) 
.
\]
Replacing by \eqref{tech1} and \eqref{tech2},
\begin{eqnarray*}
d_{\mathbb P}(\mathbf f_i, \mathbf f_i \cdot R_i(\mathbf x)) 
	&\le& \max_{\mathbf a \in A_i} |\mathbf a(\Re(\mathbf x))-\ell_i(\mathbf x)| + 
	\max_{\mathbf a \in A_i} |\mathbf a (\Im(\mathbf x))|
\\
&\le& 2 \max_{\mathbf a \in A_i} |\mathbf a(\Re(\mathbf x))| + \max_{\mathbf a \in A_i} |\mathbf a (\Im(\mathbf x))|
\\
&\le&
\sqrt{5} \max_{\mathbf a \in A_i}  |\mathbf a \mathbf x|
\end{eqnarray*}
where the last inequality comes from:
\[
\max_{c^2+s^2 \le 1} 2c+s = \max_{0 \le t \le 2\pi} 2 \cos(t)+\sin(t) = \sqrt{5}
.
\]
Let $\mathbf y = \frac{1}{\|\mathbf x\|_{i, \mathbf 0}} \mathbf x$, we obtain
\[
d_{\mathbb P}(\mathbf f_i, \mathbf f_i \cdot R_i(\mathbf x)) 
\le
\sqrt{5} \| \mathbf x\|_{i,\mathbf 0} \max_{\mathbf a \in A_i} |\mathbf a \mathbf y|
=
\sqrt{5} \|\mathbf x\|_{i,\mathbf 0} \ \nu_i 
.
\]}

Finally,
\[
d_{\mathbb P}(\mathbf f, \mathbf f \cdot R(\mathbf x))^2
=
\sum_i 
d_{\mathbb P}(\mathbf f_i, \mathbf f_i \cdot R_i(\mathbf x))^2
\le
5 
\sum_i 
\secrev{\nu_i}
\|\mathbf x\|_{i,\mathbf 0}^2
\le
5 \secrev{\nu}\|\mathbf x\|_{\mathbf 0}^2
.
\]
\end{proof}
\subsection{Proof of Theorem~\ref{th-A}}

We claim first that for $\alpha_*$ small enough,
the recurrence \eqref{rec-zero} of Definition~\ref{def-recurrence} is well-defined
in the sense that given previously produced $t_j<T$ and $x_j$,
there is $t_{j+1}>t_j$
satisfying the condition in \eqref{rec-zero}.
This will follow from the intermediate value theorem after
replacing $\mathbf f$ by $\mathbf q_{t_j}$ in 
the \secrev{lemma} below.

\begin{lemma}\label{lem-well-defined}
\changed{Let $\alpha_0$ be the constant of Theorem~\ref{th-alpha}
	and define $\psi(\alpha) = 1 - 4\alpha + 2 \alpha^2$.}
Assume that
\[
	\frac{1}{2} \secrev{
\beta( \mathbf f \cdot R(\mathbf x_{j}) )\ \mu( \mathbf f \cdot R(\mathbf x_{j}) )\ \nu} 
\le \alpha \le \alpha_0
.
\]
	Moreover, set $\mathbf x_{j+1}=N(\mathbf f \changed{\cdot}R(\mathbf x_j), \mathbf 0) + \mathbf x_j$
as in \eqref{rec-zero}.
Then,
\begin{equation}\label{thA-ai2}
	\frac{1}{2} \secrev{ 
\beta( \mathbf f \cdot R(\mathbf x_{j+1}) )\ \mu( \mathbf f \cdot R(\mathbf x_{j+1}) )\ \nu} 
	\le \alpha^2 \frac{(1-\alpha)}{\psi(\alpha)(1-2 \sqrt{5} \alpha)}
\end{equation}
\end{lemma}

Numerically, the bound in the right-hand side of \eqref{thA-ai2} is smaller than $\alpha$
for all $0< \alpha \le 0.155,098\dots$ and $\alpha_{*} < 1.555$
\changed in \eqref{rec-zero}.

\begin{proof}
	\secrev{Theorem \ref{prop-mu-gamma}} 
applied to $(\mathbf f \cdot R(\mathbf x_{j}),\mathbf 0)$ yields
\[
	\secrev{
\beta(\mathbf f \cdot R(\mathbf x_{j}) )
\ 
\gamma(\mathbf f \cdot R(\mathbf x_{j}) ) 
	}\le 
\frac{1}{2} 
	\secrev{\beta(\mathbf f \cdot R(\mathbf x_{j}) )
	\mu(\mathbf f \cdot R(\mathbf x_{j}) )}
\secrev{\nu} 
\le \alpha
.
\]
Let $\mathbf y_0 = \mathbf x_j$ and $\mathbf F(\mathbf y)=\mathbf f \cdot V(\mathbf y)$. 
Lemma~\ref{toric-to-classical} implies that
\[
\beta(\mathbf F,\mathbf y_0) \ \gamma(\mathbf F,\mathbf y_0) \le \alpha \le \alpha_0
\]
so we are in the conditions of Theorem~\ref{th-alpha}.
Moreover, $\mathbf x_{j+1} = \mathbf y_1 = N(\mathbf F,\mathbf y_0)$ 
	\changed{and $\secrev{\beta(\mathbf f \cdot R(\mathbf x_{j+1}))}=\beta(F,y_1)$. According} to \ocite{Bezout1}*{Prop. 3 p.478},
\[
\beta(\mathbf F,\mathbf y_{1}) \le \frac{1-\alpha}{\psi(\alpha)} 
\alpha \beta(\mathbf F, \mathbf y_{0})
\]
	Let $\mathbf g_i = \mathbf f_i \cdot R(\mathbf x_j) = \mathbf f_i \cdot R(\mathbf y_0)$. 
	Lemma~\ref{lem:fRdist} yields
\begin{eqnarray*}
	d_{\mathbb P}( \mathbf f_i \cdot R(\mathbf x_j), \mathbf f_i \cdot R(\mathbf x_{j+1})) &=&
	d_{\mathbb P}( \mathbf g_i , \mathbf g_i \cdot R(\mathbf x_{j+1}-\mathbf x_j)) \\
	&\le&
	\sqrt{5}\, \|\mathbf x_{j+1}-\mathbf x_j\|_{i\mathbf 0}\ \secrev{\nu}.
\end{eqnarray*}
Hence,
	\[
	d_{\mathbb P}( \mathbf f \cdot R(\mathbf x_j), \mathbf f \cdot R(\mathbf x_{j+1})) \le 
	\sqrt{5} \|\mathbf x_{j+1}-\mathbf x_j\|_{\mathbf 0}\ \secrev{\nu}
=
	\sqrt{5} \secrev{\beta(	\mathbf f \cdot R(\mathbf x_j))}\ \secrev{\nu} .
	\]
	From Proposition~\ref{prop-mu}(\ref{prop-mu-b}),
\[
	\secrev{\mu( \mathbf f \cdot R(\mathbf \mathbf x_{j+1}))} \le \frac
		{	\secrev{\mu( \mathbf f \cdot R(\mathbf x_{j}))} }
		{1 - \secrev{\mu( \mathbf f \cdot R(\mathbf x_{j}))} d_{\mathbb P}(\mathbf f \cdot R(\mathbf x_j), \mathbf f \cdot R(\mathbf x_{j+1})) }
		\le
		\frac
		{	\secrev{\mu( \mathbf f \cdot R(\mathbf x_{j}))} }
		{1 - 2 \sqrt{5}\alpha}
	.
\]
Putting all together,
	\[
\frac{1}{2} 
	\secrev{\beta( \mathbf f \cdot R(\mathbf x_{j+1}))}\ \secrev{\mu( \mathbf f \cdot R(\mathbf x_{j+1}))}\ \secrev{\nu} 
	\le \alpha^2 \frac{(1-\alpha)}{\psi(\alpha)(1-2 \sqrt{5} \alpha)} .
	\]
\end{proof}

Equation \eqref{thA-ui} \changed{in Theorem~\ref{th-A} will follow}
from item (2) of the Lemma~\ref{lem-u} below. 
The other two items will be instrumental to the complexity bound. 

\begin{lemma}\label{lem-u} With the notations of Theorem~\ref{th-alpha},
let $u_0=\frac{5-\sqrt{17}}{4}$.
For $0 <\alpha<\alpha_{0}$, 
define 
\[
u_{*}=
	\frac{\alpha r_0(\alpha)}
	{1 - 2 \sqrt{5} r_0(\alpha) \alpha}
\]
\[
u_{**}=
	\frac{\alpha r_1(\alpha)}
	{1 - 2 \sqrt{5} r_1(\alpha) \alpha}
\]
and $u_{***}=\frac{\psi(u_*)}{u_*+\psi(u_*)}\alpha$.
Then,
	\begin{enumerate}[(a)]
\item for $t_j \le t \le t_{j+1}$, $u_j(t) < u_* < u_0$, 
\item $u_{j}(t_j) \le u_{**}$, and 
\item \changed{If $t_{j+1}<T$, then} $u_j(t_{j+1}) \ge u_{***}$.
\end{enumerate}
\end{lemma}

The graphs of $u_*$, $u_{**}$ and $u_{***}$ as functions of $\alpha$
are plotted in Figure~\ref{plot-u}. In view of Lemma~\ref{lem-u}, we
define $\alpha_{**} \simeq 0.096,917\cdots$ as the solution of $u_*(\alpha_{**}) = u_0$.
The bound in the right-hand side of \eqref{thA-ai2} is strictly smaller than $\alpha$ for all $0< \alpha \le \alpha_{**} \le 0.155,098\dots$. Also notice that $\alpha_{**}$ is smaller than the constant $\alpha_0$ from Theorem~\ref{th-alpha}.

\begin{figure}
\centerline{\resizebox{\textwidth}{!}{\includegraphics{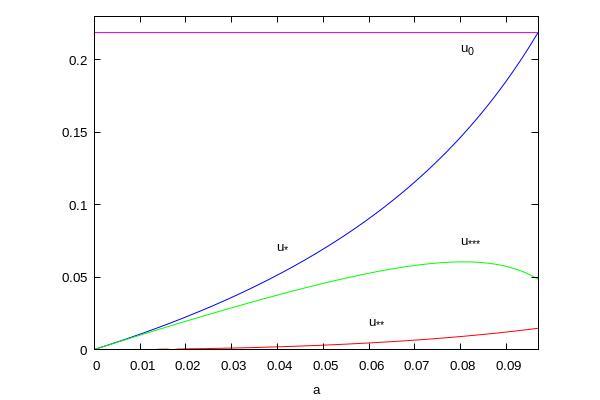}}}
\caption{The graphs of $u_0$,
$u_{*}(\alpha)$, $u_{**}(\alpha)$ and $u_{***}(\alpha)$.
\label{plot-u}}
\end{figure}

\begin{proof}
	Apply Smale's alpha-theorem (Theorem~\ref{th-alpha})
to the point $\mathbf y_1=\mathbf x_{j+1}$, for \changed{the system} $\mathbf F(\mathbf y)=\mathbf q_{t}\cdot V(\mathbf y)$, 
$t_j \le t \le t_{j+1}$.
\secrev{
	The construction of $t_{j+1}$ in
	the recurrence \eqref{rec-zero} guarantees 
	for those values of $t$ that
	\begin{equation}\label{cert-root-hypothesis}
\frac{1}{2}\secrev{\beta(\mathbf q_{t} \changed{\cdot} R(\mathbf y_{1}))}
\secrev{\mu (\mathbf q_t \cdot R(\mathbf y_{1}))} \secrev{\nu} 
\le \alpha < \alpha_0.
\end{equation}
From Lemma~\ref{toric-to-classical}, $\alpha(\mathbf F, \changed{\mathbf y_1}) = \alpha(\mathbf{f},\mathbf x_{j+1}) \le \alpha \le \alpha_0$.
	Theorem~\ref{th-alpha}(b) asserts that the Newton iterates of 
$\mathbf y_{1}$ converge to a zero $\mathbf z_t$ of $\mathbf F_t$
with
	\begin{equation}\label{y1-converges}
	\| \mathbf y_{1}- \mathbf z_t \|_{\mathbf 0} \le r_0(\alpha)\, \secrev{\beta(\mathbf q_{t} \changed{\cdot} R(\changed{\mathbf y_{1}}))} 
.
\end{equation}
%
Hence,
\begin{align*}
	\secrev{\mu (\mathbf q_t \cdot R(\mathbf z_t))} 
	&\le
	\frac{ \secrev{\mu (\mathbf q_t \cdot R(\mathbf y_1))} }
	{1 - \secrev{\mu (\mathbf q_t \cdot R(\mathbf y_1))} d_{\mathbb P}( \mathbf q_t \cdot R(\mathbf y_1), \mathbf q_t \cdot R(\mathbf z_t))}
	&& \text{by Prop. ~\ref{prop-mu}(\ref{prop-mu-b})}\\
&\le
	\frac{ \secrev{\mu (\mathbf q_t \cdot R(\mathbf y_{1}))} }
	{1 - \sqrt{5} \secrev{\mu(\mathbf q_t \cdot R(\mathbf y_{1}))} \|\mathbf y_{1}-\mathbf z_t\|_{\mathbf 0} \secrev{\nu}}
	&& \text{by Lemma~\ref{lem:fRdist}}\\
	&\le
	\changed{\frac{ \secrev{\mu (\mathbf q_t \cdot R(\mathbf y_{1}))} }
	{1 - \sqrt{5} \secrev{\mu(\mathbf q_t \cdot R(\mathbf y_{1}))} 
	r_0(\alpha) \secrev{\beta(\mathbf q_t \cdot R(\mathbf y_{1}))} \secrev{\nu}}
	}
	&&\text{by \eqref{y1-converges}}\\
&\le
	\frac{ \secrev{\mu (\mathbf q_t \cdot R(\mathbf y_{1}))} }
	{1 - 2 \sqrt{5} r_0(\alpha) \alpha} 
&&\text{by \eqref{cert-root-hypothesis}}.
\end{align*}}
Therefore,
\[
	u_j(t) =
	\frac{1}{2} \| \mathbf z_t - \mathbf x_{j+1}\| 
	\secrev{\mu (\mathbf q_t \cdot R(\mathbf z_t))} \secrev{\nu}	
	\le 
	\frac{\alpha r_0(\alpha)}
	{1 - 2 \sqrt{5} r_0(\alpha) \alpha} = u_*
.
\]
By construction $u_* < u_0$,
and equation  \eqref{thA-ui} follows.
We may obtain a sharper estimate for $u_j(t_j)$, since $\mathbf y_1=\mathbf x_{j+1}$ is the
	iterate of $\mathbf y_0=\mathbf x_j$. \secrev{The last
	item of Theorem~\ref{th-alpha} yields}
\[
	\| \mathbf y_{1}- \mathbf z_{t_j} \|_{\mathbf 0} \le r_1(\alpha) \secrev{\beta(\mathbf q_{t_j} \cdot R(\mathbf x_{j})} 
,
\]
and by the very same reasoning,
\[
u_j(t_j) \le 
	\frac{\alpha r_1(\alpha)}
	{1 - 2 \sqrt{5} r_1(\alpha) \alpha} 
= u_{**}.
\]

If $t_{j+1} \ne T$, then by construction
\[
	\frac{1}{2} \secrev{\beta( \mathbf q_{t_{j+1}} R(\mathbf x_{j+1})) \mu( \mathbf q_{t_{j+1}} R(\mathbf x_{j+1}))\ \nu} = \alpha
.
\]
Thus,
\[
\frac{1}{2} \|\mathbf x_{j+1}-\mathbf z_{t_{j+1}}\|_{\mathbf 0}
\ \secrev{\mu( \mathbf q_{t_{j+1}} R(\mathbf x_{j+1}))\ } \le u_* < u_0
.
\]
Let $\mathbf y_0=\mathbf x_{j+1}$. From \cite{BCSS}*{Proposition 1 p. 157}
the Newton iterate 
$\mathbf y_1=N(\mathbf F,\mathbf y_0)$ satisfies 
\[
\|\mathbf y_1 - \mathbf z_{t_{j+1}}\|_{\mathbf 0}
\le 
\frac{u_*}{\psi(u_*)}  \|\mathbf y_0-\mathbf z_{t_{j+1}}\|_{\mathbf 0}
.
\]
Therefore,
\begin{eqnarray*}
	\secrev{\beta( \mathbf q_{t_{j+1}} R(\mathbf x_{j+1}) )}
&=&
\|\mathbf y_0-\mathbf y_1\|_{\mathbf 0}
\\
&\le& 
\|\mathbf y_0-\mathbf z_{t_{j+1}}\|_{\mathbf 0}
+
\|\mathbf z_{t_{j+1}}-\mathbf y_1\|_{\mathbf 0}
\\
&\le&
\left( 1+ \frac{u_*}{\psi(u_*)} \right)
\|\mathbf x_{j+1}-\mathbf z_{t_{j+1}}\|_{\mathbf 0}
.
\end{eqnarray*}
It follows that 
\begin{eqnarray*}
	\alpha &\le& 
\frac{1}{2}
\left( 1+ \frac{u_*}{\psi(u_*)} \right)
\|\mathbf x_{j+1}-\mathbf z_{t_{j+1}}\|_{\mathbf 0}
	\ \secrev{\mu( \mathbf q_{t_{j+1}} R(\mathbf x_{j+1}))\ \nu} \\
&\le&
\left( 1+ \frac{u_*}{\psi(u_*)} \right) u_{j}(t_{j+1})
\end{eqnarray*}
Since $u_{***}=\frac{\psi(u_*)}{\psi(u_*)+u_*} \alpha$, 
\[
u_{***} \le u_j(t_{j+1}).
\]
\end{proof}

Towards the proof of Theorem~\ref{th-A}, let $\mu_j = \secrev{\mu(\mathbf q_{t_j} \cdot R(\mathbf z_{t_j}))}$ and let
\begin{eqnarray*}
	d_{\mathrm{max}}(t) &=& 
\max_{t_j \le \tau \le t} 
	d_{\mathbb P}( 
\mathbf q_{\tau} \cdot R(\mathbf z_{\tau}), 
	\mathbf q_{t_j} \cdot R(\mathbf z_{t_j})) 
.
\end{eqnarray*}

Clearly $d_{\mathrm{max}}(\secrev{t_j})=0$ and $d_{\mathrm{max}}(t)$ is a continuous function.

\begin{lemma}\label{lem-rt}
\[
d_{\mathrm{max}}(t_{j+1}) \mu_j \le (1+d_{\mathrm{max}}(t_{j+1}) \mu_j) \mathscr L(t_j, t_{j+1})
.\]
Furthermore if $\mathscr L(t_j, t_{j+1}) <1$,
\[
d_{\mathrm{max}}(t_{j+1}) \mu_j \le \frac{\mathscr L(t_j, t_{j+1})}{1-\mathscr L(t_j, t_{j+1})}
\]
\end{lemma}
\begin{proof}
The projective distance is always smaller than the Riemannian metric,
since they share the arc length element. The Riemannian distance between two points is smaller 
or equal than the Riemannian length of an arbitrary path between those two points. We obtain the upper bound 
\begin{eqnarray*}
d_{\mathrm{max}}(t_{j+1})  
&\le&
\max_{t_j \le \tau \le t_{j+1}} \int_{t_j}^{\tau}
\left\|
	\frac{\partial}{\partial \tau} \left( \mathbf q_{\tau} \cdot R(\mathbf z_{\tau}) \right)
	\right\|_{\mathbf q_{\tau} \cdot R(\mathbf z_{\tau})} \ \dd \tau
\\
&\le&
\int_{t_j}^{t_{j+1}}
\left\|
	\frac{\partial}{\partial \tau} \left( \mathbf q_{\tau} \cdot R(\mathbf z_{\tau}) \right)
	\right\|_{\mathbf q_{\tau} \cdot R(\mathbf z_{\tau})} \ \dd \tau
.
\end{eqnarray*}
and $0\le d_{\mathrm{max}}(t) \le d_{\mathrm{max}}(t_{j+1})$ for $t_j \le t \le t_{j+1}$.
From the definition of $d_{\mathrm{max}}(t)$, we have a trivial lower bound
\begin{equation} \label{bound-upper-dP}
	d_{\mathbb P}(
	\mathbf q_{t_j} \cdot R(\mathbf z_{t_j}),
	\mathbf q_{\tau} \cdot R(\mathbf z_{\tau})) \le d_{\mathrm{max}}(t_{j+1}) .
\end{equation}
	Proposition~\ref{prop-mu}(\ref{prop-mu-b}) combined with equation
	\eqref{bound-upper-dP} yields the estimate
\begin{equation}\label{mu-bound}
\frac{\mu_j}{1+d_{\mathrm{max}}(t_{j+1}) \mu_j}
\le 
	\secrev{\mu(\mathbf q_{t} \cdot 
	R(\mathbf z_{t}))}
\le
\frac{\mu_j}{1-d_{\mathrm{max}}(t_{j+1}) \mu_j}
.
\end{equation}
We can combine the upper and lower bounds:	
\begin{eqnarray*}
	d_{\mathrm{max}}(t_{j+1}) \mu_j &\le&
\int_{t_j}^{t_{j+1}}
\mu_j \left\|
	\frac{\partial}{\partial \tau} \left( \mathbf q_{\tau} \cdot R(\mathbf z_{\tau}) \right)
	\right\|_{\mathbf q_{\tau} \cdot R(\mathbf z_{\tau})} \ \dd \tau
\\
	&\le&
(1+d_{\mathrm{max}}(t_{j+1}) \mu_j)
\int_{t_j}^{t_{j+1}}
	\mu(\mathbf q_{\tau} \cdot R(\mathbf z_{\tau})) 
\\
	&& \hspace{8em} \times 
	\left\|
	\frac{\partial}{\partial \tau} \left( \mathbf q_{\tau} \cdot R(\mathbf z_{\tau}) \right)
	\right\|_{\mathbf q_{\tau} \cdot R(\mathbf z_{\tau})} \ \dd \tau
\\
	&\changed{\le}&
 (1+ d_{\mathrm{max}}(t_{j+1})\mu_j) 
\mathscr L(t_j, t_{j+1})
.
\end{eqnarray*}
Rearranging terms under the assumption $\mathscr L(t_j, t_{j+1})<1$,
\[
d_{\mathrm{max}}(t_{j+1}) \mu_j \le \frac{\mathscr L(t_j, t_{j+1})}{1-\mathscr L(t_j, t_{j+1})}
.
\]
\end{proof}

\begin{proof}[Proof of Theorem~\ref{th-A}]
\begin{figure}
\centerline{\resizebox{\textwidth}{!}{\includegraphics{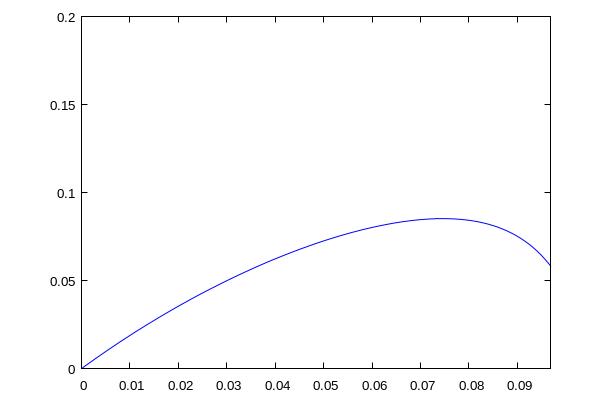}}}
\caption{The value of $\delta(\alpha)$ in function of $\alpha$ for
$0 < \alpha < \alpha_{**}$. The maximum is $\delta_* = \delta(\alpha_*)$.
\label{plot-delta}}
\end{figure}

A path $(\mathbf z_t)_{t \in [t_j, t_{j+1}]}$ can be produced as in Lemma~\ref{lem-u} for 
each value of $j$ by extending the previous definition to $t_{j+1}$: 
For each $t \in [t_j, t_{j+1}]$,
define $\mathbf y_1(t) = \mathbf x_{j+1}$ and inductively, $\mathbf y_{k+1}(t)$ as the Newton iterate of $\mathbf y_k(t)$ for the system $\mathbf F_t(\mathbf y) = \mathbf q_t \cdot V(\mathbf y)$.
Equation  \eqref{thA-ui} guarantees quadratic convergence
to a zero $\mathbf z_t$ because $u_* < u_0=\frac{3-\sqrt{7}}{2}$, so we
	can apply Theorem~\ref{th-gamma} combined with \secrev{Theorem \ref{prop-mu-gamma}}.
Moreover, for $t_j \le t \le t_{j+1}$, $\|\mathbf y_k(t) - \mathbf z_t\|_{\mathbf 0} \le 2^{-2^k+1} u_*$
so the convergence is uniform.

We claim that each $(\mathbf z_t)_{t \in [t_j, t_{j+1}]}$ is continuous. Indeed, let $\epsilon > 0$.
There is $k$ such that for all $\tau \in [t_j, t_{j+1}]$,
	$\|\mathbf y_k(\tau) - \mathbf z_{\tau}\|_{\mathbf 0} \le \epsilon/3$. Moreover, $\mathbf y_k(\tau)$ is continuous
	in $t$ so there is $\changed{\iota=\iota(t)} > 0$ with the property that for all $t' \in [t_j, t_{j+1}]$, $|t-t'|< \changed{\iota}$
	implies that $\|\mathbf y_k(t) - \mathbf y_k(t')\|_{\mathbf 0} < \epsilon/3$. Whence,
\[
	\|\mathbf z_t - \mathbf z_{t'}\|_{\mathbf 0}  \le
\|\mathbf z_t - \mathbf y_k(t)\|_{\mathbf 0}  
+
\|\mathbf y_k(t) - \mathbf y_k(t')\|_{\mathbf 0}  
+
\|\mathbf y_k(t') - \mathbf z_{t'}\|_{\mathbf 0}  
< \epsilon .
\]

To check that the constructed paths 
$(\mathbf z_t)_{t \in [t_j, t_{j+1}]}$ and
$(\mathbf z_t)_{t \in [t_{j+1}, t_{j+2}]}$
patch together, we need to compare the 
end-points at $t_{j+1}$. Recall that 
$\mathbf x_{j+2} = N(\mathbf q_{t_{j+1}} \cdot R(\mathbf x_{j+1}), \mathbf 0) + \mathbf x_{j+1} = N(\mathbf F_{t_{j+1}},\mathbf x_{j+1} )$
using Lemma~\ref{toric-to-classical}(a).  
Let $\tilde{\mathbf y}_k(t)$ denote the $k-1$-th Newton iterate 
of $\mathbf x_{j+2}$ 
for the system $\mathbf F_t(\tilde{\mathbf y}) = \mathbf q_t \cdot V(\tilde{\mathbf y})$.
By construction, $\tilde{\mathbf y}_k(t_{j+1})=\mathbf y_{k+1}$. Therefore,
\[
	\lim_{k \rightarrow \infty} \tilde{\mathbf y}_k(t_{j+1})
	=\lim_{k \rightarrow \infty} {\mathbf y}_{k+1}(t_{j+1})
\]
and the endpoints at $t_{j+1}$ are the same.
	Because of Lemma \ref{lem-u}(b), $\secrev{\mu(\mathbf q_t R(\mathbf z_t))}$ is finite and hence the implicit function theorem
guarantees that $\mathbf z_t$ has the same differentiability class than $\mathbf q_t$.

We proceed now to the lower bound
$\mathscr L(t_j, t_{j+1}) \ge \delta_*$. Assume without loss of
generality that $\delta_*<1/2$,
and that  $\delta = \mathscr L(t_j, t_{j+1}) < \delta_*$. 
Lemma~\ref{lem-rt} provides an upper bound for $d_{\mathrm{max}}(t_{j+1}) \mu_j$.
The rightmost inequality of \eqref{mu-bound} with $t=t_{j+1}$ implies that
\[
	u_j(t_{j+1}) \le \frac{1}{2} \frac{\mu_j \secrev{\nu} \|\mathbf z_{t_{j+1}}-\mathbf x_{j+1}\|_{\mathbf 0} }{1-d_{\mathrm{max}}(t_{j+1}) \mu_j}
.
\]
From Lemma~\ref{lem-u}(\changed{c}) and rearranging terms,
\[
(1-d_{\mathrm{max}}(t_{j+1}) \mu_j) u_{***}
\le
\frac{1}{2} \mu_j \secrev{\nu} \| \mathbf z_{t_{j+1}} - \mathbf x_{j+1} \|_{\mathbf 0} 
.
\]
Triangular inequality yields
\begin{eqnarray*}
\frac{1}{2} \mu_j \secrev{\nu} \| \mathbf z_{t_{j+1}} - \mathbf z_{t_j} \|_{\mathbf 0} 
&\ge&
\frac{1}{2} \mu_j \secrev{\nu} \| \mathbf z_{t_{j+1}} - \mathbf x_{j+1} \|_{\mathbf 0} 
-
\frac{1}{2} \mu_j \secrev{\nu} \| \mathbf z_{t_{j}} - \mathbf x_{j+1} \|_{\mathbf 0} 
\\
&\ge&
u_{***} (1-d_{\mathrm{max}}(t_{j+1}) \mu_j) -u_{**} 
\end{eqnarray*}
On the other hand,
\begin{eqnarray*}
\frac{1}{2} 
\mu_j \secrev{\nu} \| \mathbf z_{t_{j+1}} - \mathbf z_{t_j} \|_{\mathbf 0} 
&\le&
\frac{1}{2}
\mu_j \secrev{\nu} 
\int_{t_j}^{t_{j+1}}
	\| \dot {\mathbf z}_{t} \|_{\mathbf 0} 
\dd t
\\
&\le&
\frac{1}{2}
(1+d_{\mathrm{max}}(t_{j+1}) \mu_j)
\int_{t_j}^{t_{j+1}}
\| \dot z_{t} \|_{\mathbf 0} 
	\secrev{\mu(\mathbf q_t \cdot R(\mathbf z_t)) \nu} 
\dd t
\end{eqnarray*}
Thus,
\[
u_{***} (1-d_{\mathrm{max}}(t_{j+1}) \mu_j) -u_{**} 
\le
\frac{1}{2}
\mathscr L(t_j, t_{j+1} )(1+d_{\mathrm{max}}(t_{j+1}) \mu_j)
\]

Assume that $\delta = \mathscr L(t_j, t_{j+1}) < 1/2$. 
Lemma~\ref{lem-rt} above implies that
$(1-d_{\mathrm{max}}(t_{j+1})\mu_j) \ge 1-\frac{\delta}{1-\delta}=\frac{1-2\delta}{1-\delta}$. 
Similarly, $(1+d_{\mathrm{max}}(t_{j+1})\changed{\mu_j}) \le \frac{1}{1-\delta}$. We have
\[
\left(u_{***}\frac{1 - 2 \delta}{1-\delta} -u_{**}\right) 
\le
\frac{1}{2} \frac{\delta}{1-\delta}
.
\]
Rearranging terms,
\[
\delta \ge \frac{u_{***} - u_{**}}{2 u_{***} - u_{**} + \frac{1}{2}}
.
\]
By construction $\alpha_{**}$ satisfies $u_*(\alpha_{**})=u_0$.
The right hand side is a smooth function of $\alpha \in (0,\alpha_{**})$
(See figure~\ref{plot-delta}). Its maximum is attained for $\alpha_*
\sim 0.074,609\cdots$ with value $\delta_*=0.085,180\cdots$ 
\changed{Finally,}
Lemma~\ref{lem-well-defined} implies 
that
$x_{N+1} = $ $N( \mathbf q_{T} \cdot R(\mathbf x_{N}), \mathbf 0) + \mathbf x_N 
$ is an approximate root for $\mathbf q_T$ 
as in Equation \eqref{thA-a0}. 
\end{proof}

\section{Proof of Theorem~\ref{E-M2}}
\label{sec-integral}
\subsection{The coarea formula}
Theorem ~\ref{E-M2} claims a bound for the integral
\begin{equation}\label{Ihat}
\gls{Ihat}	
\defeq \expected{\mathbf q \sim \secrev{\mathcal N}(\secrev{\mathbf f}, \Sigma^2)} 
{ \sum_{\mathbf z \in Z_H(\mathbf q)} \| M(\mathbf q ,\mathbf z)^{-1}\|_{\mathrm F}^2 }
\end{equation}
which is independent of $\secrev{\mathbf f}$. This bound will be derived from the
coarea formula, see for instance \cite{BCSS}*{Theorem 4 p.241}. This statement is also known in other communities as the Rice formula \cite{Azais-Wschebor}
or the Crofton formula. However, it is convenient to restate this 
result in terms of complex fiber bundles (See figure~\ref{fiber-bundle}) 

\begin{figure}
\centerline{
	\resizebox{0.7 \textwidth}{!}{ 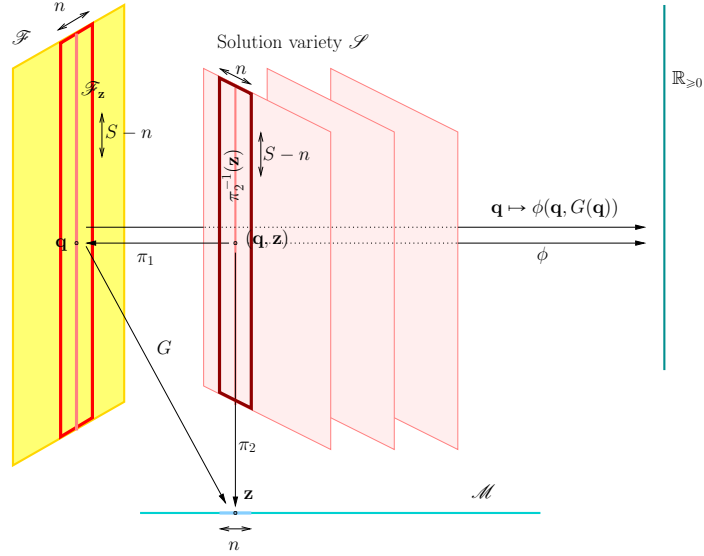} }
\caption{The solution variety $\secrev{\mathscr S}$ is a linear bundle over the complex manifold $\secrev{\mathscr M}$. The fiber is an $S-n$ subspace of $\mathscr F$, where $S=\dim \secrev{\mathscr S}$ and $n= \dim \secrev{\mathscr M}$. By endowing $\secrev{\mathscr S}$ with the pull-back of the metric of $\secrev{\mathscr M}$, Theorem~\ref{coarea} becomes trivial in case the function $\phi$ vanishes outside the domain of the implicit function $G$. The full statement then follows using the trick of the partitions of unity. Notice that the fibration is over the base space $\secrev{\mathscr M}$, so that the discriminant variety (locus of singular values of $\pi_1$) plays no role in this picture.
\label{fiber-bundle}}
\end{figure}

Recall that the solution variety is
$\secrev{\mathscr S} = \{ (\mathbf q,\mathbf z) \in \mathscr F
\times \secrev{\mathscr M}: \mathbf q \cdot \mathbf V(\mathbf z)=\mathbf 0\}
$. Let $\pi_1: \secrev{\mathscr S} \rightarrow \mathscr F$ and
$\pi_2: \secrev{\mathscr S} \rightarrow \secrev{\mathscr M}$ be the canonical projections.
Then $(\secrev{\mathscr S}, \secrev{\mathscr M}, \changed{\pi_2, F})$ is a complex smooth fiber bundle, where \changed{$F = \pi_1 \circ \pi_2^{-1}(\mathbf 0) \subseteq \mathscr F$ is the fiber}.
The solution variety $\secrev{\mathscr S}$ will be endowed here with the pull-back metric $d\secrev{\mathscr S} = \pi_1^* d\mathscr F$.
\secrev{Let $\chi_{H}(\mathbf z)$ be the indicator function for the set 
$\left \{\mathbf z \in \mathscr M: \|\Re(\mathbf z)\|_{\infty} \le H \right \}$.} 
\secrev{Assume as in \eqref{Ihat} that $\mathbf q \sim \secrev{\mathcal N}(\secrev{\mathbf f}, \Sigma^2)$.
Then $\mathbf g \defeq \mathbf q - \secrev{\mathbf f}$ has zero average and variance $\Sigma^2$.
Recall that $\mathbf g_i$ is represented as a row vector, so $\mathbf h_i \defeq \mathbf g_i \Sigma_i^{-1}
\simeq \secrev{\mathcal N}(\mathbf 0, I)$. The probability density of $\mathbf h_i$ is
is $\pi^{-S_i}
\exp(-\|\mathbf h_i\|^2)$, with $S_i=
\#A_i = \dim_{\mathbb C} \mathcal F_{A_i}$
.
By change of variables, the probability density of $\mathbf g_i$ is
\[
\mathbf g_i \sim \frac{ \exp\left(\ghostrule{2.5ex}-\|\mathbf g_i \Sigma_i^{-1}\|^2\right) }
{\pi^{S_i} |\det (\Sigma_i)|^2}
\ \dd \mathscr F_{A_i}(\mathbf g) .
\]
For short, we write $\mathbf g \Sigma^{-1}$ for the row vector $(\mathbf g_1 \Sigma_1^{-1}, \dots,
\mathbf g_n \Sigma_n^{-1})$. Equation \eqref{Ihat} becomes}
\[
I_{\secrev{\mathbf f}, \Sigma^2} = 
\int_{(\secrev{\mathbf f}+\mathbf g,\mathbf z) \in \secrev{\mathscr S}}
\| M(\secrev{\mathbf f}+\mathbf g,\mathbf z) ^{-1}\|^2_{\mathrm F} \ 
\frac{ \exp\left(\ghostrule{2.5ex}-\|\mathbf g \Sigma^{-1}\|^2\right) }
{\pi^{S} \prod_i |\det \Sigma_i|^2}
\secrev{\chi_H(\mathbf z)}
\dd \secrev{\mathscr S}
\]
with $S = \dim_{\mathbb C} \mathscr F = \sum \# A_i$.
As explained by \ocite{Malajovich-nonlinear}*{Theorem 4.9}, the coarea 
formula may be restated in terms of fiber bundles:
\begin{theorem}\label{coarea}
	Let $(\secrev{\mathscr S}, \secrev{\mathscr M}, \changed{\pi_2}, F)$ be a complex smooth fiber bundle. Assume that $\secrev{\mathscr M}$ is finite dimensional. Let $\phi: \secrev{\mathscr S} \rightarrow \mathbb R_{\ge 0}$ be measurable. Then whenever the last integral exists,
\[
	\int_{\secrev{\mathscr S}} \phi(\mathbf p) \dd \secrev{\mathscr S}(\mathbf p) =
	\int_{\secrev{\mathscr M}} \dd \secrev{\mathscr M}(\mathbf z)
	\int_{\pi_2^{-1}(\mathbf z)} \det (D\pi_2(\mathbf p) D\pi_2(\mathbf p)^*)^{-1} \phi(\mathbf p) \,\dd \pi_2^{-1}(\mathbf z)(\mathbf p)
.
\]
\end{theorem}
\noindent
Because the metric in $\secrev{\mathscr S}$ is the pull-back of the metric in $\mathscr F$,
this is the same as
\[
	\int_{\secrev{\mathscr S}} \phi(\mathbf p) \dd \secrev{\mathscr S}(\mathbf p) =
	\int_{\secrev{\mathscr M}} \dd \secrev{\mathscr M}(\mathbf z)
	\int_{\pi_1 \circ \pi_2^{-1}(\mathbf z)} \det (DG(\mathbf q) DG(\mathbf q)^*)^{-1} \phi(\mathbf q,\mathbf z) \dd \mathscr F(\mathbf q)
\]
with $G: \mathbf f \in U \subseteq \mathscr F \rightarrow \secrev{\mathscr M}$ the local implicit
function, that is the local branch of $\pi_2 \circ \pi_1^{-1}$.
While the integral on the left is independent of the volume form $\dd \secrev{\mathscr M}(\mathbf z)$ on $\secrev{\mathscr M}$, the value
of the determinants of $D\pi(\mathbf p) D\pi(\mathbf p)^*$ and of $DG(\mathbf q) DG(\mathbf q)^*$ depend on this volume form. 
In order to simplify computations, we choose to endow $\secrev{\mathscr M}$ with the canonical Hermitian metric of $\mathbb C^n$. We can now compute the
normal Jacobian.

\begin{lemma}\label{JacDet} Under the assumptions above,
\[
	\det (DG(\mathbf q) DG(\mathbf q)^*) = |\det M(\mathbf q,\mathbf z)|^{\thirdrev{-}2}
.
\]
\end{lemma}
\begin{proof}
Assume that $M(\mathbf q,\mathbf z)$ is invertible, otherwise both sides are zero.
Then we can parameterize a neighborhood of $(\mathbf q,\mathbf z) \in \secrev{\mathscr S}$
by a map $\mathbf h \mapsto (\mathbf h,G(\mathbf h))$ where the implicit function $G(\mathbf h)$ 
is defined in a neighborhood
of $\mathbf q$, satisfies $G (\mathbf q) = \mathbf z$ and $\mathbf h \cdot \mathbf V(G(\mathbf h)) \equiv \mathbf 0$.
The derivative of the implicit function can be obtained by differentiation
at $(\mathbf q, \mathbf z)$, viz.
\[
	\mathbf q \cdot D\mathbf V(\mathbf z) \dot {\mathbf z} + \dot {\mathbf q} \cdot \mathbf V(\mathbf z) = 0
.
\]
	\changed{Let $K_i(\mathbf x, \mathbf y) = V_{A_i}(y)^*\secrev{( V_{A_i}(x))}$ be
	the reproducing Kernel of $\mathscr F_{A_i}$, see
\cite{toric1}*{Section 3.1}. In terms of the
	reproducing Kernel,}
\[
	\dot {\mathbf q} \cdot \mathbf V(\mathbf z) =
\begin{pmatrix}
	\left\langle \dot {\mathbf q}_1, K_1(\cdot, \mathbf z) \right \rangle_{\mathscr F_{A_1}} \\
\vdots\\
	\left\langle \dot {\mathbf q}_n, K_n(\cdot, \mathbf z) \right \rangle_{\mathscr F_{A_n}} 
\end{pmatrix}
=
\begin{pmatrix}
K_1(\cdot, \mathbf z)^* & & \\
& \ddots & \\
& & K_n(\cdot, \mathbf z)^*
\end{pmatrix}
	\dot {\mathbf q} .
\]
It follows that
\[
	DG(\changed{\mathbf q}) = - M(\mathbf q,\mathbf z)^{-1} 
\begin{pmatrix}
\frac{1}{\|V_1(\mathbf z)\|} K_1( \cdot, \mathbf z)^* & & \\
&\ddots&\\
& & \frac{1}{\|V_n(\mathbf z)\|} K_n(\cdot, \mathbf z)^*
\end{pmatrix}
\]
	Because of the reproducing kernel property, $\|V_{A_i}(\mathbf z)\|^2=K_i(\mathbf z,\mathbf z)$ and
hence
	\begin{eqnarray*}
	\det(
		\changed{DG(\mathbf q) DG(\mathbf q)^*}) 
		&=&
		\det (M(\mathbf q, \mathbf z)^{-\changed{1}} M(\mathbf q,\mathbf z)^{-\changed{*}}) \\
		&=& |\det M(\mathbf q,\mathbf z)|^{-2} .
	\end{eqnarray*}

\end{proof}

Let $\secrev{\mathscr M_H} = \{ \mathbf z \in \secrev{\mathscr M}: \| \Re(\mathbf z) \|_{\infty} \le H\}$ and let $\chi_{H}(\mathbf z)$ be its indicator function.
We can now compute $I_{\secrev{\mathbf f}, \Sigma^2}$ by replacing, in the statement of Theorem~\ref{coarea}, 
\[
\phi(\mathbf q, \mathbf z) = 
\| M(\secrev{\mathbf f}+\mathbf g,\mathbf z) ^{-1}\|^2_{\mathrm F} \ 
\frac{ \exp\left(\ghostrule{2.5ex}-\|\mathbf g \Sigma^{-1}\|^2\right) }
{\pi^{S} \prod_i |\det \Sigma_i|^2}
\chi_{H}(\mathbf z)
.
\]
We introduce the notation $\mathscr F_{\mathbf z} = (\pi_1 \circ \pi_2^{-1}) (\mathbf z)$. \changed{Theorem~\ref{coarea}} \secrev{and Lemma~\ref{JacDet} yield}
\[
\begin{split}
		I_{\secrev{\mathbf f}, \Sigma^2}
=
\int_{\secrev{\mathscr M_H}} \dd \secrev{\mathscr M}(\mathbf z)
\int_{\mathscr F_{\mathbf z}}
|\det M(\secrev{\mathbf f}+\mathbf g,\mathbf z)|^{2} 
\| M(\secrev{\mathbf f}+\mathbf g,\mathbf z) ^{-1}\|^2_{\mathrm F} \ 
\\
\times \frac{ \exp\left(\ghostrule{2.5ex}-\|\mathbf g \Sigma^{-1}\|^2\right) }
	{\pi^{S} \prod_i |\det \Sigma_i|^2} \dd \mathscr F_{\secrev{\mathbf z}}(\mathbf g)
.
\end{split}
\]
\subsection{The integral as a root count}
\ocite{ABBCS} introduced
the following technique to integrate $|\det (M(\secrev{\mathbf f}+\mathbf g,\mathbf z))|^2 \|M(\secrev{\mathbf f}+\mathbf g,\mathbf z)^{-1}\|_{\mathrm F}^2$.
Let $S(M,k,\mathbf w)$ be the determinant of the matrix obtained
by replacing the $k$-th row of $M$ with $\mathbf w$. Cramer's
rule yields
\[
|\det M|^2
\|M^{-1}\|_{\mathrm F}^2 
=
\sum_{kl} 
|S(M,k,\mathrm e_l)|^2 .
\]
Therefore,
\[
I_{\secrev{\mathbf f},\Sigma^2}  
=
\int_{\secrev{\mathscr M_H}}
\dd\secrev{\mathscr M}(\mathbf z)
\int_{\mathscr F_{\mathbf z}} \
\sum_{k,l}
|S(M(\secrev{\mathbf f}+\mathbf g,\mathbf z),k,\mathrm e_l)|^2
	\frac{ \exp\left(\ghostrule{2.5ex}-\|\mathbf g\Sigma^{-1}\|\right) }
{\pi^S \prod_i |\det \Sigma_i|^2}
	\dd \mathscr F(\mathbf g)
\]
\ocite{ABBCS} computed that integral in the dense case.
We will proceed in a different manner.
\secrev{Rewrite
\begin{equation}\label{Ihatreord}
I_{\secrev{\mathbf f},\Sigma^2}  
=
\sum_{k,l}
\int_{\secrev{\mathscr M_H}}
\dd\secrev{\mathscr M}(\mathbf z)
\int_{\mathscr F_{\mathbf z}} \
|S(M(\secrev{\mathbf f}+\mathbf g,\mathbf z),k,\mathrm e_l)|^2
	\frac{ \exp\left(\ghostrule{2.5ex}-\|\mathbf g\Sigma^{-1}\|\right) }
{\pi^S \prod_i |\det \Sigma_i|^2}
	\dd \mathscr F(\mathbf g)
\end{equation}
and notice that for each \secrev{fixed} $k$, the 
subdeterminant $S(M(\secrev{\mathbf f}+\mathbf g,\mathbf z),k,\mathrm e_l)$
is independent of $\mathbf g_k$. For $\mathbf z$ and $k$ fixed, $\mathbf g_k$ is a Gaussian restricted to
the complex hyperplane $(\mathcal F_{A_k})_{\mathbf z}
= K_{A_k}(\cdot,\mathbf z)^{\perp} \subseteq \mathscr F_{A_k}$, the $k$-th component space of $\mathscr F_{\mathbf z}$.
Its probability density is
\[
\mathbf g_k \sim \frac{ \exp\left(\ghostrule{2.5ex}-\|\mathbf g_k \Sigma_k^{-1}\|^2\right) }
{\pi^{S_k-1} \left|\det\left( ({\Sigma_k})_{|(\mathscr F_{A_k})_{\mathbf z}}\right)\right|^2}\
\dd (\mathscr F_{A_k})_{\mathbf z}(\mathbf g) .
\]
Thus,
\[
	\int_{\mathbf g_k \in (\mathcal F_{A_k})_{\mathbf z}} 
	\frac{ \exp\left(\ghostrule{2.5ex}-\|\mathbf g_k\Sigma_k^{-1}\|\right) }
{\pi^{S_k} |\det \Sigma_k|^2}
\dd (\mathscr F_{A_k})_{\mathbf z}(\mathbf g) 
=
	\frac{\left|\det\left( ({\Sigma_k})_{|(\mathscr F_{A_k})_{\mathbf z}}\right)\right|^2}{\pi |\det (\Sigma_k)|^2  }
.
\]
Integrating out $\mathbf g_k$ as above in each summand of \eqref{Ihatreord},
we obtain}
\begin{multline*}
I_{\secrev{\mathbf f}, \Sigma^2} = 
\sum_{k,l}
\int_{\secrev{\mathscr M_H}}
	\frac{ |\det ({\Sigma_k})_{|(\mathscr F_{A_k})_{\mathbf z}}|^2 }
	{\pi  |\det (\Sigma_k)|^2  }
\dd\secrev{\mathscr M}(\mathbf z)
	\int_{
\bigoplus_{i \ne k}
	(\mathscr F_{A_i})_{\mathbf z}}
|S(M(\secrev{\mathbf f}+\mathbf g,\mathbf z),k,\mathrm e_l)|^2
	\\ 
	\frac{ |\det ({\Sigma_k})_{|(\mathscr F_{A_k})_{\mathbf z}}|^2 }
	{\pi^{\sum_{i \ne k}S_i} \prod_{i \ne k} |\det \Sigma_i|^2}
	\bigwedge_{i \ne k} \dd (\mathscr F_{A_i})_{\mathbf z}(\mathbf g) .
\end{multline*}
At this point we need the following \secrev{lemma}:
\begin{lemma}\label{det-interlacing}
Let $H$ be a Hermitian positive matrix, with eigenvalues $\sigma_1 \ge \dots \ge \sigma_n$. Let $\mathbf w = \mathbf u + i \mathbf v$ be a non-zero vector in $\mathbb C^N$, with $\mathbf u, \mathbf v \in \mathbb R^n$.
Then,
\[
	\frac{\det H}{\sigma_1} 
	\le \det H_{|\mathbf w^{\perp}} 
	\le \frac{\det H}{\sigma_N} 
\]
\end{lemma}

It follows from the Courant-Fischer minimax theorem \cite{Demmel}*{Theorem 5.2}. 
Only the first 
part of the statement is quoted below.
Recall that the Grassmannian
$\mathrm{Gr}(k,V)$ is the set of $k$-dimensional linear subspaces in a real space $V$: 

\begin{theorem}[\changed{Courant-Fischer}]
Let $\alpha_1 \ge \dots \ge \alpha_m$ be the eigenvalues of a symmetric matrix $A$,
and let $\rho(\mathbf r, A) = \frac{\mathbf r^T A \mathbf r}{\mathbf \|\mathbf r\|^2}$ be
the Rayleigh quotient.  Then,
\[
	\max_{R \in \mathrm{Gr}(j,\mathbb R^{m})} \min_{0 \ne \mathbf r \in R} \rho(\mathbf r, A)
	= \alpha_j
	=
	\min_{S \in \mathrm{Gr}(m-j+1,\mathbb R^{m})} \min_{0 \ne \mathbf s \in S} \rho(\mathbf s, A)
.
\]
\end{theorem}

\begin{proof}[Proof of Lemma \ref{det-interlacing}]
We assume without loss of generality that $W$ is diagonal and real.
A vector $\mathbf z = \mathbf x + i \mathbf y \in \mathbb C^N$ is complex-orthogonal to 
$\mathbf w = \mathbf u + i \mathbf v$ with $\mathbf x, \mathbf y, \mathbf u$ and $\mathbf v$ real,
if and only if
	$\begin{pmatrix} \mathbf x \\ \mathbf y \end{pmatrix}$ is real-\changed{orthogonal} to  
		$\begin{pmatrix} \mathbf u \\ \mathbf v \end{pmatrix}$ and \changed{to}  
$\begin{pmatrix} -\mathbf v \\ \mathbf u \end{pmatrix}$.  
We consider also the diagonal matrix 
\[
A =
\begin{pmatrix} 
\sigma_1\\
	&\ddots\\
	&&\sigma_N\\
	&&&\sigma_1\\
	&&&&\ddots\\
	&&&&&\sigma_N\\
\end{pmatrix}.  
\]
It satisfies $\det(A) = \det(W)^2$. Let $T$ be the space orthogonal
to $\begin{pmatrix} \mathbf u \\ \mathbf v \end{pmatrix}$ and  
$\begin{pmatrix} -\mathbf v \\ \mathbf u \end{pmatrix}$. 
We denote by $A_{|T}$ the restriction of $A$ to $T$, as a bilinear form. The restriction
	$A_{|T}$ is still symmetric and positive.
Let $\alpha_1=\sigma_1 \ge \alpha_2=\sigma_1 \ge \alpha_3 = \sigma_2 \ge \dots
	\ge \alpha_{2N}=\sigma_n$ be the eigenvalues of $A$, and
	$\lambda_1 \ge \dots \ge \lambda_{2N-2}$ be the eigenvalues of $\changed{A_{|T}}$. Courant-Fischer
theorem yields:
\[
\lambda_j = 
	\max_{R \in \mathrm{Gr}(j,\changed{T})} \min_{0 \ne \mathbf r \in R} \rho(\mathbf r, A_{|T})
	\le
	\max_{R \in \mathrm{Gr}(j,\mathbb R^{2N})} \min_{0 \ne \mathbf r \in R} \rho(\mathbf r, A)
	= \alpha_j
\]
and
\[
\alpha_j
=
	\min_{S \in \mathrm{Gr}(2N-j+1,\mathbb R^{2N})} \min_{0 \ne \mathbf s \in S} \rho(\mathbf s, A)
\le
	\min_{S \in \mathrm{Gr}(\changed{(2N-2)}-(j-2)+1,T)} \min_{0 \ne \mathbf s \in S} \rho(\mathbf s, A)
=
	\lambda_{j-2}
\]
	It follows that $\lambda_{j} \le \alpha_j \le \lambda_{j-2}$. Since this holds for all $j$,
	\[
\frac{\det (A)}{\sigma_1^2} =
\prod_3^{2N} \alpha_j \le
		\det A_T = \prod_1^{2N-2} \lambda_j \le \prod_1^{2N-2} \alpha_j = \frac{\det (A)}{\sigma_N^2}
	\]
and hence
	\[
\frac{\det (H)}{\sigma_1} 
\le
\det (H_{\mathbf w^{\perp}})
\le
\frac{\det (H)}{\sigma_N} 
.
\]
\end{proof}

From Lemma~\ref{det-interlacing},
\[
	\frac{|\det ({\Sigma_k})|^2}{ \max_{\mathbf a} \sigma_{k \mathbf a}^2  }
\le
|\det ({\Sigma_k})_{|(\mathscr F_{A_k})_{\mathbf z}}|^2 
\le
\frac
{|\det ({\Sigma_k})|^2}
{\min_{\mathbf a} \sigma_{k \mathbf a}^2}
.
\]
We have proved:
\begin{proposition}\label{prop-variance}
\[
I_{\secrev{\mathbf f}, \Sigma^2} \le 
\pi^{-1} \sum_{k,l=1}^n \frac{1}{\min_{\mathbf a} \sigma_{k \mathbf a}^2}
I_{kl}
\]
with
\begin{eqnarray*}
I_{kl} &\defeq&
\int_{\secrev{\mathscr M_H}}
\dd\secrev{\mathscr M}(\mathbf z)
\int_{
\bigoplus_{i \ne k}
	(\mathscr F_{A_i})_{\mathbf z}}
|\det S(M(\secrev{\mathbf f}+\mathbf g,\mathbf z),k,\mathrm e_l)|^2
\\
& & \hspace{5em}
	\frac{ \exp\left(\ghostrule{2.5ex}-\sum_{i \ne k}\|\mathbf g_i\Sigma_i^{-1}\|^2\right) }
	{\pi^{\sum_{i \ne k}S_i} \prod_{i \ne k} |\det \Sigma_i|^2}
	\bigwedge_{i \ne k} \dd (\mathscr F_{A_i})_{\mathbf z}(\mathbf g) .
\end{eqnarray*}
\end{proposition}

The integrals $I_{kl}$ of Proposition~\ref{prop-variance} do not need to be
computed. Instead, one can interpret them as the expected number
of roots of certain random mixed systems of polynomials and exponential sums. Such objects were
studied by \ocite{Malajovich-fewnomials} in greater generality. But we will obtain the following
statement by judicious use of the properties of the mixed volume:
\begin{proposition}\label{prop-subintegral}
	Let $p_l: \mathbb R^n \rightarrow \mathrm e_l^{\perp} \simeq \mathbb R^{n-1}$ be the orthogonal projection. \secrev{Let $\mathcal A_i \defeq \conv{A_i}$.} Then,
\[
I_{kl}= \frac{4\pi H}{\det \Lambda} (n-1)! V(p_l(\mathcal A_1), \dots,
	p_l(\mathcal A_{\changed{k}-1}), p_l(\mathcal A_{\changed{k}+1}), \dots, p_l(\mathcal A_n)) 
\]
where $V$ is the $n-1$-dimensional mixed volume operator.
\end{proposition}

\begin{proof}[Proof of Proposition~\ref{prop-subintegral}]
	We only need to prove Proposition~\ref{prop-subintegral}
	for $k=l=n$. If $\mathbf q \in \mathscr F_{\mathbf z}$,
	then $\mathbf z \in Z(\mathbf q)$. Recall \changed{from \eqref{M-zero}} that in that case,
\[
M(\mathbf q,\mathbf z) = 
\begin{pmatrix}
\frac{1}{\|V_{A_1}(\mathbf z)\|} \mathbf q_1 \cdot DV_{A_1}(\mathbf z) \\
\vdots \\
\frac{1}{\|V_{A_n}(\mathbf z)\|} \mathbf q_n \cdot DV_{A_n}(\mathbf z) \\
\end{pmatrix}
\]
and
\[
DV_{A_i}(\mathbf z) = \diag{V_{A_i}(\mathbf z)} A_i
\]
where on the right, $A_i$ stands for the matrix with rows $\mathbf a \in A_i$.
The rows of matrix $S(M(\mathbf q, \mathbf z), n, \mathrm e_n)$ are:
\[
S(M(\mathbf q, \mathbf z), n, \mathrm e_n) =
	\changed{\begin{pmatrix}
\hspace{\stretch{1}}
	\frac{1}{\|V_{A_1}(\mathbf z)\|}\hspace{2.5em} 
\left(\ghostrule{2.5ex}  \dots,q_{1\mathbf a} V_{1\mathbf a}(\mathbf z),\dots \right)_{\mathbf a \in A_1}\hspace{1em} 
	\hspace{1em} A_1 \hspace{1em}\\
	\vdots \\
	\frac{1}{\|V_{A_{n-1}}(\mathbf z)\|} 
	\left(\ghostrule{2.5ex}\dots,q_{n-1,\mathbf a} V_{n-1,\mathbf a}(\mathbf z),\dots\right)_{\mathbf a \in A_{n-1}}  
	A_{n-1} \\
	\mathrm e_n^T
	\end{pmatrix}}
\]
We claim that $|\det
S(M(\mathbf q, \mathbf z), n, \mathrm e_n)|^{-2}$
is the normal Jacobian for a certain system of fewnomial sums.
	We previously defined $\mathscr F_{A_1}$, \dots, $\mathscr F_{A_{n-1}}$
as spaces of fewnomials over the complex manifold $\secrev{\mathscr M} =
\mathbb C^n \mod 2 \pi \sqrt{-1} \, \Lambda^*$, where $\Lambda$ is
the lattice generated by the $A_i-A_i, 1 \le i \le n$ and $\Lambda^*$ is
its dual. 

Let $\mathscr N \subseteq \mathbb C$ be the strip $-\pi < \Im(z) \le \pi$.
Each point of $\mathbb C^n \mod 2 \pi \sqrt{-1}\, \mathbb Z^n$ is represented
by a unique point in $\mathscr N^n$. Moreover, the natural projection
\[
\mathscr N^n \rightarrow \secrev{\mathscr M}
\]
is a $\det \Lambda$-to-1 local isometry. 
	In the same spirit, we define $\mathscr N_H \subseteq \mathscr N$ as the domain $-H \le \Re(z) \le H, -\pi < \Im(z) \le \pi$. Now each point of $\mathbb C^n \mod 2 \pi \sqrt{-1}\, \mathbb Z^n$ with $\| \Re(\mathbf z) \|_{\infty} \le H$
is represented by a unique point in $\mathscr N_H^n$. The natural projection
\[
\mathscr N_H^n \rightarrow \secrev{\mathscr M_H}
\]
is again a $\det \Lambda$-to-1 local isometry.
	
We extend all our spaces $\mathscr F_{A_i}$ to spaces of functions on $\mathscr N^n$, and write
\begin{eqnarray*}
	I_{nn} &=& \frac{1}{\det \Lambda}
	\int_{\mathscr N_H^n}
\dd \mathbb C^n(\mathbf z)
\int_{
\bigoplus_{i < n}
	(\mathscr F_{A_i})_{\mathbf z}}
|\det S(M(\mathbf q,\mathbf z),\thirdrev{n},\mathrm e_n)|^2
\\
& & \hspace{5em}
	\frac{ \exp\left(\ghostrule{2.5ex}-\sum_{i \ne n}\|(\mathbf q_i-\secrev {\mathbf f}_i)\Sigma_i^{-1}\|^2\right) }
	{\pi^{\sum_{i \ne n}S_i} \prod_{i \ne n} |\det \Sigma_i|^2}
	\bigwedge_{i \ne n} \dd (\mathscr F_{A_i})_{\mathbf z}(\mathbf q) .
\end{eqnarray*}

We will recognize in the formula above the average 
number of zeros 
\secrev{in $\mathscr N_H^n$}
of a certain system of fewnomial equations. This will be
done through direct application of Theorem~\ref{coarea} (coarea formula).
We need first a fiber bundle.

Define $\mathscr H = \mathscr F_{A_1} \times \cdots \times \mathscr F_{A_{n-1}} \times \mathscr N_H$, and endow this space with the product metric (the space 
$\secrev{ \mathscr N_H \subseteq \mathbb C}$ is endowed with the canonical metric). The solution variety $\secrev{\mathscr  S_H} \subseteq
\mathscr H \times \mathscr N_H^n$ will be
\[
\begin{split}
	\secrev{\mathscr  S_H} =
\left\{
(\mathbf q_1, \dots, \mathbf q_{n-1}, w; \mathbf z) \in \mathscr H \times \mathscr N_H^n:
	\mathbf q_1 \cdot V_{A_1}(\mathbf z) =
\dots \right. 
	\hspace{3em}	
	\\ \left. 
	=
\mathbf q_{n-1} \cdot V_{A_{n-1}}(\mathbf z) = z_n -w = 0
\right\}
\end{split}
\]
with canonical projections $\pi_1: \secrev{\mathscr  S_H} \rightarrow \mathscr H$
and $\pi_2: \secrev{\mathscr  S_H} \rightarrow \mathscr N_H^n$.
The inner product in $\secrev{\mathscr  S_H}$ is the pull-back of the inner product
of $\mathscr H$ by $\pi_1$. Then the bundle $(\secrev{\mathscr  S_H}, \mathscr N_H^n,
\pi, F')$ is a fiber bundle with fiber $F' = (\mathscr F_{A_1})_{\mathbf 0} \times
\cdots \times (\mathscr F_{A_{n-1}})_{\mathbf 0} \times \thirdrev{\{ \mathbf 0 \}}$.

In order to compute the normal Jacobian, we differentiate the implicit function
$\tilde G$ for
\[
	\Phi(\mathbf q_1, \dots, \mathbf q_{n-1},w; \mathbf z) = 
\begin{pmatrix}
\mathbf q_1 
\cdot 
	V_{A_1}(\mathbf z) 
\\
\vdots\\
\mathbf q_{n-1} 
\cdot 
	V_{A_{n-1}}(\mathbf z) 
\\
z_n -w
\end{pmatrix} = \mathbf 0
.
\]
Let $\tilde {\mathbf q} = (\mathbf q_1, \dots, \mathbf q_{n-1}, w)$. We obtain:
\begin{eqnarray*}
	D\tilde G(\tilde {\mathbf q}; \mathbf z)&=&
	-D_{\mathbf z} \Phi(\tilde {\mathbf q}; \mathbf z) ^{-1} D_{\tilde {\mathbf q}}\Phi(\tilde {\mathbf q}; \mathbf z)
\\
	&=&		-S(M(\mathbf q, \mathbf z), \mathrm e_n, n) 	
\\
	&&\hspace{4ex}\times \begin{pmatrix}
	\frac{1}{\|V_{A_1}(\mathbf z)\|} K_{A_1}( \cdot, \mathbf z)^* & & & \\
&\ddots& &\\
	& & \frac{1}{\|V_{A_{n-1}}(\mathbf z)\|} K_{A_{n-1}}(\cdot, \mathbf z)^* & \\
&&&-1
\end{pmatrix}
\end{eqnarray*}
\changed{The reproducing kernel property yields}  $\| V_{A_i}(\mathbf z) \|^2 = K_{A_i}(\mathbf z, \mathbf z)$. The Normal Jacobian is therefore
\[
	NJ^{2}=|\det D\tilde G(\tilde {\mathbf q};\mathbf z) D\tilde G(\tilde {\mathbf q};\mathbf z)^*|
= 
|S(M(\mathbf q,\mathbf z),n, \mathrm e_n)|^{-2} 
.
\]
The coarea formula (Theorem~\ref{coarea}) yields
\begin{eqnarray*}
	I_{nn} &=& \frac{1}{\det \Lambda}
	\int_{\mathscr N_H^n}
\dd \mathbb C^n(\mathbf z)
\int_{ \mathscr H_{\mathbf z}}
	NJ^{-2}
	\frac{ \exp\left(\ghostrule{2.5ex}-\sum_{i \ne n}\|(\mathbf q_i-\secrev {\mathbf f}_i)\Sigma_i^{-1}\|^2\right) }
	{\pi^{\sum_{i \ne n}S_i} \prod_{i \ne n} |\det \Sigma_i|^2}
\\
	& & \hspace{20em}\bigwedge_{i \ne n} \dd (\mathscr F_i)_{\mathbf z}(\mathbf q) 
\\
&=&
\frac{1}{\det \Lambda}
\int_{\secrev{\mathscr  S_H}}
	\frac{ \exp\left(\ghostrule{2.5ex}-\sum_{i \ne n}\|(\mathbf q_i-\secrev{\mathbf f}_i)\Sigma_i^{-1}\|^2\right) }
	{\pi^{\sum_{i \ne n}S_i} \prod_{i \ne n} |\det \Sigma_i|^2}
	\pi_1^* \dd \secrev{\mathscr  S_H}(\tilde{\mathbf q}, \mathbf z)
\\
	&=&
\frac{1}{\det \Lambda}
\int_{\mathscr H}
	\# \pi_1^{-1}(\tilde{ \mathbf q}) 
	\frac{ \exp\left(\ghostrule{2.5ex}-\sum_{i \ne n}\|(\mathbf q_i-\secrev{\mathbf f}_i)\Sigma_i^{-1}\|^2\right) }
	{\pi^{\sum_{i \ne n}S_i} \prod_{i \ne n} |\det \Sigma_i|^2}
	\pi_1^* \dd \mathscr H(\tilde{\mathbf q})
\end{eqnarray*}
The integral is the average number of roots 
\secrev{in $\mathscr N_H^n$}
of the fewnomial system
$\Phi(\tilde {\mathbf q}, \mathbf z) = \mathbf 0$. The value of the variable $z_n$
at a solution is precisely $w$. If we eliminate this variable
from the other equations, we obtain 
a system of exponential sums with support  
$p_n(A_i)$, $1 \le i \le n-1$. Theorem~\ref{BKK} implies that
those systems have generically (and at most) 
$(n-1)! V( \thirdrev{p_n}(\mathcal A_1), \dots, \thirdrev{p_{n}}(\mathcal A_{n-1}))$
isolated roots in $\mathscr N^{\thirdrev{n-1}}$, whence at most
$(n-1)! V( \thirdrev{p_{n}}(\mathcal A_1), \dots, \thirdrev{p_{n}}(\mathcal A_{n-1}))$
isolated roots in $\mathscr N_H^{\thirdrev{n-1}}$. 
Thus, we integrate for $-H \le \Re(w) \le H$ and $-\pi < \Im(w) \le \pi$ to obtain:
\[
I_{nn}
\le
\frac{4 \pi H}{\det \Lambda}
\ (n-1)!\ V( \thirdrev{p_n}(\mathcal A_1), \dots, \thirdrev{p_{n}}(\mathcal A_{n-1}))
.
\]
\end{proof}

\begin{proof}[Proof of Theorem~\ref{E-M2}]
Combining propositions \ref{prop-variance} and \ref{prop-subintegral},
\[
I_{\secrev{\mathbf f},\Sigma^2} 
\le
\frac{4H}{\det(\Lambda)}
	\sum_{k,l=1}^n
	\frac{1}{\min_{\mathbf a} \sigma_{k \mathbf a}^2}
(n-1)! V(p_l(\mathcal A_1, \dots,
\widehat{p_l(\mathcal A_k}), \dots, p_l(\mathcal A_n))).
\]
	\thirdrev{For any convex polytope $\mathcal B$, 
	\[
		\frac{\partial}{\partial t}|_{t=0} \vol_n \left(\mathcal B + t 
	[\mathbf 0 \mathrm e_l]\right) = \vol_n(p_l(\mathcal B)+
	[\mathbf 0 \mathrm e_l])= \vol_{n-1}(p_l(\mathcal B)).
\]
	From definition~\ref{def-mixed} and elementary calculus, 
	it follows that}
\[
	\thirdrev{(n-1)!}
	V(p_l(\mathcal A_1), \dots,
\widehat{p_l(\mathcal A_k}), \dots, p_l(\mathcal A_n))
=
\thirdrev{n!}
	V(\mathcal A_1, \dots, [\mathbf 0 \mathrm e_l], \dots, \mathcal A_n))
.
\] 
	\changed{\thirdrev{The linearity of the mixed volume with respect
	to Minkowski linear combinations and then the monotonicity property}
	(remark \ref{properties-mixed-volume}) imply that} 
	\begin{eqnarray*}
	\thirdrev{(n-1)!}
\sum_l
		V(p_l(\mathcal A_1), \dots,
\widehat{p_l(\mathcal A_k}), \dots, p_l(\mathcal A_n)))
		&=&
\thirdrev{n!}
V(\mathcal A_1, \dots, [0,1]^n, \dots, \mathcal A_n))
\\
&\le&
\thirdrev{n!}
\frac{\sqrt{n}}{2} V(\mathcal A_1, \dots, B^n, \dots, \mathcal A_n))
\end{eqnarray*}
and
	hence \eqref{eq-mixed-area} yields
\[
I_{\secrev{\mathbf f},\Sigma^2} 
\le
	\frac{2H\sqrt{n}}{\det(\Lambda)}
	\frac{1}{\min_{\mathbf a} \sigma_{k \mathbf a}^2}
\thirdrev{n!}
V'.
\]
\end{proof}

\section{The expectation of the squared condition number}

\subsection{On conditional and unconditional Gaussians}

The condition number used in this paper and the previous one \cite{toric1}
is invariant under independent scalings of each coordinate polynomial. 
This multi-homogeneous invariance introduced by
\ocites{MRHigh} breaks with the tradition 
in dense polynomial systems 
~\cites{BCSS, BC}.
This richer invariance will be strongly exploited in this section. 

Recall that $Z(\mathbf q)$ denotes the set of isolated roots for a system
of exponential sums $\mathbf q \in \mathscr F$. It would be 
desirable here to bound the \secrev{expectation} of $\sum_{\mathbf z \in Z(\mathbf q)} \secrev{\mu(\mathbf q \cdot R(\mathbf z))}^2$ where $\mathbf q \sim \secrev{\mathcal N}(\mathbf f, \Sigma^2)$. \secrev{The} author was unable to compute \secrev{this expectation}. There is no
reason to believe at this point that this \secrev{expectation} is finite. Instead,
let $H > 0$ be arbitrary. Recall that
$Z_H(\mathbf q) = \{ \mathbf z \in Z(\mathbf q): \| \Re(\mathbf z)\|_{\infty} \le H\}$. We will \changed{use Theorem~\ref{E-M2} to bound}  
the \secrev{expectation} of $\sum_{\mathbf z \in Z_H(\mathbf q)} \secrev{\mu(\mathbf q \cdot R(\mathbf z))}^2$ where $\mathbf q \sim \secrev{\mathcal N}(\mathbf f, \Sigma^2)$. 
More precisely, let
\begin{equation}\label{eq-EfSigma}
	E_{\mathbf f, \Sigma^2} \defeq \expected{\mathbf q \sim \secrev{\mathcal N}(\mathbf f, \Sigma^2)}{\sum_{\mathbf z \in Z_H(\mathbf q)} \secrev{\mu(\mathbf q \cdot R(\mathbf z))}^2}
.
\end{equation}
The conditional \secrev{expectation} below will turn out to be easier to bound.
\secrev{It depends on an auxiliary system $\mathbf h$ to be constructed later.}
\begin{equation}\label{eq-EfSigmaK}
\begin{split}
	E_{\secrev{\mathbf {h}}, \Sigma^2,K} \defeq 
	\condexpected{\mathbf q \sim \secrev{\mathcal N}(\secrev{\mathbf {h}}, \Sigma^2)}
	{ \sum_{\mathbf z \in Z_H(\mathbf q)} \secrev{\mu(\mathbf q \cdot R(\mathbf z))}}
	{\|(\mathbf q_i - \secrev{\mathbf {h_i}})\Sigma^{-1}\|\le K \sqrt{S_i},}
	{ \left. i=1, \dots, n }
\end{split}
\end{equation}
with $S_i=\# A_i = \dim_{\mathbb C} \mathscr F_{A_i}$. 
\changed{This is similar to the truncated Gaussian distributions
used by \ocite{BC-annals}, with the difference that we truncate
the tail of the distribution of each $\mathbf f_i$ instead of the tail of the
distribution of $\mathbf f$.}
When $\mathbf f=\mathbf 0$, conditional and unconditional \secrev{expectations} coincide,
due to scaling invariance:
\[
	E_{\mathbf 0, \Sigma^2} = E_{\mathbf 0, \Sigma^2, K}
\]
for any $K>0$. \secrev{When $\mathbf h = \mathbf f$}, 
\[
E_{f,\Sigma^2,K} \le
\frac{E_{f,\Sigma^2}}
{
\probability{\mathbf q \sim \secrev{\mathcal N}(\mathbf f, \Sigma^2)}{
	\|(\mathbf q_i - \mathbf f_i) \Sigma_i^{-1}\|\le K \sqrt{S_i},\ 
	i=1, \dots, n}
}
.
\]
The reciprocal inequality for a non-centered Gaussian probability distribution is more elusive. 
In the next section we prove:
\begin{proposition}\label{prop-mu2-1} Let $S_i=\# A_i$ and  
	$L > 0$. Suppose that $K\ge1+2L + \sqrt{\frac{\log(n)}{\min (S_i)}+2\log(3/2)}$. 
	\secrev{Then
\[
\sup_{\| \mathbf f_i \Sigma_i^{-1}\| \le L \sqrt{S_i}}
E_{\mathbf f, \Sigma^2}
\le
e
	\sup_{\|\secrev{\mathbf h}_i \Sigma^{-1}\| \le L\sqrt{S_i}}
	E_{\secrev{\mathbf h}, \Sigma^2,K}
\]
	where it is assumed that $\mathbf f_i, \secrev {\mathbf h_i} \in \mathscr F_{A_i}$.}
\end{proposition}

\subsection{The truncated non-centered Gaussian}
\label{non-centered}

\ocite{BC-annals} developed a truncated Gaussian technique to bound the expected value of the squared condition number for the dense case. This technique will be generalized to the toric setting in order to prove
Proposition~\ref{prop-mu2-1}. It should be stressed that we do not have unitary invariance and that the condition number in this paper is multi-homogeneous invariant.  
The following result will be used:

\begin{lemma}\ 
\begin{enumerate}[(a)]
	\label{lem-conditional}
\item
	Let $\varphi: 
	\mathbb R^{S_1} \times \cdots \times 	\mathbb R^{S_n}  \rightarrow \mathbb R \cup \{\infty\}$ be measurable, positive and scaling invariant: for all $\mathbf 0 \ne \boldsymbol \lambda \in \mathbb R^n$,
\[
\varphi(\lambda_1 \mathbf w_1, \dots, \mathbf \lambda_n \mathbf w_n) = 
\varphi(\mathbf w_1, \dots, \mathbf w_n).
\]
		Let $\mathbf u \in \mathbb C^{S_1} \times \cdots \times 	\mathbb C^{S_n}$. Let $L=\max(\|\mathbf u_i\|/\sqrt{S_i})$ and $K \ge 1+2L+\sqrt{\frac{\log(n)}{\min (S_i)}+2\log(3/2)}$. Write $S=\sum S_i$. Then,
\[
		\expected{\mathbf w \sim \secrev{\mathcal N}(\mathbf u, I; \mathbb R^S)}{\varphi(\mathbf w)}
\le
\sqrt{e}
	\sup_{\|\hat{\mathbf u}_i\| \le \sqrt{S_i}L}
G_{\hat{\mathbf u}, K} 
\]
with
\[
G_{\hat{\mathbf u}, K} \defeq
\expected{\mathbf w \sim \secrev{\mathcal N}(\hat{\mathbf u}, I; \mathbb R^S)}
		{\varphi(\mathbf w) \conditional \| \mathbf w - \hat{\mathbf u}\| \le \sqrt{S_i}K}
.
\]
\item
\secrev{
Let $\varphi: 
	\mathbb C^{S_1} \times \cdots \times 	\mathbb C^{S_n}  \rightarrow \mathbb R \cup \{\infty\}$ be measurable, positive and complex scaling invariant.
		Let $\mathbf u \in \mathbb C^{S_1} \times \cdots \times 	\mathbb C^{S_n}$. Let $L$, $K$ and $S$ be as in item (a) above. Let
$\secrev{\mathcal N}(\mathbf u, I; \mathbb C^S)$ denote the {\em complex} Gaussian distribution
with average $\mathbf u$ and covariance matrix $I$. Then,
\[
	\expected{\mathbf w \sim \secrev{\mathcal N}(\mathbf u, I; \mathbb C^S)}{\varphi(\mathbf w)}
\le
e
	\sup_{\|\hat{\mathbf u}_i\| \le \sqrt{S_i}L}
G_{\hat{\mathbf u}, K} 
\]
with
\[
G_{\hat{\mathbf u}, K} \defeq
\expected{\mathbf w \sim \secrev{\mathcal N}(\hat{\mathbf u}, I; \mathbb C^S)}
		{\varphi(\mathbf w) \conditional \| \mathbf w - \hat{\mathbf u}\| \le \sqrt{S_i}K}
.
\]}
\end{enumerate}
\end{lemma}

This \secrev{lemma} will follow from the large deviations estimate:
\begin{lemma} \label{lem-large-dev}Let $s,t > 0$. Then,
\begin{enumerate}[(a)]
\item
\[
	\probability{\mathbf u \sim \secrev{\mathcal N}(\mathbf 0,I; \mathbb R^N)}{\|\mathbf u\|\ge \sqrt{N}+t} \le e^{\frac{-t^2}{2}} ,
\]
\item 
\[
\probability{\mathbf w \sim \secrev{\mathcal N}(\mathbf 0,I; \mathbb C^N)}{\|\mathbf w\|\ge \sqrt{N}+s} \le e^{-s^2} .
\]
\end{enumerate}
\end{lemma}

\begin{proof} Item (a) is borrowed from \ocite{BC}*{Corollary 4.6}.
	For item (b), set $\mathbf u = \sqrt{2}\ \Re(\mathbf w)$ and $\mathbf v = \sqrt{2}\ \Im(\mathbf w)$.
Then, $[\mathbf u, \mathbf v] \sim \secrev{\mathcal N}(\mathbf 0,I; \mathbb R^{2n})$ and 
\[
\probability{\mathbf w \sim \secrev{\mathcal N}(\mathbf 0,I; \mathbb C^N)}{\|\mathbf w\|\ge \sqrt{N}+s} 
=
	\probability{[\mathbf u, \mathbf v] \sim \secrev{\mathcal N}(\mathbf 0,I; \mathbb R^{2N})}{\|[\mathbf u, \mathbf v]\|\ge \sqrt{2N}+s\sqrt{2}} 
\le
e^{-s^2} 
\]
using item (a).
\end{proof}

We will also use the following \secrev{trivial} bound.

\begin{lemma}\label{triv-bound-lemma}
	Let $\mathbf u, \mathbf v \in \mathbb C^N$ (resp. $\mathbb R^N$) with $\|\mathbf u\| \le L < K \le \| \mathbf v\|$.
Then,
\[
	-\|\mathbf v\|^2 < -\frac{\|\mathbf u+\mathbf v\|^2}{\sigma^2}
\]
	with $\sigma = 1+\frac{L}{K}$.
\end{lemma}

\begin{proof}[Proof of Lemma~\ref{lem-conditional}]
The proof for the real and complex case is essentially the
same, so we only write down the proof for the complex case (b).
Let $\mathbf w = \mathbf u + \mathbf v$, 
$\mathbf v \sim \secrev{\mathcal N}(\mathbf 0, I; \mathbb C^S)$, so the vectors
	$\mathbf v_i$ are 
\changed{independently distributed}.
We subdivide the domain $\mathbb C^S$ in cells indexed
by each $J \subseteq [n]=\{1,\dots, n\}$ as follows:
\[
V_J = \{ \mathbf v = (\mathbf v_1, \dots, \mathbf v_n) \in
	\mathbb C^{S_1} \times \cdots \times 	\mathbb C^{S_n}	: 
	\|\mathbf v_i\| > \sqrt{S_i}K \Leftrightarrow i \in J\}.
\]
Under this notation $\mathbb C^S$ is the disjoint union of all
the $V_J$, $J \subseteq [n]$. We also define sets $\tilde V_J \supseteq \hat V_J \supseteq V_J$
by
\[
\begin{split}
\hat V_J = \left\{ \mathbf v = (\mathbf v_1, \dots, \mathbf v_n) \in
	\mathbb C^{S_1} \times \cdots \times \mathbb C^{S_n}:\  
	\right.
	&
	\|\mathbf v_i\| > \sqrt{S_i}(K-L) \ \text{for}\ i \in J
	\\
	&\text{and}\ \|\mathbf v_i\| \le \sqrt{S_i}K \ \text{for}\ i \not \in J
	\left.
	\right\},
\end{split}
\]
\[
\begin{split}
\tilde V_J = \left\{ \mathbf v = (\mathbf v_1, \dots, \mathbf v_n) \in
	\mathbb C^{S_1} \times \cdots \times \mathbb C^{S_n}:\  
	\right.
	&
	\sigma \|\mathbf v_i\| > \sqrt{S_i}(K-L) \ \text{for}\ i \in J
	\\
	&\text{and}\ \|\mathbf v_i\| \le \sqrt{S_i}K \ \text{for}\ i \not \in J
	\left.
	\right\},
\end{split}
\]
with $\sigma=1+\frac{L}{K} > 1$. The expectation satisfies:
\[
\expected{\mathbf w \sim \secrev{\mathcal N}(\mathbf u, I; \mathbb C^N)}{\varphi(\mathbf w)}
=
\sum_{J \in [n]}
\int_{\mathbf v \in V_J}
	\varphi(\mathbf u+ \mathbf v) \frac{e^{-\|\mathbf v\|^2}}{\pi^{S}} \dd \mathbb C^S(\mathbf v).
\]	
We will change variables inside each integral. If $i \in J$, we set
	\secrev{auxiliary variables}
$\mathbf {\hat u}_i \defeq 0$ and $\hat {\mathbf v}_i \defeq \mathbf u_i + \mathbf v_i$.
If $i \not \in J$, $\mathbf {\hat u}_i \defeq \mathbf u_i$ and
	$\mathbf {\hat v}_i \defeq \mathbf v_i$. In \changed{both} cases, 
$\hat {\mathbf u} + \hat{\mathbf v} = \mathbf u + \mathbf v$.
If $i \in J$, we have
	$\|\mathbf u_i\| \le L \sqrt{S_i}< K \sqrt{S_i} \le \|\mathbf v_i\|$ so
	that Lemma~\ref{triv-bound-lemma} yields
	$-\|\mathbf v_i\|^2 < -\|\hat {\mathbf v}_i\|^2/\sigma^2$ with
$\sigma = 1 + L/K$. 
Each integral can be bounded as follows:
\[
\begin{split}
\int_{\mathbf v \in V_J}
\varphi(\mathbf u+ \mathbf v) 
&
\frac{e^{-\|\mathbf v\|^2}}{\pi^{S}} \dd \mathbb C^S(\mathbf v)
\le\\
	&\le
	\int_{\hat {\mathbf v} \in \hat V_J}
\varphi(\hat{\mathbf u}+ \hat{\mathbf v}) 
	\prod_{i \in J} \frac{e^{-\|\mathbf {\hat v}_i\|^2/\sigma^2}}{\pi^{S_i}} 
\prod_{i \not \in J} \frac{e^{-\|\mathbf {\hat v}_i\|^2}}{\pi^{S_i}} 
	\dd \mathbb C^S(\hat{\mathbf v})
.
\end{split}
\]
For each $j \in J$, $\hat {\mathbf u}_j=0$ and the function $\varphi(\hat{\mathbf u}+ \hat{\mathbf v})$
is invariant by scaling $\hat {\mathbf v}_j$. We will replace 
$\hat {\mathbf v}_j = \sigma \tilde {\mathbf v}_j$
for $j \in J$, and $\tilde{\mathbf v}_j \defeq \hat {\mathbf v}_j$ otherwise.
Now,
\begin{equation}\label{mention-jacobian}
\begin{split}
\int_{\mathbf v \in V_J}
\varphi(\mathbf u+ \mathbf v) 
&
\frac{e^{-\|\mathbf v\|^2}}{\pi^{S}} \dd \mathbb C^S(\mathbf v)
\le\\
	&\le
\sigma^{2\sum_{i \in J} S_i}
	\int_{\tilde{\mathbf v} \in \tilde V_J}
	\varphi(\hat{\mathbf u}+ \tilde{\mathbf v}) 
	\frac{e^{-\|\tilde{\mathbf v}\|^2}}{\pi^{S}} 
	\dd \mathbb C^S(\tilde{\mathbf v})
.
\end{split}
\end{equation}
Again, we take advantage of the scaling invariance of
the function $\varphi$ with respect to $\tilde {\mathbf v}_i$ for $i \in J$.  
For those indices, we can replace the domain
of integration $\|\tilde{\mathbf v}_i\| > \sqrt{S_i}\frac{K-L}{\sigma}$ by 
$\|\tilde{\mathbf v}_i\| \le \sqrt{S_i}K$
as long as we take into account the full probability of each domain. Namely,
\[
\begin{split}
\int_{\mathbf v \in V_J}
\varphi(\mathbf u+ \mathbf v) 
&
\frac{e^{-\|\mathbf v\|^2}}{\pi^{S}} \dd \mathbb C^S(\mathbf v)
\le\\
	&\le
\prod_{i \in J}
\sigma^{2S_i}
\frac
	{\probability{\tilde{\mathbf v}_i \in \secrev{\mathcal N}(\mathbf 0,I; \mathbb C^{S_i})}{\|\tilde{\mathbf v}_i \| > \frac{\sqrt{S_i}(K-L)}{\sigma}}}
	{\probability{\tilde{\mathbf v}_i \in \secrev{\mathcal N}(\mathbf 0,I; \mathbb C^{S_i})}{\|\tilde{\mathbf v}_i \| < \sqrt{S_i}K}}
	\\
	&
	\times \prod_{i=1}^n \int_{\|\tilde{\mathbf v}_i\| \le \sqrt{S_i}K}
	\varphi(\hat{\mathbf u}+ \tilde{\mathbf v}) 
	\frac{e^{-\|\tilde{\mathbf v}\|^2}}{\pi^{S}} 
	\dd \mathbb C^S(\tilde{\mathbf v})
.
\end{split}
\]
The product of integrals is precisely
\[
\prod_{i=1}^n 
\int_{\|\tilde {\mathbf v}_i\| \le \sqrt{S_i}K}
\varphi(\hat{\mathbf u}+ \tilde{\mathbf v}) 
\frac{e^{-\|\tilde{\mathbf v}\|^2}}{\pi^{S}} 
\dd \mathbb C^S(\tilde {\mathbf v})
=
G_{\hat{\mathbf u}, K} \prod_{i=1}^n \probability{\tilde{\mathbf v}_i \sim \secrev{\mathcal N}(\mathbf 0,I; \mathbb C^{S_i})}{\|{\tilde{\mathbf v}_i \| \le \sqrt{S_i}K}}
\]
where
\[
G_{\hat{\mathbf u}, K} \defeq
\expected{\mathbf w \sim \secrev{\mathcal N}(\hat{\mathbf u}, I; \mathbb C^{\thirdrev{S}})}
{\varphi(\mathbf w) \conditional \| \mathbf w -\hat {\mathbf u}\| \le \sqrt{S_i}K}
.
\]
In the equation above, $\hat {\mathbf u}$ depends on the choice of $J$.
This is why we take the supremum 
$\sup_{\|\hat {\mathbf u}_i\| \le K \sqrt{S_i}}
G_{\hat{\mathbf u}, K}$ in the main statement.
Lemma~\ref{lem-large-dev}(b) provides the bound
\[
	\probability{\tilde{\mathbf v}_i \in \secrev{\mathcal N}(\mathbf 0,I; \mathbb C^{S_i})}{\|\tilde{\mathbf v}_i \| > \frac{\sqrt{S_i}(K-L)}{\sigma}}
\le
e^{-S_i \left( \frac{K-L}{\sigma}-1 \right)^2}
.
\]
\changed{We obtained:
\[
\begin{split}
\int_{\mathbf v \in V_J}
\varphi(\mathbf u+ \mathbf v) 
\frac{e^{-\|\mathbf v\|^2}}{\pi^{S}} \dd \mathbb C^S(\mathbf v)
\le
G_{\hat{\mathbf u}, K} \prod_{i\not \in J} \probability{\tilde{\mathbf v}_i \in \secrev{\mathcal N}(\mathbf 0,I; \mathbb C^{S_i})}{\|{\tilde{\mathbf v}_i \| \le \sqrt{S_i}K}}
\\
\times 
\prod_{i \in J} \sigma^{2S_i} 
	e^{-S_i \left( \frac{K-L}{\sigma}-1 \right)^2}
\end{split}
\]
}
For $j \not \in J$, we use the trivial bound
\[
	\probability{\tilde{\mathbf v}_i \in \secrev{\mathcal N}(\mathbf 0,I; \mathbb C^{S_i})}{\|\tilde {\mathbf v}_i \| \changed{\le} \sqrt{S_i}K} \le 1
.
\]
Adding for all subsets $J$,
\begin{eqnarray*}
\expected{\mathbf w \sim \secrev{\mathcal N}(\mathbf u, I; \mathbb C^N)}{\varphi(w)}
&\le&
\sum_{J \subseteq [n]}
\prod_{i \in J}
	e^{S_i \left( 2 \log(\sigma) -\left( \frac{K-L}{\sigma}-1 \right)^2\right)} 
\sup_{\|\hat {\mathbf u}_i\| \le K \sqrt{S_i}}
G_{\hat{\mathbf u}, K}
\\
&\le&
\prod_{i=1}^n
\left( 1+ 
	e^{S_i\left(2 \log(\sigma) - \left( \frac{K-L}{\sigma}-1 \right)^2\right)} 
\right)
\sup_{\|\hat {\mathbf u}_i\| \le K \sqrt{S_i}}
G_{\hat{\mathbf u}, K}
.
\end{eqnarray*}
We choose $K \ge 1 + 2 L + t$, with $t \secrev{\ge 0}$ to be determined.
In this case $\sigma=1+L/K < 3/2$ and
\[
\frac{K-L}{\sigma} - 1=
\frac{ K^2 - KL - K - L}{K+L} = K -1 - 2\frac{KL}{K+L} \ge K - 1 - 2L \ge t.
\]
By setting $t = \sqrt{\frac{\log(n)}{\changed{\mathrm{min}} S_i}+2\log(3/2)}$, we obtain
\[
\prod_{i=1}^n
	\left( 1+ e^{S_i \left( 2 \log(\sigma) - \left( \frac{K-L}{\sigma}-1 \right)^2\right)} 
	\right)
\le
\left(
1+\frac{1}{n}
\right)^n
\le
e
\]
and thus
\[
\expected{\mathbf w \sim \secrev{\mathcal N}(\mathbf u, I; \mathbb C^N)}{\varphi(\mathbf w)}
\le
e
\sup_{\|\hat {\mathbf u}_i\| \le K \sqrt{S_i}}
G_{\hat{\mathbf u}, K}
.
\]

The proof for the real case is the same, with the following changes.
In equation~\eqref{mention-jacobian}, the Jacobian is
$\sigma^{\sum_{i \in J} S_i}$ instead of
$\sigma^{2\sum_{i \in J} S_i}$. Lemma~\ref{lem-large-dev}(a) yields
\[
	\probability{\tilde{\mathbf v}_i \in \secrev{\mathcal N}(\mathbf 0,I; \mathbb R^{S_i})}{\|\tilde{\mathbf v}_i \| > \frac{\sqrt{S_i}(K-L)}{\sigma}}
\le
e^{-S_i \left( \frac{K-L}{\sigma}-1 \right)^2/2}
.
\]
Therefore, with the same choice of $t$, 
\[
\prod_{i=1}^n
	\left( 1+ e^{S_i \left( \log(\sigma) - \left( \frac{K-L}{\sigma}-1 \right)^2\right)/2} 
	\right)
\le
\left(
1+\frac{1}{2n}
\right)^n
\le
e^{\frac{1}{2}}.
\]
\end{proof}

\begin{proof}[Proof of Proposition~\ref{prop-mu2-1}]
	Recall that the $\mathbf f_i \in \mathscr F_{A_i}$ are always written as covectors, so
$\mathbf f_i \Sigma_i^{-1}$ is the product of $\mathbf f_i$ by the matrix $\Sigma_i^{-1}$. Under that notation, write
$\mathbf u_i = \mathbf f_i \Sigma_i^{-1}$ and $\mathbf w_i=\mathbf q_i \Sigma_i^{-1}$.
	\secrev{Suppose that $\|\mathbf f_i \Sigma_i^{-1}\| = \|\mathbf u_i\| \le L \sqrt{S_i}$}. Let
\[
\varphi(\mathbf w) \defeq \sum_{\mathbf z \in Z_H(\mathbf w \cdot \Sigma)} 
	\secrev{\mu (\mathbf w \cdot \Sigma \cdot R(\mathbf z))}^2 
.
\]
	In the notations of Lemma~\ref{lem-conditional}, 
\[
E_{\mathbf f, \Sigma^2} = 
	\expected{\mathbf q \sim \secrev{\mathcal N}(\mathbf f,\Sigma^2)}{\varphi(\mathbf q_i \Sigma_i^{-1})}=
\expected{\mathbf w \sim \secrev{\mathcal N}(\mathbf u,\mathbf I)}{\varphi(\mathbf w)}
\]
and similarly
\[
	E_{\secrev{\mathbf h}, \Sigma^2,K} = G_{\changed{\hat{\mathbf u},K}}
\]
for $\secrev{\mathbf h_i} = \hat{\mathbf u}_i \Sigma_i$.
It follows from Lemma~\ref{lem-conditional} that
\[
E_{\mathbf f, \Sigma^2}  
\le
	e \sup_{\| \secrev{\mathbf h}_i \Sigma_i^{-1} \| \le L\sqrt{S_i}}
E_{\secrev{\mathbf h}, \Sigma^2,K} 
.
\]
\end{proof}

\subsection{Renormalization and the condition number}
\label{sec-renormalization}

\changed{To complete} the proof of Theorem ~\ref{old-mainB}, we will estimate the expectation $E_{\secrev{\mathbf h}, \Sigma^2, K}$ from Proposition~\ref{prop-mu2-1} in
terms 
of the \secrev{expectation} 
\[
	I_{\secrev{\mathbf h},\Sigma^2} \defeq 
\expected{\mathbf q \sim \secrev{\mathcal N}(\secrev{\mathbf h}, \Sigma^2)} { \sum_{\mathbf z \in Z_H(\mathbf q)} \| M(\mathbf q ,z)^{-1}\|_{\mathrm F}^2 }
\]
\secrev{that was bounded in Theorem~\ref{E-M2}}.
\begin{proposition}\label{prop-mu2-2} For $i=1, \dots, n$, let $S_i=\# A_i \ge 2$
and $\delta_i = \max_{\mathbf a \in A_i} \| \mathbf a - \mathbf m_i(\mathbf 0)\|$.
	Let $\| \mathbf f_i \Sigma^{-1}\| \le \sqrt{S_i} L$. Let $K \ge 1 + 2L + \sqrt{\frac{\log(n)}{\min (S_i)}+ 2\log(3/2)}$. Then,
\[
E_{\mathbf f, \Sigma^2}
	\changed{\le}
	1.25 e (K+L)^2 
\left( \sum \delta_i^2 \right)
	\max_i \left(S_i \kappa_{\rho_i}^2
	\max_{\mathbf a \in A_i} \sigma_{i, \mathbf a}^2\right)
\sup_{\| \secrev{\mathbf h}_i \Sigma_i^{-1} \| \le L\sqrt{S_i}}
I_{\secrev{\mathbf h},\Sigma^2} 
.
\]
\end{proposition}

\begin{proof}[Proof of Theorem \ref{old-mainB}]
	Plugging Theorem~\ref{E-M2} into Proposition~\ref{prop-mu2-2},
\[
\begin{split}
E_{\mathbf f, \Sigma^2}
\le
	\frac{2.5 e H\sqrt{n}}{\det(\Lambda)} &
	\left( 1 + 3L +\sqrt{\frac{\log(n)}{\min (S_i)}+2\log(3/2)}\right)^2
\\
	&\times \frac{\max_{i}\left( S_i \kappa_{\rho_i}^2 \max_{\mathbf a \in A_i} (\sigma_{i, \mathbf a}^2)\right)}
{\min_{i,\mathbf a} (\sigma_{i, \mathbf a}^2)}
\left( \sum \delta_i^2 \right)
	\ \thirdrev{n!} V'
\end{split}
\]
with $L = \max \| \mathbf f_i \Sigma_i^{-1} \|/ \sqrt{S_i}$.
\end{proof}

In order to prove Proposition~\ref{prop-mu2-2},
we need two preliminary results. \secrev{Before stating the first,
recall that $\mathscr M$ was endowed at each point $\mathbf z$ with 
an inner product \eqref{metricM} and its induced norm. We will denote
by $\| \cdot \|_{\mathbf 0}$ the norm at the point $\mathbf 0 \in \mathscr M$.
We have several choices for the norm of the operator 
$M(\mathbf q, \mathbf z)^{-1}: \mathbb C^n \rightarrow T_{\mathbf 0}\mathscr M$.
We can use the 
operator
norm with respect to the canonical norm of $\mathbf C^n$, viz.  
$\|M(\mathbf q, \mathbf z)^{-1}\|_{\mathbf 0}$. We may compare it
to the Frobenius norm $\|M(\mathbf q, \mathbf z)^{-1}\|_{\mathrm F}$ of the
matrix. Let $\delta_i \defeq \delta_i(\mathbf 0)$ be the radius from 
\eqref{deltaix}} 

\begin{lemma}\label{lemma-condition} 
Let $\mathbf q \in \mathscr F$ and $\mathbf z \in Z(\mathbf q) \subseteq \secrev{\mathscr M}$. Then,
\begin{enumerate}[(a)]
\item
\[
\left\| M(\mathbf q , \mathbf z)^{-1} \right\|_{\mathbf 0}
\le \sqrt{\sum_i \delta_i^2 }
\left\| M(\mathbf q , \mathbf z)^{-1} \right\|_{\mathrm F},
\]
\item and
\[
\mu( \mathbf q \cdot R(\mathbf z) )
\le \sqrt{\sum_i \delta_i^2 } 
\max_i (\kappa_{\rho_i} \|\mathbf q_i\|)
\left\| M(\mathbf q , \mathbf z)^{-1} \right\|_{\mathrm F}
.
\]
	\end{enumerate}
\end{lemma}

Before proving Lemma~\ref{lemma-condition}, we need to compare
the norms of $V_{A_i}(\mathbf 0)$ and $V_{A_i}(\mathbf z)$. Let
$\ell_i(\mathbf z)=\max_{\mathbf a \in A_i}(\mathbf a\ \Re(\mathbf z))$.

\begin{lemma}\label{lem-V-coarse}
Assume that $m_i(\mathbf 0)=\mathbf 0$.
Let $A_i'$ denote the set of vertices of $\conv{A_i}$. 
	\changed{Let}
\[
\kappa_{\rho_i}\defeq\frac{
\sqrt{ \sum_{\mathbf a \in A_i} \rho_{i,\mathbf a}^2 }
}
{\min_{\mathbf a \in A_i'} \rho_{i,\mathbf a}}.
\]
	Then \secrev{for all $\mathbf z \in \mathbb C^n$},
\[
\|V_{A_i}(\mathbf 0)\| \le 
\|V_{A_i}(\mathbf z)
\|
\le e^{\ell_i(\mathbf z)}
\| V_{A_i}(\mathbf 0) 
\|
\le
\kappa_{\rho_i}
\| V_{A_i}(\mathbf z)\|
\]
In particular, if the coefficients $\rho_{i,\mathbf a}=\rho_i$ are
	the same, $\kappa_{\rho_i} = \sqrt{S_i}$ and 
\[
e^{\ell_i(\mathbf z)}
\| V_{A_i}(\mathbf 0) 
\|
\le
\sqrt{S_i}
\| V_{A_i}(\mathbf z)\|.
\]
\end{lemma}
\begin{remark}
In the context of example~\ref{ex:Weyl}, 
if $\rho_{i,\mathbf a}=\sqrt{\binomial{d}{\mathbf a}}$
and $A_i = \{ \mathbf a \in \mathbb Z^n: 0 \le a_i, \sum a_i \le d_i\}$,
$\kappa_{\rho_i}=(n+1)^{d_i}$
and all we have is
\[
	e^{d_i \max(\max_j (\Re(z_j)),0)}
\| V_{A_i}(0) 
\|
\le
(n+1)^{d_{i}}
\| V_{A_i}(\mathbf z)\|.
\]
\end{remark}
\begin{proof}[\secrev{Proof of Lemma~\ref{lem-V-coarse}}]
In order to prove the first inequality, we claim that $\mathbf 0$ is a 
	global minimum of $\|V_{A_i}(\mathbf z)\|$.
Indeed, $m_i(\mathbf z)$ is precisely the derivative
of the convex potential
\[
		\defun{\psi}{\mathbb R^n}{\mathbb R}{\mathbf x}{\frac{1}{2}
		\log(\|V_{A_i}(\mathbf x)\|^2)}
.
\] 
	Since $V_{A_i}$ and $m_i$ depend only on the real part of their argument and
$m_i(\mathbf 0)=\mathbf 0$, the point $\mathbf 0$ is a global minimum of the convex potential $\psi$.

For the next two inequalities, we compare the norms of
\[
e^{\ell_i(\mathbf z)} V_{A_i}(\mathbf 0) =
\begin{pmatrix}
\vdots \\
\rho_{i,\mathbf a} e^{\ell_i(\mathbf z)}
\\ \vdots
\end{pmatrix}_{\mathbf a \in A_i}
\ 
\text{and}
\
V_{A_i}(\mathbf z) =
\begin{pmatrix}
\vdots \\
\rho_{i,\mathbf a} e^{\mathbf a \mathbf z}
\\ \vdots
\end{pmatrix}_{\mathbf a \in A_i}
.
\]
Comparing coordinate by coordinate, 
\[
\|V_{A_i}(\mathbf z)
\|
\le e^{\ell_i(\mathbf z)}
\| V_{A_i}(\mathbf 0) 
\|
\]
The maximum of $\mathbf a \ \Re(\mathbf z)$
is attained for some $\mathbf a^* \in A_i'$. Hence,
\[
e^{\ell_i(\mathbf z)}
\| V_{A_i}(\mathbf 0) 
\|
=
e^{\mathbf a^* \ \Re(\mathbf z)}
\sqrt{ \sum_{\mathbf a \in A_i} \rho_{i,\mathbf a}^2 }
\le
\frac{
\sqrt{ \sum_{\mathbf a \in A_i} \rho_{i,\mathbf a}^2 }
}
{\rho_{i,\mathbf a^*}}
\| V_{A_i}(\mathbf z)\|.
\]
.
\end{proof}

\begin{proof}[Proof of Lemma~\ref{lemma-condition}]
\noindent \\
{\bf Item (a):}
We have to prove that that	
\[
	\|M(\mathbf q,\mathbf z)^{-1}\|_{\mathbf 0} 
	\le \sqrt{\sum_i \delta_i^2 }\,
	\|M(\mathbf q,\mathbf z)^{-1}\|_{\mathrm F}
\]
where 
$\|X\|_{\mathrm F} = \sqrt{\sum_{ij} |X_{ij}|^2}$ is the Frobenius norm.

\begin{eqnarray*}
\| M(\mathbf q,\mathbf z)^{-1}\|_{\mathbf 0} 
&=&
	\| D[\mathbf V](\mathbf 0)\ M(\mathbf q,\mathbf z)^{-1}\|
\\
&\le&
	\| D[\mathbf V](\mathbf 0) \|  \|M(\mathbf q,\mathbf z)^{-1}\|
\\
&\le&
	\| D[\mathbf V](\mathbf 0) \|  \|M(\mathbf q,\mathbf z)^{-1}\|_{\mathrm F}. 
\end{eqnarray*}
Lemma~\ref{coarse-bound-DV} yields, for each $i=1, \dots n$,
\[
	\| D[\mathbf V_{A_i}](\mathbf 0) \| \le \delta_i
\]
and hence
\[
	\|D[\mathbf V](\mathbf 0)\|_2 \le \sqrt{\sum \delta_i^2}.
\]
{\bf Item(b):}
We can assume without loss of generality that $m_i(\mathbf 0)=\mathbf 0$ for each $i$. 
Indeed, subtracting $m_i(\mathbf 0)$ from each $\mathbf a \in A_i$ will multiply 
$V_{A_i}(\mathbf z)$, $DV_{A_i}(\mathbf z)$ by the same constant $e^{-m_i(\mathbf 0) \mathbf z}$. 
In particular,
	\secrev{$[V_{A_i}](\mathbf 0)$, $D[V_{A_i}](\mathbf 0)$} do not change and the metric of 
$T_{\mathbf 0}\secrev{\mathscr M}$ is the same. 
	Also, the \secrev { 
	operator $M(\mathbf q,\mathbf z)$ does not change.}
Under the hypothesis $m_i(\mathbf 0)=\mathbf 0$,
	\begin{equation}\label{MRenorm}
		M(\mathbf q \cdot R(\mathbf z), \mathbf 0) =  
	\begin{pmatrix}
	\frac{1}{\|V_{A_1}(\mathbf 0)\| } \mathbf q_1 R_1(\mathbf z) DV_{A_1}(\mathbf 0) \\
\vdots \\
	\frac{1}{\|V_{A_n}(\mathbf 0)\| } \mathbf q_n R_n(\mathbf z) DV_{A_n}(\mathbf 0) \\
\end{pmatrix}
		=
	\begin{pmatrix}
		\frac{\secrev{1}}{\|V_{A_1}(\mathbf 0)\| } \mathbf q_1 DV_{A_1}(\mathbf z) \\
\vdots \\
		\frac{\secrev{1}}{\|V_{A_n}(\mathbf 0)\| } \mathbf q_n DV_{A_n}(\mathbf z) \\
\end{pmatrix}
.
\end{equation}
Therefore,
	\secrev{\begin{align*}
\mu(\mathbf q \cdot R(\mathbf z)) &= 
\left\| 
M(\mathbf q \cdot R(\mathbf z), \mathbf 0)^{-1} \diag { \|\mathbf q_i \cdot R(\mathbf z)\|}
	\right\|_{\mathbf 0} && \text{by definition,}
\\
&=
\left\| 
\left(
	\diag{\frac{\|V_{A_i}(\mathbf z)\|}{\|V_{A_i}(\mathbf 0)\|}} 
M(\mathbf q , \mathbf z)\right)^{-1} \diag { \|\mathbf q_i \cdot R_i(\mathbf z)\|}
\right\|_{\mathbf 0}
	&& \text{by \eqref{MRenorm},}
\\
&=
\left\| 
M(\mathbf q , \mathbf z)^{-1} 
\diag { 
	\frac{ \|\mathbf q_i \cdot R_i(\mathbf z )\|\|V_{A_i}(\mathbf 0)\| }
	{\|V_{A_i}(\mathbf z)\|}} 
\right\|_{\mathbf 0}
	&&
\\
&\le
\left\| 
M(\mathbf q , \mathbf z)^{-1} 
\right\|_{\mathbf 0}
\max_{i} 
	\frac{ \|\mathbf q_i \cdot R_i(\mathbf z)\|\|V_{A_i}(\mathbf 0)\| } 
	{\|V_{A_i}(\mathbf z)\|} .
&&
\end{align*}
Theorem~\ref{th-renormalization}(c) bounds
\[
\| \mathbf q_i R_i(\mathbf z) \| \le e^{\ell_i(\mathbf z)} \| \mathbf q_i\|. 
\]
From Lemma~\ref{lem-V-coarse},
\[
e^{\ell_i(\mathbf z)} \|V_{A_i}(\mathbf 0)\| \le
\kappa_{\rho_i}
\| V_{A_i}(\mathbf z)\| 
.
\]
}
Combining those bounds with item (a), 
\[
\secrev{\mu( \mathbf q \cdot R(\mathbf z) )}
\le \sqrt{\sum_i \delta_i^2 } 
\left\| M(\mathbf q , \mathbf z)^{-1} \right\|_{\mathrm F}
\max_i (\kappa_{\rho_i} \|\mathbf q_i\|)
.
\]
\end{proof}

\begin{proof}[Proof of Proposition~\ref{prop-mu2-2}]
	From Proposition~\ref{prop-mu2-1},
\begin{eqnarray*}
E_{\mathbf f, \Sigma^2}
	&\le&
e
	\sup_{\|\secrev{\mathbf h}_i \Sigma^{-1}\| \le L\sqrt{S_i}}
	E_{\secrev{\mathbf h}, \Sigma^2,K}
\\
&\le&
e
\sup_{\|\secrev{\mathbf h}_i \Sigma^{-1}\| \le L\sqrt{S_i}}
	\condexpected{\mathbf q \sim \secrev{\mathcal N}(\secrev{\mathbf h}, \Sigma^2)}
	{\sum_{\mathbf z \in Z_H(\mathbf q)} 
	\secrev{\mu(\mathbf q R(\mathbf z))^2}}
	{}
	{&& \hspace{11em} \left.\conditional \|(\mathbf q_i - \secrev{\mathbf h}_i)\Sigma^{-1}\|\le K \sqrt{S_i},
\ i=1, \dots, n}
.
\end{eqnarray*}
	The condition $\|(\mathbf q_i - \secrev{\mathbf h}_i)\Sigma^{-1}\|\le K \sqrt{S_i}$
	implies that $\|\mathbf q_i\| \le (K+L)\sqrt{S_i}  
\max_{\mathbf a \in A_i} \sigma_{i\mathbf a}$. From Lemma~\ref{lemma-condition},
\[
	\secrev{\mu(\mathbf q R(\mathbf z))}^2
\le
(K+L)^2 
(\sum_i \delta_i^2)
	\max_i \left( S_i \kappa_{\rho_i}^2 \max_{\mathbf a \in A_i} \sigma_{i\mathbf a}^2\right)
\| M(\mathbf q,\mathbf z)^{-1}\|_{\mathrm F}^2
.
\]
It follows that
\[
\begin{split}
E_{\mathbf f, \Sigma^2}
	\le
e
(K+L)^2 
	(\sum_i \delta_i^2)&
	\max_i \left(S_i \kappa_{\rho_i}^2 \max_{\mathbf a \in A_i} \sigma_{i\mathbf a}^2\right)
	\\
	&\times \sup_{\|\secrev{\mathbf h}_i \Sigma^{-1}\| \le L\sqrt{S_i}}
	\condexpected{\mathbf q \sim \secrev{\mathcal N}(\secrev{\mathbf h}, \Sigma^2)}
	{\sum_{\mathbf z \in Z_H(\mathbf q)} \|M(\mathbf q,\mathbf z)^{-1}\|_{\mathrm F}^2}
	{}
	{& \left. \hspace{7em}\conditional \|(\mathbf q_i - \secrev{\mathbf h}_i)\Sigma^{-1}\|\le K \sqrt{S_i}
	i=1, \dots, n}
.
\end{split}
\]
	The conditional \secrev{expectation} times the probability that $\|(\mathbf q_i - \secrev{\mathbf h}_i)\Sigma^{-1}\|\le K \sqrt{S_i}$ is bounded above by $I_{\secrev{\mathbf h},\Sigma^2}$. Thus,
\[
E_{\mathbf f, \Sigma^2}
\le
\frac{e
(K+L)^2 
(\sum_i \delta_i^2)
	\max_i \left(S_i \kappa_{\rho_i}^2 \max_{\mathbf a \in A_i} \sigma_{i\mathbf a}^2\right)
	}{\probability{\mathbf q \sim \secrev{\mathcal N}(\secrev{\mathbf h}, \Sigma^2)}
	{\|(\mathbf q_i - \secrev{\mathbf h}_i)\Sigma^{-1}\|\le K \sqrt{S_i},
\ i=1, \dots, n}
}
	\sup_{\|\secrev{\mathbf h}_i \Sigma^{-1}\| \le L\sqrt{S_i}}
	(I_{\secrev{\mathbf h},\Sigma^2})
\]
From Lemma~\ref{lem-large-dev}(b),
	\[
\probability{\mathbf q_i \sim \secrev{\mathcal N}(\secrev{\mathbf h_i}, \Sigma^2)}
{\|(\mathbf q_i - \secrev{\mathbf h}_i)\Sigma^{-1}\|> K \sqrt{S_i}, }
\le
e^{-S_i(K-1)^2} 
< \frac{1}{n} \left(\frac{4}{9}\right)^{S_i}
.
\]
	\secrev{Since we are assuming $L>0$ and $S_i \ge 2$, we can bound
\begin{eqnarray*}
	S_i(K-1)^2 &\ge& 
S_i \left(2L+\sqrt{\frac{\log(n)}{\min(S_i)}+2\log(3/2)}\right)^2
\\
&\ge& S_i \left( \frac{\log(n)}{\min(S_i)}+2\log(3/2) \right)
\\
	&\ge& \log(n) + 4 \log(3/2),
\end{eqnarray*}
hence
\[
e^{-S_i(K-1)^2} 
\le
\frac{1}{n} \frac{16}{81} 
\]
}
	Thus, the probability that $\|(\mathbf q_i - \secrev{\mathbf h}_i)\Sigma^{-1}\|> K \sqrt{S_i}$ for some $i$ is at most $\frac{16}{81}$. 
The probability of the opposite event
is therefore at least $65/81$. In the final expression, the factor $81/65$ \changed{was replaced} by the approximation $1.25$.
\end{proof}

\section{Toric infinity}

\subsection{Proof of \secrev{Condition Number Theorem~\ref{cond-num-infty}}}

Let $\mathbf 0 \ne \boldsymbol \xi_1, \dots, \boldsymbol \xi_m \in \mathbb R^n$. The 
closed polyhedral cone
spanned by the $\xi_i$ is
\[
\mathrm{Cone}(\boldsymbol \xi_1, \dots, \boldsymbol \xi_m) =
\{ s_1 \boldsymbol \xi_1 + \dots + s_m \boldsymbol \xi_m: s_1, \dots, s_m \ge 0\}
.
\]
It turns out that any $k$-dimensional polyhedral cone is actually a union of cones of the form
$\mathrm{Cone}(\boldsymbol \xi_I)$ where $\#I = k$: 

\begin{theorem}[Carathéodory] Let $\mathbf 0 \ne \boldsymbol \xi_1, \dots, \boldsymbol \xi_m \in \mathbb R^n$ and let
	$\mathbf x \in \mathrm{Cone}(\boldsymbol \xi_1, \dots,$ $\boldsymbol \xi_n)$. Then there is $I \subseteq [m]$,
	$\#I \le n$, such that $\mathbf x \in \mathrm{Cone}(\boldsymbol \xi_I)$.
\end{theorem}
A proof of Carathéodory's Theorem can be found in the book by \ocite{BCSS}*{Cor. 2 p.168}.
We will apply this theorem here to the cones of the fan of a tuple of supports $A_1, \dots, A_n$. Recall from
section \ref{infinity} that given subsets
$B_1 \subseteq A_1, \dots , B_n
\subseteq A_n$, the (open) {\em cone above $(B_1, \dots, B_n)$} is
\[
C(B_1, \dots, B_n) = \{ 
\boldsymbol \xi \in \mathbb R^n:
	B_i = A_i^{\boldsymbol \xi} \}
\]
where $A_i^{\boldsymbol \xi}$ is the set of $\mathbf a \in A_i$ maximizing $\mathbf a \boldsymbol \xi$. \changed{The closure of this cone } belongs to some stratum $\mathfrak F_{k-1}$ of the fan \secrev{(See Definition~\ref{defan})}.
\secrev{The one-facets of the closed cone are all rays in} the stratum $\mathfrak F_0$. \secrev{Each ray can be represented as the non-negative span of a vector
$\boldsymbol \xi \in S^{n-1}$. Recall that $\mathfrak R$ denotes the set of all
those representatives $\boldsymbol \xi$.}
Carathéodory's theorem directly implies:
	\begin{corollary}\label{cor-carath}
Let $\mathbf 0 \ne \mathbf x \in C(B_1, \dots, B_n)$. Then there are $k \le n$,
and 
$\boldsymbol \xi_1, \dots, \boldsymbol \xi_k \in \thirdrev{\mathfrak R}$ so that
\[
\mathbf x = s_1 \boldsymbol \xi_1 + \dots + s_k\boldsymbol \xi_k
\]
for some $s_1, \dots, s_k > 0$.
\end{corollary}

\begin{proof}[Proof of Theorem~\ref{cond-num-infty}]
	Let $\mathbf z = \mathbf x + \sqrt{-1}\, \mathbf y \in Z(\mathbf q)$ be such that $\| \Re(\mathbf z) \| \ge H$.
	By Corollary~\ref{cor-carath} \changed{applied to 
$B_1=A_1^{\mathbf x}, \dots,
	B_n=A_n^{\mathbf x}$}, there are $\boldsymbol \xi_1, \dots, \boldsymbol \xi_k \in \thirdrev{\mathfrak R}$,
$k \le n$,
so that
\[
\mathbf x = s_1 \boldsymbol \xi_1 + \dots + s_k \boldsymbol \xi_k
\]
with $s_1, \dots, s_k > 0$. By permuting the $\boldsymbol \xi_i$'s one can
assume that $s_1 \ge s_2 \ge \cdots \ge s_k>0$. In particular,
\[
	H \le \| \mathbf x\| \le s_1 \|\boldsymbol \xi_1\| + \dots + s_k \|\boldsymbol \xi_k\| \le k s_1 \le n s_1
.
\]
This provides a lower bound $s_1 \ge H/n$. 
Suppose now that $\mathbf a \mathbf x$ is maximal for $\mathbf a \in A_i$. 
Since $\mathbf a \in A_i^{\mathbf x} 
 = A_i^{\boldsymbol \xi_1} \cap  A_i^{\boldsymbol \xi_2} \cap \dots \cap 	A_i^{\boldsymbol \xi_k}$, 
\[
	\mathbf a \mathbf x = s_1 \lambda_i(\boldsymbol \xi_1) + s_2 \lambda_i(\boldsymbol \xi_2) + \dots + s_k \lambda_i(\boldsymbol \xi_k)
.
\]
	Now suppose that $\mathbf a' \in A_i \setminus A_i^{\boldsymbol \xi_1}$. In that case
\[
	\mathbf a' \mathbf x = s_1 \mathbf a'\boldsymbol \xi_1 + s_2 \mathbf a'\boldsymbol \xi_2 + \dots + s_k \mathbf a'\boldsymbol \xi_k
.
\]
Subtracting the two expressions,
\[
	(\mathbf a'-\mathbf a)\mathbf x = s_1 (\mathbf a'\boldsymbol \xi_1 - \lambda_i(\boldsymbol \xi_1)) + \cdots + s_k (\mathbf a'\boldsymbol \xi_k - \lambda_i(\boldsymbol \xi_k))
.
\]
	For $j > 1$, we estimate $\mathbf a'\boldsymbol \xi_j - \lambda_i(\boldsymbol \xi_j) \le 0$. But for $j=1$, we can bound
	\[
\mathbf a'\boldsymbol \xi_1 - \lambda_i(\boldsymbol \xi_1) \le - \eta_i(\boldsymbol \xi_1) \le -\eta_i 
.
\]
Therefore, 
\begin{equation}\label{where-max-attained}
	(\mathbf a'-\mathbf a)\mathbf x \le -\eta_i H/n.
\end{equation}
In order to produce the perturbation $\mathbf h_i$, define first
\[
g_i = 
	\sum_{\mathbf a' \in A_i \setminus A_i^{\boldsymbol \xi_1}}
	\mathbf q_{i \mathbf a'} V_{i\mathbf a'}(\mathbf z)
	= -
	\sum_{\mathbf a \in A_i^{\boldsymbol \xi_1}}
	\mathbf q_{i \mathbf a} V_{i\mathbf a}(\mathbf z).
\]
For $\mathbf a \in A_i^{\boldsymbol \xi_1}$, set
\[
	h_{i \mathbf a} = g_i \frac{ \overline{V_{i \mathbf a}(\mathbf z)}}
	{\sum_{\mathbf a \in A_i^{\boldsymbol \xi_1}}|V_{i \mathbf a}(\mathbf z)|^2}
\]
and set $h_{i \mathbf a'} =0$ for $\mathbf a'\in A_i \setminus A_i^{\boldsymbol \xi_1}$. For later usage, we save the bound
	\begin{equation}\label{h-on-g}
		\| \mathbf h_i \| \le \frac{ |g_i| }{\sqrt{\sum_{\mathbf a \in A_i^{\boldsymbol \xi_1}}|V_{i \mathbf a}(\mathbf z)|^2}}
.
\end{equation}

Let 
	\[
		W_i = \lim_{t \rightarrow \infty} 
	\frac{1}{\|V_{A_i}(\mathbf z + t \boldsymbol \xi_1)\|}
	V_{A_i}(\mathbf z + t \boldsymbol \xi_1)
\]
	so that $[\mathbf W] = \lim_{t \rightarrow \infty} [\mathbf V(\mathbf z + t \boldsymbol \xi_1)]$.
\changed{For all $\mathbf a' \in A_i \setminus A_i^{\boldsymbol \xi_1}$,
$W_{i\mathbf a'}$ vanishes. For $\mathbf a \in A_i^{\boldsymbol \xi_1}$, 
\[
	W_{i\mathbf a}=\frac{V_{i\mathbf a}\secrev{(\mathbf z)}}{\sqrt{\sum_{\mathbf a'' \in A_i^{\boldsymbol \xi_1}} |V_{i\mathbf a''}\secrev{(\mathbf z)}|^2}}
.
\]
Therefore,}
\begin{eqnarray*}
	(\mathbf q_i + \mathbf h_i) W_i &=& 
	\sum_{\mathbf a \in A_i}
	(q_{i \mathbf a}+ h_{i \mathbf a}) W_{i\mathbf a}
\\
	&=&
	\sum_{\mathbf a \in A_i^{\boldsymbol \xi_1}}
	(q_{i \mathbf a}+ h_{i \mathbf a}) W_{i\mathbf a}
\\
	&=&
	\frac{\sum_{\mathbf a \in A_i^{\boldsymbol \xi_1}}
	q_{i \mathbf a} V_{i\mathbf a}(\mathbf z) \thirdrev{+} g_i}{
		\sqrt{ \sum_{\mathbf a \in A_i^{\boldsymbol \xi_1}}
	|V_{i \mathbf a}(\mathbf z)|^2} }\\
	&=& 0.
\end{eqnarray*}

The norm of the perturbation can be estimated \changed{in terms of
the facet gap $\eta_i$:}
for each $\mathbf a' \in A_i \setminus
A_i^{\boldsymbol \xi_1}$, $\mathbf a'-\lambda_i(\mathbf x) \le -\eta_i H/n$,
so
\[
	|e^{\mathbf a' \mathbf z}| \le e^{-\eta_i H/n + \lambda_i(\mathbf x)} 
\]
and hence
\[
	|g_i| = \left| \sum_{\mathbf a' \in A_i \setminus A_i^{\boldsymbol \xi_1}}
q_{i{\mathbf a}'} V_{i \mathbf a'}(\mathbf z)  \right|
\le
\|\mathbf q_i\| \sqrt{\sum_{\mathbf a' \in A_i \setminus A_i^{\boldsymbol \xi_1}}
\rho_{i \mathbf a'}^2 }\, e^{-\eta_i H/n + \lambda_i(\mathbf x)} 
.
\]
\secrev{By construction, $\lambda_i(\mathbf x) \defeq \secrev{\max_{\mathbf a \in A_i} \mathbf a \mathbf x} = \mathbf a^* \mathbf x$ for some
vertex $\mathbf a^*$ of $A_i$. Necessarily from \eqref{where-max-attained},
$\mathbf a^* \in A_i^{\boldsymbol \xi}$.
Introducing $V_{i\mathbf a^*}(\mathbf z)$ in the equation above,} 
\begin{eqnarray*}
	|g_i| &\le& 
\|\mathbf q_i\| 
\frac{\sqrt{\sum_{\mathbf a' \in A_i \setminus A_i^{\boldsymbol \xi_1}}
\rho_{i \mathbf a'}^2 }}{\rho_{i \mathbf a^*}}
e^{-\eta_i H/n} 
|V_{i\mathbf a^*}(\mathbf z)|
\\
&\le&
\|\mathbf q_i\| 
\kappa_{\rho_i}
e^{-\eta_i H/n}  
\secrev{|V_{i\mathbf a^*}(\mathbf z)|}
\\
&\le&
\|\mathbf q_i\| 
\kappa_{\rho_i}
e^{-\eta_i H/n}  
\sqrt{\sum_{\mathbf a \in A_i^{\boldsymbol \xi_1}}|V_{i \mathbf a}(\mathbf z)|^2}
\end{eqnarray*}
\changed{where $\kappa_{\rho_i}$ is the distortion bound defined in 
\eqref{eq-rho}.} 
	It follows \secrev{from \eqref{h-on-g}} that 
\[
\frac{\| \mathbf h_i \|}{\| \mathbf q_i \|}
\le
\kappa_{\rho_i}
e^{-\eta_i H/n}
.
\]
\end{proof}

\subsection{On \secrev{strongly mixed support tuples}}

\begin{proof}[Proof of Lemma~\ref{tropical}]
	We prove first the equivalence between $(a)$ and $(b)$.	
\secrev{
Assume that the tuple $(A_1, \dots, A_n)$ is strongly mixed.
	Let $\boldsymbol \xi \in \thirdrev{\mathfrak R}$. By definition,
$\boldsymbol \xi$ does not belong to the tropical prevariety
of $(A_1, \dots, A_n)$. Hence there is $1 \le i \le n$ so that
$\boldsymbol \xi$ does not belong to tropical hypersuface 
associated to $A_i$. This means that $\max_{\mathbf a \in A_i}
\mathbf a \boldsymbol \xi$ is attained for a unique value of $\mathbf a$,
hence $\# A_i^{\boldsymbol \xi}=1$.

Reciprocally, assume that the tuple $(A_1, \dots, A_n)$ is
not strongly mixed. 
Then there is $\mathbf 0 \ne \mathbf x \in \mathbb R^n$ belonging
to the tropical prevariety of $(A_1, \dots, A_n)$. For all
$i$, $\mathbf x$ belongs to the tropical hypersurface of $A_i$.
This means that $\# A_i^{\mathbf x} \ge 2$ for all $i$. Let 
$B_i = A_i^{\mathbf x}$, by definition $\mathbf x \in C(B_1,
\dots, B_n)$. In particular, the closed cone
	$\bar C(B_1, \dots, B_n) \ne \{\mathbf 0\}$ is a cone of the fan
$\mathfrak F$. Let
	$c$ be a facet of 
	$\bar C(B_1, \dots, B_n)$ of minimal dimension, other than the
	facet $\{\mathbf 0\}$. Then $c$ has 
	dimension $1$. By Definition~\ref{defan},
	$c$ belongs to the fan $\mathfrak F$, so
$c \in \mathfrak F_0$ is of the form
	$\{t \boldsymbol \xi : t>0\}$ with $\boldsymbol \xi \in \thirdrev{\mathfrak R}$.
For all $i$, we have by construction that
$B_i \subseteq A_i^{\boldsymbol \xi}$. Thus, we found 
	$\boldsymbol \xi \in \mathfrak R$
	with the property that $\# A_i^{\boldsymbol \xi} > 1$
	for all $i$. Therefore, $\neg (a) \Rightarrow \neg (b)$.}

To show that $(b)$ implies $(c)$, we choose for each $\boldsymbol \xi
	\in \thirdrev{\mathfrak R}$ a pair $(i, \mathbf a^*)$ with $A_i^{\boldsymbol \xi} =
	\{ \mathbf a^* \}$. \secrev{Suppose by contradiction
	that $\mathbf q \in \Sigma^{\infty}$ 
	is such that 
	} $q_{i \mathbf a^*} \ne 0$
for all those pairs. In particular, for all 
	$\changed{\boldsymbol \xi} \in \thirdrev{\mathfrak R}$, $\mathbf q \not \in \Sigma^{\boldsymbol \xi}$.
We claim that $\mathbf q \in \Sigma^{\mathbf x}$ for some $\mathbf x$.
Indeed, there is some $[\mathbf V] \in \mathscr V \setminus [\mathbf V(\secrev{\mathscr M})]$
	with $\mathbf q_i \cdot 
	V_{A_i} = \mathbf 0$ for all $i$. 
	\secrev{There is also} a path \changed{$(\mathbf z(t))_{t \in (0,1]} \subseteq \secrev{\mathscr M}$} with $\changed{\lim_{t \rightarrow 0}} [\mathbf V(\mathbf z(t))] = [\mathbf V]$. 
By compacity of the sphere $S^{n-1}$, there is
an accumulation point 
$\mathbf x \in S^{n-1}$ of 
$\frac{ \Re(\mathbf z(t))}{\|\Re(\mathbf z(t))\|}$
	\changed{for $t \rightarrow 0$}, and
$[\mathbf V] \in \Sigma^{\mathbf x}$. Let $k$ be minimal so
	that $\mathbf x$ \changed{belongs} to a \secrev{closed cone $\bar c$} in $\mathfrak F_k$. \changed{Suppose also by contradiction that $k\ne 0$.}
	If $[\mathbf V] \in \Sigma^{\mathbf x}$, $[\mathbf V] \in \Sigma^{\boldsymbol \xi}$ for any $\boldsymbol \xi \in \secrev{\bar c}$.
In particular, this holds for $\mathbf 0 \ne \boldsymbol \xi \in \partial \bar c$, 
\changed{and $k \ne 0$ cannot possibly be minimal. Therefore, $k=0$ and \secrev{$\mathbf x$} \thirdrev{spans a ray of $\mathfrak F_0$, contradiction.
Thus, $(b) \Rightarrow (c)$.}}

Now suppose that \secrev{Condition} (b) does not hold for a certain 
$\boldsymbol \xi \thirdrev{\in \mathscr R}$. \secrev{This means that
$\# A_i^{\mathbf \xi} \ge 2$ for all $i$.}
	We \secrev{can choose $\mathbf q$ with} coefficients $q_{i \mathbf a} \changed{\ne 0}$ so that
\[
	\sum_{\mathbf a \in A_i^{\boldsymbol \xi}}\rho_{i\mathbf a} q_{i \mathbf a} = 0, \ i=1, \dots, n.
\]
	The value of the \changed{remaining $q_{i\mathbf a'}$,
	$\mathbf a'\not \in A_i^{\boldsymbol \xi}$,} is irrelevant. 
	Also, let $[\mathbf V]= \lim_{t \rightarrow \infty} [\mathbf V(t \boldsymbol \xi)]$. 
	We have
\[
	\mathbf q_i \frac{1}{\|V_{i}\|} V_{A_i} = 0.
\]
	Thus, \secrev{$\mathbf q \in \Sigma^{\infty}$} and $\neg (b)$ implies $\neg (c)$.
\end{proof}

\subsection{The variety of systems with solutions at toric infinity}

\secrev{In the previous section, we represented a ray by a unit vector $\boldsymbol \xi \in S^{n-1}$. In this section, we assume instead that a ray is always represented by an integer vector} $\mathbf 0 \ne \boldsymbol \xi \in \mathbb Z^n$ 
\secrev{with}
$\gcd(\xi_i)=1$.
\secrev{As before,} $\Sigma^{\boldsymbol \xi}$ \secrev{denotes} the Zariski closure of the
set of all 
$\mathbf q \in \mathscr F$ with a root at infinity in the
direction $\boldsymbol \xi$. This means that the overdetermined
system
\[
	\sum_{\mathbf a \in A_i^{\boldsymbol \xi}} q_{i\mathbf a} V_{i,\mathbf a}(\mathbf z)
\]
has a common root in $\boldsymbol \xi^{\perp} \subseteq \mathbb C^n$, possibly
\changed{`at infinity'}. More formally, for each $i$ we can write 
$\mathscr F_{A_i} = \mathscr F_{A_i^{\boldsymbol \xi}} \times \mathscr F_{A_i \setminus A_i^{\boldsymbol \xi}}$.
Let $\Sigma_0^{\boldsymbol \xi}$ be the subvariety of systems in
$\mathscr F_{A_1^{\boldsymbol \xi}} \times \cdots \times \mathscr F_{A_n^{\boldsymbol \xi}}$ that are
solvable in $\boldsymbol \xi^{\perp}$. Then
\[
\Sigma^{\boldsymbol \xi} = \Sigma_0^{\boldsymbol \xi} \times  \mathscr F_{A_1 \setminus A_1^{\boldsymbol \xi}}
\times \cdots \times  \mathscr F_{A_n \setminus A_n^{\boldsymbol \xi}}
.
\]
It follows that
\[
	\codim (\Sigma^{\boldsymbol \xi}) = 	\codim (\Sigma_0^{\boldsymbol \xi}) 
\]
both as subvarieties of a linear space or as subvarieties of a projective space.

\ocite{Pedersen-Sturmfels} and \ocite{Sturmfels}*{Lemma 1.1} proved that the 
closure of the locus
of sparse overdetermined systems with a common root is an irreducible
variety in the product of the projectivizations of the coefficient spaces.
This variety is defined over the rationals. In the setting of this paper,
this implies that whenever $\boldsymbol \xi$ \thirdrev{spans a ray of} $\mathfrak F_0$, 
$\Sigma_0^{\boldsymbol \xi}$ is an irreducible variety in
\[
	\mathbb P( \mathbb C^{\#A_1^{\boldsymbol \xi}}) 
	\times \cdots \times
	\mathbb P( \mathbb C^{\#A_n^{\boldsymbol \xi}}) 
\]
with rational coefficients. The same is trivially true for $\Sigma^{\boldsymbol \xi}$.

\changed{When $\Sigma^{\boldsymbol \xi}$ has codimension one, its defining
polynomial is known as the {\em mixed resultant}, and its degree is known
\cite{GKZ}*{Proposition 1.6}. But in order to apply this bound, one needs to
check condition (1) p.252 in \cite{GKZ} which assures codimension 1. 
In our setting, this condition 
requires each of the $A_i^{\boldsymbol \xi}$ to \secrev{span an $n-1$-dimensional affine
space}.  
The original definition by 
\ocite{Sturmfels} avoids this restriction, but sets the mixed resultant
equal to one if the codimension is more than one.}

\changed{We shall see in example ~\ref{ex-resultant} below
that $\Sigma^{\boldsymbol \xi}$ may have arbitrary codimension between $1$
and $n-1$. Therefore,
the locus $\Sigma^{\infty}$ of systems with a root at toric infinity may
have components of codimension $1$ to $n-1$.}

\changed{This is why we need to proceed differently in this paper.} 
We do not need to find
the sparse resultant ideal $I(\Sigma^{\boldsymbol \xi})$ but just
a non-zero polynomial in it. If there is $i$ with $\#A_i^{\boldsymbol \xi} = 1$ for instance, the variety $\Sigma^{\boldsymbol \xi}$ is contained
into the hyperplane $q_{i\mathbf a} = 0$. This will be
enough to prove Theorem
~\ref{degenerate-locus} and to derive probabilistic complexity bounds.
Item (b) of the Theorem follows trivially from Lemma~\ref{tropical}(c). 
We assume from now on the
hypothesis $\#A_i^{\boldsymbol \xi} \ge 2$.

If $I \subseteq [n]$, then we denote by
$\Lambda_I$ the lattice spanned by $\bigcup_{i \in I}(A_i - A_i)$, and by $\Lambda_I^{\boldsymbol \xi}$ the lattice spanned by $\bigcup_{i \in I}(A_i^{\boldsymbol \xi} - A_i^{\boldsymbol \xi})$. 
The variety \changed{$\Sigma_0^{\boldsymbol \xi}$} is the variety of {\em solvable systems} with support $(A_i^{\boldsymbol \xi})_{i=1, \dots, n}$.

\ocite{Sturmfels}*{Theorem 1.1} computed the codimension of the variety of solvable systems.
In the particular case of $\Sigma_0^{\boldsymbol \xi}$, his bound reads:

\begin{theorem}\cite{Sturmfels}\label{Sturmfels}
\[
\mathrm{codim}(\Sigma_0^{\boldsymbol \xi}) = 
\max_{I \subseteq [n]} \left( \#I - \mathrm{rank}( \Lambda_I^{\boldsymbol \xi} ) \right)
.
\]
\end{theorem}

\begin{corollary}
\[
\mathrm{codim}(\Sigma^{\boldsymbol \xi}) = 
\max_{I \subseteq [n]} \left( \#I - \mathrm{rank}( \Lambda_I^{\boldsymbol \xi} ) \right)
.
\]
\end{corollary}

From this result it is easy to construct an example of supports with non-zero mixed volume
and
a variety $\Sigma^{\boldsymbol \xi}$ of large codimension. 

\begin{example}\label{ex-resultant}
Choose $\boldsymbol \xi = -\mathrm e_n$.
Let $A_n$ be the hypercube 
$A_n = \{ \mathbf a: a_i \in \{0, 1\}, i=1, \dots, n \}$.
Let $0 \le k < n$ be arbitrary,
and let $\Delta_k$ be the $k$-dimensional simplex,
\[
	\Delta_k =\{\mathbf 0, \mathrm e_1, \dots, \mathrm e_k\} \secrev{\ \subseteq \mathbb R^n }
\]
and set
\[
	A_1 = \dots = A_{n-1} = \Delta_k \cup (\mathrm e_n + A_n).
\]
	For all non-empty $I \subseteq [n]$, $\mathrm{rank} \Lambda_I^{\boldsymbol \xi} = k$
	and therefore the maximum \thirdrev{from Theorem~\ref{Sturmfels}} is attained for $I \thirdrev{=} [n]$.
	\thirdrev{Hence,} the codimension of $\Sigma^{\boldsymbol \xi}$ is precisely $n-k$,
which ranges between $1$ and $n-1$.
\end{example}

\begin{theorem}\label{sigma-xi-general}
Let $\Lambda$ be the lattice 
spanned by $\bigcup_i A_i - A_i$ 
	and let $\mathcal A = \conv{A_1} + \dots + \conv{A_n}$.
	Let $\boldsymbol \xi$ span a ray of $\mathfrak F_0$.
Then the variety 
	$\Sigma^{\boldsymbol \xi}$ is contained in some \secrev{hypersurface} of the form $Z(p)$, where $p$ is an irreducible polynomial
	of degree at most 
\[
	d \le \changed{\frac{\max_i \diam{\mathcal A_i}}{2 \det \Lambda}}
	\max_{0 \le k \le n-1} 
	(\thirdrev{e^k v_k})
,
\]
	\secrev{where the coefficients $v_k$ are defined 
	in equations \eqref{v-poly} and \eqref{v-k}.}

	In the strongly mixed case, $p$ \secrev{can be taken to be} linear.

	In the unmixed case $A_1 = A_2 = \dots = A_n$, or if there is
some $A$ with $A_1 = d_1 A$, \dots, $A_n = d_n A$, then $p$
is an irreducible polynomial of degree at most
\[
	d \le \changed{\frac{\max_i \diam{\mathcal A_i}}{2 \det \Lambda}}
	\secrev{V'}
,
\]
	\secrev{where $V'=v_0$ is the mixed area.}
\end{theorem}

\secrev{The proof of Theorem~\ref{sigma-xi-general} will use the
Lemma~\ref{mixed-elementary} below. We will need to extend
Definition~\ref{def-mixed} to compact convex subsets of a lower dimensional
space.
\begin{definition}\label{def-mixed2} 
Let $\mathcal A_1, \dots, \mathcal A_{\changed{j}} \subseteq \mathbb R^n$
be bounded convex sets with joint affine span $W$, and assume that
$W$ is $j$-dimensional. Let $\vol_W$ denote the ordinary volume in $W$.
	The {\em mixed volume} of 
$(\mathcal A_1,$ $\dots, \mathcal A_j)$ 
in $W$ is
\[
	\glsdisp{MixedVolume2}{V_W(\mathcal A_1, \dots, \mathcal A_j)}
\defeq
\frac{1}{j!}\ 
\frac{\partial^j}
{\partial t_1 \partial t_2 \cdots \partial t_j}
\
	\vol_W(t_1 \mathcal A_1 + \cdots + t_j \mathcal A_j)
\]
where $t_1, \dots, t_j \ge 0$ and the derivative is taken at
$t_1=\dots=t_j=0$. 
\end{definition}
}

\changed{
\begin{lemma}\label{mixed-elementary}
Let $\mathcal A_1, \dots, \mathcal A_{\changed{j}} \subseteq \mathbb R^n$
	be bounded convex sets with joint affine span $W$\secrev{, assumed
	to be $j$-dimensional}. Let
$\mathbf u_{{\changed{j}}+1}, \dots, \mathbf u_n$ be vectors in $W^{\perp}
\subseteq \mathbb R^n$ and let
$C=[\mathbf 0 \mathbf u_{{\changed{j}}+1}] + \dots + [\mathbf 0\mathbf u_n]$. Then,
\[
{\changed{j}}!\, V_W (\mathcal A_1, \dots, \mathcal A_{\changed{j}})
= 
	\frac{n!\, \secrev{V}(\mathcal A_1, \dots,
\mathcal A_{\changed{j}}, [\mathbf 0 \mathbf u_{{\changed{j}}+1}], \dots, [\mathbf 0 \mathbf u_n])}
	{\vol_{W^{\perp}}(C)}
\]
\end{lemma}
}

\changed{\begin{proof}[Proof of Lemma~\ref{mixed-elementary}]
From Definition \ref{def-mixed},
\begin{eqnarray*}
	V &\defeq & n!\, 
	\secrev{V}(\mathcal A_1, \dots,
	\mathcal A_{\changed{j}}, [\mathbf 0 \mathbf u_{{\changed{j}}+1}], \dots, [\mathbf 0 \mathbf u_n])
	\\&=&
\frac{\partial^n}
{\partial t_1 \partial t_2 \cdots \partial t_n}
\
\vol_n \left(t_1 \mathcal A_1 + \dots + 
	t_{\changed{j}}\mathcal A_{\changed{j}} + t_{{\changed{j}}+1} [\mathbf 0 \mathbf u_{{\changed{j}}+1}] + \dots + t_n [\mathbf 0 \mathbf u_n]\right)
\end{eqnarray*}
with $t_1, \dots, t_n \ge 0$ and the derivative taken at
$t_1=\dots=t_n=0$. By orthogonality,
\begin{eqnarray*}
V &=& 
\frac{\partial^n}
{\partial t_1 \cdots \partial t_n}
\
	\secrev{\vol_W} \left(t_1 \mathcal A_1 + \dots + 
	t_{\changed{j}}\mathcal A_{\changed{j}}\right) \\
	& & \hspace{6em} \secrev{\times} \secrev{\vol_{W^{\perp}}} (t_{{\changed{j}}+1} [\mathbf 0 \mathbf u_{{\changed{j}}+1}] + \dots + t_n [\mathbf 0 \mathbf u_n])))\\
&=& 
\frac{\partial^{\changed{j}}}
{\partial t_1 \cdots \partial t_{\changed{j}}}
\
	\left(\vol_W \left(t_1 \mathcal A_1 + \dots + 
	t_{\changed{j}}\mathcal A_{\changed{j}}\right) 
	\right) 
	\\ &&\hspace{6em}  \times \left(
	\frac{\partial^{n-{\changed{j}}}}
	{\partial t_{{\changed{j}}+1} \cdots \partial t_n}
\
	\vol_{W^{\perp}} \left(
	t_{{\changed{j}}+1} [\mathbf 0 \mathbf u_{{\changed{j}}+1}] + \dots + t_n [\mathbf 0 \mathbf u_n]\right)\right)\\
\\
	&=&
{\changed{j}}!\,V_W(\mathcal A_1, \dots, \mathcal A_{\changed{j}})
	\ \vol_{W^{\perp}}(C).
\end{eqnarray*}
\end{proof}}

\begin{proof}[\changed{Proof of Theorem~\ref{sigma-xi-general}}] 
Let $I$ be minimal such that
\[
	\#I - \mathrm{rank}( \Lambda_I^{\boldsymbol \xi} )=1 < \mathrm{codim} \Sigma^{\boldsymbol \xi}
.
\]
	In the unmixed case, we suppose that $\mathrm{rank}( \Lambda_I^{\boldsymbol \xi})=n-1$ for all $I$, so the only possible choice is $I=[n]$.
The strongly mixed case has minimal $I$ with $\#I = 1$ and $p$ is therefore linear.
We consider the general case now. In order to simplify notations, we reorder the supports so that
$I=[{\changed{j}}+1]$ where ${\changed{j}}=\mathrm{rank}( \Lambda_I^{\boldsymbol \xi} )$.
Suppose that $\mathbf q \in \Sigma^{\boldsymbol \xi}$. In that case,
there is $\mathbf z$ such that
\begin{equation}\label{system-full}
	\sum_{\mathbf a \in A_i^{\boldsymbol \xi}} q_{i \mathbf a} V_{i \mathbf a}(\mathbf z) = 0, \hspace{4em} 1 \le i \le n
\end{equation}
and in particular the subsystem 
\begin{equation}\label{system-small}
	\sum_{\mathbf a \in A_i^{\boldsymbol \xi}} q_{i \mathbf a} V_{i \mathbf a}(\mathbf z) = 0, \hspace{4em} 1 \le i \le {\changed{j}}+1 
\end{equation}
admits a solution $\mathbf z \in \mathbb C^n$. 
Let $Z$ be the complex linear span of $\Lambda_I^{\boldsymbol \xi}$.
Let $Z'$ be the smallest complex space containing $Z$ and $\boldsymbol \xi$.
For later usage, let $Z_{\mathbb R}$ be the real linear
span of $\Lambda_I^{\boldsymbol \xi}$ and let $Z_{\mathbb R}'$ be
the smallest real space containing $Z_{\mathbb R}$ and $\boldsymbol \xi$.
Without loss of generality, we can assume that $\mathbf z \in Z$.

	\secrev{It turns out that the set of systems 
$\mathbf q$ for which 
 \eqref{system-small} admits a solution $\mathbf z \in Z$ is a hypersurface
of degree $d$, related to some mixed resultant (See Remark~\ref{mixed-resultant} below). 
But it is more convenient to prodceed by 
direct computation: 
	the degree $d$ of this algebraic surface
	is by definition the generic number of intersections with a generic line, that is
	} the 
number of values $t \in \mathbb C$ for which
	the system \secrev{with generic coefficients} below
	admits a solution $\mathbf z \in Z$:
	\begin{equation}\label{aug-t}
t \sum_{\mathbf a \in A_i^{\boldsymbol \xi}} f_{i \mathbf a}
\rho_{i, \mathbf a} e^{\mathbf a  
\mathbf z}
+
\sum_{\mathbf a \in A_i^{\boldsymbol \xi}} g_{i \mathbf a}
\rho_{i, \mathbf a} e^{\mathbf a \mathbf z}
= 0,
\
1 \le i \le {\changed{j}}+1
.
	\end{equation}
	We \secrev{assumed} $\mathbf f$ and $\mathbf g$ generic, \secrev{hence} there are no
	solutions for $t=0$ or for $t=\infty$. \secrev{Recall that for $\mathbf a \in A_i^{\boldsymbol \xi}$, we have $\mathbf a \boldsymbol \xi = 
	\lambda_i(\boldsymbol \xi)$.} If we set 
	$t= e^{\|\boldsymbol \xi\|^{-2}s}$ \secrev{and divide by $e^{-\|\boldsymbol \xi\|^{-2}s\lambda_i(\boldsymbol \xi)}$} , we recover a solution
$\mathbf w = \mathbf z + \|\boldsymbol \xi\|^{-2} s \boldsymbol \xi \in Z'$ 
for the system
\begin{equation}\label{augmented}
\sum_{\mathbf a \in A_i^{\boldsymbol \xi}} f_{i \mathbf a}
\rho_{i, \mathbf a} e^{(\mathbf a + 
\boldsymbol \xi^T)\mathbf w}
+
\sum_{\mathbf a \in A_i^{\boldsymbol \xi}} g_{i \mathbf a}
\rho_{i, \mathbf a} e^{\mathbf a \mathbf w}
= 0,
	\hspace{1em}
1 \le i \le {\changed{j}}+1
.
\end{equation}
Let $\Lambda'$ be the lattice spanned by $A_1^{\boldsymbol \xi}-
A_1^{\boldsymbol \xi}, \dots, A_{{\changed{j}}+1}^{\boldsymbol \xi}-
A_{{\changed{j}}+1}^{\boldsymbol \xi}$, and $\{\boldsymbol \xi^T\}$. 
The dual lattice of $\Lambda'$ is $\left(\Lambda^{\boldsymbol \xi}_{[{\changed{j}}+1]}\right)^*
+ \left \{ \frac{k}{\|\boldsymbol \xi\|^2} \boldsymbol \xi^T: k \in \mathbb Z \right \}$.
For each solution of the system \eqref{augmented} in 
$\mathbb C^n \mod 2 \pi \sqrt{-1}\, (\Lambda')^*$ there is at most one 
complex value of $t\changed{=e^{\|\boldsymbol \xi\|^{-2} s \boldsymbol \xi}}$ for which 
the system \eqref{aug-t} has a solution \changed{$\mathbf z$} in $\mathbb C^n \mod 2 \pi \sqrt{-1} \left(\Lambda^{\boldsymbol \xi}_{[{\changed{j}}+1]}\right)^*$.

We scaled $\boldsymbol \xi$ so that $\boldsymbol \xi \in \mathbb Z^n$ 
and the system \eqref{augmented} is an exponential sum with
integer coefficients.
According to Theorem~\ref{BKK2}, its generic number of solutions 
in $Z'$ is precisely

\begin{equation}\label{value-of-D}
	\changed{d}
	= \changed{({\changed{j}}+1)} !\ \frac{	V_{Z_{\mathbb R}'}\left( 
	\conv{A_1^{\boldsymbol \xi}}+[\mathbf 0, 
\boldsymbol \xi^T],
	\dots, 
	\conv{A_{{\changed{j}}+1}^{\boldsymbol \xi}}+[\mathbf 0, 
\boldsymbol \xi^T]\right)
	}{
		\| \boldsymbol \xi\| \det \Lambda_{I}^{\boldsymbol \xi}}
\end{equation}
using the identity 
$\det \Lambda'= \| \boldsymbol \xi\| \det \Lambda_{I}^{\boldsymbol \xi}$.
In the formula above, $V_{Z_{\mathbb R}'}$ denotes the mixed volume 
restricted to the ${\changed{j}}+1$-dimensional space $Z_{\mathbb R}'$. 

We assumed that $\Lambda$ was an $n$-dimensional lattice, and that
$\boldsymbol \xi$ spanned a ray of $\mathfrak F_0$. These hypotheses imply that 
	$\boldsymbol \xi$ is orthogonal to \changed{the} sublattice $
\changed{\Lambda \cap \boldsymbol \xi^{\perp}} \subseteq \Lambda$ of rank $n-1$, and we have the inclusion
\[
\Lambda_I^{\boldsymbol \xi} 
\subseteq 
\Lambda \cap Z_{\mathbb R}
\subseteq 
\Lambda \cap \boldsymbol \xi^{\perp}
\subseteq \Lambda .
\]
The first inclusion is equidimensional.
There is a lattice basis $(\mathbf u_1, \dots, \mathbf u_n)$
of $\Lambda$ with the following properties:
$\mathbf u_1, \dots, \mathbf u_{\changed{j}} \in \Lambda \cap Z_{\mathbb R}$
and $\mathbf u_{{\changed{j}}+1}, \dots, \mathbf u_{n-1} \in \Lambda \cap 
\boldsymbol \xi^{\perp}$.

Let $C_k=[\mathbf 0 \mathbf u_{1}] + \dots + [\mathbf 0 \mathbf u_{k}]$. In particular,
$C_n$ is a fundamental domain of $\Lambda$ and $C_{n-1}$ is a
a fundamental domain of $\Lambda \cap \boldsymbol \xi^{\perp}$. We denote by
$\pi$ the orthogonal projection onto $(Z_{\mathbb R}')^{\perp}$.
In particular, $C'=\pi(C_{n-1})= [\mathbf 0, \pi(\mathbf u_{{\changed{j}}+1})] + \dots + [\mathbf 0, \pi(\mathbf u_{n-1})]$ is
a fundamental domain of $\pi(\Lambda) \subseteq (Z_{\mathbb R}')^{\perp}$. 
\changed{From the lattice inclusions,
\begin{eqnarray*}
	\det(\Lambda) &=& 
\det (\Lambda \cap \boldsymbol \xi^{\perp})\,
\frac{\min_{ \mathbf b \in \Lambda, \mathbf b \boldsymbol \xi \ne \mathbf 0} |\mathbf b \boldsymbol \xi|}{\| \boldsymbol \xi\|}
\\
&=&
	\det (\Lambda \cap Z_{\mathbb R}) \,\vol_{(Z_{\mathbb R}')^{\perp}}(C')
\, 
\frac{\min_{ \mathbf b \in \Lambda, \mathbf b \boldsymbol \xi \ne \mathbf 0} |\mathbf b \boldsymbol \xi|}{\| \boldsymbol \xi\|}
\\
&\le&
	\det (\Lambda_I^{\boldsymbol \xi})\,  \vol_{(Z_{\mathbb R}')^{\perp}}(C')
\, 
\frac{\min_{ \mathbf b \in \Lambda, \mathbf b \boldsymbol \xi \ne \mathbf 0} |\mathbf b \boldsymbol \xi|}{\| \boldsymbol \xi\|}
.
\end{eqnarray*}
For some $i$, there must be $\mathbf a$, $\mathbf a' \in A_i$
so that $(\mathbf a - \mathbf a')\boldsymbol \xi \ne \mathbf 0$. 
Therefore,
\secrev{ there is $\mathbf b \in \Lambda$ with $\mathbf b \boldsymbol \xi \ne \mathbf 0$
and the minimum above is attained, with} 
$\min_{\secrev{\mathbf b \in \Lambda,}\mathbf b \boldsymbol \xi \ne \mathbf 0} |\mathbf b \boldsymbol \xi|
\le \|\boldsymbol \xi\| \max_i \diam{\mathcal A_i}$}
Equation \eqref{value-of-D} implies the bound:
\begin{equation}\label{value-of-D2}
	\changed{
\begin{split}
	d \le & 
	\max_i(\diam{\mathcal A_i})
	\, \vol_{(Z_{\mathbb R}')^{\perp}}(C') \,
	\\
	& \hspace{2em}
	\frac{	
	({\changed{j}}+1) !\, 
	V_{Z_{\mathbf R}'}\left( 
	\conv{A_1^{\boldsymbol \xi}}+[\mathbf 0, 
\boldsymbol \xi^T],
	\dots, 
	\conv{A_{{\changed{j}}+1}^{\boldsymbol \xi}}+[\mathbf 0, 
\boldsymbol \xi^T]\right)
	}{
	\|\boldsymbol \xi\| \,	\det \Lambda}
\end{split}
}
\end{equation}
The domain $C'$ is orthogonal to $\changed{Z_{\mathbb R}'} \supseteq \conv{A_i^{\boldsymbol \xi}},
[\mathbf 0 \boldsymbol \xi]$. Lemma~\ref{mixed-elementary} implies that
\[
\changed{
	\begin{split}
d
		\le \max_i(\diam{\mathcal A_i})	&\\
		&\hspace{-3em}
\times \frac{n!\, V\left( 
\conv{A_1^{\boldsymbol \xi}}+[\mathbf 0, 
		\boldsymbol \xi^{T}],
		\dots, 
\conv{A_{{\changed{j}}+1}^{\boldsymbol \xi}}+[\mathbf 0, 
	\boldsymbol \xi^{T}], 
	C', \dots, C'\right)}{\|\boldsymbol \xi\| \,\det \Lambda}.\end{split} }
\]
To simplify notations we assume now that $\mathbf 0 \in A_i$ for each $i$.
The sum $\mathcal A= \conv{A_{1}} + \dots + \conv{A_n}$ is
$n$-dimensional, so it admits a linearly independent set
of vectors $\mathbf w_l = \sum_i \mathbf a_{il}$
with $\mathbf a_{il} \in A_i, 1 \le j \le n$. At least $n-{\changed{j}}-1$
of the $\pi(\mathbf w_l)$ are linearly independent, say those are
$\pi(\mathbf w_{{\changed{j}}+1}), \dots, \pi(\mathbf w_{n-1})$. 
Since that $C'$ is a fundamental domain of $\pi(\Lambda)$,
$\vol_{(Z_{\mathbb R}')^{\perp}}(C') \le \vol_{(Z_{\mathbb R}')^{\perp}}(\pi(C''))$, for
\[
	C'' = [\mathbf 0 \mathbf w_{{\changed{j}}+1}] + \dots + [\mathbf 0 \mathbf w_{n-1}]
\]
and hence
\[
\changed{
	\begin{split}
		d
		\le \max_i(\diam{\mathcal A_i})&\\  
		& \hspace{-3em}\times
\frac{n!\, 
V\left( \conv{A_1^{\boldsymbol \xi}}+[\mathbf 0, \boldsymbol \xi^T], \dots, \conv{A_{{\changed{j}}+1}^{\boldsymbol \xi}}+[\mathbf 0, \boldsymbol \xi^T], C'', \dots, C''\right)}
	{\|\boldsymbol \xi\| \,\det \Lambda}.\end{split} }
\]
\secrev{Using the multi-linearity property (see Remark~\ref{properties-mixed-volume}),
the mixed volume above can be expanded with respect to $[0, \boldsymbol \xi^T]$.
The `constant' term
\[
V\left( \conv{A_1^{\boldsymbol \xi}}, \dots, \conv{A_{{\changed{j}}+1}^{\boldsymbol \xi}}, C'', \dots, C''\right)
\]
vanishes because all of its argument lie in affine planes orthogonal
to $\boldsymbol \xi$, so the volume in Definition~\ref{def-mixed} is uniformly
zero. `Quadratic' terms like
\[
	V\left( [0, \boldsymbol \xi^T], [0, \boldsymbol \xi^T], \conv{A_3^{\boldsymbol \xi}}, \dots, \conv{A_{{\changed{j}}+1}^{\boldsymbol \xi}}, C'', \dots, C''\right)
\]
also vanish because the volume in Definition~\ref{def-mixed} has vanishing term in $t_1 t_2 \dots t_n$. `Higher order' terms vanish for the same reason, and we are left with the linear terms. We can rewrite the previous bound as
\[
\begin{split}
d
\le \frac{ \max_i(\diam{\mathcal A_i}) n!}{\det \Lambda}
\, \partialat{\epsilon}{0}\
V\left( \conv{A_1^{\boldsymbol \xi}}+\epsilon[\mathbf 0, \frac{1}{\|\boldsymbol \xi\|}\boldsymbol \xi^T], \dots,\right.\\\left.
\conv{A_{{\changed{j}}+1}^{\boldsymbol \xi}}+ \epsilon [\mathbf 0, \frac{1}{\|\boldsymbol \xi\|}\boldsymbol \xi^T], C'', \dots, C''\right)
\end{split}
\]
}

By convexity, the simplex with vertices $\mathbf 0$, $w_{{\changed{j}}+1}$, \dots,
$w_{n-1}$ is contained in $\mathcal A$. It follows that
$C'' \subseteq (n-{\changed{j}}-1) \mathcal A$ and thus \secrev{by monotonicity,}
\secrev{
\[
\begin{split}
d
\le \frac{ \max_i(\diam{\mathcal A_i}) n!  \thirdrev{(n-j-1)^{n-j-1}}
	}{\det \Lambda}
\, \partialat{\epsilon}{0}\
V\left( \conv{A_1^{\boldsymbol \xi}}+\epsilon[\mathbf 0, \frac{1}{\|\boldsymbol \xi\|}\boldsymbol \xi^T], \dots,\right.\\\left.
\conv{A_{{\changed{j}}+1}^{\boldsymbol \xi}}+ \epsilon [\mathbf 0, \frac{1}{\|\boldsymbol \xi\|}\boldsymbol \xi^T], \mathcal A, \dots, \mathcal A\right).
\end{split}
\]
}
\secrev{Multi-linearity and monotonicity can be combined to replace
\secrev{$\conv{A_i^{\boldsymbol \xi}}$ by $\mathcal A_i \defeq \conv{A_i}$} and
$[\mathbf 0 \frac{1}{\|\boldsymbol \xi\|}\boldsymbol \xi^T]$ by $\frac{1}{2} B^n$.}
In the particular case ${\changed{j}}=n-1$, 
\[
\changed{
d
\le 
\frac{n! \, \max_i(\diam{\mathcal A_i})}  
{2 \det \Lambda}
\, \partialat{\epsilon}{0}\
V( 
\secrev{\mathcal A_1}+\epsilon B^n,
		\dots, 
\secrev{\mathcal A_{n}}+\epsilon B^n )
. }
\]
In the general case ${\changed{j}}<n-1$, let $k=n-j-1 > 0$,
\secrev{
\begin{eqnarray*}
d
	&\le&
\frac{n! \, \max_i(\diam{\mathcal A_i})}  
{2 \det \Lambda}
	\thirdrev{k^k}
\, \partialat{\epsilon}{0}\
V( 
\secrev{\mathcal A_1}+\epsilon B^n,
		\dots, 
\secrev{\mathcal A_{j+1}}+\epsilon B^n, \mathcal A, \dots, \mathcal A )
\\
&\le&
\frac{n! \, \max_i(\diam{\mathcal A_i})}  
{2 \det \Lambda}
	\frac{	\thirdrev{k^k}}{k!}
\, \partialat{\epsilon}{0}\
\partialhigh{k}{t}{0}\
V( 
\secrev{\mathcal A_1}+\epsilon B^n,
		\dots, 
\secrev{\mathcal A_{j+1}}+\epsilon B^n, \\
& & \hspace{25em}t \mathcal A, \dots, t \mathcal A )
\\
&\le&
\frac{n! \, \max_i(\diam{\mathcal A_i})}  
{2 \det \Lambda}
	\frac{	\thirdrev{k^k}}{k!}
\, \partialat{\epsilon}{0}\
\partialhigh{k}{t}{0}\
V( 
\secrev{\mathcal A_1}+\epsilon B^n + t \mathcal A,
		\dots,\\
& & \hspace{23em}
	\secrev{\mathcal A_{n}}+\epsilon B^n + t \mathcal A )
	,
\end{eqnarray*}
the last step by monotonicity. Since $k \ge 1$, we can 
using Stirling's formula to bound
 $k^k \le \frac{1}{\sqrt{2 \pi k}} e^k \le e^k$.

In both cases, using $k=n-{\changed{j}}-1$, we have
\[
d \le 
\frac{\max_i(\diam{\mathcal A_i})}{2 \det \Lambda} \thirdrev{e^k} v_k
.
\]}
\end{proof}

\secrev{
\begin{remark}\label{mixed-resultant} The hypersurface of all $\mathbf q$ so
		that  \eqref{system-small} admits a solution can be constructed as follows:
Let $\Omega$ be a $n \times j$ matrix whose columns are a basis for the lattice dual
	$(\Lambda_I^{\boldsymbol \xi})^*$. We can change variables so that $\mathbf z = 
	\Omega \mathbf y$ for $\mathbf y \in \mathbb C^j$. Equation \eqref{system-small}
is equivalent to
\begin{equation}\label{system-small-y}
	\sum_{\mathbf b \in A_i^{\boldsymbol \xi} \Omega} 
	q_{i \mathbf a} \rho_{i,\mathbf a} e^{\mathbf b \mathbf y}
 = 0, \hspace{4em} 1 \le i \le {\changed{j}}+1
.
\end{equation}
	Recall that we represent points of lattices $\Lambda_I$ and 
	$\Lambda_I^{\boldsymbol \xi}$ as row
	vectors. 
	With that notation, the lattice spanned by  $\bigcup_i A_i^{\boldsymbol \xi}\Omega - A_i^{\boldsymbol \xi}\Omega$ is $\Lambda_I^{\boldsymbol \xi} \Omega$. 
	Since $\Lambda_I^{\boldsymbol \xi}$ and $(\Lambda_I^{\boldsymbol \xi})^*$ are
	dual to each other, the canonical basis of $\mathbb R^j$ is a lattice basis
	for $\bigcup_i A_i^{\boldsymbol \xi}\Omega - A_i^{\boldsymbol \xi}\Omega$.
	By the momentum action, that is by adding an appropriate constant vector to each
	$A_i^{\boldsymbol \xi}$, one can assume without loss of generality that 
	$\mathbf 0 \in A_i^{\boldsymbol \xi}\Omega$ and hence that $A_i^{\boldsymbol \xi}\Omega
	\subseteq \mathbb Z^j$.
	System \eqref{system-small-y} admits a solution if and only if the sparse
	resultant 
	for \eqref{system-small-y}
	vanishes. Therefore, \eqref{system-small} admits a solution $\mathbf z \in Z$ 
	if and only if, $\mathbf q$ belongs to a certain algebraic hypersurface of degree
\[
	d={\changed{j}}! \sum_{
		\substack{K \subseteq[{\changed{j}}+1]\\\#K = {\changed{j}}}} 
	V_{\mathbb R^{\changed{j}}} (\conv{A_{K_1}^{\boldsymbol \xi}\Omega^{\dagger}}, 
\dots,
	\conv{A_{K_{\changed{j}}}^{\boldsymbol \xi}\Omega^{\dagger}})
\]
\end{remark}
}

\begin{proof}[Proof of Theorem \ref{degenerate-locus}]
	\thirdrev{Recall that $\mathfrak R$ contains one representative
	$\boldsymbol \xi$ for every ray in $\mathfrak F_0$.}
	We claim that
\[
	\Sigma^{\infty} = \bigcup_{\boldsymbol \xi \in \thirdrev{\mathfrak R}}
		\Sigma^{\boldsymbol \xi}
.
\]
Indeed, let $\mathbf q \in \Sigma^{\infty}$. Let $[\mathbf W] \in \mathcal V$
be the root at infinity, so $\mathbf q \cdot \mathbf W = \mathbf 0$. Let 
	$B_i = \{ \mathbf a \in A_i: W_{i\mathbf a} \ne 0\}$.
There must exist $\mathbf x \in S^{n-1}$ with $B_i \subseteq
	A_i^{\mathbf x}$, for otherwise there would exist some $\mathbf z \in \secrev{\mathscr M}$ with $[\mathbf W] = [\mathbf V(\mathbf z)]$.
The closed cone
$\bar C=\bar C(A_1^{\mathbf x}, \dots, A_n^{\mathbf x})$ is a polyhedral cone
	with \secrev{extremal rays} in $\mathfrak F_0$.
\changed{For each vertex $\boldsymbol \xi$ of $\bar C$,
$B_i \subseteq A_i^{\mathbf \xi} \subseteq A_i^{\boldsymbol \xi}$
	so that $\mathbf q \in \Sigma^{\boldsymbol \xi}$.}
Thus, $\Sigma^{\infty}$ is a finite union of $\Sigma^{\boldsymbol \xi}$
	for $\boldsymbol \xi$ in a subset of \thirdrev{$\mathfrak R$}.

Item (a) follows directly from 
Proposition \ref{sigma-xi-general}. The particular case in
Item (b) was already proved in Lemma~\ref{tropical}(c). Let's prove item (c) now.

Let $r: \mathscr F \rightarrow \mathbb C$ be the polynomial of
item (a). We produce now two real polynomials $p$ and $q$ depending
on the real and on the imaginary parts of $\mathbf g \in \mathscr F$,
and on a real parameter $t$:
\begin{eqnarray*}
p(\Re(\mathbf g), \Im(\mathbf g), t) &=& \Re( r(\mathbf g + t \mathbf f))
\\
q(\Re(\mathbf g), \Im(\mathbf g), t) &=& \Im( r(\mathbf g + t \mathbf f))
.
\end{eqnarray*}
	We also write $p(t)=p(\Re(\mathbf g), \Im(\mathbf g), t)=\sum_{i=0}^{d_r} p_i t^i$ and similarly for $q$. The coefficient $p_i$ (resp. $q_i$) is a polynomial of degree at most $d_r-i$ on the real and imaginary parts of $\mathbf g$. We assumed that $r(\mathbf f) \ne 0$, therefore the leading term
$t^{d^r} r(\mathbf f)$ of $r(\mathbf g + t \mathbf f)$ does not vanish, and
$r(\mathbf g + t \mathbf f)$ has degree $d_r$ in $t$.
We do not know the degree of $p(t)$ and $q(t)$.
We can nevertheless assume that $\deg p(t)=a$ and $\deg q(t)=b$ with
$\max(a, b)= d_r$. The Sylvester resultant is
\[
s(\mathbf g)\defeq
S(\Re(\mathbf g), \Im(\mathbf g)) = 
\det
	\left(\rule[-.3\baselineskip]{0pt}{5.5\baselineskip} \right.
	\raisebox{1ex}{$\overbrace{
	\begin{matrix} 
	p_0     &       &	\\
	p_1	&\ddots	&	\\
	\vdots	&\ddots	& p_0	\\
	\vdots	&	& p_1	\\
	p_a     &       &\vdots \\
		&\ddots	&\vdots	\\
		&	& p_a    
	\end{matrix}}^{\text{$b$ columns}}
	\hspace{1em}
	\overbrace{
	\begin{matrix} 
	q_0 	&       &       & \\
	q_1	&\ddots &       & \\
	\vdots	&\ddots &\ddots &\\
	q_b     &	&\ddots & q_0 \\
	        &\ddots &       & q_1 \\
 	        &       &\ddots&\vdots \\
                &	&      & q_b 
	\end{matrix}}^{\text{$a$ columns}}$}
	\left.\rule[-.3\baselineskip]{0pt}{5.5\baselineskip} \right)
\]
which has clearly degree $ab \le d_r^2$.
If $r(\mathbf g + t \mathbf f)\changed{=0}$ for a real, possibly infinite $t$, 
then $p(t)$ and $q(t)$ have a common real root. 
In particular, $s(\mathbf g)=0$.  
The situation $s \equiv 0$ cannot arise because $r(\mathbf f) \ne 0$.

\end{proof}
Notice also that the Sylvester resultant \changed{vanishes also} when $p(t)$ and $q(t)$ have
a common non-real root, which does not correspond to a  
$\mathbf g + t \mathbf f \in \Sigma^{\infty}$, $t \in \mathbb R$.

\subsection{Probabilistic estimates}
We start with a trivial \secrev{lemma}: 
\begin{lemma}\label{probest-lemma}  Let $c > 0$ and $S \in \mathbb N$ be arbitrary.
\[
	\probability{\mathbf x \sim \secrev{\mathcal N}(\mathbf 0,I; \mathbb C^S)}{\|x\| \le c \sqrt{S}} \le
	\frac{e^{S (1 + 2 \log(c))}}{\sqrt{2\pi S}}
.
\]
\end{lemma}
\begin{proof}
\begin{eqnarray*}
\probability{\mathbf x \sim \secrev{\mathcal N}(\mathbf 0,I; \mathbb C^S)}{\|x\| \le c \sqrt{S}} &=&
	\pi^{-S}\int_{B(\mathbf 0,c\sqrt{S};\mathbb C^S)} e^{-\| \mathbf x\|^2} \dd \mathbb C^S(\mathbf x) \\
	&\le&
	\pi^{-S} c^{2S} S^S \vol_{2S} B^{2S} \\
        &=&
	\frac{c^{2S} S^S}{S!}
        \\
	&\le&
	\frac{e^{S (1 + 2 \log(c))}}{\sqrt{2\pi S}}
\end{eqnarray*}
	using Stirling's approximation $S! \ge \sqrt{2\pi} S^{S+\frac{1}{2}}e^{-S}$.
\end{proof}

\secrev{In order to prove Theorem~\ref{vol-omega}, we introduce now a conic condition number.
We start by fixing $f \in \mathscr F$ with $r(\mathbf f) \ne 0$, where $r$ is the polynomial
from Theorem~\ref{degenerate-locus}(a). Let $s$ be the polynomial from Theorem~\ref{degenerate-locus}(c), and recall that its degree is at most $d_r^2$. The conic condition number 
is defined for $s(\mathbf g) \ne 0$ by}
\[
	\mathscr C(\mathbf g) \defeq 
\frac{ \|\mathbf g \| }{ \inf_{\mathbf h \in \mathscr F:s(\mathbf g + \mathbf h) = 0} \| \mathbf h \|}
=
\frac{ \sqrt{\sum_{i=1}^N\|\mathbf g_i \|^2} }{ \inf_{\mathbf h \in \mathscr F:s(\mathbf g + \mathbf h) = 0} \sqrt{\sum_{i=1}^N \| \mathbf h_i \|^2}}
.
\]
We stress that this is the only part of this paper where we forsake the complex multi-projective structure (and invariance) in $\mathscr F$.
At this point, we look at the zero-set of $s$ as a real algebraic variety in $\mathbb R^{2S}$ or in $S^{2S-1}$. According to
\ocite{BC}*{Theorem 21.1}, 
\begin{equation}\label{th21-1}
	\probability{ \mathbf g \sim \secrev{\mathcal N}(\mathbf 0,I; \mathscr F) }{ \mathscr C(\mathbf g) \ge \epsilon^{-1} } \le 
	\thirdrev{4} e d_r^2 \thirdrev{(2S-1)} \epsilon 
\end{equation}
for $\epsilon^{-1} \thirdrev{\ge} (2d_r^2+1)\thirdrev{(S-1)}$.

\begin{proof}[Proof of Theorem~\ref{vol-omega}]
	\secrev{Assume that } $\mathbf g \in \Omega_H$. 
	In particular there are $t \in [0,T]$, $\mathbf z \in Z(\mathbf g + t \mathbf f)$ with
	$\| \Re(\mathbf z) \|_{\infty} \ge H$. From Theorem~\ref{cond-num-infty}, there is $\mathbf h$ such that
        $\mathbf g + \mathbf h + t \mathbf f \in \Sigma^{\infty} \subseteq Z(r)$, with

\[
	\frac{\|\mathbf h_i \|}{\|\mathbf g_i + t \mathbf f_i\|} \le \kappa_{\rho_i} e^{-\eta_i H/n} \hspace{4em} i=1, \dots, n
.
\]
In particular,
\[
	\|\mathbf h\| \le \| \mathbf g + t \mathbf f\| \max_i (\kappa_{\rho_i}) e^{-\eta H/n}
.
\]
	\secrev{
Assume furthermore
		that $\|\mathbf g \|>c \sqrt{S}$ where $c$ is a number to be determined.}
We can bound $\| \mathbf g + t \mathbf f\| \le 
	\| \mathbf g \| + T \| \mathbf f\| \le 
	\|\mathbf g \| \left(1 + \frac{T \|\mathbf f\|}{c \sqrt{S}}\right)$. Therefore,
\[
	\mathscr C(\mathbf g)^{-1} \le \frac{\|\mathbf h\|}{\|\mathbf g \|} \le \epsilon \defeq 
	\left(1 + \frac{T \|\mathbf f \|}{c \sqrt{S}}\right) \max_i \kappa_{\rho_i} e^{-\eta H/n}
\]
	We can now pass to probabilities: \secrev{for $\epsilon$ as above},
\begin{eqnarray*}
	\probability{\mathbf g \sim \secrev{\mathcal N}(\mathbf 0,I; \mathscr F)}{\mathbf g \in \Omega_H}
	&\le&
	\probability{\mathbf g \sim \secrev{\mathcal N}(\mathbf 0,I; \mathscr F)}{\|\mathbf g\| \le c \sqrt{S}}
	+
	\probability{\mathbf g \sim \secrev{\mathcal N}(\mathbf 0,I; \mathscr F)}{\mathscr C(\mathbf g)^{-1} > \epsilon^{-1}} \\
	&\le&
	\frac{e^{S (1 + 2 \log(c))}}{\sqrt{2\pi S}}
+
	8 e d_r^2 S \epsilon
\\
\end{eqnarray*}
	\secrev{using Lemma~\ref{probest-lemma} and \eqref{th21-1}. Replacing
	$\epsilon$ by its value,}
\[
		\probability{\mathbf g \sim \secrev{\mathcal N}(\mathbf 0,I; \mathscr F)}{\mathbf g \in \Omega_H}
\le
	\frac{e^{S (1 + 2 \log(c))}}{\sqrt{2\pi S}}
+
	8 e d_r^2 S 
	\left(1 + \frac{T \|\mathbf f\|}{c \sqrt{S}}\right) \max_i
	\kappa_{\rho_i} e^{-\eta H/n}
\]
We want this expression smaller that an arbitrary $\delta > 0$, A non-optimal solution is
to set $c=e^{\frac{\log \delta}{2S}-\frac{1}{2}} = \frac{ \sqrt[2S]{\delta}}{\sqrt{e}}$. This guarantees that
\[
	\frac{e^{S (1 + 2 \log(c))}}{\sqrt{2\pi S}} \le \frac{\delta}{2}
.
\]
Then we can set 
\[
	H \ge \frac{n}{\eta} \log\left( 
	16 e \delta^{-1} d_r^2 S \max_i (\kappa_{\rho_i})
	\left(1 + \frac{T \|\mathbf f\|}{ \sqrt[2S]{\delta} \sqrt{S}}\sqrt{e}\right) \right)
.
\]
	Because the $\rho_{i\mathbf a} = 1$ are constant, we can replace
	$\max_i (\kappa_{\rho_i})$ by $\thirdrev{\sqrt{\max_i (S_i)}}$ \changed{inside the logarithm}.
\end{proof}

\section{Analysis of linear homotopy}
\label{sec-linear}

\subsection{Overview}\label{overview-linear}
The proof of Theorem~\ref{mainD} is long. 
Recall \secrev{from Section \ref{sec:unconditional}} that $\mathbf q_{t} = \mathbf g + t \mathbf f$, where $\mathbf f$ is assumed `fixed' and $\mathbf g$ is assumed `Gaussian',
conditional to the event
$\mathbf g \not \in \Lambda_{\epsilon} \cup \Omega_H \cup Y_K$, where
\begin{eqnarray*}
	\Lambda_{\epsilon} &=&\left\{ \mathbf g \in \mathscr F: 
\ \exists \ 1 \le i \le n, \ \exists \mathbf a \in A_i, \ 
\left| \arg \left(\frac{g_{i{\mathbf a}}}{f_{i\mathbf a}}\right)\right| \ge \pi-\epsilon,
\right\} ,
\\
	\Omega_H &=& \Omega_{\mathbf f, T, H} =
\left\{
\mathbf g \in \mathscr F: \exists t \in [0,T], 
\exists \mathbf z \in Z(\mathbf g+t \mathbf f) \subseteq \secrev{\mathscr M}, 
	\|\Re(\mathbf z)\|_{\infty} \ge H \right \} \text{and}
\\
	Y_K &=& \left\{ \mathbf g \in \mathscr F: \exists i, \| \mathbf g_i \| \ge K \sqrt{S_i}\right\} .
\end{eqnarray*}

\secrev{In section~\ref{sec:unconditional}, it was assumed that each of the
first two sets had probability $\le 1/72$}. 
The event $\mathbf g \in  \Lambda_{\epsilon} \cup \Omega_H$
has therefore probability $\le 1/36$. 
By choosing $K=1+\sqrt{\frac{\log(n)+\log(10)}{\min_i S_i}}$, Lemma~\ref{lem-large-dev}(b) implies that 
the event $\|\mathbf g_i\| \ge \changed{K\sqrt{S_i}}$ has probability $\le \frac{1}{10n}$, hence
the event $\mathbf g \in Y_K$ has probability
$\le 1/10$.
Moreover, it is independent from the two other events.
Thus the event 
$\mathbf g \not \in  \Lambda_{\epsilon} \cup \Omega_H \cup Y_K$ has probability $\ge 7/8$.
Under this condition we need to estimate the expectation of the \secrev{renormalized} condition length
\secrev{of Definition~\ref{def-rencondlength}, that is
\[
\mathscr L((\mathbf q_{\tau}, \mathbf z_{\tau}); 0,T)
= \int_{0}^{T} 
\left(
\left\| \frac{\partial}{\partial \tau} \mathbf p_{\tau} \right\|_{\mathbf p_{\tau} }
+\secrev{\nu} \|\dot {\mathbf z}_{\tau}\|_{\mathbf 0} \right)
\mu( \mathbf p_{\tau}) \dd \tau
\]
	where $\mathbf p_{\tau}= \mathbf q_{\tau} \cdot R(\mathbf z_{\tau})$ is
the renormalized path, the norm of the first summand is the Fubini-Study
multi-projective metric and the norm in the second summand is the norm
of the inner product defined in Equation \eqref{metricM}, at $\mathbf x = \mathbf 0$.
}

\begin{theorem} \label{condlength} 
There is a constant $C_0$ with the following properties.
	\secrev{Assume the hypotheses and notations of Theorem~\ref{mainD}.}
	Let $T > 2K$, and set 
\[
\secrev{\mathrm{LOGS}_T} \defeq \log (d_r) +\log(S) +\log (T)
.
\]
%
	\secrev{Let $\kappa_{\mathbf f}$ be as defined in \eqref{kappaf}.
Take 
	\[
\mathbf g \sim N\left(\mathbf 0,I; \mathscr F
\conditional	
	\mathbf g \not \in \Lambda_{\epsilon} \cup \Omega_H \cup Y_K \right).
\]}
Then with probability $\ge 6/7$, 
\[
	\sum_{\mathbf z_{\tau} \in \mathscr Z(\mathbf q_{\tau})} \mathscr L (\mathbf q_t, \mathbf z_t; 0, T)
	\le C_0
	Q
	\changed{n^{2}\sqrt{S}}\max_i (S_i) 
	\ \thirdrev{\max \left(K, \sqrt{\kappa_{\mathbf f}}\right) }
	\ T 
	\ \secrev{\mathrm{LOGS}_T}
.
\]
\end{theorem}

\begin{corollary} If the conditional probability for $\mathbf g$ is replaced by
unconditional $\mathbf g \sim \secrev{\mathcal N}(\mathbf 0,I; \mathscr F)$, then the inequality above holds
	with probability $\ge \frac{6}{7} \frac{7}{8}= \frac{3}{4}$.
\end{corollary}

The first step toward the proof of Theorem~\ref{condlength} is
to break the condition length into two integrals,
\begin{eqnarray}
\label{L1}
\mathscr L_1((\mathbf q_{\tau}, \mathbf z_{\tau}); 0,T) &=& \int_{0}^{T} 
\left\| \dot{\mathbf q}_{\tau} \cdot R(\mathbf z_{\tau}) \right\|_{\mathbf q_{\tau} \cdot R(\mathbf z_{\tau})}
\mu( \mathbf q_{\tau} \cdot R(\mathbf z_{\tau}) , \mathbf 0) \dd \tau,
\\
\label{L2}
\mathscr L_2((\mathbf q_{\tau}, \mathbf z_{\tau}); 0,T) &=& 
2 \secrev{\nu} \int_{0}^{T} 
	\| \dot {\mathbf z}_{\tau} \|_{\mathbf 0}
\mu( \mathbf q_{\tau} \cdot R(\mathbf z_{\tau}) , \mathbf 0)
\dd \tau .
\end{eqnarray}

\begin{lemma}\label{cond-length-split} Assume that $\mathbf q_{\tau} \cdot \mathbf V(\mathbf z_{\tau}) \equiv \mathbf 0$
	and $\mathbf m_i(\mathbf 0)=\mathbf 0$ for all $i$. Then,
\begin{enumerate}[(a)] 
\item 
\[
	\left\| \frac{\partial}{\partial \tau} \left(\mathbf q_{\tau} \cdot R(\mathbf z_{\tau})\right) \right\|_{\mathbf q_{\tau} \cdot R(\mathbf z_{\tau})}
\le
		\left\| \left(\dot{\mathbf q}_{\tau} \cdot R(\mathbf z_{\tau})\right) \right\|_{\mathbf q_{\tau} \cdot R(\mathbf z_{\tau})}
+
		\secrev{\nu} \| \dot {\mathbf z}_{\tau}\|_{\mathbf 0}	
\]
\item
\[
\mathscr L((\mathbf q_{\tau}, \mathbf z_{\tau}); 0,T)
\le
\mathscr L_1((\mathbf q_{\tau}, \mathbf z_{\tau}); 0,T)
+
\mathscr L_2((\mathbf q_{\tau}, \mathbf z_{\tau}); 0,T)
.
\]
\end{enumerate}
\end{lemma}
\begin{proof}
	\secrev{
		Recall the notation $\mathbf p_{i \tau} \defeq
	\mathbf q_{\changed{i}\tau} \cdot R_{\changed{i}}(\mathbf z_{\tau})$.
Item (a) follows from the product rule and the triangle inequality:
	}

\secrev{
	\[
	\left\| \frac{\partial}{\partial \tau} \mathbf p_{\changed{i}\tau} \right\|_{\mathbf p_{\changed{i}\tau}}
	\le
	\left\| \dot{\mathbf q}_{\changed{i}\tau} \cdot R_{\changed{i}}(\mathbf z_{\tau})\right\|_{\mathbf p_{\changed{i}\tau}}
+
	\left\| {\mathbf p}_{\changed{i}\tau}  
	\cdot \diag{\mathbf a \dot {\mathbf z}_{\tau}}
	\right\|_{\mathbf p_{\changed{i}\tau}} 
.\]
The Fubini-Study norm is by definition
	\[
		\| \mathbf r \|_{\mathbf x} \defeq
		\frac{ \left\| P_{\mathbf x^{\perp}} (\mathbf r) \right\|}
		{\left\| \mathbf x \right\|}
		\le
		\frac{\|\mathbf r\|}{\|\mathbf x\|}
	\]
	where $P_{\mathbf x^{\perp}}$ is the orthogonal projection onto
	$\mathbf x^{\perp}$. By the triangle inequality again,
\[
	\left\| {\mathbf p}_{\changed{i}\tau}  
	\cdot \diag{\mathbf a \dot {\mathbf z}_{\tau}}
	\right\|_{\mathbf p_{\changed{i}\tau}} 
\le
	\frac{\left\| \mathbf p_{\changed{i}\tau}\right\|_{\mathscr F_{A_i}} 
	\max_{i,\mathbf a}\left(|\mathbf a \dot {\mathbf z}_{\tau}|\right)}
	{\left\|\mathbf p_{\changed{i}\tau} \right\|_{\mathscr F_{A_i}}} 
\le
	\max_{i,\mathbf a}\left(|\mathbf a \dot {\mathbf z}_{\tau}|\right)
.\]
Replacing in Equation \eqref{distortion} with $\nu_i = \nu_i(\mathbf 0)$ and $m_i(\mathbf 0)=\mathbf 0$, 
	the last term is at most 
	$\nu_i \|\dot {\mathbf z}_{\tau}\|_{i,\mathbf 0}$.
	Hence, }
\[
\changed{
\left\| \frac{\partial}{\partial \tau} 
	\left(\mathbf q_{\tau} \cdot R(\mathbf z_{\tau})\right) 
	\right\|_{\mathbf q_{\tau} \cdot R(\mathbf z_{\tau})}
\le
\left\| \dot{\mathbf q}_{\tau} \cdot R(\mathbf z_{\tau}) \right\|_{\mathbf q_{\tau} \cdot R(\mathbf z_{\tau})} 
+
	\secrev{\nu} \|\dot {\mathbf z}_{\tau}\|_{\mathbf 0}
	.}
\]
	Item (b) is now trivial.

\end{proof}

\subsection{Expectation of the condition length (part 1)}

The objective of this section is to bound the \secrev{expectation} of the integral 
\[
\mathscr L_1((\mathbf q_{\tau}, \mathbf z_{\tau}); 0,T) = \int_{0}^{T} 
\left\| {\dot{\mathbf q}}_{t} \cdot R(\mathbf z_{t}) \right\|_{\mathbf q_{t} \cdot R(\mathbf z_{t})}
\mu( \mathbf q_{t} \cdot R(\mathbf z_{t}) , \mathbf 0) \dd t.
\]

\begin{proposition}\label{prop-C1}\secrev{Let $Q$ be as defined in Equation \eqref{mainD-Q}.} Under the conditions of Theorem~\ref{condlength},
there is a constant $C_1$ such that 
\begin{eqnarray*}
\expected{\mathbf g \sim \secrev{\mathcal N}(\mathbf 0,I; \mathscr F)}
{
\sum_{\mathbf z_{\tau} \in \mathscr Z(\mathbf q_{\tau})}
	\mathscr L_1((\mathbf q_{\tau}, \mathbf z_{\tau}); 0,T)
\conditional \mathbf g \not \in \Lambda_{\epsilon} \cup \Omega_H \cup Y_K
} 
	&\le& \\
	&&
\hspace{-16em}\le\   
        C_1 Q 
	n^{\frac{1}{4}}
	\sqrt{1+\frac{K}{4}+\frac{\kappa_{\mathbf f}}{8}} \,
	S \max(S_i) 
	\ T \ 
	\secrev{\mathrm{LOGS}_T}
\end{eqnarray*}
\end{proposition}
We need first a preliminary result:
\begin{lemma}\label{worst-speed-proj}
	Let $\mathbf q_t \in \mathscr F$ be smooth at $t_0$, with $\|(\mathbf q_i)_{t_0}\| \ne 0$ for $i=1, \dots, n$. 
Let $\mathbf z \in \mathbb C^n$. Then for all $i$,
\[
	\| \dot {\mathbf q}_{it} \cdot R_i(\mathbf z) \|_{{\mathbf q}_{it} \cdot R_i(\mathbf z)}
\le
	\max_{\mathbf a \in A_i} \frac
	{|(\dot q_{i\mathbf a})_{t_0}|}
	{|(q_{i\mathbf a})_{t_0}|}
.
\]
In particular, if
$\mathbf q_t = t \mathbf f + \mathbf g \in \mathscr F$, then
\[
	\| \dot {\mathbf q}_{it} \cdot R_i(\mathbf z) \|_{{q}_{it} \cdot R_i(\mathbf z) }
	\le \max_{\mathbf a \in A_i} 
	\frac{1}{\left| \frac{g_{i\mathbf a}}{f_{i\mathbf a}} + t\right|}
\]
	and
\[
\| \dot {\mathbf q}_{t} \cdot R(\mathbf z) \|_{{\mathbf q}_{t} \cdot R(\mathbf z) }
	\le \sqrt{\sum_i \left( \max_{\mathbf a \in A_i} 
	\frac{1}{\left| \frac{g_{i\mathbf a}}{f_{i\mathbf a}} + t\right|}\right)^2
	}
	.
\]
\end{lemma}
\begin{proof}
Let $\mathbf u = (\mathbf q_{i})_{t_0} \cdot R_i(\mathbf z)$ and $\mathbf v= (\dot {\mathbf q}_{i})_{t_0} \cdot R_i(\mathbf z)$.
\changed{The projection of $\mathbf v$ onto $\mathbf u^{\perp}$ is
	$P_{\mathbf u^{\perp}}(\mathbf v) = \mathbf v - \frac{\langle \mathbf v, \mathbf u \rangle}{\| \mathbf u\|^2} \mathbf u$.}
We compute 
\changed{
	\begin{eqnarray*}
		\| \mathbf v\|_{\mathbf u}^2 &=& \frac{\|P_{\mathbf u^{\perp}}(\mathbf v)\|^2}{\|\mathbf u\|^2}
		\\
		&=&\frac{1}{\|\mathbf u\|^6} \left\| \| \mathbf u\|^2 \mathbf v - \langle \mathbf v,\mathbf u \rangle \mathbf u \right\|^2
\\
	&=&
	\frac{1}{\|\mathbf u\|^4} \left( \|\mathbf u\|^2 \|\mathbf v\|^2 - |\langle \mathbf u,\mathbf v \rangle|^2\right) 
	\end{eqnarray*}}
\changed{Let $w_{\mathbf a} = v_{\mathbf a} / u_{\mathbf a}$. Then,
\[
	\|\mathbf v\|_{\mathbf u}^2 = \frac{\sum_{\mathbf a, \mathbf b} 
	|u_{\mathbf a}|^2|u_{\mathbf b}|^2 w_{\mathbf a} (\bar w_{ \mathbf a} - \bar w_{ \mathbf b})}
{\sum_{\mathbf a, \mathbf b} 
	|u_{\mathbf a}|^2|u_{\mathbf b}|^2 }
=
	\frac{1}{2}
\frac{\sum_{\mathbf a, \mathbf b} 
	|u_{\mathbf a}|^2|u_{\mathbf b}|^2 |w_{ \mathbf a} - w_{ \mathbf b}|^2}
{\sum_{\mathbf a, \mathbf b} 
	|u_{\mathbf a}|^2|u_{\mathbf b}|^2 }
\]
	and so $\secrev{\|\mathbf v\|_{\mathbf u}} \le \max |w_{\mathbf a}|$}.
This proves the first part of the Lemma. The second part comes from taking absolute values of
the expression
\[
	\changed{	w_{\mathbf a}} = \frac{f_{i \mathbf a}}{g_{i \mathbf a}+t f_{i \mathbf a}} = \frac{1}{\frac{g_{i\mathbf a}}{f_{i\mathbf a}} + t}
.
\]
\end{proof}

Lemma~\ref{worst-speed-proj} yields a convenient bound for the integral
$\mathscr L_1$ \secrev{from \eqref{L1}}, viz.
\[
\mathscr L_1((\mathbf q_{\tau}, \mathbf z_{\tau}); 0,T) \le \int_{0}^{T} 
\sqrt{\sum_i \max_{\mathbf a \in A_i}\left( \frac{1}{\left|
\frac{g_{i\mathbf a}}{f_{i\mathbf a}}+t \right|}\right)^2}
\
\mu( \mathbf q_{t} \cdot R(\mathbf z_{t}) , \mathbf 0) \dd t,
\]
and we also apply Lemma~\ref{lemma-condition}(b)
with $\|(\mathbf q_i)_t\| \le (t+K) \sqrt{S_i}$ \changed{and $\kappa_{\rho_i}=\sqrt{S_i}$}. 
Adding over all paths,
\[
\begin{split}
\sum_{\mathbf z_{\tau} \in \mathscr Z(\mathbf q_{\tau})}
\mathscr L_1((\mathbf q_{\tau}, \mathbf z_{\tau}); 0,T) \le 
	c_1 & \int_{0}^{T} 
\sqrt{ \sum_i \max_{\mathbf a \in A_i} \left(\frac{t+K}{\left|
\frac{g_{i\mathbf a}}{f_{i\mathbf a}}+t \right|}\right)^2 }
\\
	&\hspace{4em} \times \sum_{\mathbf z \in Z(\mathbf q_t)} 
\| M( \mathbf q_{t}, \mathbf z)^{-1}\|_{\mathrm F} \dd t.
\end{split}
\]
with $c_1=\sqrt{\sum_i \delta_i^2 } 
\max_i (S_i)
$.
Cauchy-Schwartz inequality applied to the rightmost sum 
yields
\secrev{
\[
\begin{split}
\sum_{\mathbf z_{\tau} \in \mathscr Z(\mathbf q_{\tau})}
\mathscr L_1((\mathbf q_{\tau}, \mathbf z_{\tau}); 0,T) \le 
	c_1 & \int_{0}^{T} 
\sqrt{ \sum_i \max_{\mathbf a \in A_i} \left(\frac{t+K}{\left|
\frac{g_{i\mathbf a}}{f_{i\mathbf a}}+t \right|}\right)^2 }
\\
	&\hspace{4em} 
	\times \sqrt{\# Z(\mathbf q_t)} \sqrt{
\sum_{\mathbf z \in Z(\mathbf q_t)} 
\| M( \mathbf q_{t}, z_{t})^{-1}\|_{\mathrm F}^2  
}
\,	\dd t .
\end{split}
\]
Since $\# Z(\mathbf q_t) = \frac{n!V}{\det(\Lambda)}$, we
replace   
$c_2 \defeq c_1
\sqrt{\frac{n!V}{\det \Lambda}}$ in the integral above.
Applying Cauchy-Schwartz again,
we obtain:
\begin{eqnarray*}
\sum_{\mathbf z_{\tau} \in \mathscr Z(\mathbf q_{\tau})}
	\mathscr L_1((\mathbf q_{\tau}, \mathbf z_{\tau}); 0,T) \le \hspace{-7em}&& 
	\\
	&\le&
	c_2  \int_{0}^{T} 
\sqrt{ \sum_i \max_{\mathbf a \in A_i} \left(\frac{t+K}{\left|
\frac{g_{i\mathbf a}}{f_{i\mathbf a}}+t \right|}\right)^2 }
	\sqrt{
\sum_{\mathbf z \in Z(\mathbf q_t)} 
\| M( \mathbf q_{t}, z_{t})^{-1}\|_{\mathrm F}^2
}
\,	\dd t .
\\
	&\le&
	c_2  
\sqrt{ 
	\int_{0}^{T} 
	\sum_i \max_{\mathbf a \in A_i} \left(\frac{t+K}{\left|
\frac{g_{i\mathbf a}}{f_{i\mathbf a}}+t \right|}\right)^2 \, \dd t}
	\sqrt{
	\int_{0}^{T} 
\sum_{\mathbf z \in Z(\mathbf q_t)} 
\| M( \mathbf q_{t}, z_{t})^{-1}\|_{\mathrm F}^2
\,	\dd t 
} .
\end{eqnarray*}
}

We would like at this point to pass to the conditional \secrev{expectations}. 
Using Cauchy-Schwartz \secrev{once more}, 
\[
\begin{split}
\expected{\mathbf g \sim \secrev{\mathcal N}(\mathbf 0,I; \mathscr F)}
{
\sum_{\mathbf z_{\tau} \in \mathscr Z(\mathbf q_{\tau})}
	\mathscr L_1((\mathbf q_{\tau}, \mathbf z_{\tau}); 0,T)
\conditional \mathbf g \not \in \Lambda_{\epsilon} \cup \Omega_H \cup Y_K
} 
	&\le 
	\\
	&\hspace{-22em}\le
	c_2  
\
\sqrt{
\expected{\mathbf g \sim \secrev{\mathcal N}(\mathbf 0,I; \mathscr F)}{
\int_{0}^{T} 
	\sum_i \max_{\mathbf a \in A_i} \frac{(t+K)^2}{\left|
\frac{g_{i\mathbf a}}{f_{i\mathbf a}}+t \right|^2}
\dd t
\conditional \mathbf g \not \in \Lambda_{\epsilon} \cup \Omega_H \cup Y_K
}} 
\\
	&	\hspace{-22em}\times \sqrt{
\expected{\mathbf g \sim \secrev{\mathcal N}(\mathbf 0,I; \mathscr F)}{
\int_{0}^{T} 
\sum_{\mathbf z \in Z_H(\mathbf q_t)} 
	\|M( \mathbf q_{t},\mathbf z_{t})^{-1}\|_{\mathrm F}^2
\dd t
\conditional \mathbf g \not \in \Lambda_{\epsilon} \cup \Omega_H \cup Y_K
}}. 
\end{split}
\]
Above, we used the fact that $\mathbf g \not \in \Omega_H$ \changed{to replace
the sum over 
$Z(\mathbf q_t)$ by the sum over $Z_H(\mathbf q_t)$}.
\changed{We simplify the conditions on the choice of $\mathbf g$ as follows:}
for any positive measurable function $\phi: \mathscr F \rightarrow \mathbb R$, we have
\begin{eqnarray*}
\expected{\mathbf g \sim \secrev{\mathcal N}(\mathbf 0,I; \mathscr F)}{\phi(g) \conditional g \not \in \Lambda_{\epsilon} \cup \Omega_H \cup Y_K}
	&\le&\\
	&& \hspace{-15em}
	\le \expected{\mathbf g \sim \secrev{\mathcal N}(\mathbf 0,I; \mathscr F)}{\phi(g) \conditional g \not \in \Lambda_{\epsilon}}
\frac{ \probability{\mathbf g \sim \secrev{\mathcal N}(\mathbf 0,I; \mathscr F)}{\mathbf g \not \in \Lambda_{\epsilon}}}
{ \probability{\mathbf g \sim \secrev{\mathcal N}(\mathbf 0,I; \mathscr F)}{\mathbf g \not \in \Lambda_{\epsilon} \cup \Omega_H \cup Y_K}}.
\end{eqnarray*}
\secrev{As explained in the beginning of Section~\ref{overview-linear}, the}
\changed{exclusion sets satisfy the bounds
$\frac{71}{72} \le \probability{\mathbf g \sim \secrev{\mathcal N}(\mathbf 0,I; \mathscr F)}{\mathbf g \not \in \Lambda_{\epsilon}}\le 1$ and $\frac{7}{8} \le \probability{\mathbf g \sim \secrev{\mathcal N}(\mathbf 0,I; \mathscr F)}{\mathbf g \not \in \Lambda_{\epsilon} \cup \Omega_H \cup Y_K} \le 1$.
Therefore,}
\[
	\changed{\expected{\mathbf g \sim \secrev{\mathcal N}(\mathbf 0,I; \mathscr F)}{\phi(g) \conditional g \not \in \Lambda_{\epsilon} \cup \Omega_H \cup Y_K}
	\le
\frac{8}{7} \expected{\mathbf g \sim \secrev{\mathcal N}(\mathbf 0,I; \mathscr F)}{\phi(g)\conditional g \not \in \Lambda_{\epsilon}}
}
\]
\changed{and similarly,}
\[
\expected{\mathbf g \sim \secrev{\mathcal N}(\mathbf 0,I; \mathscr F)}{\phi(g) \conditional g \not \in \Lambda_{\epsilon} \cup \Omega_H \cup Y_K}
	\le
\frac{8}{7} \expected{\mathbf g \sim \secrev{\mathcal N}(\mathbf 0,I; \mathscr F)}{\phi(g)}
.
\]
We have proved that
\begin{lemma}\label{lem-C1} On the hypotheses above,
\[
\begin{split}
\expected{\mathbf g \sim \secrev{\mathcal N}(\mathbf 0,I; \mathscr F)}
{
\sum_{\mathbf z_{\tau} \in \mathscr Z(\mathbf q_{\tau})}
	\mathscr L_1((\mathbf q_{\tau}, \mathbf z_{\tau}); 0,T)
\conditional \mathbf g \not \in \Lambda_{\epsilon} \cup \Omega_H \cup Y_K
} 
&\le \\
	& \hspace{-16em} \le c_3
\
\sqrt{
\expected{\mathbf g \sim \secrev{\mathcal N}(\mathbf 0,I; \mathscr F)}{
\int_{0}^{T} 
	\sum_i \max_{\mathbf a \in A_i} \frac{(t+K)^2}{\left|
\frac{g_{i\mathbf a}}{f_{i\mathbf a}}+t \right|^2}
\dd t
\conditional \mathbf g \not \in \Lambda_{\epsilon}
}} 
\\
	&\hspace{-14em}\times \sqrt{
\expected{\mathbf g \sim \secrev{\mathcal N}(\mathbf 0,I; \mathscr F)}{
\int_{0}^{T} 
\sum_{\mathbf z \in Z_H(\mathbf q_t)} 
	\|M( \mathbf q_{t},\mathbf z_{t})^{-1}\|_{\mathrm F}^2
\dd t
}}. 
\end{split}
\]
	with $c_3 = \changed{\frac{8}{7}} \sqrt{\sum_i \delta_i^2 } 
	\max_i (S_i) 
\sqrt{\frac{n!V}{\det \Lambda}}$.
\end{lemma}

Now we need to bound the two expectations in the bound above.
\secrev{Theorem~\ref{E-M2}
with $\sigma_{k,\mathbf a}=1$ provides a bound for the second one:
\begin{equation}\label{E-C2-2}
	\expected{\mathbf g \sim \secrev{\mathcal N}(\mathbf 0,I; \mathscr F)}{
\int_{0}^{T} 
\sum_{\mathbf z \in Z_H(\mathbf q_t)} 
	\|M( \mathbf q_{t},\mathbf z_{t})^{-1}\|_{\mathrm F}^2
\dd t
}
\le
	\frac{2HT\sqrt{n}}{\det(\Lambda)}
	\thirdrev{n}! V'.
\end{equation}}

\begin{lemma} \label{lem-C2}
\[
\begin{split}
\expected{\mathbf g \sim \secrev{\mathcal N}(\mathbf 0,I; \mathscr F)}{
\int_{0}^{T} 
\sum_i
	\max_{\mathbf a \in A_i} \frac{(t+K)^2}{\left|
\frac{g_{i\mathbf a}}{f_{i\mathbf a}}+t \right|^2}
\dd t
\conditional \mathbf g \not \in \Lambda_{\epsilon}
} 
	&\le
	\\\hspace{10em}&\hspace{-10em}
\changed{
	\le \left( \frac{\sqrt{\pi}}{2}+1\right)}
\changed{	\left( \pi \log(2/\epsilon) + \frac{\epsilon^2}{\sin(\epsilon)} \right)
	}\changed{(K^2 + \kappa_{\mathbf f}/2) S}
+2 S T
\end{split}
\]
\end{lemma}

\begin{proof}
We start by bounding the maximum by the sum,
\[
\int_{0}^{T} 
	\sum_i \max_{\mathbf a \in A_i} \frac{(t+K)^2}{\left|
\frac{g_{i\mathbf a}}{f_{i\mathbf a}}+t \right|^2}
\dd t
\le 
\int_{0}^{T} 
	\sum_{i} \sum_{\mathbf a \in A_i} \frac{(t+K)^2}{\left|
\frac{g_{i\mathbf a}}{f_{i\mathbf a}}+t \right|^2}
\dd t
= 
	\sum_{i} \sum_{\mathbf a \in A_i} 
\int_{0}^{T} 
	\frac{(t+K)^2}{\left|
\frac{g_{i\mathbf a}}{f_{i\mathbf a}}+t \right|^2}
\dd t
.
\]
An elementary change of variables yields
\[
\int_0^T
	\frac{(t+K)^2}{\left|
\frac{g_{i\mathbf a}}{f_{i\mathbf a}}+t \right|^2}
\dd t
=
\int_0^T
	\frac{\left( |f_{i\mathbf a}|t + K|f_{i\mathbf a}|\right)^2}{\left|
\frac{g_{i\mathbf a}}{f_{i\mathbf a}}
|f_{i\mathbf a}|
+
|f_{i\mathbf a}|
t \right|^2}
\dd t
=
|f_{i\mathbf a}|^{-1}
\int_0^{|f_{i\mathbf a}|T}
	\frac{\left(s + K|f_{i\mathbf a}|\right)^2}{|z+s|^2}\dd s
\]
	\changed{where $z =  g_{i\mathbf a} |f_{i\mathbf a}|/f_{i\mathbf a}$ is a random variable with probability distribution $\secrev{\mathcal N}(\mathbf 0,1; \mathbb C)$
	subject to the condition 
\[
	\left| \arg \left( z\right) \right| < \pi-\epsilon
\]
	because $\mathbf g \not \in \Lambda_{\epsilon}$. 
	Let $\Delta_{\epsilon} \secrev{\defeq} \{ z \in \mathbb C: 
	-\pi + \epsilon \le \arg(z) \le \pi-\epsilon\}$. For any rotationally invariant distribution, ${\mathrm {Prob}}[z \in \Delta_{\epsilon}]=\frac{\pi-\epsilon}{\pi}$.
Therefore, the probability density function of $z$ is
\[
\frac{1}{\pi-\epsilon} 
e^{-|z|^2}
.
\]}
	We need another change of variables: write $z=x+\sqrt{-1}y$ and replace $s$ \changed{by} $y \tau- x$. Assume first that $y >0$:
\begin{eqnarray*}
\int_0^T
	\frac{(t+K)^2}{\left|
\frac{g_{i\mathbf a}}{f_{i\mathbf a}}+t \right|^2}
\dd t
	&=&
|f_{i\mathbf a}|^{-1}
\frac{1}{y}
\int_{x/y}^{\left(|f_{i\mathbf a}|T + x \right)/y}
\frac{\left(y \tau -  x + K|f_{i\mathbf a}|\right)^2}{1+\tau^2} \dd\tau 
\\
&\le&
2
|f_{i\mathbf a}|^{-1}
\frac{1}{y}
\int_{x/y}^{\left(|f_{i\mathbf a}|T + x \right)/y}
	\frac{ y^2 \tau^2 + \left( x - K|f_{i\mathbf a}|\right)^2}{1+\tau^2} \dd\tau 
\end{eqnarray*}
The case $y<0$ is the same with the sign and the integration limits reversed.
The expression above can be expanded as follows:
\begin{equation}\label{three-parts}
\begin{split}
		\int_0^T
	\frac{(t+K)^2}{\left|
\frac{g_{i\mathbf a}}{f_{i\mathbf a}}+t \right|^2}
\dd t
\le
	2\left(
|f_{i\mathbf a}|
	K^2 y^{-1}
-
	2Kx y^{-1}
+ |f_{i\mathbf a}|
	^{-1} x^2 y^{-1}
	\right)	
	A_0(z)
\\
+
2 |f_{i\mathbf a}|^{-1}
	y A_2(z)
\end{split}
\end{equation}
with
\[
	A_i(z) = 
\int_{x/y}^{\left(|f_{i\mathbf a}|T+x\right)/y}
	\frac{\tau^i}{1+\tau^2}\dd \tau
.\]
We can integrate, assuming again $y>0$:
\begin{eqnarray*}
	A_0(z) &=& 
\left[ \arctan(\tau) \right]
_{\tau=xy^{-1}}^{\tau=\left(|f_{i\mathbf a}|T+x\right)y^{-1}}
\\
	&\le&
\left[ \arctan(\tau) \right]
_{\tau=xy^{-1}}^{\infty}
\\
&=& \arg(z)
\\
	A_2(z) &=& 
\left[ \tau-\arctan(\tau) \right]
_{\tau=xy^{-1}}^{\tau=\left(|f_{i\mathbf a}|T+x\right)y^{-1}}
\\
	&=& |f_{i\mathbf a}|Ty^{-1} - A_0(z)
\end{eqnarray*}
Replacing in Equation \eqref{three-parts},
\[
\int_0^T
\frac{(t+K)^2}{\left|
\frac{g_{i\mathbf a}}{f_{i\mathbf a}}+t \right|^2}
\dd t
\le
2\left(
|f_{i\mathbf a}|
	K^2 
-
	2Kx 
+ |f_{i\mathbf a}|^{-1} (x^2-y^2)
	\right)	
		\frac{A_0(z)}{y}
+
2 T
\]
Passing to polar coordinates $x=r \cos(\theta)$, $y=r \sin(\theta)$ we
can bound
\[
\int_0^T
\frac{(t+K)^2}{\left|
\frac{g_{i\mathbf a}}{f_{i\mathbf a}}+t \right|^2}
\dd t
\le
2\changed{
\left(
|f_{i\mathbf a}|
	K^2 
+
	2Kr 
+ |f_{i\mathbf a}|^{-1} r^2 
		\right)}
		\frac{\theta}{r \sin(\theta)}
+
2 T
\]
Above, we used trivial bounds $-1 \le \cos(\theta) \le 1$, $\cos^2(\theta)-\sin^2(\theta) \le 1$. Notice also that the left hand side of \eqref{three-parts} is 
symmetric
with respect to $y = \Im\left(\frac{g_{i\mathbf a}}{f_{i\mathbf a}}\right) |f_{i\mathbf a}|$. 
We need to bound
\begin{eqnarray*}
	\expected{\mathbf z \in \Delta_{\epsilon}}{\sum_{i, \mathbf a} 
\int_{0}^{T} 
	\frac{(t+K)^2}{\left|
\frac{g_{i\mathbf a}}{f_{i\mathbf a}}+t \right|^2}
	\dd t}
	\le\hspace{-12em}&&\\
&\le& \frac{2}{\pi - \epsilon}
\sum_{i, \mathbf a}
\int_{-\pi + \epsilon}^{\pi - \epsilon}
\int_{0}^{\infty}
	\changed{\left( 
|f_{i\mathbf a}|
	K^2 
+
	2Kr 
+ 
	|f_{i\mathbf a}|^{-1} r^2 
	\right)}
e^{-r^2}
\frac{\theta}{\sin(\theta)}
\dd r
\dd \theta
+ 2S T
\end{eqnarray*}
The integral above clearly splits. The integral in $r$ is
trivial:
\begin{eqnarray*}
\sum_{i, \mathbf a}
\int_{0}^{\infty}
\left(
|f_{i\mathbf a}|
	K^2 
+
	2Kr 
+ |f_{i\mathbf a}|^{-1} r^2 
	\right)	
e^{-r^2}
\dd r
	=\hspace{-10em}	&&\\
	&=&\frac{1}{2}
\sum_{i, \mathbf a}
\left(
|f_{i\mathbf a}|
	K^2 \sqrt{\pi}
+
	2K 
+ \frac{1}{2} |f_{i\mathbf a}|^{-1} \sqrt{\pi}
\right)
\\
&\le&
\frac{1}{2} \|  \mathbf f\| K^2 \sqrt{S} \sqrt{\pi}
+ KS + \frac{1}{4}	\sqrt{\pi} 
	\sum_{i,\mathbf a} |f_{i\mathbf a}|^{-1}
\\
&\le&
\frac{1}{2} K^2 S \sqrt{\pi}
	+ KS + \frac{1}{4}	\kappa_{\mathbf f}  \sqrt{n}\sqrt{S} \sqrt{\pi} 
\end{eqnarray*}
\changed{using $\|\mathbf f\|=\sqrt{S}$ and 
$\kappa_{\mathbf f}=\max_{i,\mathbf a} \frac{\|\mathbf f_i\|}{|f_{i\mathbf a}|}
=
\max_{i,\mathbf a} \frac{\sqrt{S_i}}{|f_{i\mathbf a}|}
$.}
An elementary bound of the right-hand side yields
\[\sum_{i, \mathbf a}
\int_{0}^{\infty}
\left(
|f_{i\mathbf a}|
	K^2 
+
	2Kr 
+ |f_{i\mathbf a}|^{-1} r^2 
	\right)	
e^{-r^2}
\dd r
\le
\changed{\left( \frac{\sqrt{\pi}}{2}+1\right)}	(K^2 + \kappa_{\mathbf f}/2) S
\]

\changed{We still need to bound the integral with respect to $\theta$.}
A primitive $F(\theta)$ for $\frac{\theta}{\sin(\theta)}$ is known:
\[
\begin{split}
	F(\theta) = &\theta \left(\log\left(1-e^{\sqrt{-1}\theta}\right) - \log\left(1+e^{\sqrt{-1}\theta}\right)\right)
	+ \\
	& \hspace{10em} +	\sqrt{-1} \left( 
\Li\left(-e^{\sqrt{-1}\theta}\right)  
-
\Li\left(e^{\sqrt{-1}\theta}\right)  
\right)
\end{split}
\]
where $\Li(z) = \sum_{k=1}^{\infty} \frac{z^k}{k^2}$ is the {\em polylogarithm}
or {\em Jonquière's function}, and the identity $F'(\theta)=\frac{\theta}{\sin(\theta)}$ can be deduced from the property $\Li(z)'= -\frac{\log(1-z)}{z}$. 
Since $\lim_{\theta \rightarrow 0} \Re(F(\theta))=0$, we have
\begin{eqnarray*}
\int_{0}^{\pi-\epsilon}
\frac{\theta}{\sin(\theta)}
\dd \theta
	&=&
	\Re(F(\pi-\epsilon)) \\
	&=&
	\pi \Re\left(\log\left(1+e^{-\sqrt{-1}\epsilon}\right) - \log\left(1-e^{-\sqrt{-1}\epsilon}\right)\right)-\Re(F(-\epsilon))
\\
	&\le&
	\pi \log(2/\epsilon)+ \epsilon^2/\sin(\epsilon)
\end{eqnarray*}
Putting all together and bounding $\sqrt{n} \le \sqrt{S}$,
\[
\begin{split}
	\expected{\mathbf z \in \Delta_{\epsilon}}{\sum_{i, \mathbf a} 
\int_{0}^{T} 
	\frac{(t+K)^2}{\left|
\frac{g_{i\mathbf a}}{f_{i\mathbf a}}+t \right|^2}
	\dd t}
\le
\changed{
\left( \frac{\sqrt{\pi}}{2}+1\right)
\left( \pi \log(2/\epsilon) + \frac{\epsilon^2}{\sin(\epsilon)} \right)
(K^2 + \kappa_{\mathbf f}/2) S}
\\
+2 S T
\end{split}
\]
\end{proof}

\begin{proof}[Proof of Proposition~\ref{prop-C1}]
	In Lemma~\ref{lem-C2}, we replace $\epsilon$ by \changed{its value
	$\pi/(72 S)$. We have}
	$\log (2/{\epsilon}) \le \log S + \log(\changed{144}/\pi)$. Also
	\thirdrev{for small $\epsilon$, for instance for $\epsilon \le 
	\pi/288$,}
	$\epsilon^2/\sin(\epsilon)$ is bounded by $\thirdrev{
		\frac{\epsilon}{1-\frac{1}{3!}\epsilon^2}} < \log(2/\epsilon)$
	and
\[
\pi \log (2/\epsilon) + \frac{\epsilon^2}{\sin(\epsilon)} \le c_4 \log(S)
\]
for some constant $c_4$. 
	\[
\sqrt{
\expected{\mathbf g \sim \secrev{\mathcal N}(\mathbf 0,I; \mathscr F)}{
\int_{0}^{T} 
	\max_{i, \mathbf a} \frac{(t+K)^2}{\left|
\frac{g_{i\mathbf a}}{f_{i\mathbf a}}+t \right|^2}
\dd t
\conditional \mathbf g \not \in \Lambda_{\epsilon}
}} 
\le
	\sqrt{c_5 \left(K^2+\frac{\kappa_{\mathbf f}}{2}\right) S \log(S) + 2 S T} 
\]
\changed{for $c_5=(\sqrt{\pi}/2+1)c_4$.}
	\secrev{In the statement of Theorem~\ref{condlength}, we
	required that
	$T > 2K$. Since $K \ge 1$, the expression above simplifies to}
\[
\sqrt{
\expected{\mathbf g \sim \secrev{\mathcal N}(\mathbf 0,I; \mathscr F)}{
\int_{0}^{T} 
	\max_{i, \mathbf a} \frac{(t+K)^2}{\left|
\frac{g_{i\mathbf a}}{f_{i\mathbf a}}+t \right|^2}
\dd t
\conditional \mathbf g \not \in \Lambda_{\epsilon}
}} 
	\le 
	\sqrt{2 c_5
	S \log(S) 
	\left(1+\frac{K}{4}+\frac{\kappa_{\mathbf f}}{8}\right) 
	T}.
\]
	Combining with \secrev{\eqref{E-C2-2}},
	\secrev{Lemma~\ref{lem-C1} yields:}	
	\begin{equation}\label{E1}
\begin{split}
	E_1 &\defeq \expected{\mathbf g \sim \secrev{\mathcal N}(\mathbf 0,I; \mathscr F)}
{
\sum_{\mathbf z_{\tau} \in \mathscr Z(\mathbf q_{\tau})}
	\mathscr L_1((\mathbf q_{\tau}, \mathbf z_{\tau}); 0,T)
\conditional \mathbf g \not \in \Lambda_{\epsilon} \cup \Omega_H \cup Y_K
} 
\\
	&\le 
	\changed{\frac{\thirdrev{16}}{7}} 
	n^{\frac{1}{4}}
	\max_i (S_i) 
	\sqrt{
		c_5 \changed{H}
		S \log(S) 
		\left(1+\frac{K}{4}+\frac{\kappa_{\mathbf f}}{8}\right) 
		\frac{
	(\sum_i \delta_i^2)
	(n!V) (n! V' \eta)
	}
	{
		\changed{\eta} \det(\Lambda)^2
	}} \ T. 
\end{split}
\end{equation}
We picked $H \le \frac{n}{\eta} \changed{O(\secrev{\mathrm{LOGS}_T})}$
\changed{in Equation \eqref{eq-H}}. Recall from
Remark~\ref{remark-on-eta} that $\eta_i \le 2 \delta_i$, whence
$\eta = \min \eta_i \le 2 \delta_i$ for all $i$. Thus,
\[
	\eta^2 \le \frac{4}{n}\sum_{i=1}^n \delta_i^2
\]
This allows us to bound
\[
\secrev{ H \le
\frac{n}{\eta} O(\mathrm{LOGS}_T)
\le
\frac{4}{\eta^3} 
\left(\sum_{i=1}^n \delta_i^2\right)
O(\mathrm{LOGS}_T)
}
\]
\secrev{
	Replacing $H$ in \eqref{E1} by the estimate above, we obtain
\[
E_1 \le
\frac{32}{7} 
	n^{\frac{1}{4}}
	\max_i (S_i) 
	\sqrt{
		c_5 
		S \log(S) 
		\left(1+\frac{K}{4}+\frac{\kappa_{\mathbf f}}{8}\right) 
		Q^2 \mathrm{LOGS}_T
	}\ T 
\]
with $Q$ as defined in \eqref{mainD-Q}. Since
$\log(S) \le \secrev{\mathrm{LOGS}_T}$, we can remove $Q^2 \log(S)\mathrm{LOGS}_T$ from the radical. For some constant $C_1$, 
\[
E_1 \le
C_1 
n^{\frac{1}{4}}\sqrt{S} \max S_i
\sqrt{1+\frac{K}{4}+\frac{\kappa_{\mathbf f}}{8}}
	\	Q
	\ T \ 
	\secrev{\mathrm{LOGS}_T} \ .
\]}
\end{proof}

\subsection{Expectation of the condition length (part 2)}
\begin{proposition}\label{prop-C2}
Under the conditions of Theorem~\ref{condlength},
there is a constant $C_2$ such that
\begin{eqnarray*}
\expected{\mathbf g \sim \secrev{\mathcal N}(\mathbf 0,I; \mathscr F)}
{
\sum_{\mathbf z_{\tau} \in \mathscr Z(\mathbf q_{\tau})}
\mathscr L_2((\mathbf q_{\tau}, \mathbf z_{\tau}); 0,T)
\conditional \mathbf g \not \in \Lambda_{\epsilon} \cup \Omega_H \cup Y_K
} 
&\le&\\
&& \hspace{-22em}\le \ 
	C_2	Q \changed{n^{2}} 
	\sqrt{S}\max_i (S_i) 
	\ K \ T \ 
 \secrev{\mathrm{LOGS}_T}
\end{eqnarray*}
\end{proposition}

We need first an auxiliary Lemma.

\begin{lemma}\label{cond-length-triv-bounds} \ 
\begin{enumerate}[(a)] 
\item 
Assume that $\mathbf q_{\tau} \cdot \mathbf V(\mathbf z_{\tau}) \equiv 0$. Then
for any $\mathbf y \in \secrev{\mathscr M}$,
\[
	\| \dot {\mathbf z}_{\tau}\|_{\mathbf y}	
\le
		\changed{\sqrt{n}}
		\| P_{\mathbf q_{\tau}} (\dot{\mathbf q}_{\tau}) \|
		\| M(\mathbf q_{\tau}, \mathbf z_{\tau})^{-1}\|_{\mathbf y}
.
\]
\item Assume that $\tau \ne 0$ and $\mathbf q_{\tau}=\tau\mathbf f + \mathbf g$.
	Then,
\[
\left\|	P_{\tau\mathbf f + \mathbf g}(\mathbf f) \right\|
	\le \min( \| \mathbf f\|, 2 \| \mathbf g\| /\tau) 
.
\]
\item Assume $\tau \ne 0$, $\mathbf q_{\tau}=\tau\mathbf f + \mathbf g$
and $\mathbf q_{\tau} \cdot \mathbf V(\mathbf z_{\tau}) \equiv 0$. Then,
\[
	\| \dot {\mathbf z}_{\tau}\|_{\mathbf 0}	
		\secrev{\mu( \mathbf q_{\tau} \cdot R(\mathbf z_{\tau}) )}
\le
		\changed{\sqrt{n}}
\left(\sum_i \delta_i^2 \right) 
\min( \| \mathbf f\|, 2 \| \mathbf g\| /\tau) 
\| M(\mathbf q_{\tau}, \mathbf z_{\tau})^{-1}\|_{\mathrm F}^2
		\max_i (\sqrt{S_i} \|\mathbf q_i\|)
.
\]
\end{enumerate}
\end{lemma}

\begin{proof}[Proof of Lemma \ref{cond-length-triv-bounds}]
Item (a) is obtained by differentiating 
	$\mathbf q_{\tau} \cdot \mathbf V(\mathbf z_{\tau}) \changed{\equiv} 0$:
\[
	\begin{pmatrix}
	\ddots \\
		&\| V_{A_i}(\mathbf z_{\tau})\|&\\
		& & \ddots
	\end{pmatrix}
M( \mathbf q_{\tau} , \mathbf z_{\tau} ) \dot {\mathbf z}_{\tau}
	= -\dot {\mathbf q}_{\tau}  \cdot 
\mathbf V(\mathbf z_{\tau}) 
.
\]
\changed{Because $\mathbf q_{\tau} \cdot \mathbf V(\mathbf z_{\tau}) \equiv 0$,
each coordinate satisfies 
$
	P_{\mathbf q_{i\tau}^{\perp}} (\dot{\mathbf q}_{i\tau}) V_{A_i}(\mathbf z_{\tau})
=
	\dot{\mathbf q}_{i\tau} V_{A_i}(\mathbf z_{\tau})$ and the equation above}
is equivalent to
\[
	\begin{pmatrix}
	\ddots \\
		&\| V_{A_i}(\mathbf z_{\tau})\|&\\
		& & \ddots
	\end{pmatrix}
M( \mathbf q_{\tau} , \mathbf z_{\tau} ) \dot {\mathbf z}_{\tau}
	= -P_{\mathbf q_{\tau}^{\perp}} (\dot {\mathbf q}_{\tau})  \cdot 
\mathbf V(\mathbf z_{\tau}) 
.
\]
Thus,
	\[
	\dot {\mathbf z}_{\tau}
=-
M( \mathbf q_{\tau} ,\mathbf z_{\tau})^{-1} 
\left(P_{\mathbf q_{\tau}^{\perp}} (\dot {\mathbf q}_{\tau})   
\cdot 
\begin{pmatrix} \vdots \\
	\frac{1}{\|V_{A_i}(\mathbf z_{\tau})\|}	V_{A_i}(\mathbf z_{\tau}) \\
	\vdots
\end{pmatrix}\right)
\]
	\changed{Item (a) follows from passing to norms.}
Now we prove item (b). Because $P$ is a projection operator,$ \left\|	P_{\tau\mathbf f + \mathbf g}(\mathbf f) \right\| \le \| \mathbf f\|$. We prove the remaining inequality below:
For each $i$, let $\mathbf q_i = \tau \mathbf f_i + \mathbf g_i$. 
\begin{eqnarray*}
P_{\mathbf q_i^{\perp}} (\mathbf f_i) &=&
	\frac{1}{\tau}P_{\mathbf q_i^{\perp}} (\mathbf q_i - \mathbf g_i) \\
&=&
	\frac{1}{\tau}	\secrev{\left(\mathbf q_i - \mathbf g_i - \frac{1}{\|\mathbf q_i\|^2} \mathbf q_i \mathbf q_i^*  (\mathbf q_i - \mathbf g_i)\right)} \\
	&=&
	\frac{1}{\tau} \left( \mathbf q_i - \mathbf g_i - \mathbf q_i + \mathbf q_i \frac{\langle \mathbf g_i, \mathbf q_i \rangle}{\|\mathbf q_i\|^2} \right)
\\
&=&
\frac{1}{t} \left( - \mathbf g_i + \mathbf q_i \frac{\langle \mathbf g_i, \mathbf q_i \rangle}{\|\mathbf q_i\|^2} \right)
\end{eqnarray*}
so passing to norms, $\|P_{\mathbf q_i^{\perp}} (\mathbf f_i) \| \le 2 \frac{\|\mathbf g_i\|}{\tau}$
and $\|P_{(\tau\mathbf f + \mathbf g)^{\perp}} (\mathbf f) \| \le
2 \frac{\| \mathbf g \|}{\tau}$.

To prove item (c), recall
	from Lemma~\ref{lemma-condition}(b) that
\[
	\secrev{\mu( \mathbf q \cdot R(\mathbf z))}
\le \sqrt{\sum_i \delta_i^2 } 
\left\| M(\mathbf q , \mathbf z)^{-1} \right\|_{\mathrm F}
\max_i (\kappa_{\rho_i}\|\mathbf q_i\|)
\]
with $\kappa_{\rho_i}=\sqrt{S_i}$ because $\rho_{i, \mathbf a}=1$.
\secrev{Then combine with the two previous items at $\mathbf y=\mathbf 0$.}
\end{proof}

\begin{proof}[Proof of Proposition~\ref{prop-C2}]
	Since $\|\mathbf f_i\|= \sqrt{S_i}$, $\|\mathbf f\| = \sqrt{S}$
where $S=\sum S_i$. For $\mathbf g \not \in Y_K$, we also have
 $\|\mathbf g\| \le K \sqrt{S}$.  
As in the proof of Proposition~\ref{prop-C1},
we bound the conditional expectation by
\begin{eqnarray*}
	\secrev{E_2}&\secrev{\defeq}&
\expected{\mathbf g \sim \secrev{\mathcal N}(\mathbf 0,I; \mathscr F)}
{
\sum_{\mathbf z_{\tau} \in \mathscr Z(\mathbf q_{\tau})}
\mathscr L_2((\mathbf q_{\tau}, \mathbf z_{\tau}); 0,T)
\conditional \mathbf g \not \in \Lambda_{\epsilon} \cup \Omega_H \cup Y_K
} \\
&\le&
\frac{
\probability{\mathbf g \sim \secrev{\mathcal N}(\mathbf 0,I; \mathscr F)}{\mathbf g \not \in \Omega_H \cup Y_K}
}
{
\probability{\mathbf g \sim \secrev{\mathcal N}(\mathbf 0,I; \mathscr F)}{\mathbf g \not \in  \Lambda_{\epsilon} \cup  \Omega_H \cup Y_K}
}
\\
	&& \hspace{1em} \times\ 
\expected{\mathbf g \sim \secrev{\mathcal N}(\mathbf 0,I; \mathscr F)}
{
\sum_{\mathbf z_{\tau} \in \mathscr Z(\mathbf q_{\tau})}
\mathscr L_2((\mathbf q_{\tau}, \mathbf z_{\tau}); 0,T)
\conditional \mathbf g \not \in \Omega_H \cup Y_K
} 
\\
	&\le& \hspace{1em} 
	\changed{\frac{8}{7}}\int_{\mathscr F \setminus \Omega_H \cup Y_K}
	\frac{e^{-\|\mathbf g \|^2}}
	{\pi^S } 
\sum_{\mathbf z_{\tau} \in \mathscr Z(\mathbf q_{\tau})}
\mathscr L_2((\mathbf q_{\tau}, \mathbf z_{\tau}); 0,T)
\dd \mathscr F(\mathbf g).
\end{eqnarray*}
	Replacing the integral $\mathscr L_2$ by its definition \eqref{L2}
and then changing the order of integration,
\\
\begin{eqnarray*}
	E_2 & \le &  
	\changed{\frac{16}{7}}\secrev{\nu} \int_0^T
	\int_{\mathscr F \setminus \Omega_H \cup Y_K}
	\frac{e^{-\|\mathbf g \|^2}}
	{\pi^S } 
\sum_{\mathbf z \in Z(\mathbf q_{\tau})}
	\| \dot {\mathbf z} \|_{\mathbf 0}
	\secrev{\mu( \mathbf q_{\tau} \cdot R(\mathbf z))}
\dd \mathscr F(\mathbf g)
\dd \tau
\\
	& \le &  
	\changed{\frac{16}{7}}
	\secrev{\nu} 
	\changed{\sqrt{n}}
	\left(\sum_i \delta_i^2 \right) 
	\int_0^T
	\int_{\mathscr F \setminus \Omega_H \cup  Y_K}
	\frac{e^{-\|\mathbf g \|^2}}
	{\pi^S } 
\\
&& \hspace{1em}\times \ 
\sum_{\mathbf z \in Z(\mathbf q_{\tau})}
\min( \| \mathbf f\|, 2 \| \mathbf g\| /\tau) 
\| M(\mathbf q_{\tau}, \mathbf z_{\tau})^{-1}\|_{\mathrm F}^2
	\max_i (\sqrt{S_i}\|\mathbf q_i\|) \dd \mathscr F(\mathbf g)
\dd \tau
\end{eqnarray*}
where the last step is Lemma \ref{cond-length-triv-bounds}(c) above.
We can bound
\[
\min( \| \mathbf f\|, 2 \| \mathbf g\| /\tau) 
\le  \sqrt{S}
\min(1, 2K/\tau) 
\]
and
\[
\max_i\|\mathbf q_i\| \le \max_{i} \sqrt{S_i}
(\tau + K)
.
\]
As in sections 
\ref{sec-integral} 
and 
\ref{sec-renormalization} 
but with $\Sigma^2 = I$, define
\begin{eqnarray*}
	I_{\secrev{\mathbf h}, I} &=& \expected{\mathbf q \sim \secrev{\mathcal N}(\secrev{\mathbf h}, I)} { \sum_{\mathbf z \in Z_H(\mathbf q)} \| M(\mathbf q ,z)^{-1}\|_{\mathrm F}^2 }
\\
&=&
\int_{\mathscr F}
\frac{e^{-\|\mathbf g \|^2}}
{\pi^S } 
\sum_{\mathbf z \in Z_H(\secrev {\mathbf h} + \mathbf g)}
\| M(\secrev{\mathbf h}+\mathbf g, \mathbf z_{\tau})^{-1}\|_{\mathrm F}^2
\dd \mathscr F(\mathbf g) .
	\end{eqnarray*}
\secrev{We obtain}
\begin{equation}\label{boundE}
E_2 \le
	\changed{\frac{16}{7}}\secrev{\nu} 
	\changed{\sqrt{n}}\sqrt{S}\max_i (S_i) 
\left(\sum_i \delta_i^2 \right) 
\int_0^T
(\tau + K)\min(1,2K/\tau)\ I_{\tau \mathbf f, I}\ \dd \tau .
\end{equation}
\secrev{
Theorem~\ref{E-M2} provides gives a bound for 
$I_{\tau \mathbf f, I}$ in \eqref{boundE} above. Let $\mathbf h= \tau \mathbf f$,}
\[
I_{\tau \mathbf f, I}=
I_{\secrev{\mathbf h},I} 
\le
	\frac{2H\sqrt{n}}{\det(\Lambda)}
\thirdrev{n!} V',
\]
and this bound does not depend on $\tau$ so it can be removed from
the integral in \eqref{boundE}. We still have to
integrate, assuming $T \ge 2K$,
\begin{eqnarray*}
\int_0^T
(\tau + K)\min(1,2K/\tau) \dd \tau
&=&
\int_0^{2K}
\tau + K \ \dd \tau
+
\int_{2K}^T
2K + 2K^2/\tau \ \dd \tau
\\
&=& 
	2KT + 2K^2 \log \left( \frac{T}{2K} \right)
\\
	&\le& 3KT\ \changed{\log(T)}.
\end{eqnarray*}
Putting all together,
\begin{equation}\label{almost-it}
E_2 \le
	\changed{\frac{96}{7}} H \secrev{\nu} 
	\changed{n} \sqrt{S}\max_i (S_i)
	KT\changed{\log(T)} 
	\frac{\left(\sum_i \delta_i^2 \right)
	\thirdrev{n}! V'}{\det \Lambda}.
\end{equation}
Recall \secrev{from equation \eqref{eq-H}} that \secrev{$H \le \frac{n}{\eta} \secrev{\mathrm{LOGS}_T}$ and that}
\[
	Q = \eta^{-2} \left(\sum_{i=1}^n \delta_i^2\right) \frac {\max( n! V , \thirdrev{n!} V' \eta)}{\det \Lambda}
\]
This allows to simplify expression \eqref{almost-it} to
\[
	\secrev{E_2} \le
	C_2 Q \changed{n^2}  
	\sqrt{S}\max_i (S_i)
	KT 
\  \secrev{\mathrm{LOGS}_T}
\]
for some constant $C_2$.
\end{proof}

\subsection{Proof of Theorem~\ref{condlength}}
\begin{proof}[Proof of Theorem~\ref{condlength}]
We need to put together Propositions~\ref{prop-C1} and~\ref{prop-C2},
\begin{eqnarray*}
\secrev{E_1} &\defeq&
\expected{\mathbf g \sim \secrev{\mathcal N}(\mathbf 0,I; \mathscr F)}
{
\sum_{\mathbf z_{\tau} \in \mathscr Z(\mathbf q_{\tau})}
	\mathscr L_1((\mathbf q_{\tau}, \mathbf z_{\tau}); \mathbf 0,T)
\conditional \mathbf g \not \in \Lambda_{\epsilon} \cup \Omega_H \cup Y_K
} 
\\
	&\le& 
        C_1      Q 
	n^{\frac{1}{4}}
	\changed{\sqrt{S}} \max(S_i) 
	\sqrt{1+\frac{K}{4}+\frac{\kappa_{\mathbf f}}{8}} 
	\ T \ 
	\secrev{\mathrm{LOGS}_T}
\\
	\secrev{E_2} &\defeq&
\expected{\mathbf g \sim \secrev{\mathcal N}(\mathbf 0,I; \mathscr F)}
{
\sum_{\mathbf z_{\tau} \in \mathscr Z(\mathbf q_{\tau})}
\mathscr L_2((\mathbf q_{\tau}, \mathbf z_{\tau}); 0,T)
\conditional \mathbf g \not \in \Lambda_{\epsilon} \cup \Omega_H \cup Y_K
} 
\\
&\le&
	C_2	Q \changed{n^{2}} 
	\sqrt{S}\max_i (S_i) 
	\ K \ T \ 
 \secrev{\mathrm{LOGS}_T}
\end{eqnarray*}
To simplify notations, let $\mathscr L$ be the random variable $\sum_{\mathbf z_{\tau} \in \mathscr Z(\mathbf q_{\tau})} \mathscr L((\mathbf q_{\tau}, \mathbf z_{\tau}); 0,T)$. 
	\secrev{Adding together, we conclude that 
\begin{eqnarray*}
\expected{\mathbf g \sim \secrev{\mathcal N}(\mathbf 0,I; \mathscr F)}
{
\mathscr L
\conditional \mathbf g \not \in \Lambda_{\epsilon} \cup \Omega_H \cup Y_K
} 
	&\le& E_1 + E_2\\
	&\le&
	(C_1+C_2) E'
\end{eqnarray*}
with 
\[
	E'= 		
	Q
	\changed{n^2 \sqrt{S}}\max_i (S_i)~
	\thirdrev{\max\left(K,n^{-1} \sqrt{1+\frac{K}{4}+\frac{\kappa_{\mathbf f}}{8}}\right)} 
	\ T 
	\ \secrev{\mathrm{LOGS}_T}
.
\]}

	\thirdrev{We assumed in this paper that $n \ge 2$, hence $n^{-1} \le \frac{1}{2}$. We claim that
\[
	\max\left(K, \frac{1}{2} \sqrt{1+\frac{K}{4}+\frac{\kappa_{\mathbf f}}{8}}\right)
\le
	\max(K, \sqrt{\kappa_{\mathbf f}}).
\]
Indeed, assume first that $K \le \frac{1}{2} \sqrt{\kappa_{\mathbf f}}$.
In this case $1 \le K \le  \frac{1}{2} \sqrt{\kappa_{\mathbf f}} 
\le  \frac{1}{2} \kappa_{\mathbf f}$ and
\[
\frac{1}{2} \sqrt{1+\frac{K}{4}+\frac{\kappa_{\mathbf f}}{8}}
\le
\frac{1}{2} \sqrt{\frac{3}{4} \kappa_{\mathbf f}} 
\le \sqrt{\kappa_{\mathbf f}}
.
\]
Suppose instead that $K \ge \frac{1}{2} \sqrt{\kappa_{\mathbf f}}$.
Now,
\[
\frac{1}{2} \sqrt{1+\frac{K}{4}+\frac{\kappa_{\mathbf f}}{8}}
\le
\frac{1}{2} \sqrt{1+\frac{K}{4}+\frac{K^2}{2}}
\le K
.
\]
Thus,
\[
	E'\le 		
	Q
	\changed{n^2 \sqrt{S}}\max_i (S_i) 
	\max\left(K,\sqrt{\kappa_{\mathbf f}}\right) 
	\ T 
	\ \secrev{\mathrm{LOGS}_T}
.
\]}
Let $C_0=t (C_1 + C_2)$ for $t>1$. Markov's inequality says that
\[
\probability{\mathbf g \sim \secrev{\mathcal N}(\mathbf 0,I; \mathscr F)}
{\mathscr L \ge C_0 E'  
\conditional \mathbf g \not \in \Lambda_{\epsilon} \cup \Omega_H \cup Y_K
} \le 1/t . 
\]
The main statement follows from setting $t=7$.
%
\end{proof}

\subsection{Finite, non-degenerate roots}

Assume that $\mathbf q_{t} = \mathbf g + t \mathbf f \in \mathscr F$ is
given, and that there is a continuous path $\mathbf z_t \in
\secrev{\mathscr M}$ with $\mathbf q_t \cdot \mathbf V(\mathbf z_t) \secrev{\equiv} 0$.
Assume furthermore that $\boldsymbol \zeta = \lim_{t \rightarrow \infty} \mathbf z_t$ exists and is a point in $\secrev{\mathscr M}$.
\changed{The objective of this section is to show that for a large enough
value of $T$, the condition length of the renormalized path
$(\mathbf q_t, \mathbf z_t)_{t \in [T, \infty)}$ is smaller than the
constant $\delta_*$ from Theorem~\ref{th-A}. In particular, \secrev{$\mathbf z_{T}$}
is a certified root of $\mathbf f = \lim_{t \rightarrow \infty} \frac{1}{t}\mathbf q_t$}.

The \secrev{lemma} below gives precise values for $T$ so that 
Theorem~\ref{condlength} can be used for tracking the homotopy path: 
we may want to find approximations of $\mathbf z_0$ out of approximations
of $\boldsymbol \zeta$, or the opposite.
\secrev{The Lipschitz properties of the path $\tau \mapsto (\mathbf q_{\tau}, \mathbf z_{\tau})$ for large values of $\tau$ will be investigated by means of the invariants below. Recall that $d_{\mathbb P}$ stands for the multi-projective
distance defined in \eqref{multi-projective-distance}}:
\[
	\Delta_0(T) \secrev{\defeq} \max_{\tau \ge T} d_{\mathbb P}(\mathbf q_{\tau} \cdot R(\boldsymbol \zeta),
	\mathbf f\cdot  R(\boldsymbol \zeta) ),
\]
\[
	\Delta_1(T) \secrev{\defeq}  \max_{\tau \ge T} d_{\mathbb P}(\mathbf q_{\tau} \cdot R(\mathbf z_{\tau}),
	\mathbf f\cdot  R(\boldsymbol \zeta) )
\]
and
\[
	\Delta_2(T) \secrev{\defeq}  \secrev{\nu}\ \max_{\tau \ge T} \| \mathbf z_{\tau} - \boldsymbol \zeta \|_{\mathbf 0}
\le \tilde \Delta_2(T) \secrev{\defeq} \secrev{\nu} \int_{T}^{\infty} \| \dot{\mathbf z}_{\tau}\|_{\mathbf 0} \ \dd \tau 
.
\]
Those functions are decreasing and, by continuity, 
\[
\lim_{T \rightarrow \infty} \Delta_0(T) = \lim_{T \rightarrow \infty} \Delta_1(T) = 
\lim_{T \rightarrow \infty} \Delta_2(T) 
\changed{= 
\lim_{T \rightarrow \infty} \tilde \Delta_2(T)} = 0.
\]
\changed{Let $\mu \ge \secrev{\mu(\mathbf f \cdot R(\boldsymbol \zeta))}$.}
Recall that $\kappa_{\mathbf f} = \max_{i, \mathbf a} \frac{\|\mathbf f_i\|}{|f_{i\mathbf a}|}$.

\begin{lemma}\label{lem-delta}
	Assume that $n \ge 2$, $S_i \ge 2$ for all $i$, and 
	$\rho_{i \mathbf a}=1$ always. 
	Suppose that $\mathbf f$ is scaled so that
	$\|\mathbf f_i \|=\sqrt{S_i}$ exactly, for each $i$.
	\secrev{Assume also} 
	that
	$\|\mathbf g_i \|\le K \sqrt{S_i}$ for some constant $K$.
	\secrev{Let}
	\begin{equation}\label{bound-T}
		T \ge \theta \kappa_{\mathbf f} K 
		\secrev{\sqrt{n}}
		\sqrt{S} 
		\mu^2 \secrev{\nu}
	\end{equation}
	for $\theta \ge \theta_0 \simeq  \thirdrev{13.977.369\cdots<14}$. \secrev{Suppose that} there is a smooth path
$\mathbf z_t \in \secrev{\mathscr M}$, $t \ge T$, with 
	$\mathbf q_{t} \cdot \mathbf V(\mathbf z_t) \secrev{\equiv} 0$ for $\mathbf q_t = \mathbf g + t \mathbf f$ with $\boldsymbol \zeta = \lim_{t \rightarrow \infty}(\mathbf z_t) \secrev{\ \in \mathcal M}$.
Then,
\begin{eqnarray}\label{bound-Delta0}
	\Delta_0(T) \mu &\le& \frac{1}{2}\ \theta^{-1} \mu^{-1},
\\
	\label{bound-Delta1}
	\Delta_1(T) \mu &<& k_1 \theta^{-1} \hspace{1em}
	\text {with $k_1 \simeq \secrev{6.988,684\dots}$ and}
\\
\label{bound-Delta2}
	\Delta_2(T) \mu 
	\le \tilde \Delta_2(T) \mu
	&<& k_2 \theta^{-1}\hspace{1em}
	\text {with $k_2 \simeq \secrev{2.901,827\dots}$}
\end{eqnarray}
\end{lemma}

\begin{proof} {\secrev Before starting the proof, we assume without loss of generality that
	$\mathbf 0 \in \mathcal A_i = \conv{A_i}$. (See item (b) of the proof of  Lemma~\ref{lemma-condition} in Section~\ref{sec-renormalization}).
	This proof is divided into 6 parts.}
\medskip

\noindent 
	\secrev{{\bf Part 1:} For all $\tau \ge T$,
\begin{eqnarray*}
d_{\mathbb P}(
\mathbf q_{\tau} \cdot R(\boldsymbol \zeta),
\mathbf f\cdot  R(\boldsymbol \zeta) ) 
	&\le&
	\sqrt{ \sum_{i=1}^n 
	\frac{ \|(\frac{1}{\tau}(\mathbf q_i)_{\tau} - \mathbf f_i)\cdot  R_i(\boldsymbol \zeta)\|^2 }{ \|\mathbf f_i\cdot  R_i(\boldsymbol \zeta)\|^2}
}
\\
	&=&
\frac{1}{\tau}
\sqrt{ \sum_{i=1}^n 
\frac{ \|\mathbf g_i \cdot R_i(\boldsymbol \zeta)\|^2 }{ \|\mathbf f_i\cdot  R_i(\boldsymbol \zeta)\|^2}
}.
\end{eqnarray*}
	Recall from 
	Equation \eqref{def-elli} the notation
	$\ell_i(\boldsymbol \zeta) \defeq
	\max_{\mathbf a \in A_i} \mathbf a \Re(\boldsymbol \zeta)$.
	Theorem \ref{th-renormalization}(c) yields
	\[
		 \|\mathbf g_i \cdot R_i(\boldsymbol \zeta)\|
		 \le
		 e^{\ell_i(\boldsymbol \zeta)} \| \mathbf g_i\|
	.
\]
	Using the bound $\|\cdot\|_{\infty} \le \| \cdot\|_2$,
\[
\frac{ \|\mathbf g_i \cdot R_i(\boldsymbol \zeta)\| }{ \|\mathbf f_i\cdot  R_i(\boldsymbol \zeta)\|}
\le
	\frac{\|\mathbf g_i\|
	e^{\ell_i(\boldsymbol \zeta)}
	}
	{ \| \mathbf f_{i} \cdot 
	R_i(\boldsymbol \zeta)
	\|_{\infty}}
	\le
	\frac{ \|\mathbf g_i \| }
	{\|\mathbf f_i 
	\|}
	\frac{\|\mathbf f_i\|}
	{ \max_{\mathbf a \in A_i} |f_{i\mathbf a} e^{\mathbf a \boldsymbol \zeta -\ell_i(\boldsymbol \zeta)}|}
.
\]
	By construction, 
	$|e^{\mathbf a \boldsymbol \zeta -\ell_i(\boldsymbol \zeta)}| \le 1$ with the bound attained for at least one value 
	of $\mathbf a$, say $\mathbf a^* \in A_i$. 
	In particular,  
	\[
		\max_{\mathbf a \in A_i} |f_{i\mathbf a} e^{\mathbf a \boldsymbol \zeta -\ell_i(\boldsymbol \zeta)}|
\ge
		|f_{i\mathbf a^*} e^{\mathbf a^* \boldsymbol \zeta -\ell_i(\boldsymbol \zeta)}|
	=  |f_{i\mathbf a^*}|
.	\]
Recall that $\kappa_{\mathbf f} = 
	\max_{i, \mathbf a} \frac{\| \mathbf f_i \|}
	{|\mathbf f_{i\mathbf a}|}$, we may bound 
	\[
		\frac{ \|\mathbf g_i \cdot R_i(\boldsymbol \zeta)\| }{ \|\mathbf f_i\cdot  R_i(\boldsymbol \zeta)\|}
\le 
	\frac{ \|\mathbf g_i \| }
	{\|\mathbf f_i\|}
	\frac{\|\mathbf f_i\|}
		{|\mathbf f_{i\mathbf a^*}|}
\le \kappa_{\mathbf f}
	\frac{ \|\mathbf g_i \| }
	{\|\mathbf f_i\|}
.
\]
	Using that $\|\mathbf g_i\| \le K \|\mathbf f_i\|$,
	\[
d_{\mathbb P}(
\mathbf q_{\tau} \cdot R(\boldsymbol \zeta),
\mathbf f\cdot  R(\boldsymbol \zeta) ) 
\le
	\frac{\kappa_{\mathbf f} K \sqrt{n}}{\tau}
.
\]
Since $\tau \ge T$, \eqref{bound-T} implies:
\[
\Delta_0(T) \le 
\frac{1}{\theta \sqrt{S} \mu^2 \nu}
\]
The estimate \eqref{bound-Delta0} follows from 
the bounds: $\secrev{\nu} \ge 1$ from Lemma~\ref{lem-nu-kappa}, 
and the 
hypotheses $n \ge 2$ and $S_i\ge 2$.
}
\medskip

\noindent 
\secrev{{\bf Part 2:}}
	We bound now $\Delta_1(T)$ \secrev{in terms of $\Delta_2(T)$}.
Suppose that the maximum in its definition is attained
for $\tau = t \ge T$.
\begin{eqnarray*}
	\Delta_1(T) &=& d_{\mathbb P}(\mathbf q_{t} \cdot R(\mathbf z_{t}),
	\mathbf f\cdot  R(\boldsymbol \zeta) ) \\
	&\le&
	d_{\mathbb P}(
	\mathbf q_{t} \cdot R(\mathbf z_{t}),
	\mathbf q_{t} \cdot R(\boldsymbol \zeta))
	+
	d_{\mathbb P}(
	\mathbf q_{t} \cdot R(\boldsymbol \zeta),
	\mathbf f\cdot  R(\boldsymbol \zeta) ) 
	\\
	&=&
	d_{\mathbb P}(
	\mathbf q_{t} \cdot R(\mathbf z_{t}),
	\mathbf q_{t} \cdot R(\boldsymbol \zeta))
	+
	\Delta_0(t)
	.
\end{eqnarray*}
Setting $\mathbf h = \mathbf q_{t} \cdot R(\boldsymbol \zeta)$, 
Lemma~\ref{lem:fRdist}
applied to the first term yields
\[
	d_{\mathbb P}(
\mathbf h \cdot R(\mathbf z_t - \boldsymbol \zeta),
\mathbf h) \le \sqrt{5} \| \mathbf z_t - \boldsymbol \zeta\|_{\mathbf 0} \secrev{\nu}
\le \sqrt{5} \Delta_2(\secrev{T}).
\]
Using \eqref{bound-Delta0},

\begin{equation}\label{Delta1}
\Delta_1(T) \le
\sqrt{5} \Delta_2(T)
+
	\frac{1}{2 \theta \mu}
.
\end{equation}
\smallskip
\medskip

\noindent 
\secrev{{\bf Part 3:}}
\secrev{In order to obtain \eqref{bound-Delta1} and \eqref{bound-Delta2},
we need a bound on $\Delta_2(T)$. Since the distance between two points is
less or equal than the path length joining those points, we always have
$\Delta_2(T) \le \tilde \Delta_2(T)$. Then we can use
Lemma \ref{cond-length-triv-bounds}(a) and (b) to bound the integral
\begin{eqnarray*}
\tilde \Delta_2(T) 
&=&
	\secrev{\nu} \int_{T}^{\infty} \| \dot{\mathbf z}_{\tau}\|_{\mathbf 0} \ \dd \tau 
\\
&\le& 
	\nu \sqrt{n} \int_{T}^{\infty} 
	\|M(\mathbf q_{\tau}, \mathbf z_{\tau})^{-1}\|_{\mathbf 0}
	\min( \|\mathbf f\|, 2 \|\mathbf g\|\tau^{-1}) 	
	\ \dd \tau .
\end{eqnarray*}
We assumed that $\| \mathbf f_i\| = \sqrt{S_i}$
and $\| \mathbf g_i\| \le K \sqrt{S_i}$. 
Hypotheses $n \ge 2$, $S_i \ge 2$
in bound \eqref{bound-T}  
imply that $T \ge 2\theta K$ and in particular,
$\tau \ge T \ge 2 K$ so the minimum inside the integral is precisely  
$2 \| \mathbf g\| \tau^{-1} \le 2 K \sqrt{S} \tau^{-1}$. Thus, 
\begin{equation}\label{Delta2Part1}
\tilde \Delta_2(T) 
	\le 
	2 K 
	\secrev{\nu} \sqrt{n} \sqrt{S} 
	\int_{T}^{\infty} 
	\|M(\mathbf q_{\tau}, \mathbf z_{\tau})^{-1}\|_{\mathbf 0}
	\tau ^{-1}	
	\ \dd \tau 
.
\end{equation}
}
\medskip

\noindent
\secrev{{\bf Part 4:}}
\secrev{
We claim that 
\begin{equation}\label{Delta2Part2}
	\|M(\mathbf q_{\tau}, \mathbf z_{\tau})^{-1}\|_{\mathbf 0}
\le
\frac{\mu}{1-\mu \Delta_1(\tau)}
	\frac{\kappa_{\mathbf f} }{\tau}\frac{\thirdrev{2}}{2\thirdrev{\sqrt{2}}-\theta^{-1}}
.
\end{equation}
Indeed, define $\mathbf p_i = \mathbf q_{i\tau}R_i(\mathbf z_{\tau}) e^{-\ell_i(\mathbf z_{\tau})}$.
Since
\[
\mathrm{diag}\left( 
\|V_{A_i}(\mathbf z_{\tau})\|\right)
	M(\mathbf q_{\tau}, \mathbf z_{\tau}) =
\mathrm{diag}\left( 
\|V_{A_i}(\mathbf 0)\|
e^{\ell_i(\mathbf z_{\tau})}
\right)
M(\mathbf p, \mathbf 0),
\]
we have
\[
	M(\mathbf q_{\tau}, \mathbf z_{\tau})^{-1} =
M(\mathbf p, \mathbf 0)^{-1} 
\mathrm{diag}\left( 
\frac
{\|V_{A_i}(\mathbf z_{\tau})\| }
{\|V_{A_i}(\mathbf 0)\|
e^{\ell_i(\mathbf z_{\tau})}
}\right),
\]
and Lemma~\ref{lem-V-coarse} 
implies:
\[
	\|M(\mathbf q_{\tau}, \mathbf z_{\tau})^{-1}\|_{\mathbf 0}
\le
\|M(\mathbf p,\mathbf 0 )^{-1}\|_{\mathbf 0}
.
\]}

\secrev{Recall that $\mu(\mathbf p) = \left\| M(\mathbf p, \mathbf 0)^{-1} \diag{ \| \mathbf p_i\|}\right\|_{\mathbf 0}$ so that $\| M(\mathbf p, \mathbf 0)^{-1}\|_{\mathbf 0} \le
\frac{\mu(\mathbf p)}{\min \|\mathbf p_i\|}$. }
Triangular inequality yields 
\secrev{
\[
	\| \mathbf p_i \| =
	\|\mathbf q_{i\tau} \cdot R_i( \mathbf z_{\tau}) e^{-\ell_i(\mathbf z_{\tau})} \| \ge
\tau \|\mathbf f_i  \cdot R_i( \mathbf z_{\tau})e^{-\ell_i(\mathbf z_{\tau})}\| - \| \mathbf g_i \cdot R_i( \mathbf z_{\tau})e^{-\ell_i(\mathbf z_{\tau})}\|
.
\]
As in Part 1 of this proof, there is $\mathbf a^* \in A_i$ 
with $e^{\mathbf a^* \mathbf  z_{\tau} -\ell_i(\mathbf z_{\tau})}=1$. Because $\mathbf 0 \in \mathcal A_i$, $\ell_i(\mathbf z_{\tau}) \ge 0$.
Since $\|\mathbf f_i\|=\sqrt{S_i}$, $\|\mathbf g_i\| \le K \sqrt{S_i}$
and because of Theorem~\ref{th-renormalization}(c) again,
\begin{equation}\label{lower-bound-p}
\| \mathbf p_i \| 
	\ge \tau |\mathbf f_{i \mathbf a^*}| - K \sqrt{S_i} e^{-\ell_i(\mathbf z_{\tau})}
\ge (\tau \kappa_{\mathbf f}^{-1} -K) \sqrt{S_i}
.
\end{equation}
}
\secrev{Since $\tau \ge T$, Equation \eqref{bound-T} implies that
\begin{eqnarray*}
	\min (\|\mathbf p_{i}\|) &\ge& 
	\tau \kappa_{\mathbf f}^{-1}
	\min(\sqrt{S_i})
\left( 1 - \frac{K}{\tau \kappa_{\mathbf f}^{-1}}\right) 
\\
	&\ge&
	\tau \kappa_{\mathbf f}^{-1}
	\min(\sqrt{S_i})
	\left( 1 - \frac{1}{\theta \sqrt{n} \sqrt{S} \mu^2 \secrev{\nu}
 }\right) .
\end{eqnarray*}
In particular, using $n \ge 2$ and $S_i \ge 2$, 
\[
	\min (\|\mathbf p_{i}\|) \ge 
\tau \kappa_{\mathbf f}^{-1} \sqrt{2} \left(1-\frac{1}{2 \thirdrev{\sqrt{2}}\theta}\right)
.
\]
At this point, we established that
\[
	\|M(\mathbf q_{\tau}, \mathbf z_{\tau})^{-1}\|_{\mathbf 0}
\le
\mu(\mathbf p) 
\frac{\kappa_{\mathbf f} }{\tau}
{\thirdrev{\frac{2}{2\sqrt{2}-\theta^{-1}}}}
\]
The condition number is invariant by multiprojective scaling, 
so $\mu(\mathbf p) = \mu(\mathbf q_{\tau} R(\mathbf z_{\tau}))$.
Proposition \ref{prop-mu}(\ref{prop-mu-b}) yields:
\[
\mu(\mathbf q_{\tau} R(\mathbf z_{\tau})) 
\le
\frac{\mu}{1-\mu \Delta_1(\tau)},
\]
and Equation~\ref{Delta2Part2} follows.
}
\medskip

\noindent
\secrev{{\bf Part 5:}} 
We claim that $\mu \Delta_1(T) \le \frac{1}{2}$. Suppose by contradiction
that  $\mu \Delta_1(T) > \frac{1}{2}$, then we can increase $T$ such that
$\mu \Delta_1(T) = \frac{1}{2}$.
\secrev{Using the bound $\Delta_2(T) \le \tilde \Delta_2(T)$,
Equation \eqref{Delta1}  reads}
\begin{equation} \label{half-le-bound}
	\secrev{\frac{1}{2}}
	\le \sqrt{5} \mu \tilde \Delta_2(T) + \frac{1}{2\theta}.
\end{equation}
\secrev{Equations \eqref{Delta2Part1} and \eqref{Delta2Part2}
together provide the estimate
\[
\tilde \Delta_2(T) 
\le
\frac{\mu }{1-\mu \Delta_1(T)} 
	\frac{\thirdrev{4}\kappa_{\mathbf f} K \nu \sqrt{n}\sqrt{S}}{2\thirdrev{\sqrt{2}}-\theta^{-1}}
\int_{T}^{\infty} 
\tau^{-2}
\
\dd \tau .
\]
Integrating and then combining with \eqref{bound-T},
\[
		\begin{split}
		\tilde \Delta_2(T) 
		&\le
	\frac{\mu }{1-\mu \Delta_1(T)} 
		\frac{\thirdrev{4}\kappa_{\mathbf f} K \nu \sqrt{n} \sqrt{S}}{2\thirdrev{\sqrt{2}}-\theta^{-1}}
	T^{-1} 
\\
&\le
	\frac{1}{1-\mu \Delta_1(T)} 
	\frac{\thirdrev{4}}{2\thirdrev{\sqrt{2}}\theta-1}\frac{1}{\mu}
.
	\end{split}
	\]
}
Since $\mu \Delta_1(T) = \frac{1}{2}$,
\secrev{the bound above becomes}
\begin{equation}\label{Delta2}
\mu \tilde \Delta_2(T) \le 
	\thirdrev{\frac{8}{2\sqrt{2}\theta-1}}
.
\end{equation}
Combining ths bound with~\eqref{half-le-bound}, we conclude that
\[
	\frac{1}{2} \le  
	\thirdrev{\frac{8 \sqrt{5}}{2\sqrt{2} \theta-1}} + \frac{1}{2\theta }
.
\]
When $\theta \rightarrow \infty$, we clearly get a contradiction. To find the
smaller $\theta_0$ that guarantees \secrev{the contradiction}, we compute the largest solution of
\thirdrev{
\[
	\theta_0^2 - \left(1 + \frac{\sqrt{2} + 16 \sqrt{10}}{4}\right) \theta_0 + \frac{\sqrt{2}}{4}=0
\]
}

that is $\theta_0 \simeq \thirdrev{13.977.369\cdots<14}$.
This contradiction establishes that $\mu \Delta_1(T) \le \frac{1}{2}$.
\medskip

\noindent
\secrev{{\bf Part 6:} Inequation \eqref{Delta2} is still valid when
$\mu \Delta_1(T) \le \frac{1}{2}$. Combining with
\eqref{Delta1},
we obtain as before}
\[
	\mu \Delta_1(T) \le  
	\thirdrev{\frac{8 \sqrt{5}}{2\sqrt{2} \theta-1}} + \frac{1}{2\theta }
,
\]
but this bound is inconvenient. \secrev{
Assuming $\theta \ge \theta_0$, 
	we write instead
\[
	\mu \Delta_1(T) \le  
\theta^{-1} \left( 
\frac{8\sqrt{5}}{2 \sqrt{2} - \theta^{-1}}
+ \frac{1}{2}
\right)
\le
\theta^{-1} \left( 
\frac{8\sqrt{5}}{2 \sqrt{2} - \theta_0^{-1}}
+ \frac{1}{2}
\right) = k_1 \theta^{-1}
\]
for some constant $k_1$. The approximation 
$k_1 \simeq 6.988,684\cdots$. The  was obtained numerically.
This proves the Inequation \eqref{bound-Delta1}. 

Similarly, we can derive \eqref{bound-Delta2} from the bound \eqref{Delta2}:
\[
\mu \tilde \Delta_2(T) \le 
	\frac{8}{2\sqrt{2}\theta-1}
	=
\theta^{-1}
	\frac{8}{2\sqrt{2}-\theta^{-1}}
	\le
\theta^{-1}
	\frac{8}{2\sqrt{2}-\theta_0^{-1}}
=
k_2 \theta^{-1}
\]
for a constant $k_2 \simeq 2,901.827\dots$.
}
\end{proof}

\subsection{Expectation of the condition length (part 3)}
\begin{proposition}\label{last-part}
	Assume that the hypotheses (\ref{mainD-supports}) to (\ref{mainD-fixed}) of Theorem~\ref{mainD} hold, and \secrev{let $K$ be as defined in \eqref{mainD-K}}.
	Let $\mathbf q_t = \mathbf g + t \mathbf f$ \secrev{for $0 \le t \le \infty$},
	where $\mathbf g \in \mathscr F$ satisfies $\| \mathbf g_i \| \le K\sqrt{S_i}$.
        To this path
	associate the set $\mathscr Z(\mathbf q_t)$ of continuous solutions \secrev{$\mathbf z_t$} of $\mathbf q_t \cdot \mathbf V(\mathbf z_t) \equiv 0$. 
	Suppose that $T \ge \theta \kappa_{\mathbf f} K \sqrt{n} \sqrt{S} \mu_{\mathbf f}^2 \secrev{\nu}$
	with $\theta \ge \theta_0 \simeq  \thirdrev{13.977,369}$ and $\mu_{\mathbf f} = \max_{\mathbf z \in Z(\mathbf f)} 
	\secrev{\mu(\mathbf f \cdot R(\mathbf z))}$. 
Then unconditionally,
\[
	\sum_{\mathbf z_{\tau} \in \mathscr Z(\mathbf q_{\tau})} \mathscr L (\mathbf q_t, \mathbf z_t; T, \infty)
	\le 2 Q k_3
.
\]
for $k_3 \le \thirdrev{0.903,836\dots} $.
\end{proposition}

The proof of Proposition~\ref{last-part} requires the Lemma below.

\begin{lemma}\label{speed-decay}
	Assume that $t \ge T$. Under the hypotheses of Proposition~\ref{last-part},
\[	
\left\| \dot {\mathbf q}_t \cdot R(\mathbf z_t) \right\|_{\mathbf q_t \cdot R(\mathbf z_t) }
	\changed{\le}
	\frac{1}{t^2}
	\secrev{
		\frac{\kappa_{\mathbf f}K\changed{\sqrt{n}}}{1 - T^{-1} \kappa_{\mathbf f}K }
	}
	.
\]
\end{lemma}

\begin{proof}[Proof of Lemma~\ref{speed-decay}]
\begin{eqnarray*}
\left\| \dot {\mathbf q}_t \cdot R(\mathbf z_t) \right\|_{\mathbf q_t \cdot R(\mathbf z_t) }
&=&
\left\| 
	\changed{\mathbf f} \cdot R(\mathbf z_t) 
\right\|_{\mathbf q_t \cdot R(\mathbf z_t) }
\\
&=&
	\frac{1}{t} \left\| {\mathbf q}_t \cdot R(\mathbf z_t) -\changed{\mathbf g} \cdot R(\mathbf z_t) 
\right\|_{\mathbf q_t \cdot R(\mathbf z_t) }
\\
&\le&
	\frac{1}{t} \left\| {\mathbf q}_t \cdot R(\mathbf z_t) 
\right\|_{\mathbf q_t \cdot R(\mathbf z_t) }
	+
	\frac{1}{t} \left\| 
	{\mathbf g}\cdot R(\mathbf z_t) 
\right\|_{\mathbf q_t \cdot R(\mathbf z_t) }.
\end{eqnarray*}
The first term vanishes.
	\secrev{The second term is bounded as in Part 4 of the proof of
	Lemma~\ref{lem-delta}. For each value of $i$,
\[
\frac{1}{t} \left\| 
{\mathbf g_i} \cdot R(\mathbf z_t) 
	\right\|_{(\mathbf q_i)_t \cdot R(\mathbf z_t) }
\le
	\frac{1}{t} \frac{\left\| {\mathbf g_i} \cdot R(\mathbf z_t) \right\|}
	{\|(\mathbf q_i)_t \cdot R(\mathbf z_t)\| }
\le
	\frac{1}{t} \frac{\left\| 
	\mathbf g_i 
	\right\|
	e^{\ell_i(\mathbf z_t)}}
	{\|(\mathbf q_i)_t \cdot R(\mathbf z_t)\| }
\]
	using Theorem~\ref{th-renormalization}(c). 
	We set $\mathbf p_i = (\mathbf q_i)_t \cdot R(\mathbf z_t) e^{-\ell_i(\mathbf z_t)}$ and recover 
	\[
		\| \mathbf p_i \| 
\ge (t \kappa_{\mathbf f}^{-1} -K) \sqrt{S_i}
	\]
	as in \eqref{lower-bound-p}. Hence,}
\[
\frac{1}{t} \left\| 
{\mathbf g_i} \cdot R(\mathbf z_t) 
	\right\|_{(\mathbf q_i)_t \cdot R(\mathbf z_t) }
\le
	\frac{1}{t} \frac{K}{t \changed{\kappa_{\mathbf f}^{-1}} -K} .
\]
It follows that
\[
\left\| \dot {\mathbf q}_t \cdot R(\mathbf z_t) \right\|_{\mathbf q_t \cdot R(\mathbf z_t) }
\le
\frac{1}{t^2}
	\secrev{
	\frac{\kappa_{\mathbf f}K\changed{\sqrt{n}}}{1 - T^{-1} \kappa_{\mathbf f}K }
}
\]
\end{proof}

\begin{proof}[Proof of Proposition~\ref{last-part}]
\secrev{From Proposition~\ref{prop-mu}(b),
	\begin{eqnarray*}
\mathscr L_1((\mathbf q_{\tau}, \mathbf z_{\tau}); T, \infty)
&=& \int_T^{\infty} 
\left\| \dot {\mathbf q}_t \cdot R(\mathbf z_t) \right\|_{\mathbf q_t \cdot R(\mathbf z_t) }
		\secrev{\mu( \mathbf q_t \cdot R(\mathbf z_t))} \secrev{\nu} \dd t
\\
&\le& 
	\frac{\mu_{\mathbf f} \secrev{\nu}}{1-\mu_{\mathbf f} \Delta_1(T)}
\int_T^{\infty} 
\left\| \dot {\mathbf q}_t \cdot R(\mathbf z_t) \right\|_{\mathbf q_t \cdot R(\mathbf z_t) }
\dd t .
\end{eqnarray*}
Lemma~\ref{speed-decay} above yields
\begin{eqnarray*}
\mathscr L_1((\mathbf q_{\tau}, \mathbf z_{\tau}); T, \infty)
	&\le&
\frac{\mu_{\mathbf f} \secrev{\nu}}{1-\mu_{\mathbf f} \Delta_1(T)}
\frac{\secrev{\kappa_{\mathbf f}} K \sqrt{n}}
	{1-T^{-1} \kappa_{\mathbf f} K}
\int_T^{\infty} 
	t^{-2} \dd t .
\\
	&\le&
T^{-1}
\frac{\mu_{\mathbf f} \secrev{\nu}}{1-\mu_{\mathbf f} \Delta_1(T)}
\frac{\secrev{\kappa_{\mathbf f}} K \sqrt{n}}{1-\frac{1}{2 \sqrt{2} \theta}}
.
\end{eqnarray*}
Lemma~\ref{lem-delta} with $\mu=\mu_{\mathbf f}$ provides 
	the bound for $\mu_{\mathbf f} \Delta_1(T) \le \theta^{-1} k_1$.
	Replacing \eqref{bound-T} with the usual bounds
$n \ge 2$, $S_i \ge 2$, $\mu_{\mathbf f} \ge 1$,
	\[
\mathscr L_1((\mathbf q_{\tau}, \mathbf z_{\tau}); T, \infty)
	\le
		\frac{1}{\theta - k_1} \frac{1}{2 - \frac{\sqrt{2}}{2\theta_0}}
	.
\]}

Similarly,
\begin{eqnarray*}
\mathscr L_2((\mathbf q_{\tau}, \mathbf z_{\tau}); T, \infty)
&=& 
2 \secrev{\nu} 
\int_T^{\infty} 
	\left\| \dot {\mathbf z}_t \right\|_{\mathbf z_t}
	\secrev{\mu( \mathbf q_t \cdot R(\mathbf z_t))} \dd t
\\
&\le&
	2 \frac{\mu_{\mathbf f} \tilde \Delta_2(T)}{1-\mu_{\mathbf f} \Delta_1(T)} 
\\
&\le&
	2 \frac{k_2 }{\theta -k_1 } .
\end{eqnarray*}
	It follows from Lemma~\ref{cond-length-split} that for every solution path
	$\mathbf z_t \in \mathcal Z(\mathbf q_t)$,
	\begin{eqnarray*}
\mathscr L((\mathbf q_{\tau}, \mathbf z_{\tau}); T, \infty)
		&\le&
\mathscr L_1((\mathbf q_{\tau}, \mathbf z_{\tau}); T, \infty)+
\mathscr L_2((\mathbf q_{\tau}, \mathbf z_{\tau}); T, \infty)
\\
&\le&
	\secrev{
		\frac{1}{\theta - k_1} \frac{1}{2 - \frac{\sqrt{2}}{2\theta_0}}
	}
+
2 \frac{k_2 }{\theta -k_1 } 
	\end{eqnarray*}
We set $k_3= \secrev{
		\frac{1}{\theta_0 - k_1} \frac{1}{2 - \frac{\sqrt{2}}{2\theta_0}}
	} 
+
2 \frac{k_2 }{\theta_0 -k_1 } \simeq 
\thirdrev{0.903,836\dots} $.
From Remark~\ref{num-paths},
the total number of paths
is at most $2\ Q$.
\changed{Proposition \ref{last-part}} follows.
\end{proof}

\begin{proof}[Proof of Theorem~\ref{mainD}]
	We will combine Theorem~\ref{condlength} with Proposition~\ref{last-part}.
	Fix $T=\theta_0 \kappa_{\mathbf f} K \secrev{\sqrt{n}} \sqrt{S} \mu_{\mathbf f}^2 \secrev{\nu}$.
	\changed{We know from Section~\ref{overview-linear} that} with probability at least $7/8$, the random system $\mathbf g$ does not belong to the
	exclusion set 
	$\Lambda_{\epsilon} \cup \Omega_H \cup Y_K$. \secrev{under that assumption,}
	\[
\sum_{\mathbf z_{\tau} \in \mathscr Z(\mathbf q_{\tau})} \mathscr L (\mathbf q_t, \mathbf z_t; 0, \infty)
	=
\sum_{\mathbf z_{\tau} \in \mathscr Z(\mathbf q_{\tau})} \mathscr L (\mathbf q_t, \mathbf z_t; 0, T)
+
\sum_{\mathbf z_{\tau} \in \mathscr Z(\mathbf q_{\tau})} \mathscr L (\mathbf q_t, \mathbf z_t; T, \infty)
.
\]
	\changed{Theorem~\ref{condlength} guarantees under the same assumption that} with probability at least $6/7$,
\[
\sum_{\mathbf z_{\tau} \in \mathscr Z(\mathbf q_{\tau})} \mathscr L (\mathbf q_t, \mathbf z_t; 0, T)
	\le C_0
	Q
	\changed{n^2 \sqrt{S}}\max_i (S_i) 
	\ \thirdrev{\max(K,\sqrt{\kappa_{\mathbf f}})} 
	\ T 
	\ \secrev{\mathrm{LOGS}_T}
\]
	Also, we know from Proposition~\ref{last-part} that 
\[
\sum_{\mathbf z_{\tau} \in \mathscr Z(\mathbf q_{\tau})} \mathscr L (\mathbf q_t, \mathbf z_t; T, \infty)
	\le \thirdrev{2} Q k_3 .
\]
Adding and replacing $T$ by its value, we obtain that
\[
\begin{split}
\sum_{\mathbf z_{\tau} \in \mathscr Z(\mathbf q_{\tau})} \mathscr L (\mathbf q_t, \mathbf z_t; 0, \infty)
	\le&
	(2k_3+C_0 \theta_0) Q	\changed{n^{\thirdrev{\frac{5}{2}}} S}\secrev{\max_i (S_i)}
\\
	&\times K \ \thirdrev{\max \left(K,\sqrt{\kappa_{\mathbf f}}\right)}
	\kappa_{\mathbf f} \mu_{\mathbf f}^2 \secrev{\nu}
	\ \secrev{\mathrm{LOGS}_T} .
\end{split}
\]
The conclusion above is valid with probability $\ge 6/7$ for
$\mathbf g$ not in the exclusion set. Hence, it holds for
	$\mathbf g$ unconditionally with probability $\ge \frac{6}{7}\frac{7}{8}=\frac{3}{4}$.

	\secrev{Lemma~\ref{lem-nu-kappa}(c) bounds $\nu \le \max S_i$, so
	$\log(T) \in O(\log(d_r) + \log(S))$ and hence 
\[
	\mathrm{LOGS}_T 
	\in O(\mathrm{LOGS}_{\mathbf f}) 
.
\]
	The constant $C$ is the product of $2k_3+C_0 \theta_0$ times the hidden constants in $\mathrm{LOGS}_{T}$ and $\mathrm{LOGS}_{\mathbf f}$ above.}

\end{proof}

\section{Conclusions and further research}

A theory of homotopy algorithms over toric varieties is now within reach. In this paper, \secrev{a particular {\em renormalization} was introduced. It} allowed to obtain complexity bounds for homotopy between two fixed systems, as long as they satisfy some conditions: they should be well-posed, and have no root at infinity. New invariants that play an important role in the theory were identified: the mixed \changed{area}, and the facet gap $\eta$. The cost of a `cheater's homotopy' between two fixed, non-degenerate systems with same support was bounded here. 

Theorem~\ref{cond-num-infty} paves the way for rigorously detecting roots at infinity, and furthermore finding out {\em which} toric infinity the root may be converging to. Then one can think of replacing the original system with the appropriate overdetermined system at infinity, and attempt to solve it. 
There are some technical difficulties to certify the global solution set with roots at toric infinity, that also deserve some investigation. \secrev{More general techniques to deal with roots in the neighborhood of toric infinity will be introduced in a future paper.}

Degenerate roots are more challenging. The hypothesis \secrev{$r(\mathbf f) \ne 0$} in Theorems~\ref{mainD} 
and~\ref{mainE} already implies that $\mathbf f \not \in \Sigma^{\infty}$ and hence, by Bernstein's second
theorem (Theorem~\ref{BKK3} here) the number of finite roots is $n!V/\det(\Lambda)$ and the roots are isolated.
If we do not assume  $r(\mathbf f) \ne 0$, then more general singular solutions may arise.
There are numerical methods to deal with this
situation, see for instance \ocites{Li-Sauer-Yorke,
Sommese-Verschelde,
Dayton-Li-Zeng, 
Giusti-Yakoubsohn, Li-Sang, Hauenstein-Mourrain-Szanto} and references.  
\medskip
\par

Finding a convenient starting system is usually one of the big challenges for homotopy algorithms. \secrev{\aboutsampling{As we saw in Section~\ref{sampling}, this}{This} problem is comparable to the problems of solving or sampling random sparse systems. Several
viable options are suggested by numerical practice.} One of them is the use of {\em polyhedral homotopy}, also known as {\em nonlinear homotopy} 
\cites{ 
Verschelde-Verlinden-Cools, 
Huber-Sturmfels, 
Li-Polyhedral, 
Verschelde-795}. 
It `reduces' a generic system to a tropical polynomial system. Several approaches are available for solving tropical polynomial systems. A complexity bound in terms of mixed volumes
and quermassintegralen for solving generic tropical systems was given by \cite{Malajovich-mixed}. 
A procedure to solve arbitrary tropical systems with roughly the same complexity bound was given independently by~\ocites{Jensen, Jensen2}.
The results of those papers disprove the belief by practitioners that
\begin{quotation}
	{\em In general, finding the exact maximal root count for given sparse structure and setting up a
	compatible homotopy is a combinatorial process with a high computational complexity}
\cite{BERTINI}*{p.71} 
\end{quotation}
and provide usable implementations for finding the starting systems.

The situation is different for polyhedral homotopy continuation itself. While the same numerical evidence, as together as this author's experience show that this is a highly effective numerical method, theoretical justifications are missing. It is important here to point out our findings in Theorem ~\ref{old-mainB}: the variance of the coefficients appears in the average bound for the condition. \changed{This variance tends to zero at the starting point of the polyhedral homotopy when renormalized. In this sense Theorem~\ref{old-mainB} is an obstruction to such complexity bounds.}
No complexity bound for polyhedral homotopy is known at this time.

Polyhedral homotopy is not the only possible algorithm for solving \secrev{random} sparse systems. One can also experiment with monodromy as in \cites{Krone-Leykin,
Leykin-Rodriguez-Sottile,DHJLLS,BRSY}. Finding a point in the solution variety is easy, just project a random system into the subspace vanishing at a fixed point. Then the other roots can be found by homotopy continuation through several random loops. No complexity analysis for this procedure is known either.

Finally, there is the situation where \changed{a space $\mathscr F$} is
given and one is asked to solve many systems in $\mathscr F$. 
In this case the cost of finding one generic 
`cheater' system in $\mathscr F$ is irrelevant. The `cheater' system
can be obtained by total degree homotopy as in~\cites{Li-Sauer-Yorke, BST3264}.
\medskip

Experimental validation of the results in this paper is still to be done. Theorem~\ref{mainD} uses a conditional probability estimate. By performing experiments with this conditional probability or with adversarial probability distributions, one can determine if the domains in the proof of the Theorem are really necessary or if they are a side-product of the proof technique.

The complexity bound in \changed{Theorem~\ref{mainE}}
should be compared to the problem size. Since we are considering the problem of finding all the roots \secrev{and the number of roots is exponential in the input size, a reasonable definition for {\em problem size} is the input size plus the output size. The running time depends also on the condition numbers
$\mu_{\mathbf f}$ and $\kappa_{\mathbf f}$. It is also possible
to interpret the geometric invariant $Q$ from \eqref{mainD-Q} as} 
an abstract condition number,
see for instance 
\cites{Cucker-condition, Malajovich-Shub} for the rationale of introducing such an object.
\secrev{It would be interesting to know 
if there are natural examples of families of polynomial systems,
arising from practice, with $Q$ not polynomial on the scaled mixed volume $n!V$, or with
an isoperimetric ratio $V'/V$ tending to infinity.}

Last but not least, a large number of implementation issues remain unsettled. Several choices in this paper were done to simplify the theory, but do not seem reasonable in practice. For instance, it would be reasonable to replace the trial and error procedure of Theorem~\ref{mainE} by early detection that the Gaussian system is outside of the domain of the conditional probability. 

Complexity analysis in this paper is done in terms of total cost. But each path can be followed independently of the others, so the algorithm is massively parallelizable. In those situations, the computational bottleneck is the communication between processes. If it is possible to detect failure early from data at the path, one can avoid communication almost completely. It would be desirable in this case to estimate the expected parallel running time.

A more fundamental question is the following: most implementation of homotopy algorithms 
use a predictor-corrector scheme, as explained for instance by ~\ocite{Allgower-Georg}.  
Up to now, the tightest rigorous complexity bounds for homotopy algorithms 
refer to a corrector-only homotopy, which no one actually uses in practice. Is it possible to improve the
complexity bound of Theorem~\ref{th-A} by more than a constant by using a higher order method? What about the bound in Theorem~\ref{mainD}?

\renewcommand{\MR}[1]{}
\newcommand{\readablebib}[3]{\bib{#1}{#2}{#3}\smallskip}

\end{document}